\documentclass{osuthesis}

\usepackage{amsfonts}
\usepackage{amsmath,amscd,amssymb,amsthm}
\usepackage{mathrsfs}
\usepackage{pstricks}
\usepackage{pst-plot}
\usepackage[cmtip,arrow]{xy}
\usepackage{pb-diagram,pb-xy}
\usepackage{eucal}
\usepackage[left=1.1875in,top=1in,text={6.5in,9in}]{geometry}
\usepackage{syntonly}
\usepackage{url}
\usepackage{verbatim}

\newtheorem{theorem}{Theorem \rm}
\theoremstyle{plain}
\newtheorem{lemma}[theorem]{Lemma}
\newtheorem{prop}[theorem]{Proposition}
\newtheorem{cor}[theorem]{Corollary}
\newtheorem{conj}[theorem]{Conjecture}
\theoremstyle{definition}
\newtheorem{definition}[theorem]{Definition}
\theoremstyle{remark}

\newtheorem{rmk}[theorem]{Remark}

\newcommand{\Z}{\mathbb{Z}}
\newcommand{\ltimescirc}{\textrm{\textcircled{$\ltimes$}}}
\newcommand{\ds}{\displaystyle}
\DeclareMathOperator*{\colim}{colim}
\DeclareMathOperator*{\coeq}{coequalizer}

\title{ON THE SYMMETRIC HOMOLOGY OF ALGEBRAS}
\author{Shaun Van Ault}
\presentdegrees{B.A., B.Mus.}
\unit{Graduate Program in Mathematics}
\gradyear{2008}
\advisor{Professor Zbigniew Fiedorowicz}
\member{Professor Dan Burghelea}
\member{Professor Roy Joshua}

\begin{document}

\frontmatter

\maketitle

\disscopyright

\begin{abstract}
\noindent
  The theory of symmetric homology, in which the symmetric groups
  $\Sigma_k^\mathrm{op}$, for $k \geq 0$, play the role that the cyclic
  groups do in cyclic homology, begins with the definition of the category
  $\Delta S$, containing the simplicial category $\Delta$ as subcategory.
  Symmetric homology of a unital algebra, $A$, over a commutative ground ring, $k$,
  is defined using derived functors and the symmetric bar construction of
  Fiedorowicz.  If $A = k[G]$ is a group ring, then $HS_*(k[G])$ is related to stable
  homotopy theory.
  Two chain complexes that compute $HS_*(A)$ are constructed, both making
  use of a symmetric monoidal category $\Delta S_+$ containing $\Delta S$, which
  also permits homology operations to be defined on $HS_*(A)$.  
  Two spectral sequences are found that aid in computing symmetric homology.
  In the second spectral sequence, the complex $Sym_*^{(p)}$ is constructed.  This
  complex turns out to be isomorphic to the suspension of the cycle-free chessboard
  complex, $\Omega_{p+1}$, of Vre\'{c}ica and \v{Z}ivaljevi\'{c}.  Recent results
  on the connectivity of $\Omega_n$ imply finite-dimensionality of the symmetric
  homology groups of finite-dimensional algebras.
  Finally, an explicit
  partial resolution is presented, permitting the calculation of $HS_0(A)$ and
  $HS_1(A)$ for a finite-dimensional algebra $A$.
\end{abstract}

\newpage

\dedication{
  {\it To my Mother, Crystal, whose steadfast encouragement spurred me to 
    begin this journey.}
  
  {\it To my Wife, Megan, whose love and support served as fuel for the journey.}
  
  {\it To my son, Joshua Leighton, and my daughter, Holley Anne, who will
   dream their own dreams and embark on their own journeys someday.}
}

\begin{acknowl}
  \medskip
  \noindent

  First, I would like to express my gratitude to my advisor, Zbigniew Fiedorowicz;
  without his hands-on guidance this dissertation would not be possible.
  Zig's extensive knowledge of the field and its literature, his incredible insight 
  into the problems, his extreme patience with me as I took the time to convince
  myself of the validity of his suggestions, and his sacrifice of time and energy
  proofing my work are qualities that I have come to admire greatly.
  
  I would also like to thank Dan Burghelea, who sparked my interest in algebraic
  topology in the summer of 2003, as he turned my interest in mathematical coding
  towards an application involving cyclic homology.  I would be remiss not to 
  acknowledge the support and contributions of many more
  faculty of the Ohio State University:  Roy Joshua, Thomas Kerler, Sergei Chmutov,
  Ian Leary, Mike Davis, Atabey Kaygun,
  and Markus Linckelmann (presently at the University of Aberdeen).
  
  A very special thanks goes out to Birgit Richter (Universit\"at Hamburg) whose
  work in the field has served as inspiration.  Moreover, after she graciously took the 
  time to proof an early draft and suggest many improvements, she was (still!) willing 
  to write a letter of recommendation on my behalf. 
  I am also grateful to Sini\v{s}a Vre\'{c}ica (Belgrade University) and 
  Rade \v{Z}ivaljevi\'{c} 
  (Math. Inst. SANU, Belgrade), whose results about cycle-free chessboard
  complexes contributed substantially to my research; Robert Lewis (Fordham
  University), whose program \verb+Fermat+ and helpful suggestions aided in computations;
  Fernando Muro (Universitat de Barcelona), for his answer to my question about 
  $\pi_2^s(B\Gamma)$
  and pointing out the paper by Brown and Loday; 
  Rainer Vogt (Universit\"at Osnabr\"uck) for pointing out the papers by Kapranov
  and Manin, as well as Dold's work on the universal coefficient theorem;
  and to Peter May and others at The 
  University of Chicago, for helpful comments and suggestions.
  
  In conclusion, I must acknowledge that it takes a village to raise a thesis, and I
  am deeply indebted to all those who have lent a hand along the way.
\end{acknowl}

\begin{vita}
  \begin{datelist}
  
  \item[April 30, 1979]{Born - Bellaire, OH, USA}
  \item[May 27, 2002]{B.A. Mathematics (Computer Science minor)}
  \item[May 27, 2002]{B.Mus. Music Composition}
  \item[June 2002 - May 2007] VIGRE Fellow, The Ohio State University
  \item[June 2007 - August 2008] Graduate Teaching Assistant, OSU

  \end{datelist}
    
  \begin{publist}
    \begin{itemize}
 
    \item 
    S. Ault and Z. Fiedorowicz.  \emph{Symmetric homology of algebras}.  Preprint on
    arXiv, arXiv:0708.1575v4 [math.AT] (11-5-07).
    
    \item
    R. Joshua and S. Ault.  \emph{Extension of Stanley's Algorithm for group imbeddings}.  
    Preprint.  Available at:  \url{http://www.math.ohio-state.edu/~ault/Stan9.pdf}
    (5-25-08).
    \end{itemize}
  \end{publist}

  \begin{fieldsstudy} 
    \majorfield{Mathematics}
    \specialization{Algebraic Topology}
    \begin{studieslist}
      \studyitem{Symmetric Homology}{Dr. Zbigniew Fiedorowicz}
      \studyitem{Computational Methods in Algebraic Geometry}{Dr. Roy Joshua}
      \studyitem{Computational Methods in Cyclic Homology}{Dr. Dan Burghelea}
    \end{studieslist}
  \end{fieldsstudy}

\end{vita}

\tableofcontents

\listoffigures

\listoftables 

\begin{tabular}{ll}
  $\#S$ &= The number of elements in the set $S$.\\

  $N\mathscr{C}$ &= The nerve of the category $\mathscr{C}$ (as a simplicial set).\\

  $B\mathscr{C}$ &= The geometric realization of the category $\mathscr{C}$.  (\textit{i.e.},
                    $|N\mathscr{C}|$.)\\
                    
  $E_*G$ &= The standard resolution of the group $G$ (as a simplicial set).\\
  
  $EG$ &= Contractible space on which $G$ acts.\\
  
  $\mathrm{Mor}\mathscr{C}$ &= The class of all morphisms of the category
                               $\mathscr{C}$.\\
                               
  $\mathrm{Mor}_{\mathscr{C}}(X, Y)$ &= The set of all morphisms in
                               $\mathscr{C}$ from $X$ to $Y$.\\

  $\mathrm{Obj}\mathscr{C}$ &= The class of all objects of the category 
                               $\mathscr{C}$.\\
 
  $S_n$ &= Symmetric group on the letters $\{1, 2, \ldots, n\}$. \\
 
  $\Sigma_n$ &= Symmetric group on the letters $\{0, 1, \ldots, n-1\}$.\\
  
  $\Delta$ &= The simplicial category.\\

  $\textbf{Sets}$ &= The category of sets and set maps.\\
  
  $\textbf{SimpSets}$ &= The category of simplicial sets and simplicial set maps.\\
  
  $\textbf{Mon}$ &= The category of monoids and monoid maps.\\
  
  $k$-\textbf{Mod} &= The category of left $k$-modules, for a ring $k$.\\

  $k$-\textbf{Alg} &= The category of $k$-algebras and algebra homomorphisms.\\

  $k$-\textbf{SimpMod} &= The category of simplicial left $k$-modules and $k$-linear chain maps.\\

  $k$-\textbf{Complexes} &= The category of complexes of $k$-modules and chain maps.\\
  
  $\textbf{Cat}$ &= The category of small categories and functors.\\
  
  $IS_\lambda$ &= The identity representation of a subgroup $S_\lambda$ of $S_n$.\\
  
  $AS_\lambda$ &= The alternating (sign) representation of a subgroup
  $S_\lambda$ of $S_n$.\\
  
  $R \uparrow G$ &= Representation of $G$ induced from a representation of
  a subgroup.\\
  
  $A^{\otimes n}$ &= The $n$-fold tensor product of the algebra $A$ over its ground ring.\\
\end{tabular}
\newpage

\mainmatter

\parindent=0pt

\chapter{PRELIMINARIES AND DEFINITIONS}\label{chap.pre_def}

\section{The Category $\Delta S$}\label{sec.deltas}                 %

Denote by $[n]$ the ordered set $\{0, 1, \ldots, n\}$.  The category $\Delta S$ 
has as objects, the sets $[n]$ for $n \geq 0$, and morphisms are pairs $(\phi,
g)$, where $\phi : [n] \to [m]$ is a non-decreasing map of sets (\textit{i.e.}, a
morphism in $\Delta$), and $g \in \Sigma_{n+1}^{\mathrm{op}}$.  The element $g$ 
represents
an automorphism of $[n]$, and as a set map, takes $i \in [n]$ to $g^{-1}(i)$.
Indeed, a morphism $(\phi, g): [n] \to [m]$ of $\Delta S$ may be represented as a 
diagram:
\[
  \begin{diagram}
    \node{[n]}
    \\
    \node{[n]}
    \arrow{n,l}{g}
    \arrow{e,t}{\phi}
    \node{[m]}
  \end{diagram}
\]
Equivalently, a morphism in $\Delta S$ is a morphism in $\Delta$ together with
a total ordering of the domain $[n]$.
Composition of morphisms is achieved as in~\cite{FL}, namely:
\[
  (\phi, g) \circ (\psi, h) := (\phi \cdot g^*(\psi), \psi^*(g) \cdot h),
\]
where $g^*(\psi)$ is the morphism of $\Delta$ defined by sending the first
$\#\psi^{-1}(g(0))$ points of $[n]$ to $0$, the next $\#\psi^{-1}(g(1))$ points
to $1$, etc.  Note, if $\#\psi^{-1}(g(i)) = 0$, then the point $i \in [m]$ is not 
hit. $\psi^*(g)$ is determined as follows:  For each
$i \in [m]$, $(g^*(\psi))^{-1}(g^{-1}(i))$ has the same number of elements as 
$\psi^{-1}(i)$  If these sets are non-empty, take the order-preserving bijection
$(g^*(\psi))^{-1}(g^{-1}(i)) \to \psi^{-1}(i)$.  

It is often helpful to represent morphisms of $\Delta S$ as diagrams of points
and lines, indicating images of set maps.  Using these diagrams, it is easy to 
see how $g^*(\psi)$ and $\psi^*(g)$ are related to $(\psi, g)$ (see 
Figure~\ref{diag.morphism}).

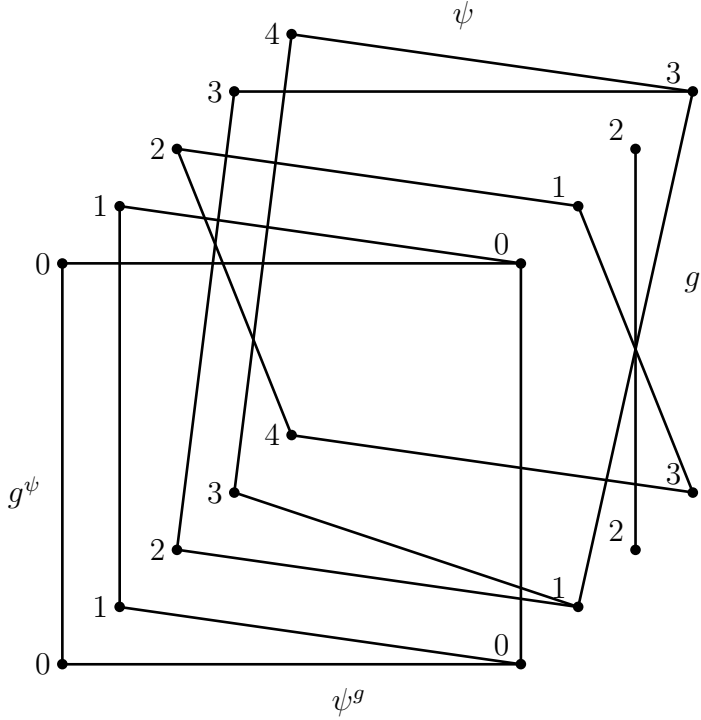
\begin{figure}[ht]
  \psset{unit=1in}
  \begin{pspicture}(4,4)
  
  \psdots[linecolor=black, dotsize=4pt]
  (0.3, 2.4)(0.6, 2.7)(0.9, 3.0)(1.2, 3.3)(1.5, 3.6)
  (2.7, 2.4)(3.0, 2.7)(3.3, 3.0)(3.6, 3.3)
  (2.7, 0.3)(3.0, 0.6)(3.3, 0.9)(3.6, 1.2)
  (0.3, 0.3)(0.6, 0.6)(0.9, 0.9)(1.2, 1.2)(1.5, 1.5)
  
  \psline[linewidth=1pt, linecolor=black](0.3, 2.4)(0.3, 0.3)
  \psline[linewidth=1pt, linecolor=black](0.6, 2.7)(0.6, 0.6)
  \psline[linewidth=1pt, linecolor=black](0.9, 3.0)(1.5, 1.5)
  \psline[linewidth=1pt, linecolor=black](1.2, 3.3)(0.9, 0.9)
  \psline[linewidth=1pt, linecolor=black](1.5, 3.6)(1.2, 1.2)

  \psline[linewidth=1pt, linecolor=black](0.3, 0.3)(2.7, 0.3)
  \psline[linewidth=1pt, linecolor=black](0.6, 0.6)(2.7, 0.3)
  \psline[linewidth=1pt, linecolor=black](0.9, 0.9)(3.0, 0.6)
  \psline[linewidth=1pt, linecolor=black](1.2, 1.2)(3.0, 0.6)
  \psline[linewidth=1pt, linecolor=black](1.5, 1.5)(3.6, 1.2)

  \psline[linewidth=1pt, linecolor=black](0.3, 2.4)(2.7, 2.4)
  \psline[linewidth=1pt, linecolor=black](0.6, 2.7)(2.7, 2.4)
  \psline[linewidth=1pt, linecolor=black](0.9, 3.0)(3.0, 2.7)
  \psline[linewidth=1pt, linecolor=black](1.2, 3.3)(3.6, 3.3)
  \psline[linewidth=1pt, linecolor=black](1.5, 3.6)(3.6, 3.3)

  \psline[linewidth=1pt, linecolor=black](2.7, 2.4)(2.7, 0.3)
  \psline[linewidth=1pt, linecolor=black](3.0, 2.7)(3.6, 1.2)
  \psline[linewidth=1pt, linecolor=black](3.3, 3.0)(3.3, 0.9)
  \psline[linewidth=1pt, linecolor=black](3.6, 3.3)(3.0, 0.6)

  \rput(0.2, 2.4){$0$}
  \rput(0.5, 2.7){$1$}
  \rput(0.8, 3.0){$2$}
  \rput(1.1, 3.3){$3$}
  \rput(1.4, 3.6){$4$}

  \rput(0.2, 0.3){$0$}
  \rput(0.5, 0.6){$1$}
  \rput(0.8, 0.9){$2$}
  \rput(1.1, 1.2){$3$}
  \rput(1.4, 1.5){$4$}
  
  \rput(2.6, 2.5){$0$}
  \rput(2.9, 2.8){$1$}
  \rput(3.2, 3.1){$2$}
  \rput(3.5, 3.4){$3$}
  
  \rput(2.6, 0.4){$0$}
  \rput(2.9, 0.7){$1$}
  \rput(3.2, 1.0){$2$}
  \rput(3.5, 1.3){$3$}
    
  \rput(0.1, 1.2){$g^\psi$}
  \rput(2.4, 3.7){$\psi$}
  \rput(3.6, 2.3){$g$}
  \rput(1.8, 0.1){$\psi^g$}

  \end{pspicture}

  \caption[Morphisms of $\Delta S$]{Morphisms of $\Delta S$: $g \cdot \psi = 
  \big(g^*(\psi), \psi^*(g)\big)  = \big(\psi^g, g^{\psi}\big)$}
  \label{diag.morphism}
\end{figure}

\newpage
\begin{rmk}
  Observe that the properties of $g^*(\phi)$ and $\phi^*(g)$ stated in 
  Prop.~1.6 of~\cite{FL} are 
  formally rather similar to the properties of exponents.  Indeed, if we denote:
  \[
    g^\phi := \phi^*(g), \qquad\qquad \phi^g := g^*(\phi),
  \]
  then Prop. 1.6 becomes:
  \begin{prop}\label{prop.comp}
    For $g, h \in G_n$ and $\phi, \psi \in \mathrm{Mor}{\Delta}$,
    \[
    \begin{array}{ll}
    (1.h)' \quad& g^{\phi \psi} = (g^\phi)^{\psi} \\
    (1.v)' \quad& \phi^{gh} = (\phi^g)^h \\
    (2.h)' \quad& (\phi\psi)^g = \phi^g \psi^{(g^\phi)} \\
    (2.v)' \quad& (gh)^\phi = g^\phi h^{(\phi^g)} \\
    (3.h)' \quad& g^{\mathrm{id_n}} = g, \qquad 1^\phi = 1 \\
    (3.v)' \quad& \phi^1 = \phi, \qquad \mathrm{id_n}^g = \mathrm{id_n} \\
    \end{array}
    \]
  \end{prop}
  In what follows, the exponent notation may be used interchangeably with the standard
  notation.
\end{rmk}

The above construction for $\Delta S$ shows that the family of groups $\{
\Sigma_n\}_{n \geq 0}$ forms a crossed simplicial group in the sense of Def. 1.1
of~\cite{FL}.  The inclusion $\Delta \hookrightarrow \Delta S$ is given by
$\phi \mapsto (\phi, 1)$, where $1$ is the identity element of $\Sigma_{n+1}^
{\mathrm{op}}$, and $[n]$ is the domain of $\phi$.  For each $n$, let $\tau_n$
be the $(n+1)$-cycle $(0, n, n-1, \ldots, 1) \in \Sigma_{n+1}^\mathrm{op}$.  Thus,
the subgroup generated by $\tau_n$ is isomorphic to $\big(\Z / (n+1)\Z\big)^\mathrm{op}
= \Z / (n+1)\Z$.  We may define the category $\Delta C$ as the subcategory
of $\Delta S$ consisting of all objects $[n]$ for $n \geq 0$, together with those
morphisms $(\phi, g)$ of $\Delta S$ for which $g = \tau_n^i$ for some $i$ (cf.~\cite{L}).
In this way, we get a natural chain of inclusions,
\[
  \Delta \hookrightarrow \Delta C \hookrightarrow \Delta S
\]

An equivalent characterization of $\Delta S$ comes from Pirashvili (cf.~\cite{P}),
as the category $\mathcal{F}(\mathrm{as})$ of `non-commutative' sets.  The objects are 
sets $\underline{n} := \{1, 2, \ldots, n\}$ for $n \geq 0$.  By convention, 
$\underline{0}$ is the empty set.  A morphism in
$\mathrm{Mor}_{\mathcal{F}(\mathrm{as})}
(\underline{n}, \underline{m})$ consists of a map (of sets) $f : \underline{n} 
\to \underline{m}$ together with a total ordering, $\Pi_j$, on $f^{-1}(j)$ 
for all $j \in \underline{m}$.  In such a case, denote by $\Pi$ the partial
order generated by all $\Pi_j$.
If $(f, \Pi) : \underline{n} \to \underline{m}$ and
$(g, \Psi) : \underline{m} \to \underline{p}$, their composition will be
$(gf, \Phi)$, where $\Phi_j$ is the total ordering on $(gf)^{-1}(j)$ (for all $j \in 
\underline{p}$) induced by $\Pi$ and $\Psi$.  Explicitly, for each pair $i_1, i_2 \in
(gf)^{-1}(j)$, we have 
\begin{center}
  $i_1 < i_2$ under $\Phi$ if and only if [$f(i_1) < f(i_2)$ under
  $\Psi$] or [$f(i_1) = f(i_2)$ and $i_1 < i_2$ under $\Pi$].
\end{center}
 
For example, let $f : \underline{9} \to \underline{5}$ be given by:
\[
  f \;:\; \left\{
    \begin{array}{ll}
      1, 5, 8 &\mapsto 1\\
      2, 7 &\mapsto 2\\
      3, 9 &\mapsto 3\\
      4 &\mapsto 4\\
      6 &\mapsto 5
    \end{array}
    \right.
\]
Let the preordering $\Pi$ on pre-image sets be defined by: $8 < 1 < 5$, $2 < 7$, and
$9 < 3$.

Let $g : \underline{5} \to \underline{3}$ be given by:
\[
  g \;:\; \left\{
    \begin{array}{ll}
      1, 2 &\mapsto 2\\
      3, 4 &\mapsto 1\\
      5 &\mapsto 3
    \end{array}
    \right.
\]
Let the preordering $\Psi$ on pre-image sets be defined by: $1 < 2$ and
$3 < 4$.

Then, the composition $(g, \Psi)(f, \Pi) = (gf, \Phi)$ will consist of the map $gf$:
\[
  gf \;:\; \left\{
    \begin{array}{ll}
      1,2,5,7,8 &\mapsto 2\\
      3,4,9 &\mapsto 1\\
      6 &\mapsto 3
    \end{array}
    \right.
\]
and the corresponding preordering $\Phi$, defined by: $9 < 3 < 4$ and $8 < 1 < 5 < 2 < 7$.

There is an obvious inclusion of categories,
$\Delta S \hookrightarrow \mathcal{F}(\mathrm{as})$, taking $[n]$ to $\underline{n+1}$,
but there is no object of $\Delta S$ that maps to $\underline{0}$.  It will be
useful to define $\Delta S_+ \supset \Delta S$ which is isomorphic to
$\mathcal{F}(\mathrm{as})$:
\begin{definition}
  $\Delta S_+$ is the category consisting of all objects and morphisms of
  $\Delta S$, with the additional object $[-1]$, representing the empty set,
  and a unique morphism $\iota_n : [-1] \to [n]$ for each $n \geq -1$.
\end{definition}

\begin{rmk}
  Pirashvili's construction is a special case of a more general construction
  due to May and Thomason \cite{MT}. This construction associates to any topological operad
  $\{\mathcal{C}(n)\}_{n \geq 0}$ a topological category $\widehat{\mathcal{C}}$ together 
  with a functor $\widehat{\mathcal{C}} \to \mathcal{F}$, where $\mathcal{F}$ is the
  category of finite sets, such that the inverse image of any function
  $f: \underline{m} \to \underline{n}$ is the space 
  \[
    \prod_{i=1}^n \mathcal{C}(\# f^{-1}(i) ).
  \] 
  Composition in $\widehat{\mathcal{C}}$
  is defined using the composition of the operad. May and Thomason refer to  
  $\widehat{\mathcal{C}}$
  as the {\it category of operators} associated to $\mathcal{C}$. They were interested 
  in the case of an $E_\infty$ operad, but their construction evidently works for any 
  operad. The category of operators associated to the discrete $A_\infty$ operad 
  $\mathcal{A}ss$, which parametrizes monoid structures,  is precisely Pirashvili's 
  construction of $\mathcal{F}(as)$, {\it i.e.} $\Delta S_+$.  (See 
  Chapter~\ref{chap.prod} for more on operads.)
\end{rmk}

One very useful advantage in enlarging our category to $\Delta S$ to $\Delta S_+$
is the added structure inherent in $\Delta S_+$.

\begin{prop}\label{prop.deltaSpermutative}
  $\Delta S_+$ is a permutative category.
\end{prop}
\begin{proof}
  Recall from Def.~4.1 of~\cite{M3} that a \textit{permutative category} is a 
  category $\mathscr{C}$ with the following additional structure:
  
  An associative bifunctor $\odot : \mathscr{C} \times \mathscr{C} \to \mathscr{C}$,
  
  A \textit{unit} object $e \in \mathrm{Obj}_\mathscr{C}$, which acts as two-sided identity
  for $\odot$.

  A natural transformation \mbox{$\gamma : \odot \to \odot T$}, where 
  \mbox{$T : \mathscr{C} \times \mathscr{C} \to \mathscr{C}$} is the transposition functor 
  \mbox{$(A,B) \mapsto (B,A)$}, such that 
  
  $\gamma^2 = \mathrm{id}$,
  
  $\gamma_{(A,e)} = \gamma_{(e, A)} = \mathrm{Id}_A$, for all objects $A$,
  
  and the following diagram is commutative for all objects $A, B, C$:
  \[
    \begin{diagram}
    \node{A \odot B \odot C}
    \arrow[2]{e,t}{\gamma_{A\odot B, C}}
    \arrow{se,l}{\mathrm{Id}\odot \gamma_{B,C}}
    \node[2]{C \odot A \odot B}\\
    \node[2]{A \odot C \odot B}
    \arrow{ne,r}{\gamma_{A,C} \odot \mathrm{Id}}
    \end{diagram}
  \]
  
  For $\Delta S_+$, let $\odot$ be the functor defined on objects by:
  \[
    [n] \odot [m] := [n+m+1], \qquad \textrm{(disjoint union of sets)},
  \]
  and for morphisms $(\phi, g) : [n] \to [n']$, $(\psi, h) :[m] \to [m']$,
  \[
    (\phi,g) \odot (\psi,h) = (\eta, k) : [n+m+1] \to [n'+m'+1],
  \]
  where
  \[
    \eta \;:\; i \mapsto
    \left\{
    \begin{array}{lll}
      \phi(i), &\qquad& \textrm{for $i=0,\ldots n$}\\
      \psi(i-n-1) + (n'+1), &\qquad& \textrm{for $i=n+1, \ldots n+m$}.
    \end{array}
    \right.
  \]
  and
  \[
    k \; : \; i \mapsto
    \left\{
    \begin{array}{lll}
     g^{-1}(i), &\qquad& \textrm{for $i=0,\ldots n$}\\
      h^{-1}(i-n-1) + (n'+1), &\qquad& \textrm{for $i=n+1, \ldots n+m$}.
    \end{array}
    \right.
  \]
  In short, $(\phi,g) \odot (\psi, h)$ is just the morphism $(\phi, g)$ acting on
  the first $n+1$ points of $[n+m+1]$, and $(\psi, h)$ acting on the remaining points.
  
  The unit object will be $[-1] = \emptyset$. $\odot$ is clearly associative, and
  $[-1]$ acts as two-sided identity.
  
  Finally, define $\gamma_{n,m} : [n] \odot [m] \to [m] \odot [n]$ to be the identity on
  objects, and on morphisms to be precomposition with the block transposition 
  $\beta_{n,m} : [n+m+1] \to [n+m+1]$.  That is, $\beta(i) = i + m+1$, if $i \leq n$, and
  $\beta(i) = i-n-1$, if $i > n$.  
  
  $\gamma_{m,n} \gamma_{n,m} = \mathrm{id}$, which is true since $\beta_{m,n} \beta_{n,m} = 
  \mathrm{id}$, and $\gamma_{m, -1}$ is precomposition by $\beta_{m, -1}$, which is clearly
  the identity (similarly for $\gamma_{-1, m}$).
  
  For $[n], [m], [p]$, we have the following commutative diagram:
  \[
    \begin{diagram}
    \node{ [n] \odot [m+p+1] }
    \arrow{e,t}{=}
    \arrow{s,r}{\mathrm{id}\odot \gamma}
    \node{ [n+m+1] \odot [p] }
    \arrow{e,t}{ \gamma }
    \node{ [p] \odot [n+m+1] }
    \arrow{s,l}{ = }
    \\
    \node{ [n] \odot [p+m+1] }
    \arrow{e,t}{ = }
    \node{ [n+p+1] \odot [m] }
    \arrow{e,t}{ \gamma \odot \mathrm{id} }
    \node{ [p+n+1] \odot [m] }
    \end{diagram}
  \]
  This diagram commutes because the block transposition $\beta_{n+m,p}$ can be accomplished
  by first transposing the blocks $\{n+1, \ldots, n+m+1\}$ and $\{n+m+2, \ldots n+m+p+2\}$ while
  keeping the block $\{0, \ldots n\}$ fixed, then transposing the blocks $\{0, \ldots, n\}$
  and $\{n+1, \ldots, n+p+1\}$ while keeping the block $\{n+p+2,\ldots,n+p+m+2\}$ fixed.
\end{proof}

For the purposes of computation, a morphism $\alpha : [n] \to [m]$ of $\Delta S$
may be conveniently represented as a tensor product of monomials in the variables $\{x_0,
x_1, \ldots, x_n\}$.  Let $\alpha = (\phi, g)$, with $\phi \in \mathrm{Mor}_
\Delta([n],[m])$ and $g \in \Sigma_{n+1}^\mathrm{op}$.  The tensor representation 
of $\alpha$ will have $m + 1$ tensor factors.  Each $x_i$ will occur exactly once,
in the order $x_{g(0)}, x_{g(1)}, \ldots, x_{g(n)}$.  The $i^{th}$ tensor factor
consists of the product of $\#\phi^{-1}(i-1)$ variables, with the convention that
the empty product will be denoted $1$.  Thus, the $i^{th}$ tensor factor 
records the total ordering of $\phi^{-1}(i)$.
As an example, the tensor representation of
the morphism depicted in Fig.~\ref{diag.morphism-tensor} is 
$x_1x_0 \otimes x_3x_4 \otimes 1 \otimes x_2$.  

\begin{figure}[ht]
  \psset{unit=1in}
  \begin{pspicture}(4,4)
  
  \psdots[linecolor=black, dotsize=4pt]
  (0.3, 2.4)(0.6, 2.7)(0.9, 3.0)(1.2, 3.3)(1.5, 3.6)
  (2.7, 0.3)(3.0, 0.6)(3.3, 0.9)(3.6, 1.2)
  (0.3, 0.3)(0.6, 0.6)(0.9, 0.9)(1.2, 1.2)(1.5, 1.5)
  
  \psline[linewidth=1pt, linecolor=black](0.3, 2.4)(0.6, 0.6)
  \psline[linewidth=1pt, linecolor=black](0.6, 2.7)(0.3, 0.3)
  \psline[linewidth=1pt, linecolor=black](0.9, 3.0)(1.5, 1.5)
  \psline[linewidth=1pt, linecolor=black](1.2, 3.3)(0.9, 0.9)
  \psline[linewidth=1pt, linecolor=black](1.5, 3.6)(1.2, 1.2)

  \psline[linewidth=1pt, linecolor=black](0.3, 0.3)(2.7, 0.3)
  \psline[linewidth=1pt, linecolor=black](0.6, 0.6)(2.7, 0.3)
  \psline[linewidth=1pt, linecolor=black](0.9, 0.9)(3.0, 0.6)
  \psline[linewidth=1pt, linecolor=black](1.2, 1.2)(3.0, 0.6)
  \psline[linewidth=1pt, linecolor=black](1.5, 1.5)(3.6, 1.2)

  \rput(0.2, 2.4){$0$}
  \rput(0.5, 2.7){$1$}
  \rput(0.8, 3.0){$2$}
  \rput(1.1, 3.3){$3$}
  \rput(1.4, 3.6){$4$}

  \rput(0.2, 0.3){$0$}
  \rput(0.5, 0.6){$1$}
  \rput(0.8, 0.9){$2$}
  \rput(1.1, 1.2){$3$}
  \rput(1.4, 1.5){$4$}
  
  \rput(2.6, 0.4){$0$}
  \rput(2.9, 0.7){$1$}
  \rput(3.2, 1.0){$2$}
  \rput(3.5, 1.3){$3$}
    
  \rput(2.7, 2.4){$x_1x_0 \otimes x_3x_4 \otimes 1 \otimes x_2$}

  \end{pspicture}

  \caption[Morphisms of $\Delta S$ as Tensors]{Morphisms of $\Delta S$ in tensor
  notation}
  \label{diag.morphism-tensor}
\end{figure}
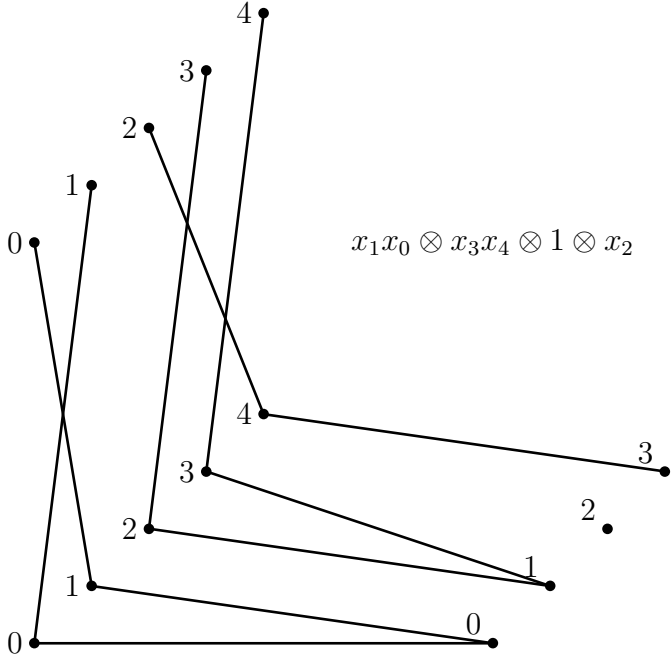

With this notation, the
composition of two morphisms $\alpha = X_0 \otimes X_1 \otimes \ldots \otimes X_m
: [n] \to [m]$ and $\beta = Y_1 \otimes Y_2 \otimes \ldots Y_n : [p] \to [n]$
is given by:
\[
  \alpha \cdot \beta = Z_0 \otimes Z_1 \otimes \ldots \otimes Z_m,
\]
where $Z_i$ is determined by replacing each variable in the monomial $X_i = 
x_{j_1} \ldots x_{j_s}$ in $\alpha$ by the corresponding monomials $Y_{j_k}$ in
$\beta$.  So, $Z_i = Y_{j_1} \ldots Y_{j_s}$.  Thus, for example,
\[
  x_4 x_0 \otimes 1 \otimes x_2 x_3 \otimes x_1 \;\cdot\; 
  x_1 x_6 x_0 \otimes x_7 x_4 \otimes 1 \otimes x_3 \otimes x_2 x_5 \;=\; 
  x_2 x_5 x_1 x_6 x_0 \otimes 1 \otimes x_3 \otimes x_7 x_4.
\]

\section{The Symmetric Bar Construction}\label{sec.symbar}          %

\begin{definition}\label{def.symbar}
  Let $A$ be an associative, unital algebra over a commutative ground ring $k$.  
  Following~\cite{F}, define a (covariant) functor $B_*^{sym}A : \Delta S \to 
  k$-\textbf{Mod} by:
  \[
    B_n^{sym}A := B_*^{sym}A[n] := A^{\otimes n+1}
  \]
  \[  
    B_*^{sym}A(\alpha) : (a_0 \otimes a_1 \otimes \ldots \otimes a_n) \mapsto
       \alpha(a_0, \ldots, a_n),
  \]
  where $\alpha : [n] \to [m]$ is represented in tensor notation, and evaluation
  at $(a_0, \ldots, a_n)$ simply amounts to substituting each $a_i$ for $x_i$
  and multiplying the resulting monomials in $A$.  If the pre-image
  $\alpha^{-1}(i)$ is empty, then the unit of $A$ is inserted.
\end{definition}

\begin{rmk}
  Fiedorowicz~\cite{F} defines the symmetric bar construction functor for morphisms
  $\alpha = (\phi, g)$, where $\phi \in \mathrm{Mor}\Delta([n],[m])$ and $g \in
  \Sigma_{n+1}^{\mathrm{op}}$, via
  \begin{eqnarray*}
    B_*^{sym}A(\phi)(a_0 \otimes a_1 \otimes \ldots \otimes a_n) &=&
    \left( \prod_{a_i \in \phi^{-1}(0)} a_i \right)\otimes \ldots
    \otimes \left( \prod_{a_i \in \phi^{-1}(n)} a_i \right)\\
    B_*^{sym}A(g)(a_0 \otimes a_1 \otimes \ldots \otimes a_n) &=&
    a_{g^{-1}(0)} \otimes a_{g^{-1}(1)} \otimes \ldots \otimes a_{g^{-1}(n)}
  \end{eqnarray*}
  However, in order that this becomes consistent with earlier notation, we should
  require $B_*^{sym}A(g)$ to permute the tensor factors in the inverse sense:
  \[
    B_*^{sym}A(g)(a_0 \otimes a_1 \otimes \ldots \otimes a_n) =
    a_{g(0)} \otimes a_{g(1)} \otimes \ldots \otimes a_{g(n)}.
  \]
\end{rmk}

\begin{prop}\label{prop.symbar-natural}
  The symmetric bar construction $B_*^{sym}A$ is natural in $A$.
\end{prop}
\begin{proof}
  If $f : A \to A'$ is a morphism of
  $k$-algebras (sending $1_A \mapsto 1_{A'}$), then there is a family of induced functors
  $B_n^{sym}f : B_n^{sym}A \to B_n^{sym}A'$ defined by
  \[
    (B_n^{sym}f)( a_0 \otimes a_1 \otimes \ldots \otimes a_n )
    = f(a_0) \otimes f(a_1) \otimes \ldots \otimes f(a_n)
  \]
  It is easily verified that the square below commutes for each $\Delta S$
  morphism $\phi : [n] \to [m]$.
  \[
    \begin{diagram}
      \node{[n]}
      \arrow{s,r}{\phi}      
      \node{A^{\otimes(n+1)}}
      \arrow{s,l}{B_*^{sym}A(\phi)}
      \arrow{e,t}{f^{\otimes(n+1)}}
      \node{B^{\otimes(n+1)}}
      \arrow{s,r}{B_*^{sym}B(\phi)}
      \\
      \node{[m]}      
      \node{A^{\otimes(m+1)}}
      \arrow{e,t}{f^{\otimes(m+1)}}
      \node{B^{\otimes(m+1)}}
    \end{diagram}
  \]    
\end{proof}

Note that $B_*^{sym}A$ can be regarded as a simplicial \mbox{$k$-module} (\textit{i.e.},
a functor $\Delta^{\mathrm{op}} \to k$-\textbf{Mod}) via the chain of 
functors:
\begin{equation}\label{eq.delta-C-S-chain}
  \Delta^{\mathrm{op}} \hookrightarrow \Delta C^{\mathrm{op}} 
  \stackrel{\cong}{\to} \Delta C \hookrightarrow \Delta S.
\end{equation}
Here, the isomorphism $D : \Delta C^{\mathrm{op}} \to \Delta C$ is the standard
duality (see~\cite{L}), which is defined on generators by:
\[
  \left\{
  \begin{array}{llll}
  D(d_i) &=& \sigma_i, &(0 \leq i \leq n-1)\\
  D(d_n) &=& \sigma_0 \cdot \tau^{-1} & \\
  D(s_i) &=& \delta_{i+1}, &(0 \leq i \leq n-1)\\
  D(t) &=& \tau^{-1} &
  \end{array}
  \right.
\]
\begin{rmk}\label{rmk.cyclic}
  Fiedorowicz showed in~\cite{F} that $B_*^{sym}A \cdot D = B_*^{cyc}A$,
  the \textit{cyclic bar construction} (cf.~\cite{L}).  By duality of $\Delta C$,
  it is equivalent to use the functor $B_*^{sym}A$, restricted to morphisms of
  $\Delta C$ in order to do computations of $HC_*(A)$.
\end{rmk}

\section{Definition of Symmetric Homology}\label{sec.symhom}        %

\begin{definition}\label{def.C-modules}
  For a category $\mathscr{C}$, a covariant functor $F : \mathscr{C} \to
  \mbox{\textrm{$k$-\textbf{Mod}}}$ will be called a \mbox{$\mathscr{C}$-module}.  Similarly,
  a contravariant functor $G : \mathscr{C} \to \mbox{\textrm{$k$-\textbf{Mod}}}$ will be
  called a $\mathscr{C}^\mathrm{op}$-module (since $G^\mathrm{op} : \mathscr{C}^\mathrm{op}
  \to \mbox{\textrm{$k$-\textbf{Mod}}}$ is covariant).
\end{definition}

\begin{definition}\label{def.categorical_tensor}
  For a category $\mathscr{C}$, if $N$ is a \mbox{$\mathscr{C}$-module} and 
  $M$ is a \mbox{$\mathscr{C}^{\mathrm{op}}$-module}, define the tensor product 
  (over $\mathscr{C}$) thus:
  \[
    M \otimes_{\mathscr{C}} N := \bigoplus_{X \in \mathrm{Obj}\mathscr{C}}
    M(X) \otimes_k N(X) / \approx,
  \]
  where the equivalence $\approx$ is generated by the following:  For every morphism
  $f \in \mathrm{Mor}_{\mathscr{C}}(X, Y)$, and every $x \in N(X)$ and $y \in M(Y)$, 
  we have $ y \otimes f_*(x) \approx f^*(y) \otimes x$.
\end{definition}

Note, MacLane defines the tensor product as a \textit{coend}:
\[
  M \otimes_\mathscr{C} N := \int^X (MX) \otimes (NX),
\]
where we consider
$(M-)\otimes(N-)$ as a bifunctor $\mathscr{C}^\mathrm{op} \times \mathscr{C} 
\to \mathscr{L}$, for a given cocomplete category $\mathscr{L}$ equipped with
a functor $\otimes : \mathscr{L} \times \mathscr{L} \to \mathscr{L}$ (see~\cite{ML}).
When $\mathscr{L} = $ \mbox{$k$-\textbf{Mod}}, this construction yields the same as that above.

Alternatively, consider $k[\mathrm{Mor}\mathscr{C}]$, the free \mbox{$k$-module} with
basis $\mathrm{Mor}\mathscr{C}$.  We may define a ring structure on
$k[\mathrm{Mor}\mathscr{C}]$ by defining products of basis elements thus:
\[
  f \cdot g :=
  \left\{
  \begin{array}{ll}
    f \circ g, &\textrm{if $f$ is composable with $g$} \\
    0, &\textrm{otherwise}
  \end{array}
  \right.
\]
Note, $k[\mathrm{Mor}\mathscr{C}]$ will in general not have a unit, but only 
\textit{local units}.
Indeed, for any finitely generated $k$-submodule, $M$, with basis $\{v_1, \ldots,
v_t\}$, each $v_i$ is the sum of only finitely many terms $c_{ij} f_j$, with
$f_j \in \mathrm{Mor}\mathscr{C}$.  Let $\{f_1, f_2, \ldots, f_n\}$ be the set 
of those morphisms that occur in any of the $v_i$'s.  Then there is an element
in $k[\mathrm{Mor}\mathscr{C}]$ that acts as a two-sided unit 
for any element of $M$, namely: $\sum \mathrm{id}_{X}$, where the sum extends 
over all
those $X \in \mathrm{Obj}\mathscr{C}$ that appear as domains or codomains of the
$f_i$'s.

Now, the category of \mbox{$\mathscr{C}$-modules} (with natural transformations as
morphisms) is equivalent to the category of left \mbox{$k[\mathrm{Mor}\mathscr{C}]$-modules}.  
The correspondence is as follows:  For a \mbox{$\mathscr{C}$-module} $M$, let $\overline{M}$
be the \mbox{$k$-module} $\bigoplus_{X \in \mathrm{Obj}\mathscr{C}} MX$.  For a morphism
$f : X \to Y$ in $\mathrm{Mor}\mathscr{C}$ and a homogeneous $x \in \overline{M}$
(\text{i.e.}, $x \in MW$ for some $W \in \mathrm{Obj}\mathscr{C}$), put:
\[
  f . x :=
  \left\{
  \begin{array}{ll}
    f_*(x), &\textrm{if $x \in MX$} \\
    0, &\textrm{otherwise}
  \end{array}
  \right.
\]
Extend the action of $f$ to arbitrary $v \in \overline{M}$ by linearity.  This formula
provides a module structure, since if $x \in MX$ and $X \stackrel{g}{\to} Y
\stackrel{f}{\to} Z$ is a chain of morphisms in $\mathscr{C}$, we have
\[
  (f \cdot g). x = (fg).x= (fg)_*(x) = f_*\big(g_*(x)\big) = f.\big(g_*(x)\big)
  = f.(g.x).
\]

Similarly, the category of \mbox{$\mathscr{C}^\mathrm{op}$-modules} is equivalent to
the category of right $k[\mathrm{Mor}\mathscr{C}]$-modules, with action
$y \cdot f = f^*(y)$ for any $y \in Y$ and $f \in 
\mathrm{Mor}_\mathscr{C}(X, Y)$.

Under this equivalence, the tensor product $M \otimes_{\mathscr{C}} N$ construction
is simply the standard tensor product $\overline{M} 
\otimes_{k[\mathrm{Mor}\mathscr{C}]} \overline{N}$, and thus we can define the modules
\[
  \mathrm{Tor}_n^\mathscr{C}(M, N) := 
  \mathrm{Tor}_n^{k[\mathrm{Mor}\mathscr{C}]}\left(\overline{M}, \overline{N}\right).
\]
Note, it is also possible to define $\mathrm{Tor}_n^\mathscr{C}(M, N)$ directly as
the derived functors of the categorical tensor product construction
(see~\cite{FL}).

The trivial \mbox{$\mathscr{C}$-module}, denoted
by $\underline{k}$, is the functor $\mathscr{C} \to k$-\textbf{Mod} which takes 
every object
to $k$ and every morphism to $\mathrm{id}_k$.  Under the above equivalence,
this becomes the trivial left \mbox{$k[\mathrm{Mor}\mathscr{C}]$-module} 
$\overline{k} = \bigoplus_{X \in \mathrm{Obj}\mathscr{C}}k$.  
We also denote by $\underline{k}$ the trivial \mbox{$\mathscr{C}^\mathrm{op}$-module}.

\begin{definition}
  The \textbf{symmetric homology} of an associative, unital $k$-algebra, $A$,
  is denoted $HS_*(A)$, and is defined as:
  \[
    HS_*(A) := \mathrm{Tor}_*^{\Delta S}\left(\underline{k},B_*^{sym}A\right)
  \]
\end{definition}

\begin{rmk}
  Note, the existing literature based on the work of Connes, Loday and Quillen
  consistently defines the categorical tensor product in the reverse sense:
  $N \otimes_{\mathscr{C}} M$ is the direct sum of copies of $NX \otimes_k MX$
  modded out by the equivalence $x \otimes f^*(y) \approx f_*(x) \otimes y$ for
  all $\mathscr{C}$-morphisms $f : X \to Y$.  In this context, $N$ is covariant,
  while $M$ is contravariant.  I chose to follow the convention of Pirashvili
  and Richter \cite{PR} in writing tensor products as $M \otimes_{\mathscr{C}} N$ so 
  that the equivalence $\xi : $\mbox{$\mathscr{C}$-\textbf{Mod}} $\to$
  \mbox{$k[\mathrm{Mor}\mathscr{C}]$-\textbf{Mod}} passes to tensor products in a
  straightforward way: $\xi( M \otimes_{\mathscr{C}} N ) =
  \xi(M) \otimes_{k[\mathrm{Mor}\mathscr{C}]} \xi(N)$.
\end{rmk}

\begin{rmk}
  Since 
  \[
    \underline{k}\otimes_{\Delta S} M \cong \colim_{\Delta S}M,
  \]
  for any $\Delta S$-module $M$,
  we can alternatively describe symmetric homology as derived functors
  of $\colim$:
  \[
    HS_i(A) = {\colim_{\Delta S}}^{(i)}B_*^{sym}A.
  \]
  (To see the relation with higher colimits, we need to tensor a projective
  resolution of $B_*^{sym}A$ with $\underline{k}$).
\end{rmk}

\section{The Standard Resolution}\label{sec.stdres}                 %

Let $\mathscr{C}$ be a category.  Henceforth, we shall use interchangeably the
notion of \mbox{$\mathscr{C}$-module} and 
\mbox{$k[\mathrm{Mor}\mathscr{C}]$-module}, under
the equivalence mentioned in section~\ref{sec.symhom}.
The rank $1$ free \mbox{$\mathscr{C}$-module} is
$k[\mathrm{Mor}\mathscr{C}]$, with the left action of composition of morphisms.
Now as $k$-module, $k[\mathrm{Mor}\mathscr{C}]$ decomposes into the direct sum,
\[
  k[\mathrm{Mor}\mathscr{C}] = \bigoplus_{X \in \mathrm{Obj}\mathscr{C}}
  \left( \bigoplus_{Y \in \mathrm{Obj}\mathscr{C}} 
  k\left[\mathrm{Mor}_\mathscr{C}(X, Y)\right] \right).
\]
By abuse of notation, denote $\ds{\bigoplus_{Y \in \mathrm{Obj}\mathscr{C}}} 
  k\left[\mathrm{Mor}_\mathscr{C}(X, Y)\right]$ by 
  $k\left[\mathrm{Mor}_\mathscr{C}(X, -)\right]$.
So there is a direct sum decomposition as left \mbox{$\mathscr{C}$-module}),
\[
  k[\mathrm{Mor}\mathscr{C}] = \bigoplus_{X \in \mathrm{Obj}\mathscr{C}}
  k\left[\mathrm{Mor}_\mathscr{C}(X, -)\right].
\]
Thus, the submodules $k\left[\mathrm{Mor}_\mathscr{C}(X, -)\right]$ are projective left
\mbox{$\mathscr{C}$-modules}.

Similarly, $k[\mathrm{Mor}\mathscr{C}]$ is the rank $1$ free right 
\mbox{$\mathscr{C}$-module}, with right action of pre-composition of morphisms, and as such, 
decomposes as:
\[
  k[\mathrm{Mor}\mathscr{C}] = \bigoplus_{Y \in \mathrm{Obj}\mathscr{C}}
  k\left[\mathrm{Mor}_\mathscr{C}(-, Y)\right]
\]
Again, the notation $k\left[\mathrm{Mor}_\mathscr{C}(-, Y)\right]$ is shorthand
for $\ds{\bigoplus_{X \in \mathrm{Obj}\mathscr{C}}}
  k\left[\mathrm{Mor}_\mathscr{C}(X, Y)\right]$.
The submodules $k\left[\mathrm{Mor}_\mathscr{C}(-, Y)\right]$ are projective as right
\mbox{$\mathscr{C}$-modules}.

It will also be important to note that 
$k\left[\mathrm{Mor}_\mathscr{C}(-, Y)\right] \otimes_{\mathscr{C}} N
\cong N(Y)$ as $k$-module via the evaluation map $f \otimes y \mapsto f_*(y)$.
Similarly, $M \otimes_{\mathscr{C}} k\left[\mathrm{Mor}_\mathscr{C}(X, -)\right]
\cong M(X)$.

Following Quillen (\cite{Q}, Section~1), we make the following definition:
\begin{definition} 
  Given a functor $F : \mathscr{B} \to \mathscr{C}$ and a fixed object $Y$ in $\mathscr{C}$,
  let $F / Y$ denote the category whose objects are pairs $(X, \phi)$ where $X$ is an
  object of $\mathscr{B}$ and $\phi : FX \to Y$ is a morphism in $\mathscr{C}$.  A morphism
  from $(X, \phi)$ to $(X', \phi')$ is a morphism $\psi : X \to X'$ such that 
  $\phi' \circ F\psi= \phi$.  When $F$ is the identity functor on $\mathscr{C}$, this
  construction is called the over-category (objects over $Y$), and is denoted by
  $\mathscr{C} / Y$.
  
  Dually, let $Y \setminus F$ denote the category whose objects are pairs $(X, \phi)$ for 
  $X$ in $\mathscr{B}$ and $\phi : Y \to FX$ in $\mathscr{C}$.  Here, a morphism
  from $(X, \phi)$ to $(X', \phi')$ is a morphism $\phi : X \to X'$ such that
  $F\psi \circ \phi = \phi'$.
  When $F$ is the identity functor,
  this is called the under-category (objects under $Y$), and is denoted by $Y \setminus
  \mathscr{C}$.
\end{definition}

Given a functor $F : \mathscr{B} \to \mathscr{C}$ of small categories
define a functor $(F/-) : \mathscr{C} \to \textbf{Cat}$ as follows:
The object $Y$ is sent to the category $F / Y$.  If $\nu : Y \to Y'$ is a morphism
in $\mathscr{C}$, the functor $(F / \nu) : F / Y \to F / Y'$ is defined on objects by
$(X, \phi) \mapsto (X, \nu \phi)$.  For a morphism $\psi : (X, \phi) \to (X', \phi')$
in $F / Y$, $\psi$ may also represent a morphism in $F / Y'$, 
since $\phi' F\psi =\phi \implies \nu\phi' F\psi = \nu\phi$ (as morphisms in $\mathscr{C}$).

Again, we may dualize to obtain a \textit{contravariant} functor $(-\setminus F) :
\mathscr{C} \to \textbf{Cat}$, defined on objects by $Y \mapsto Y\setminus F$, and if
$\nu : Y \to Y'$, then $(\nu \setminus F)$ is a functor $(Y' \setminus F) \to
(Y \setminus F)$ which takes $(X, \phi)$ to $(X, \phi\nu)$.

Thus, $(F / -)$ is a $\mathscr{C}$-category, and $(- \setminus F)$ is a $\mathscr{C}^
\mathrm{op}$-category.  In what follows, we shall assume $F$ is the identity functor
of $\mathscr{C}$.
As noted in~\cite{FL}, the nerve of $(\mathscr{C} / -)$
is a simplicial $\mathscr{C}$-set, and the complex $L_*$, given by:
\[
  L_n := k\left[ N(\mathscr{C}/-)_n \right]
\]
is a resolution by projective \mbox{$\mathscr{C}$-modules} of the trivial 
\mbox{$\mathscr{C}$-
module}, $k$.  Here, the boundary map is $\partial := \sum_i (-1)^i d_i$, where the
$d_i$'s come from the simplicial structure of the nerve of 
$(\mathscr{C}/-)$.  

For the definition of $HS_*(A)$, we
shall be more interested in the dual construction, which yields a resolution by projective 
\mbox{$\mathscr{C}^\mathrm{op}$-modules} of the trivial 
\mbox{$\mathscr{C}^\mathrm{op}$-module},
$k$.  Explicitly, define the complex $\overline{L}_*$ by:
\[
  \overline{L}_n := k\left[ N(- \setminus\mathscr{C})_n \right]
\]
\[
  \overline{L}_n(C) := k\left[ \{C \stackrel{g}{\to} A_0 \stackrel{f_1}{\to} A_1 
  \stackrel{f_2}{\to} \ldots \stackrel{f_n}{\to} A_n \} \right]
\]
For completeness, we shall provide a proof of:
\begin{prop}\label{prop.contractibility_under-category}
  $\overline{L}_*$ is a resolution of $k$ by projective 
  \mbox{$\mathscr{C}^\mathrm{op}$-modules}.
\end{prop}
\begin{proof}
  Fix $C \in \mathrm{Obj}\mathscr{C}$.  Let $\epsilon : \overline{L}_0(C) \to k$ be
  the map defined on generators by 
  \[
    \epsilon( C \to A_0 ) := 1_k.
  \]
  We shall show the complex
  \[
    k \stackrel{\epsilon}{\leftarrow} \overline{L}_0(C) 
    \stackrel{\partial}{\leftarrow} \overline{L}_1(C) \stackrel{\partial}{\leftarrow}
    \ldots
  \]
  is chain homotopic to the $0$ complex.  Explicitly, define
  \[
    h_{-1} : 1 \mapsto (C \stackrel{id}{\to} C)
  \]
  \[
    h_n : (C \to A_0 \to \ldots \to A_n) \mapsto (C \stackrel{id}{\to} C \to A_0 \to
      \ldots A_n), \qquad \textrm{for $i \geq 0$}
  \]
  We have $\epsilon h_{-1} (1) = \epsilon(C \to C) = 1$, so $\epsilon h_{-1} = \mathrm{id}$.
  Next, in degree $0$,
  \[
    (\partial h_0 + h_{-1} \epsilon)( C \to A_0 ) = \partial ( C \to C \to A_0) +
    h_{-1} ( 1 )
  \]
  \[
    = (C \to A_0) - (C \to C) + (C \to C)
  \]
  \[
    = (C \to A_0).
  \]
  Finally, let $n \geq 1$.
  \[
    (\partial h_n + h_{n-1} \partial)(C \to A_0 \to \ldots \to A_n)
  \]
  \[
    = \partial(C \to C \to A_0 \to \ldots \to A_n) +
    h_{n-1} \sum_{i = 0}^{n} (-1)^i(C \to A_0 \to \ldots \widehat{A_i} \ldots \to A_n),
  \]
  where $\widehat{A_i}$ means to omit the object $A_i$ by composing the map with target $A_i$
  with the map with source $A_i$.
  \[
    = (C \to A_0 \to \ldots \to A_n) + \sum_{i=0}^n (-1)^{i+1} (C \to C \to A_0 \to \ldots
    \widehat{A_i} \ldots \to A_n) \,+
  \]
  \[
    \qquad\sum_{i = 0}^{n} (-1)^i(C \to C \to A_0 \to \ldots \widehat{A_i} \ldots \to A_n)
  \]
  \[
    = (C \to A_0 \to \ldots \to A_n).
  \]
  Hence, $\partial h_n + h_{n-1} \partial = \mathrm{id}$, and so $h$ determines a
  chain homotopy $0 \simeq \mathrm{id}$, proving that the complex is contractible.
  
  Next, we show that the \mbox{$\mathscr{C}^\mathrm{op}$-module} $\overline{L}_n$ is projective.
  Indeed,
  \[
    \overline{L}_n = \bigoplus_{C_n} k\left[\mathrm{Mor}_\mathscr{C}( -, A_0)\right],
  \]
  where the direct sum is indexed over the set $C_n$ of all chains 
  \mbox{$A_0 \to A_1 \to \ldots \to A_n$}.  As
  we have seen above, $k\left[\mathrm{Mor}_\mathscr{C}(-,A_0)\right]$ is projective as 
  \mbox{$\mathscr{C}^\mathrm{op}$-module}, therefore $\overline{L}_n$ is projective. 
\end{proof}

Thus, we may compute $HS_*(A)$ as the homology groups of the following complex:
\begin{equation}\label{symhomcomplex}
  0 \longleftarrow 
  \overline{L}_0 \otimes_{\Delta S} B_*^{sym}A \longleftarrow
  \overline{L}_1 \otimes_{\Delta S} B_*^{sym}A \longleftarrow
  \overline{L}_2 \otimes_{\Delta S} B_*^{sym}A \longleftarrow
  \ldots
\end{equation}

\begin{cor}\label{cor.SymHomComplex}
  For an associative, unital $k$-algebra $A$,
  \[
    HS_*(A) = H_*\left( k[N(- \setminus \Delta S)] \otimes_{\Delta S} B_*^{sym}A;\,k  \right)
  \]
\end{cor}

\begin{rmk}\label{rmk.HC}
  By remark~\ref{rmk.cyclic}, it is clear that the related complex
  $k[N(- \setminus \Delta C)] \otimes_{\Delta C} B_*^{sym}A$ computes
  $HC_*(A)$.
\end{rmk}

\begin{rmk}\label{rmk.uniqueChain}
  Observe that every element of $\overline{L}_n \otimes_{\Delta S} B^{sym}_*A$
  is equivalent to one in which the first morphism of the $\overline{L}_n$
  factor is an identity:
  \[
    [p]\stackrel{\alpha}{\to}
    [q_0]\stackrel{\beta_1}{\to}[q_1]
    \stackrel{\beta_2}{\to}\ldots\stackrel{\beta_n}{\to}[q_n]
    \,\otimes\, (y_0 \otimes \ldots \otimes y_p)
  \]
  \[
    = \alpha^*\big(
    [q_0]\stackrel{\mathrm{id}_{[q_0]}}{\longrightarrow}[q_0]\stackrel{\beta_1}{\to}[q_1]
    \stackrel{\beta_2}{\to}\ldots\stackrel{\beta_n}{\to}[q_n]\big)
    \,\otimes\, (y_0 \otimes \ldots \otimes y_p)
  \]
  \[
    \approx 
    [q_0]\stackrel{\mathrm{id}_{[q_0]}}{\longrightarrow}[q_0]\stackrel{\beta_1}{\to}[q_1]
    \stackrel{\beta_2}{\to}\ldots\stackrel{\beta_n}{\to}[q_n]
    \,\otimes\,\alpha_*(y_0, \ldots, y_p)
  \]
  Thus, we may consider $\overline{L}_n \otimes_{\Delta S} B_*^{sym}A$ to be the
  \mbox{$k$-module} generated by the elements 
  \begin{equation}\label{chain_elem}
    \left\{ 
    [q_0]\stackrel{\beta_1}{\to}[q_1]
    \stackrel{\beta_2}{\to}\ldots\stackrel{\beta_n}{\to}[q_n]
    \,\otimes\,
    (y_0 \otimes \ldots \otimes y_{q_0})  
    \right\},
  \end{equation}
  where the tensor product is now over $k$. 

  The face maps 
  $d_0, d_1, \ldots, d_n$ are defined on generators by:
  \[
    d_0\left( [q_0]\stackrel{\beta_1}{\to}[q_1]
    \stackrel{\beta_2}{\to}\ldots\stackrel{\beta_n}{\to}[q_n]
    \,\otimes\,(y_0 \otimes \ldots \otimes y_{q_0})\right)
    \,=\,
    [q_1]\stackrel{\beta_2}{\to}  
    \ldots\stackrel{\beta_n}{\to}[q_n]\,\otimes\,
    \beta_1(y_0, \ldots, y_{q_0}),
  \]
  \[
    d_j\left([q_0]\stackrel{\beta_1}{\to}\ldots\stackrel{\beta_n}{\to}[q_n]
    \,\otimes\,(y_0 \otimes \ldots \otimes y_{q_0})
    \right) \,=
  \]
  \[
    [q_0]\stackrel{\beta_1}{\to}
    \ldots \to [q_{j-1}] \stackrel{\beta{j+1}\beta{j}}{\longrightarrow} [q_{j+1}] \to \ldots
    \stackrel{\beta_n}{\to}[q_n]
    \,\otimes\,
    (y_0 \otimes \ldots \otimes y_{q_0}) , \;\; (0 < j < n),
  \]
  \[
    d_n\left( [q_0]\stackrel{\beta_1}{\to}\ldots\stackrel{\beta_n}{\to}[q_n]
    \,\otimes\,(y_0 \otimes \ldots \otimes y_{q_0})\right)
    \,=\,
    [q_0]\stackrel{\beta_1}{\to}
    \ldots\stackrel{\beta_{n-1}}{\to}[q_{n-1}]\,\otimes\,
    (y_0 \otimes \ldots \otimes y_{q_0}).
  \]
\end{rmk}

We now have enough tools to compute $HS_*(k)$.  First, we need to show:

\begin{lemma}\label{lem.DeltaScontractible}
  $N(\Delta S)$ is a contractible complex.
\end{lemma}
\begin{proof}
  Define a functor $\mathscr{F} : \Delta S \to \Delta S$ by 
  \[
    \mathscr{F} : [n] \mapsto [0] \odot [n],
  \]
  \[
    \mathscr{F} : f \mapsto \mathrm{id}_{[0]} \odot f,
  \]
  using the multiplication $\odot$ defined in section~\ref{sec.deltas}.

  There is a natural transformation $\mathrm{id}_{\Delta S} \to \mathscr{F}$
  given by the following commutative diagram for each $f : [m] \to [n]$:
  \[
    \begin{diagram}
      \node{ [m] }
      \arrow{e,t}{f}
      \arrow{s,l}{\delta_0}
      \node{ [n] }
      \arrow{s,r}{\delta_0}\\
      \node{ [m+1] }
      \arrow{e,t}{\mathrm{id} \odot f}
      \node{ [n+1] }
    \end{diagram}
  \]
  Here, $\delta^{(k)}_j : [k-1] \to [k]$ is the $\Delta$ morphism that misses the point
  $j \in [k]$.
  
  Consider the constant functor $\Delta S \stackrel{[0]}{\to} \Delta S$ that sends 
  all objects to
  $[0]$ and all morphisms to $\mathrm{id}_{[0]}$.  There is a natural transformation
  $[0] \to \mathscr{F}$ given by the following commutative diagram for each
  $f : [m] \to [n]$.
  \[
    \begin{diagram}
      \node{ [0] }
      \arrow{e,t}{\mathrm{id}}
      \arrow{s,l}{0_0}
      \node{ [0] }
      \arrow{s,r}{0_0}\\
      \node{ [m+1] }
      \arrow{e,t}{\mathrm{id} \odot f}
      \node{ [n+1] }
    \end{diagram}
  \]
  Here, $0^{(k)}_j : [0] \to [k]$ is the morphism that sends the point $0$ to
  $j \in [k]$.
  
  Natural transformations induce homotopy equivalences (see~\cite{Se} or
  Prop.~1.2 of~\cite{Q}), so in
  particular, the identity map on $N(\Delta S)$ is homotopic to the map that 
  sends $N(\Delta S)$ to the nerve of a trivial category.  Thus, $N(\Delta S)$
  is contractible.
\end{proof}

\begin{cor}\label{cor.HS_of_k}
  The symmetric homology of the ground ring $k$ is isomorphic to $k$, concentrated
  in degree $0$.
\end{cor}
\begin{proof}
  By Cor.~\ref{cor.SymHomComplex} and Remark~\ref{rmk.uniqueChain}, $HS_*(k)$ is
  the homology of the chain complex generated (freely) over $k$ by the chains
  \[
    \left\{ 
    [q_0]\stackrel{\beta_1}{\to}[q_1]
    \stackrel{\beta_2}{\to}\ldots\stackrel{\beta_n}{\to}[q_n]
    \,\otimes\,
    (1 \otimes \ldots \otimes 1)  
    \right\},
  \]
  where $\beta_i \in \mathrm{Mor}_{\Delta S}\left( [q_{i-1}], [q_i] \right)$.  Each
  chain $[q_0] \to [q_1] \to \ldots \to [q_n]\,\otimes\,
    (1 \otimes \ldots \otimes 1)$ may be identified with the chain
  $[q_0] \to [q_1] \to \ldots \to [q_n]$ of $N(\Delta S)$, and clearly this
  defines a chain isomorphism to $N(\Delta S)$. The result now
  follows from Lemma~\ref{lem.DeltaScontractible}.
\end{proof}

\section{Tensor Algebras}\label{sec.tensoralg}                           %

For a general $k$-algebra $A$, the standard resolution is often too difficult
to work with.  In the following chapters, we shall see some methods of reducing
the size of the standard resolution.
In order to prove the results of chapter~\ref{chap.alt}, it is necessary
to prove these results first for the special case of tensor algebras.  Indeed,
tensor algebra arguments are also key in the proof of Fiedorowicz's Theorem
(Thm.~1({\it i}) of~\cite{F}) about the symmetric homology of group algebras.

Let $T : k$-\textbf{Alg} $\to k$-\textbf{Alg} be the functor
sending an algebra $A$ to the tensor algebra generated by $A$.
\[
  TA := k \oplus A \oplus A^{\otimes 2}
  \oplus A^{\otimes 3} \oplus \ldots
\]
The functor $T$ takes an algebra homomorphism 
$A \stackrel{f}{\to} B$ to the induced homomorphism $Tf$ defined on generators by:
\[
  Tf( a_0 \otimes a_1 \otimes \ldots \otimes a_k ) = f(a_0) \otimes f(a_1) \otimes
  \ldots \otimes f(a_k).
\]
There is an algebra homomorphism $\theta : TA \to A$, defined by multiplying tensor factors:
\[
  \theta( a_0 \otimes a_1 \otimes \ldots \otimes a_k ) := a_0a_1 \cdots a_k.
\]
In fact, $\theta$ defines a natural transformation $T \to \mathrm{id}$, as can
be verified by the following commutative diagram (valid for all $A \stackrel{f}{\to}
B$ in $k$-\textbf{Alg}).
\[
  \begin{diagram}
    \node{TA}
    \arrow{e,t}{\theta_A}
    \arrow{s,l}{Tf}
    \node{A}
    \arrow{s,r}{f}\\
    \node{TB}
    \arrow{e,t}{\theta_B}
    \node{B}
  \end{diagram}
\]
\[
  \begin{diagram}    
    \node{a_0 \otimes \ldots \otimes a_k}
    \arrow[2]{e,t,T}{\theta_A}
    \arrow{s,l,T}{Tf}
    \node[2]{a_0\cdots a_k}
    \arrow{s,r,T}{f}\\
    \node{f(a_0) \otimes \ldots \otimes f(a_k)}
    \arrow{e,t,T}{\theta_B}
    \node{f(a_0)\cdots f(a_k)}
    \arrow{e,t,=}{}
    \node{f(a_0\cdots a_k)}
  \end{diagram}
\]  

We shall also make use of a $k$-module homomorphism $h$ sending the
algebra $A$ identically onto the summand $A$ of $TA$.  This map is a natural
transformation from the forgetful functor $U : k$-\textbf{Alg} $\to k$-\textbf{Mod}
to the functor $UT$.
\[
  \begin{diagram}
    \node{UA}
    \arrow{s,l}{Uf}
    \arrow{e,t}{h_A}
    \node{UTA}
    \arrow{s,r}{UTf}\\
    \node{UB}
    \arrow{e,t}{h_B}
    \node{UTB}
  \end{diagram}
\]
Henceforth, context will make it clear whether we are working with algebras or
underlying $k$-modules, and so the functor $U$ shall be omitted.

Denote by $\mathscr{Y}_*A$, the complex $k[ N(- \setminus \Delta S) ]
\otimes_{\Delta S} B_*^{sym}A$ of Cor.~\ref{cor.SymHomComplex}.

\begin{prop}\label{prop.Y-functor}
  The assignment $A \mapsto \mathscr{Y}_*A$ is functorial.
\end{prop}
\begin{proof}
  We have to say what happens to morphisms.  If $f : A \to B$ is a morphism of
  $k$-algebras (sending $1_A \mapsto 1_B$), then there is an induced chain map
  \[
    \mathrm{id} \otimes B_*^{sym}f \;:\; k[ N(- \setminus \Delta S) ]\otimes_{\Delta S} 
    B_*^{sym}A \to k[ N(- \setminus \Delta S) ]\otimes_{\Delta S} B_*^{sym}B,
  \]
  defined on $k$-chains by:
  \begin{eqnarray*}
    [n] \to [p_0] \to [p_1] \to \ldots \to [p_k] \otimes \left(
    a_0 \otimes \ldots \otimes a_n\right)
    \\  
    \mapsto \;\;
    [n] \to [p_0] \to [p_1] \to \ldots \to [p_k] \otimes \left(
    f(a_0) \otimes \ldots \otimes f(a_n)\right)
  \end{eqnarray*}
  $B_*^{sym}f$ is a natural transformation by Prop.~\ref{prop.symbar-natural}.
\end{proof}

For a general
$k$-algebra $A$, resolve $A$ by tensor algebras.  The resulting
long exact sequence may be regarded as a $k$-complex, where $TA$ is regarded as
degree $0$.
\begin{equation}\label{eq.res_tensor_alg}
  0 \gets A \stackrel{\theta_A}{\gets} TA \stackrel{\theta_1}{\gets} T^2A 
  \stackrel{\theta_2}{\gets} \ldots
\end{equation}

The maps $\theta_n$ for $n \geq 1$ are defined in terms of face maps:
\begin{equation}\label{eq.theta_n}
  \theta_n := \sum_{i = 0}^{n} (-1)^i T^{n-i}\theta_{T^iA}.
\end{equation}

In Section~\ref{sec.deltas_plus}, we shall need to use an important property of the maps 
$\theta_n$:
\begin{prop}\label{prop.naturality-of-theta_n}
  $\theta_n$ defines a natural transformation $T^{p+1} \to T^p$.
\end{prop}
\begin{proof}
  The map $\theta_n$ depends on the algebra $A$.  When we wish to distinguish
  which algebra is associated with $\theta_n$, we shall use the notation $\theta_n(A)$.
  Now, let $f : A \to B$ be any unital map of algebras.  Consider the diagram below:
  \[
    \begin{diagram}
      \node{T^{n+1}A}
      \arrow{e,t}{\theta_n(A)}
      \arrow{s,l}{T^{n+1}f}
      \node{T^nA}
      \arrow{s,r}{T^nf}\\
      \node{T^{n+1}B}
      \arrow{e,t}{\theta_n(B)}
      \node{T^nB}
    \end{diagram}
  \]
  We must show this diagram commutes.  Now, $\ds{\theta_n(A) = \sum_{i = 0}^{n} (-1)^i 
  T^{n-i}\theta_{T^iA}}$.  Let $0 \leq i \leq n$, and consider the following diagram:
  \[
    \begin{diagram}
      \node{T(T^i A)}
      \arrow{e,t}{\theta_{T^i A}}
      \arrow{s,l}{T(T^i f)}
      \node{T^i A}
      \arrow{s,r}{T^i f}\\
      \node{T(T^i B)}
      \arrow{e,t}{\theta_{T^i B}}
      \node{T^i B}
    \end{diagram}
  \]
  This diagram commutes by naturality of $\theta$.  Now, apply the functor $T^{n-i}$
  to each object and functor to get the corresponding commutative diagram for the
  $i^{th}$ face map of $\theta_n$.
  
  \multiply \dgARROWLENGTH by2
  \[
    \begin{diagram}
      \node{T^{n+1} A}
      \arrow{e,t}{T^{n-i}\theta_{T^i A}}
      \arrow{s,l}{T^{n+1} f}
      \node{T^n A}
      \arrow{s,r}{T^n f}\\
      \node{T^{n+1} B}
      \arrow{e,t}{T^{n-i}\theta_{T^i B}}
      \node{T^n B}
    \end{diagram}
  \]
  \divide \dgARROWLENGTH by2
  
  This proves each face map is natural, so the differential $\theta_n$ is natural.
\end{proof}

\begin{rmk}
  Note that the complex~(\ref{eq.res_tensor_alg}) is nothing more than the complex
  associated to May's 2-sided bar construction 
  $B_*(T, T, A)$ (See chapter 9 of~\cite{M}).  If we denote by $A_0$ the 
  chain complex consisting of $A$ in degree $0$ and $0$ in higher degrees,
  then there is a homotopy $h_n : B_n(T, T, A) \to
  B_{n+1}(T, T, A)$ that establishes
  a strong deformation retract $B_*(T, T, A) \to A_0$.
  In fact, the homotopy maps are given by $h_n := h_{T^{n+1}A}$, where
  $h$ is the natural transformation $U \to UT$ given above.
\end{rmk}

For each $q \geq 0$, if we
apply the functor $\mathscr{Y}_q$ to the complex~(\ref{eq.res_tensor_alg}),
we obtain the sequence below:
\begin{equation}\label{eq.ex-seq-TA}
  \begin{diagram}
    \node{0}
    \node{\mathscr{Y}_qA}
    \arrow{w,t}{ }
    \node{\mathscr{Y}_qTA}
    \arrow{w,t}{\mathscr{Y}_q\theta_{A}}
    \node{\mathscr{Y}_qT^2A}
    \arrow{w,t}{\mathscr{Y}_q\theta_1}
    \node{\mathscr{Y}_qT^3A}
    \arrow{w,t}{\mathscr{Y}_q\theta_2}
    \node{\cdots}
    \arrow{w,t}{ }
  \end{diagram}
\end{equation}
This sequence is exact via the induced homotopy $\mathscr{Y}_q h_*$.  
Denote by $d_i(A)$ the $i^{th}$ differential map of $\mathscr{Y}_*A$.  When the 
context it clear,
the differential will be simply written $d_i$.
Now, the bigraded module $\left\{\mathscr{Y}_q T^{p+1}A\right\}_{p,q \geq 0}$
is not quite a double complex, since the induced maps 
$\mathscr{Y}_* T^{p+1}A \to \mathscr{Y}_*T^p A$ are chain maps (the
corresponding squares commute).  In order
to make the squares of the bigraded module anti-commute, introduce the
sign $(-1)^p$ on each vertical differential.  Call the resulting double complex
$\mathscr{T}_{*,*}$.
\begin{equation}\label{eq.scriptT}
  \begin{diagram}
    \node{\vdots}
    \arrow{s,l}{ }
    \node{\vdots}
    \arrow{s,l}{ }
    \node{\vdots}
    \arrow{s,l}{ }
    \\
    \node{\mathscr{Y}_2TA}
    \arrow{s,l}{d_2}
    \node{\mathscr{Y}_2T^2A}
    \arrow{w,t}{\mathscr{Y}_2\theta_1}
    \arrow{s,l}{-d_2}
    \node{\mathscr{Y}_2T^3A}
    \arrow{w,t}{\mathscr{Y}_2\theta_2}
    \arrow{s,l}{d_2}
    \node{\cdots}
    \arrow{w,t}{ }
    \\
    \node{\mathscr{Y}_1TA}
    \arrow{s,l}{d_1}
    \node{\mathscr{Y}_1T^2A}
    \arrow{w,t}{\mathscr{Y}_1\theta_1}
    \arrow{s,l}{-d_1}
    \node{\mathscr{Y}_1T^3A}
    \arrow{w,t}{\mathscr{Y}_2\theta_2}
    \arrow{s,l}{d_1}
    \node{\cdots}
    \arrow{w,t}{ }
    \\
    \node{\mathscr{Y}_0TA}
    \node{\mathscr{Y}_0T^2A}
    \arrow{w,t}{\mathscr{Y}_0\theta_1}
    \node{\mathscr{Y}_0T^3A}
    \arrow{w,t}{\mathscr{Y}_1\theta_2}
    \node{\cdots}
    \arrow{w,t}{}
  \end{diagram}
\end{equation}

Consider a second double complex, $\mathscr{A}_{*,*}$, consisting of the complex
$\mathscr{Y}_*A$ as the $0^{th}$ column, and $0$ in all positive columns.

\begin{equation}\label{eq.scriptA}
  \begin{diagram}
    \node{\vdots}
    \arrow{s,l}{ }
    \node{\vdots}
    \arrow{s,l}{ }
    \node{\vdots}
    \arrow{s,l}{ }
    \\
    \node{\mathscr{Y}_2A}
    \arrow{s,l}{d_2}
    \node{0}
    \arrow{w}
    \arrow{s}
    \node{0}
    \arrow{w}
    \arrow{s}
    \node{\cdots}
    \arrow{w,t}{ }
    \\
    \node{\mathscr{Y}_1A}
    \arrow{s,l}{d_1}
    \node{0}
    \arrow{w}
    \arrow{s}
    \node{0}
    \arrow{w}
    \arrow{s}
    \node{\cdots}
    \arrow{w,t}{ }
    \\
    \node{\mathscr{Y}_0A}
    \node{0}
    \arrow{w}
    \node{0}
    \arrow{w}
    \node{\cdots}
    \arrow{w,t}{}
  \end{diagram}
\end{equation}

\begin{theorem}\label{thm.doublecomplexiso}
  There is a map of double complexes, $\Theta_{*,*} : \mathscr{T}_{*,*} \to
  \mathscr{A}_{*,*}$ inducing isomorphism in homology
  \[
    H_*\left( \mathrm{Tot}(\mathscr{T});\,k\right) \to
    H_*\left( \mathrm{Tot}(\mathscr{A});\, k\right)
  \]
\end{theorem}
\begin{proof}
  The map $\Theta_{*,*}$ is defined as:
  \[
    \Theta_{p,q} := \left\{\begin{array}{ll}
                             \mathscr{Y}_q\theta_A, & p = 0 \\
                             0, & p > 0
                           \end{array}\right.
  \]
  This map is easily verified to be a map of double complexes, since most components
  of $\mathscr{A}_{*,*}$ are $0$.  On the $0^{th}$ column, we just verify that
  $d_q(A) \circ \mathscr{Y}_q\theta_A = \mathscr{Y}_{q-1}\theta_A 
  \circ d_q(TA)$, but this
  follows since $\mathscr{Y}_*$ is functorial ($\mathscr{Y}_*\theta_A$ is
  a chain map).  The isomorphism follows from the exactness of the
  sequence~(\ref{eq.ex-seq-TA}).  
\end{proof}

\begin{rmk}\label{rmk.HS-A-bicomplex}
  Observe that
  \[
     H_*\left( \mathrm{Tot}(\mathscr{A});\, k\right)
     = H_*\left(\mathscr{Y}_*A;\,k\right) = HS_*(A).
  \]
\end{rmk}

This permits the computation of symmetric homology of any given algebra $A$ in
terms of tensor algebras:

\begin{cor}\label{cor.HS_A_via-tensoralgebras}
  For an associative, unital $k$-algebra $A$,
  \[
    HS_*(A) \cong H_*\left( \mathrm{Tot}(\mathscr{T});\,k  \right),
  \]
  where $\mathscr{T}_{*,*}$ is the double complex $\{ \mathscr{Y}_q T^{p+1}A \}_{p,q \geq 0}$.
\end{cor}

The following lemma shows why it is advantageous to work with tensor algebras.

\begin{lemma}\label{lem.tensor_alg_constr}
  For a unital, associative $k$-algebra $A$, 
  there is an isomorphism of $k$-complexes:
  \begin{equation}\label{eq.YA-decomp}
    \mathscr{Y}_*TA \cong \bigoplus_{n\geq -1} Y_n,
  \end{equation}
  where
  \[
    Y_n = \left\{\begin{array}{ll}
                    k\left[N(\Delta S)\right], &n = -1\\
                    k\left[N([n] \setminus \Delta S)\right]
                      \otimes_{k\Sigma_{n+1}^\mathrm{op}} A^{\otimes(n+1)},
                      &n \geq 0
                    \end{array}\right.
  \]
  Moreover, the differential respects the direct sum decomposition.
\end{lemma}
\begin{proof}
  Any generator of $\mathscr{Y}_*TA$
  has the form
  \[
    [p] \stackrel{\alpha}{\to} [q_0] \to \ldots \to [q_n] \otimes u,
  \]
  where
  \[
    u = \left(\bigotimes_{a\in A_0}a\right) \otimes
              \left(\bigotimes_{a\in A_1}a\right) \otimes \ldots \otimes
              \left(\bigotimes_{a\in A_p}a\right),
  \]
  and $A_0, A_1, \ldots, A_p$ are finite ordered lists of elements of $A$.  Indeed,
  each $A_j$ may be thought of as an element of $A^{m_j}$ (set product).  If $A_j = \emptyset$,
  then $m_j = 0$, and we use the convention that an empty tensor product is
  equal to $1 \in k$. We say that the corresponding tensor factor is
  {\it trivial}.  (Caution, $1_A \in A$ is not considered trivial, since it
  has degree $1$ in the tensor algebra.)  Now,
  let $m = \left(\sum m_j\right) - 1$.  We shall use the convention that
  $A^0 = \emptyset$.  Let
  \[
    \pi \;:\; A^{m_0} \times A^{m_1} \times \ldots \times
    A^{m_p} \longrightarrow A^{m+1}
  \]
  be the evident isomorphism.  Let $A_m = \pi( A_0, A_1, \ldots, A_p )$.  
  
  {\bf Case 1.} If $u$ is
  non-trivial (\textit{i.e.}, $A_m \neq \emptyset$), then construct the element
  \[
    u' = \bigotimes_{a \in A_m}a
  \]
  Next, construct a $\Delta$-morphism $\zeta_u : [m] \to [p]$ as follows:  $\zeta_u$
  sends the point $0, 1, \ldots, m_0-1$ identically onto $0$, then
  sends the the points $m_0, m_0 + 1, \ldots, m_0+m_1 -1 $ onto $1$, etc.  It should
  be clear that $(\zeta_u)_*(u') = u$.  An example will clarify.  Suppose
  \[
    u = (a_0 \otimes a_1) \otimes 1 \otimes (a_2 \otimes a_3 \otimes a_4) \in
    (A \otimes A) \oplus k \oplus (A \otimes A \otimes A).
  \]  
  Then 
  \[
    u' = a_0 \otimes a_1 \otimes a_2 \otimes a_3 \otimes a_4,
  \]
  and $\zeta_u : [4] \to [2]$ has
  preimages: $\zeta_u^{-1}(0) = \{0, 1\}, \zeta_u^{-1}(1) = \emptyset, \zeta^{-1}(2)
  = \{2, 3, 4\}$ (or, in tensor notation, $\zeta_u = x_0x_1 \otimes 1 \otimes x_2x_3x_4$).
  Note, the elements $a_i$ need not be distinct.
  
  Then under the $\Delta S$-equivalence,
  \[
    [p] \stackrel{\alpha}{\to} [q_0] \to \ldots \to [q_n] \otimes u =
    [p] \stackrel{\alpha}{\to} [q_0] \to \ldots \to [q_n] \otimes (\zeta_u)_*(u')
  \]
  \[
    \approx
    [m] \stackrel{\alpha\zeta_u}{\to} [q_0] \to \ldots \to [q_n] \otimes u'
  \]
  
  The assignment is well-defined with respect to
  the $\Delta S$-equivalence since the total number of non-trivial tensor
  factors in $u$ is the same as the total number of non-trivial tensor factors
  in $\phi_*(u)$ for any $\phi \in \mathrm{Mor}\Delta S$.  It is this property
  of tensor algebras that is essential in making the proof work.

  Note that the only equivalence that persists after rewriting the generators
  is invariance under the symmetric group action:
  \[
    [m] \stackrel{\alpha\sigma}{\to}
    [q_0] \to \ldots \to [q_n] \otimes u' \approx
    [m] \stackrel{\alpha}{\to} 
    [q_0] \to \ldots \to [q_n] \otimes \sigma_*(u'), \;\; \textrm{for $\sigma \in
    \Sigma_{m+1}^\mathrm{op}$}
  \]

  This shows that any such non-trivial element in $\mathscr{Y}_*TA$ may be written 
  uniquely as an element of 
  \[
    k\left[N([m] \setminus \Delta S)\right] \otimes_{k\Sigma_{m+1}^\mathrm{op}} 
    A^{\otimes(m+1)}.
  \]
 
  {\bf Case 2.} If $u$ is
  trivial (\textit{i.e.}, $A_m = \emptyset$), then
  \[
    [p] \stackrel{\alpha}{\to} [q_0] \to \ldots \to [q_n] \otimes u
    = [p] \stackrel{\alpha}{\to} [q_0] \to \ldots \to [q_n] \otimes 1_k^{\otimes(p + 1)}.
  \]
  This element is equivalent to:
  \[
    [q_0] \stackrel{\mathrm{id}}{\to} [q_0] \to \ldots \to [q_n] \otimes 
    1_k^{\otimes(q_0 + 1)},
  \]
  and so this element can be identified with $[q_0] \to \ldots \to [q_n] \in 
  k\left[N(\Delta S)\right]$.
  
  Thus, the isomorphism~(\ref{eq.YA-decomp}) is proven.  Note, the fact that total number
  of non-trivial tensor factors is preserved under $\Delta S$ morphisms also proves that
  the differential respects the direct sum decomposition.
\end{proof}

\section{Symmetric Homology with Coefficients}\label{sec.symhom_coeff}%

Following the conventions for Hochschild and cyclic homology in 
Loday~\cite{L}, when we need to indicate explicitly the ground ring $k$ 
over which we compute symmetric homology of $A$, we shall use the notation:
\[
  HS_*(A\;|\;k)
\]
Furthermore, since the notion ``$\Delta S$-module'' does not explicitly state the
ground ring we shall use the bulkier ``$\Delta S$-module over $k$'' when the
context is ambiguous.  

If $\mathscr{Y}_*$ is a complex that computes symmetric homology of
the algebra $A$ over $k$, we may
make the following definition:

\begin{definition}\label{def.HS-with-coeff}
  The symmetric homology of $A$ over $k$, with coefficients in a left 
  $k$-module $M$ is
  \[
    HS_*(A;M) := H_*( \mathscr{Y}_* \otimes_k M)
  \]
\end{definition}

Note, this definition is independent of the particular choice of complex $\mathscr{Y}_*$,
so we shall generally use the complex $\mathscr{Y}_*A =
k[ N(- \setminus \Delta S) ]
\otimes_{\Delta S} B_*^{sym}A$ of Cor.~\ref{cor.SymHomComplex} in this section.

\begin{prop}\label{prop.M-flat}
  If $M$ is flat over $k$, then
  \[
    HS_*(A ; M) \cong HS_*(A) \otimes_k M
  \]
\end{prop}
\begin{proof}
  Since $M$ is $k$-flat, the functor $ - \otimes_k M$ is exact, and so
  commutes with homology functors.   In particular,
  \[
    H_n(\mathscr{Y}_*A \otimes_k M) \cong H_n(\mathscr{Y}_*A) \otimes_k M
  \]
\end{proof}

\begin{cor}\label{cor.HS-Q-Z_p}
  For any $\Z$-agebra $A$, $HS_n(A ; \mathbb{Q}) \cong HS_n(A) \otimes_{\Z} 
  \mathbb{Q}$.
\end{cor}

\begin{lemma}\label{lem.YA-flat}

  {\it i.} If $A$ is a flat $k$-algebra, then $\mathscr{Y}_nA$ is flat for each $n$.
  
  {\it ii.} If $A$ is a projective $k$-algebra, then $\mathscr{Y}_nA$ is projective for
  each $n$.
\end{lemma}
\begin{proof}
  By Remark~\ref{rmk.uniqueChain} we may identify:
  \[
    \mathscr{Y}_nA \cong \bigoplus_{m \geq 0}\left(
    k[ N([m] \setminus\Delta S)_{n-1} ] \otimes_k  A^{\otimes(m+1)}\right).
  \]
  Note, $k[ N([m] \setminus\Delta S)_{n-1} ]$ is free, so if $A$ is
  flat, then $\mathscr{Y}_nA$ is a direct sum of modules that are 
  tensor products of free with flat modules, hence $\mathscr{Y}_nA$ is flat.
  Similarly, if $A$ is projective, $\mathscr{Y}_nA$ is also, since tensor products
  and direct sums of projectives are projective.
\end{proof}

\begin{prop}\label{prop.HS-A-B}
  If $B$ is a commutative $k$-algebra, then there is an 
  isomorphism
  \[
    HS_*(A \otimes_k B \;|\; B) \cong HS_*(A ; B)
  \]
\end{prop}
\begin{proof}
  Here, we are viewing $A \otimes_k B$ as a $B$-algebra via the inclusion
  $B \cong 1_A \otimes_k B \hookrightarrow A \otimes_k B$.  Observe, there
  is an isomorphism 
  \[
    (A \otimes_k B) \otimes_B (A \otimes_k B) \stackrel{\cong}{\longrightarrow}
    A \otimes_k A \otimes_k (B \otimes_B B)\stackrel{\cong}{\longrightarrow}
    (A \otimes_k A) \otimes_k B.
  \]
  Iterating this for $n$-fold tensors of $A \otimes_k B$,
  \[
    \underbrace{(A \otimes_k B) \otimes_B \ldots \otimes_B (A \otimes_k B)}_{n}
     \cong \underbrace{A \otimes_k \ldots \otimes_k A}_n \otimes_k B
  \]
  This shows that the $\Delta S$-module over $B$, $B_*^{sym}(A \otimes_k B)$
  is isomorphic as $k$-module to \mbox{$(B_*^{sym}A) \otimes_k B$} over $k$.
  The proposition now follows essentially by definition.  Let $\overline{L}_*$
  be the resolution of $\underline{k}$ by projective $\Delta S^\mathrm{op}$-modules
  (over $k$) given by $\overline{L}_* = k[N(- \setminus \Delta S)]$.  Then, if we take
  tensor products (over $k$) with the algebra $B$, we obtain
  \[
    \overline{L}_* \otimes_k B \cong B[N(- \setminus \Delta S)],
  \]
  which is a projective resolution of the trivial $\Delta S^\mathrm{op}$-module over $B$,
  $\underline{B}$.  Thus,
  \begin{equation}\label{eq.H_AotimesB_B}
    HS_*(A \otimes_k B \;|\; B) = H_*\left( (\overline{L}_* \otimes_k B) 
    \otimes_{B[\mathrm{Mor}\Delta S]}
    B_*^{sym}(A \otimes_k B);\;B \right)
  \end{equation}
  On the chain level, there are isomorphisms:
  \[
    (\overline{L}_* \otimes_k B) 
    \otimes_{B[\mathrm{Mor}\Delta S]}
    B_*^{sym}(A \otimes_k B) \cong (\overline{L}_* \otimes_k B) 
    \otimes_{B[\mathrm{Mor}\Delta S]}
    (B_*^{sym}A \otimes_k B)
  \]
  \begin{equation}\label{eq.H_AotimesB_B-red}
    \cong (\overline{L}_* \otimes_{k[\mathrm{Mor}\Delta S]}
     B_*^{sym}A) \otimes_k B
  \end{equation}
  The complex~(\ref{eq.H_AotimesB_B-red}) computes $HS_*( A ; B)$
  by definition.
\end{proof}

\begin{rmk}
  Since $HS_*(A \;|\; k) = HS_*(A \otimes_k k\;|\; k)$, Prop.~\ref{prop.HS-A-B} allows
  us to identify $HS_*(A \;|\; k)$ with $HS_*(A ; k)$.
\end{rmk}

The construction $HS_*(A ; -)$ is a covariant functor, as
is immediately seen on the chain level.  Moreover, Prop.~\ref{prop.Y-functor}
implies $HS_*( - ; X)$ is a covariant functor for any left $k$-module, $X$.

\begin{prop}\label{prop.les-for-HS}
  Suppose $0 \to X \to Y \to Z \to 0$ is a short exact sequence of left $k$-modules,
  and suppose $A$ is a flat $k$-algebra.
  Then there is an induced long exact sequence in symmetric homology:
  \begin{equation}\label{eq.les_HS}
    \ldots \to HS_n(A ; X) \to HS_n(A ; Y) \to HS_n(A ; Z) \to HS_{n-1}(A ; X) \to \ldots
  \end{equation}
  Moreover, a map of short exact sequences, $(\alpha, \beta, \gamma)$, as in the
  diagram below, induces a map of the corresponding long exact sequences 
  (commutative ladder)
  \begin{equation}\label{eq.ses-morphism}
    \begin{diagram}
      \node{0}
      \arrow{e}
      \node{X}
      \arrow{e}
      \arrow{s,l}{\alpha}
      \node{Y}
      \arrow{e}
      \arrow{s,l}{\beta}
      \node{Z}
      \arrow{e}
      \arrow{s,l}{\gamma}
      \node{0}
      \\
      \node{0}
      \arrow{e}
      \node{X'}
      \arrow{e}
      \node{Y'}
      \arrow{e}
      \node{Z'}
      \arrow{e}
      \node{0}
    \end{diagram}    
  \end{equation}
\end{prop}
\begin{proof}
  By Lemma~\ref{lem.YA-flat}, the hypothesis $A$ is flat implies that the following 
  is an exact sequence of chain complexes:
  \[
    0 \to \mathscr{Y}_*A \otimes_k X \to \mathscr{Y}_*A \otimes_k Y \to
    \mathscr{Y}_*A \otimes_k Z \to 0.
  \]
  This induces a long exact sequence in homology
  \[
    \ldots \to H_n(\mathscr{Y}_*A \otimes_k X) \to H_n(\mathscr{Y}_*A \otimes_k Y) \to
    H_n(\mathscr{Y}_*A \otimes_k Z) \to H_{n-1}(\mathscr{Y}_*A \otimes_k X)
    \to \ldots
  \]
  as required.
  
  Now let $(\alpha, \beta, \gamma)$ be a morphism of short exact sequences, as
  in diagram~(\ref{eq.ses-morphism}).  
  Consider the diagram,
  \begin{equation}\label{eq.les-morphism}
    \begin{diagram}
      \node{ \vdots }
      \arrow{s}
      \node{ \vdots }
      \arrow{s}
      \\
      \node{ HS_n(A ; X) }
      \arrow{s}
      \arrow{e,t}{ \alpha_* }
      \node{ HS_n(A ; X') }
      \arrow{s}
      \\
      \node{ HS_n(A ; Y) }
      \arrow{s}
      \arrow{e,t}{ \beta_* }
      \node{ HS_n(A ; Y') }
      \arrow{s}
      \\
      \node{ HS_n(A ; Z) }
      \arrow{s,l}{\partial}
      \arrow{e,t}{ \gamma_* }
      \node{ HS_n(A ; Z') }
      \arrow{s,l}{\partial'}
      \\
      \node{ HS_{n-1}(A ; X) }
      \arrow{s}
      \arrow{e,t}{ \alpha_* }
      \node{ HS_{n-1}(A ; X') }
      \arrow{s}
      \\      
      \node{ \vdots }
      \node{ \vdots }
    \end{diagram}
  \end{equation}      
  Since $HS_n(A ; -)$ is functorial, the upper two squares of the diagram commute.
  Commutativity of the lower square follows from the naturality of the connecting 
  homomorphism in the snake lemma.
\end{proof}

\begin{rmk}\label{rmk.les_functors}
  Any family of additive covariant functors $\{T_n\}$ between two abelian categories
  is said to be a {\it long exact sequence of functors} if it takes
  short exact sequences to long exact sequences such as~(\ref{eq.les_HS}) 
  and morphisms of short exact
  sequences to commutative ladders of long exact sequences such as~(\ref{eq.les-morphism}).  
  See~\cite{D}, 
  Definition~1.1 and also~\cite{Mc}, section~12.1.  The content of
  Prop.~\ref{prop.les-for-HS} is that for $A$ flat, $\{HS_n(A ; - )\}_{n \in \Z}$ is a 
  long exact sequence of functors.
\end{rmk}

We now state the {\it Universal Coefficient Theorem for symmetric homology}.
\begin{theorem}\label{thm.univ.coeff.}
  If $A$ is a flat $k$-algebra, and $B$ is a commutative $k$-algebra, then there is 
  a spectral sequence with
  \[
    E_2^{p,q} := \mathrm{Tor}^k_p\left( HS_q( A \;|\; k) , B \right) \Rightarrow
    HS_*( A ; B).
  \]
\end{theorem}
\begin{proof}
  Let $T_q : k\textrm{-$\mathbf{Mod}$} \to k\textrm{-$\mathbf{Mod}$}$ be the functor
  $HS_q( A ; - )$.  Observe, since $A$ is flat, $\{T_q\}$ is a long exact sequence 
  of additive covariant functors 
  (Rmk.~\ref{rmk.les_functors} and Prop.~\ref{prop.les-for-HS}); $T_q = 0$ for 
  sufficiently small $q$ (indeed, for
  $q < 0$); and $T_q$ commutes with arbitrary direct sums, since tensoring and
  taking homology always commutes with direct sums.
  Hence, by the Universal Coefficient Theorem of Dold (2.12 of~\cite{D}.  See also
  McCleary~\cite{Mc}, Thm.~12.11), there is a spectral sequence
  with
  \[
    E_2^{p,q} := \mathrm{Tor}^k_p\left( T_q(k) , B \right) \Rightarrow
    T_*(B).
  \]
\end{proof}

As an immediate consequence, we have the following result.
\begin{cor}\label{cor.iso_in_HS_with_coeff.}
  If $f : A \to A'$ is a $k$-algebra map between flat algebras which induces
  an isomorphism in symmetric homology, $HS_*(A) \stackrel{\cong}{\to} HS_*(A')$,
  then for a commutative $k$-algebra $B$, the map $f \otimes \mathrm{id}_B$ induces
  an isomorphism $HS_*(A;B) \stackrel{\cong}{\to} HS_*(A' ; B)$.
\end{cor}

Under stronger hypotheses, the universal coefficient spectral sequence
reduces to short exact sequences.  Recall some notions of ring theory
(c.f. the article Homological Algebra: Categories of Modules (200:K), Vol. 1,
pp. 755-757 of~\cite{I}).  A commutative ring $k$ is said to have 
{\it global dimension $\leq n$} if for all \mbox{$k$-modules} $X$ and $Y$, 
$\mathrm{Ext}_k^m(X,Y) = 0$ 
for $m > n$.  $k$ is said to have {\it weak global dimension $\leq n$} if for
all \mbox{$k$-modules} $X$ and $Y$, $\mathrm{Tor}_m^k(X, Y) = 0$ for $m>n$. 
Note, the weak global dimension of a ring is less than or equal to its
global dimension, with equality holding for Noetherian rings but not in
general.  A ring is said to be {\it hereditary} if all submodules of projective
modules are projective, and this is equivalent to the global dimension of
the ring being no greater than $1$.
\begin{theorem}\label{thm.univ.coeff.ses}
  If $k$ has weak global dimension $\leq 1$,
  then the spectral sequence 
  of Thm.~\ref{thm.univ.coeff.} reduces to 
  short exact sequences,
  \begin{equation}\label{eq.UCTses}
    0 \longrightarrow HS_n(A\;|\;k) \otimes_k B \longrightarrow
    HS_n(A ; B) \longrightarrow \mathrm{Tor}^k_1( HS_{n-1}(A \;|\;k), B)
    \longrightarrow 0.
  \end{equation}
  Moreover, if $k$ is hereditary and and $A$ is projective over $k$, then
  these sequences split (unnaturally).
\end{theorem}
\begin{proof}
  Assume first that $k$ has weak global dimension $\leq 1$.  
  So $\mathrm{Tor}_p^k(T_q(k), B) = 0$ for all $p > 1$.  Following Dold's
  argument (Corollary~2.13 of~\cite{D}), we obtain the required exact sequences,
  \[
    0 \longrightarrow T_n(k)\otimes_k B \longrightarrow
    T_n(B) \longrightarrow \mathrm{Tor}^k_1( T_{n-1}(k), B )
    \longrightarrow 0.
  \]
  Assume further that $k$ is hereditary and $A$ is projective.  
  Then by Lemma~\ref{lem.YA-flat}, $\mathscr{Y}_nA$ is projective for each n. 
  Theorem~8.22 of Rotman~\cite{R3} then gives us the desired splitting.
\end{proof}

\begin{rmk}
  The proof given above also proves UCT for cyclic homology.  A partial result
  along these lines exists in Loday (\cite{L}, 2.1.16).  
  There, he shows 
  $HC_*(A\;|\; k) \otimes_k K \cong HC_*(A \;|\; K)$
  and \mbox{$HH_*(A\;|\; k) \otimes_k K \cong$} \mbox{$HH_*(A \;|\; K)$} in the case 
  that $K$
  is a localization of $k$, and $A$ is a $K$-module, flat over $k$.  I am
  not aware of a statement of UCT for cyclic or Hochschild homology in its full
  generality in the literature.
\end{rmk}

For the remainder of this section, we shall obtain a converse to 
Cor.~\ref{cor.iso_in_HS_with_coeff.} in the case $k = \Z$.
\begin{theorem}\label{thm.conv.HS_iso}
  Let $f : A \to A'$ be an algebra map between torsion-free $\Z$-algebras.
  Suppose for $B = \mathbb{Q}$ and $B = \Z/p\Z$ for
  any prime $p$, the map  
  $f \otimes \mathrm{id}_B$ induces an isomorphism
  $HS_*(A ; B) \to HS_*(A' ; B)$.  Then $f$ also induces an isomorphism 
  $HS_*(A) \stackrel{\cong}{\to} HS_*(A')$.
\end{theorem}

First, note that Prop.~\ref{prop.les-for-HS} allows one to construct the Bockstein
homomorphisms 
\[
  \beta_n : HS_n(A ; Z) \to HS_{n-1}(A ; X)
\]
associated to a short exact sequence of $k$-modules, $0 \to X \to Y \to Z \to 0$, 
as long as $A$ is flat over $k$.  These Bocksteins are natural in the following
sense:
\begin{lemma}\label{lem.bockstein}
  Suppose $f : A \to A'$ is a map of flat $k$-algebras.  and $0 \to X \to Y \to Z \to 0$
  is a short exact sequence of left $k$-modules.  Then the following diagram
  is commutative for each $n$:
  \[
    \begin{diagram}
      \node{HS_n(A; Z)}
      \arrow{e,t}{\beta}
      \arrow{s,l}{f_*}
      \node{HS_{n-1}(A; X)}
      \arrow{s,l}{f_*}
      \\
      \node{HS_n(A'; Z)}
      \arrow{e,t}{\beta'}
      \node{HS_{n-1}(A'; X)}
    \end{diagram}
  \]
  Moreover if the induced map $f_*: HS_*(A;W) \to HS_*(A';W)$
  is an isomorphism for any two of $W=X$, $W=Y$, $W=Z$, then it is an isomorphism
  for the third.
\end{lemma}
\begin{proof}
  $A$ and $A'$ flat imply both sequences of complexes are exact:
  \[
    0 \to \mathscr{Y}_*A \otimes_k X \to \mathscr{Y}_*A \otimes_k Y \to
    \mathscr{Y}_*A \otimes_k Z \to 0.
  \]
  \[
    0 \to \mathscr{Y}_*A' \otimes_k X \to \mathscr{Y}_*A' \otimes_k Y \to
    \mathscr{Y}_*A' \otimes_k Z \to 0.
  \]
  The map $\mathscr{Y}_*A \to \mathscr{Y}_*A'$ induces a map of short
  exact sequences, hence induces a commutative ladder of long exact
  sequences of homology groups.  In particular, the squares involving the
  boundary maps (Bocksteins) must commute.
  
  Now, assuming further that $f_*$ induces isomorphisms $HS_*(A;W) \to HS_*(A';W)$
  for any two of $W = X$, $W = Y$, $W = Z$, let $V$ be the third module.  The
  $5$-lemma implies isomorphisms $HS_n(A; V) \stackrel{\cong}{\longrightarrow}
  HS_n(A'; V)$ for each $n$.
\end{proof}

We shall now proceed with the proof of Thm.~\ref{thm.conv.HS_iso}.
All tensor products will be over $\Z$ for the rest of this section.

\begin{proof}
  Let $A$ and $A'$ be torsion-free $\Z$-modules.  Over $\Z$, torsion-free implies
  flat.  Let $f : A \to A'$ be an algebra map inducing isomorphism in
  symmetric homology with coefficients in $\mathbb{Q}$ and also in $\Z/p\Z$
  for any prime $p$.
  For $m \geq 2$, there is a short exact sequence,
  \[
    0 \longrightarrow \Z/p^{m-1}\Z \stackrel{p}{\longrightarrow} \Z/p^m\Z
    \longrightarrow \Z/p\Z \longrightarrow 0.
  \]
  Consider first the case $m = 2$.  Since $HS_*(A ; \Z/p\Z) \to HS_*(A' ; \Z/p\Z)$ 
  is an isomorphism, Lemma~\ref{lem.bockstein} implies the induced map
  is an isomorphism for the middle term:
  \begin{equation}\label{eq.HS_A_Zp2}
    f* : HS_*(A; \Z/p^2\Z) \stackrel{\cong}{\longrightarrow} HS_*(A'; \Z/p^2\Z)
  \end{equation}
  (Note, all maps induced by $f$ on symmetric homology will be denoted by
  $f_*$.)
  
  For the inductive step, fix $m > 2$ and suppose $f$ induces an isomorphism in
  symmetric homology,
  $f_* : HS_*(A; \Z/p^{m-1}\Z) \stackrel{\cong}{\longrightarrow} HS_*(A'; \Z/p^{m-1}\Z)$.
  Again, Lemma~\ref{lem.bockstein} implies the induced map is an isomorphism
  on the middle term.
  \begin{equation}\label{eq.HS_A_Zpm}
    f_* : HS_*(A; \Z/p^{m}\Z) \stackrel{\cong}{\longrightarrow} HS_*(A'; \Z/p^{m}\Z)
  \end{equation}
  
  Denote $\ds{\Z / p^\infty \Z := \lim_{\longrightarrow} \Z/ p^m\Z}$.  Note, this 
  is a {\it direct limit} in the sense that it is a colimit over a directed
  system.  The direct limit functor is exact (Prop.~5.3 of~\cite{S2}),
  so the maps $HS_n(A ; \Z / p^\infty \Z) \to HS_n(A' ; \Z / p^\infty \Z)$ induced
  by $f$ are isomorphisms, given by the chain of isomorphisms below:
  \[
    HS_n(A; \Z / p^\infty \Z)
    \cong
    H_n(\lim_{\longrightarrow} \mathscr{Y}_*A \otimes \Z/p^m\Z)
    \cong
    \lim_{\longrightarrow}H_*(\mathscr{Y}_*A \otimes \Z/p^m\Z)
    \stackrel{f_*}{\longrightarrow}
  \]
  \[
    \qquad\qquad\qquad
    \lim_{\longrightarrow}H_*(\mathscr{Y}_*A' \otimes \Z/p^m\Z)
    \cong
    H_*(\lim_{\longrightarrow} \mathscr{Y}_*A' \otimes \Z/p^m\Z)
    \cong
    HS_*(A'; \Z / p^\infty \Z)
  \]
  (Note, $f_*$ here stands for $\ds{\lim_{\longrightarrow}H_n(\mathscr{Y}_*f \otimes 
  \mathrm{id})}$.)
    
  Finally, consider the short exact sequence of abelian groups,
  \[
    0 \longrightarrow \Z \longrightarrow \mathbb{Q} \longrightarrow
    \bigoplus_{\textrm{$p$ prime}} \Z / p^\infty \Z \longrightarrow 0
  \]
    
  The isomorphism $f_* : HS_*(A; \Z / p^\infty \Z) \to HS_*(A'; \Z / p^\infty \Z)$
  passes to direct sums, giving isomorphisms for each $n$,
  \[
    f_* : HS_{n}\left(A; \bigoplus_p \Z/p^\infty \Z\right) \stackrel{\cong}
    {\longrightarrow}
    HS_{n}\left(A'; \bigoplus_p \Z/p^\infty \Z\right).
  \]
  Together with the assumption that $HS_*(A; \mathbb{Q}) \to HS_*(A';\mathbb{Q})$
  is an isomorphism, another appeal to Lemma~\ref{lem.bockstein} gives the
  required isomorphism in symmetric homology induced by $f$:
  \[
    f_* : HS_n(A) \stackrel{\cong}{\longrightarrow} HS_n(A')
  \]
\end{proof}

\begin{rmk}
  Theorem~\ref{thm.conv.HS_iso} may be useful for determining integral
  symmetric homology, since rational computations are generally simpler (see
  Section~\ref{sec.char0}),
  and computations mod $p$ may be made easier due to the presence of
  additional structure, namely homology operations (see Chapter~\ref{chap.prod}).
\end{rmk}

Finally, we state a result along the lines of McCleary~\cite{Mc}, Thm.~10.3.  Denote
 the torsion submodule of the graded module
$HS_*(A; X)$ by $\tau\left(HS_*(A ; X) \right)$.

\begin{theorem}\label{thm.bockstein_spec_seq}
  Suppose $A$ is free of finite rank over $\Z$.  Then there is a singly-graded
  spectral sequence with
  \[
    E_*^1 := HS_*( A ; \Z/p\Z) \Rightarrow HS_*(A)/\tau\left(HS_*(A) \right)
    \otimes \Z/p\Z,
  \]
  with differential map $d^1 = \beta$, the standard Bockstein map
  associated to $0 \to \Z / p\Z \to \Z / p^2\Z \to \Z / p\Z \to 0$.  Moreover,
  the convergence is strong.
\end{theorem}

The proof McCleary gives on p.~459 carries over to our case intact.
All that is required for this proof is that each $H_n(\mathscr{Y}_*A)$ be a 
finitely-generated abelian group.  The hypothesis that $A$ is finitely-generated,
coupled with a result of Chapter~\ref{chap.spec_seq2}, namely Cor.~\ref{cor.fin-gen},
guarantees this.  Note, over $\Z$, {\it free of finite rank} is equivalent
to {\it flat and finitely-generated}.

Theorem~\ref{thm.bockstein_spec_seq} is a version of the Bockstein spectral
sequence for symmetric homology.

\section{Symmetric Homology of Monoid Algebras}\label{sec.symhommonoid}  %

The symmetric homology for the case of a monoid algebra 
$A = k[M]$ has been studied by Fiedorowicz in~\cite{F}.  In the most general
formulation (Prop. 1.3 of~\cite{F}), we have:
\begin{theorem}\label{thm.HS_monoidalgebra}
  \[
    HS_*(k[M]) \cong H_*\left(B(C_{\infty}, C_1, M); k\right),
  \]
  where $C_1$ is the little $1$-cubes monad, 
  $C_{\infty}$ is the little
  $\infty$-cubes monad, and $B(C_{\infty}, C_1, M)$ is May's functorial
  2-sided bar construction (see~\cite{M}).  
\end{theorem}

The proof makes use of a variant of the
symmetric bar construction:

\begin{definition}
  Let $M$ be a monoid.  Define a functor $B_*^{sym}M : \Delta S \to \textbf{Sets}$ by:
  \[
    B_n^{sym}M := B_*^{sym}M[n] := M^{n+1}, \; \textrm{(set product)}
  \]
  \[  
    B_*^{sym}M(\alpha) : (m_0, \ldots, m_n) \mapsto
       \alpha(m_0, \ldots, m_n), \qquad \textrm{for $\alpha \in \mathrm{Mor}
       \Delta S$}.
  \]
  where $\alpha : [n] \to [k]$ is represented in tensor notation, and evaluation
  at $(m_0, \ldots, m_n)$ is as in definition~\ref{def.symbar}.  (This makes sense, as $M$ 
  is closed under multiplication).
\end{definition}

\begin{definition}
  For a $\mathscr{C}$-set $X$ and $\mathscr{C}^\mathrm{op}$-set $Y$, define the
  $\mathscr{C}$-equivariant set product:
  \[
    Y \times_\mathscr{C} X := \left(\coprod_{C \in \mathrm{Obj}\mathscr{C}}
    Y(C) \times X(C)\right) / \approx, 
  \]
  where the equivalence $\approx$ is generated by the following:  For every morphism
  $f \in \mathrm{Mor}_{\mathscr{C}}(C, D)$, and every $x \in X(C)$ and $y \in Y(D)$, 
  we have $\big(y, f_*(x)\big) \approx \big(f^*(y), x\big)$.
\end{definition}

Note that $B_*^{sym}M$ is a $\Delta S$-set, and also a simplicial set, via the
chain of functors in section~\ref{sec.symbar}.
Let $\mathscr{X}_* := N(- \setminus \Delta S) \times_{\Delta S} B^{sym}_*M$.  
\begin{prop}\label{prop.SymHomComplexMonoid}
  $\mathscr{X}_*$ is a simplicial set whose homology computes $HS_*(k[M])$.
\end{prop}
\begin{proof}
  It is clear that $\mathscr{X}_*$ is a simplicial set.  The standard construction
  for finding the homology of a simplicial set is to create the complex $k[\mathscr{X}_*]$,
  with face maps induced by the face maps of $\mathscr{X}_*$.  Since $M$ is a $k$-basis for
  $k[M]$, $B^{sym}_*M$ acts as a $k$-basis for $B^{sym}_*k[M]$.  Then, observe that
  $k[ N( - \setminus \Delta S) \times_{\Delta S} B_*^{sym}M ] =
   k[N(-\setminus \Delta S)] \otimes_{\Delta S} B_*^{sym}k[M]$.
\end{proof}

If $M = JX_+$ is a free monoid on a generating set $X$, then $k[M] = T(X)$, the
(free) tensor algebra over $k$ on the set $X$.  In this case, we have the following:
\begin{lemma}\label{lem.HS_tensoralg}
  \[
    HS_*\left(T(X)\right) \cong H_*\left( \coprod_{n\geq -1} \widetilde{X}_n; k\right),
  \]
  where
  \[
    \widetilde{X}_n = \left\{\begin{array}{ll}
                               N(\Delta S), &n = -1\\
                               N([n] \setminus \Delta S) 
                               \times_{\Sigma_{n+1}^\mathrm{op}} X^{n+1},
                               &n \geq 0
                             \end{array}\right.
  \]
\end{lemma}
\begin{proof}
  This is a consequence of Lemma~\ref{lem.tensor_alg_constr} when the tensor algebra
  is free, generated by $X = \{ x_i \;|\; i \in \mathscr{I} \}$. By the lemma,
  we obtain a decomposition
  \[
    \mathscr{Y}_*T(X) \cong k\left[N(\Delta S)\right] \oplus 
      \left(\bigoplus_{n \geq 0} k\left[N([n] \setminus \Delta S)\right]
       \otimes_{k\Sigma_{n+1}^\mathrm{op}} k[X^{n+1}]\right),
  \]
  \[
    \cong k\left[ N(\Delta S) \amalg \coprod_{n \geq 0}
    N([n] \setminus \Delta S) \times_{\Sigma_{n+1}^\mathrm{op}} X^{n+1} \right],
  \]
  computing $HS_*\left( T(X) \right)$.
\end{proof}

\begin{rmk}
  This proves Thm.~\ref{thm.HS_monoidalgebra} in the special case that $M$ is a
  free monoid.
\end{rmk}

If $M$ is a group, $\Gamma$, then Fiedorowicz~\cite{F} found:
\begin{theorem}\label{thm.HS_group}
  \[
    HS_*(k[\Gamma]) \cong H_*\left(\Omega\Omega^{\infty}S^{\infty}(B\Gamma); k\right)
  \]
\end{theorem}

This final formulation shows in particular that $HS_*$ is a non-trivial theory.  While
it is true that $H_*(\Omega^{\infty}S^{\infty}(X)) = H_*(QX)$ is well understood, 
the same cannot be said of the homology of $\Omega\Omega^{\infty}S^{\infty}X$.
Indeed, May states that $H_*(QX)$ may be regarded as the
free allowable Hopf algebra with conjugation over the Dyer-Lashof algebra and
dual of the Steenrod algebra (See~\cite{CLM}, preface to chapter 1, and also
Lemma~4.10).  Cohen and Peterson~\cite{CP} found the homology of 
$\Omega\Omega^{\infty}S^{\infty}X$,
where $X = S^0$, the zero-sphere, but there is little hope of extending
this result to arbitrary $X$ using the same methods.

We shall have more to say about $HS_1(k[\Gamma])$ in section~\ref{sec.2-torsion}.

\chapter{ALTERNATIVE RESOLUTIONS}\label{chap.alt}

\section{Symmetric Homology Using $\Delta S_+$}\label{sec.deltas_plus} %

In this section, we shall show that replacing $\Delta S$ by $\Delta S_+$ in an
appropriate way does not affect the computation of $HS_*$.

\begin{definition}\label{def.symbar_plus}
  For an associative, unital algebra, $A$, over a commutative ground ring $k$,  
  define a functor $B_*^{sym_+}A : \Delta S_+ \to 
  k$-\textbf{Mod} by:
  \[
    \left\{
    \begin{array}{lll}
      B_n^{sym_+}A &:=& B_*^{sym}A[n] := A^{\otimes n+1} \\
      B_{-1}^{sym_+}A &:=& k,
    \end{array}
    \right.
  \]
  \[  
    \left\{
    \begin{array}{ll}
      B_*^{sym_+}A(\alpha) : (a_0 \otimes a_1 \otimes \ldots \otimes a_n) \mapsto
       \alpha(a_0, \ldots, a_n), \;&\textrm{for $\alpha \in \mathrm{Mor}\Delta S$},\\
      B_*^{sym_+}A(\iota_n) : \lambda \mapsto \lambda(1_A \otimes \ldots \otimes 1_A),
      \;&(\lambda \in k).
    \end{array}
    \right.
  \]
\end{definition}

Consider the functor $\mathscr{Y}_*^+ : k$-\textbf{Alg} $ \to k$-\textbf{complexes}
given by:
\begin{equation}\label{eq.deltasplus}
   \mathscr{Y}_*^+A \;:=\; k[ N(- \setminus \Delta S_+) ]\otimes_{\Delta S_+} B_*^{sym_+}A.
\end{equation}
\[
  \mathscr{Y}_*^+f = \mathrm{id} \otimes B_*^{sym_+}f
\]
The functoriality of $\mathscr{Y}_*^+$ depends on the naturality of $B_*^{sym_+}f$, 
which follows from the naturality of $B_*^{sym}f$ in most cases.  The only case
to check is on a morphism $[-1] \stackrel{\iota_m}{\to} [m]$.  For the object $[-1]$,
$B_{-1}^{sym_+}f$ will be the identity map of $k$.
\[
  \begin{diagram}
    \node{k}
    \arrow{s,l}{B_*^{sym}A(\iota_m)}
    \arrow{e,t}{\mathrm{id}}
    \node{k}
    \arrow{s,r}{B_*^{sym}B(\iota_m)}
    \\
    \node{A^{\otimes(m+1)}}
    \arrow{e,t}{f^{\otimes(m+1)}}
    \node{B^{\otimes(m+1)}}
  \end{diagram}
\]    
The commutativity of this diagram is clear, since $f(1) = 1$, and $B_*^{sym}A(\iota_m)$,
$B_*^{sym}B(\iota_m)$ are simply the unit maps.

Note, the differential of $\mathscr{Y}_*^+A$ will be denoted $d_*(A)$.  As before,
when the context is clear, the differential will simply be denoted $d_*$.

Our goal is to prove the following:
\begin{theorem}\label{thm.SymHom_plusComplex}
  For an associative, unital $k$-algebra $A$,
  \[
    HS_*(A) = H_*\left( \mathscr{Y}_*^+A;\,k \right)
  \]
\end{theorem}

As the preliminary step, we shall prove the theorem in the special case of
tensor algebras.  We shall need an analogue of Lemma~\ref{lem.tensor_alg_constr}
for $\Delta S_+$.
\begin{lemma}\label{lem.tensor_alg_constr-plus}
  For a unital, associative $k$-algebra $A$, 
  there is an isomorphism of $k$-complexes:
  \begin{equation}\label{eq.YplusA-decomp}
    \mathscr{Y}^+_*TA \cong \bigoplus_{n\geq -1} Y^+_n,
  \end{equation}
  where
  \[
    Y^+_n = \left\{\begin{array}{ll}
                    k\left[N(\Delta S_+)\right], &n = -1\\
                    k\left[N([n] \setminus \Delta S_+)\right]
                      \otimes_{k\Sigma_{n+1}} A^{\otimes(n+1)},
                      &n \geq 0
                    \end{array}\right.
  \]
  Moreover, the differential respects the direct sum decomposition.
\end{lemma}
\begin{proof}
  The proof follows verbatim as the proof of Lemma~\ref{lem.tensor_alg_constr}, only   
  with $\Delta S$ replaced with $\Delta S_+$ throughout.
\end{proof}

\begin{lemma}\label{lem.Y-to-Y-plus}
  There is a chain map $J_A : \mathscr{Y}_*A \to \mathscr{Y}_*^+A$, which is
  natural in $A$.
\end{lemma}
\begin{proof}
  First observe that the the inclusion $\Delta S 
  \hookrightarrow \Delta S_+$ induces an inclusion of nerves:
  \[
    N(-\setminus \Delta S) 
    \hookrightarrow
    N(-\setminus \Delta S_+),
  \]
  which in turn induces the chain map
  \[
    k\left[ N(-\setminus \Delta S) \right] \otimes_{\Delta S} B_*^{sym}A
    \longrightarrow
    k\left[ N(-\setminus \Delta S_+) \right] \otimes_{\Delta S} B_*^{sym}A
  \]
  $k\left[ N(-\setminus \Delta S_+) \right]$ is a right $\Delta S$-module
  as well as a right $\Delta S_+$-module.  Similarly, $B_*^{sym_+}A$ is
  both a left $\Delta S$-module and a left $\Delta S_+$-module.  There is
  a natural transformation \mbox{$B_*^{sym}A \to B_*^{sym_+}A$}, again induced by
  inclusion of categories $\Delta S \hookrightarrow \Delta S_+$, hence there is
  a chain map
  \[
    k\left[ N(-\setminus \Delta S_+) \right] \otimes_{\Delta S} B_*^{sym}A
    \longrightarrow
    k\left[ N(-\setminus \Delta S_+) \right] \otimes_{\Delta S} B_*^{sym_+}A.
  \]
  Finally, pass to tensors over $\Delta S_+$:
  \[
    k\left[ N(-\setminus \Delta S_+) \right] \otimes_{\Delta S} B_*^{sym_+}A
    \longrightarrow
    k\left[ N(-\setminus \Delta S_+) \right] \otimes_{\Delta S_+} B_*^{sym_+}A.    
  \]
  The composition gives a chain map $J_A : \mathscr{Y}_*A \to \mathscr{Y}_*^+A$.  We
  must verify that $J$ is a natural transformation $\mathscr{Y}_* \to \mathscr{Y}_*^+$.
  Suppose $f : A \to B$ is a unital algebra map.
  \[
    \begin{diagram}
      \node{k\left[ N(-\setminus \Delta S) \right] \otimes_{\Delta S} B_*^{sym}A}
      \arrow{s,l}{\mathscr{Y}_*f}
      \arrow{e,t}{J_A}
      \node{k\left[ N(-\setminus \Delta S_+) \right] \otimes_{\Delta S_+} B_*^{sym_+}A}
      \arrow{s,r}{\mathscr{Y}_*^+f}
      \\
      \node{k\left[ N(-\setminus \Delta S) \right] \otimes_{\Delta S} B_*^{sym}B}
      \arrow{e,t}{J_B}
      \node{k\left[ N(-\setminus \Delta S_+) \right] \otimes_{\Delta S_+} B_*^{sym_+}B}
    \end{diagram}
  \]    
  Let $y = [p]\to[q_0] \to \ldots \to [q_n] \otimes (a_0 \otimes \ldots \otimes a_p)$
  be an $n$-chain of $k\left[ N(-\setminus \Delta S) \right] \otimes_{\Delta S} B_*^{sym}A$.
  Then $J_A(y)$ has the same form in
  $k\left[ N(-\setminus \Delta S_+) \right] \otimes_{\Delta S_+} B_*^{sym_+}A$.
  So,
  \[
    \mathscr{Y}_*^+f\left( J_A(y) \right) = 
    [p]\to[q_0] \to \ldots \to [q_n] \otimes \left(f(a_0) \otimes \ldots \otimes f(a_p)\right)
    = J_B\left( \mathscr{Y}_*f(y) \right)
  \]
\end{proof}

Our goal now will be to show the following:
\begin{theorem}\label{thm.J-iso}
  For a unital, associative $k$-algebra $A$, the chain map
  \[
    J_A : \mathscr{Y}_*A \to \mathscr{Y}^+_*A
  \]
  induces an isomorphism on homology
  \[
    H_*\left(\mathscr{Y}_*A;\,k\right) \stackrel{\cong}{\longrightarrow}
    H_*\left(\mathscr{Y}^+_*A;\,k\right)
  \]
\end{theorem}

\begin{lemma}\label{lem.SymHom_plusComplex-tensalg}
  For a unital, associative $k$-algebra $A$, the chain map
  \[
    J_{TA} : \mathscr{Y}_*TA \longrightarrow \mathscr{Y}_*^+TA
  \]
  induces an isomorphism in homology, hence
  \[
    HS_*(TA) = H_*\left( \mathscr{Y}_*^+TA;\,k \right)
  \]
\end{lemma}
\begin{proof}
  \[
    HS_*(TA) = H_*\left( \mathscr{Y}_*TA;\,k \right), \quad\textrm{by definition.}
  \]
  There is a commutative square of complexes:
  \[
    \begin{diagram}
      \node{\mathscr{Y}_*TA}
      \arrow{e,t}{J_{TA}}
      \node{\mathscr{Y}_*^+TA}\\
      \node{\bigoplus_{n\geq -1} Y_n}
      \arrow{n,l}{\cong}
      \arrow{e,t}{j_*}
      \node{\bigoplus_{n \geq -1} Y_n^+}
      \arrow{n,l}{\cong}
    \end{diagram}
  \]
  The isomorphisms on the left and right follow from Lemmas~\ref{lem.tensor_alg_constr}
  and~\ref{lem.tensor_alg_constr-plus}.  The map $j_*$ defined as follows.  For $n = -1$,
  \begin{equation}\label{eq.NDeltaS-to-NDeltaS-plus}
    j_* : k\left[N(\Delta S)\right] \to k\left[N(\Delta S_+)\right]
  \end{equation}
  is induced by inclusion of categories $\Delta S \hookrightarrow \Delta S_+$.
  
  For $n \geq 0$,
  \begin{equation}\label{eq.Y_n-to-Y_n-plus}
    j_* : 
    k\left[N([n] \setminus \Delta S)\right]
                      \otimes_{k\Sigma_{n+1}} A^{\otimes(n+1)} \longrightarrow
    k\left[N([n] \setminus \Delta S_+)\right]
                      \otimes_{k\Sigma_{n+1}} A^{\otimes(n+1)}
  \end{equation}
  is again induced by inclusion of categories.
  
  Observe that $N(\Delta S_+)$ is 
  contractible, since $[-1] \in \mathrm{Obj}(\Delta S_+)$
  is initial.  Thus by Lemma~\ref{lem.DeltaScontractible},
  the map $j_*$ of~(\ref{eq.NDeltaS-to-NDeltaS-plus}) 
  is a homotopy equivalence.  Now, for $n \geq 0$,
  there is equality $N([n] \setminus \Delta S_+) = N([n] \setminus \Delta S)$, since
  there are no morphisms $[n] \to [-1]$ for $n \geq 0$, so~(\ref{eq.Y_n-to-Y_n-plus}),
  and therefore $j_*$, is a homotopy equivalence.  This implies
  that $J_{TA}$ is must also be a homotopy equivalence.
\end{proof}
\begin{rmk}\label{rmk.cyclic-departure}
  Observe, this lemma provides our first major departure from the theory of
  cyclic homology.  The proof above
  would not work over the categories $\Delta C$ and $\Delta C_+$, as $N(\Delta C)$
  is not contractible.
\end{rmk}

Consider a double complex $\mathscr{T}^+_{*,*}$, the analogue of 
complex~(\ref{eq.scriptT}) for $\Delta S_+$.
\begin{equation}\label{eq.scriptT-plus}
  \begin{diagram}
    \node{\vdots}
    \arrow{s,l}{ }
    \node{\vdots}
    \arrow{s,l}{ }
    \node{\vdots}
    \arrow{s,l}{ }
    \\
    \node{\mathscr{Y}^+_2TA}
    \arrow{s,l}{d_2}
    \node{\mathscr{Y}^+_2T^2A}
    \arrow{w,t}{\mathscr{Y}^+_2\theta_1}
    \arrow{s,l}{-d_2}
    \node{\mathscr{Y}^+_2T^3A}
    \arrow{w,t}{\mathscr{Y}^+_2\theta_2}
    \arrow{s,l}{d_2}
    \node{\cdots}
    \arrow{w,t}{ }
    \\
    \node{\mathscr{Y}^+_1TA}
    \arrow{s,l}{d_1}
    \node{\mathscr{Y}^+_1T^2A}
    \arrow{w,t}{\mathscr{Y}^+_1\theta_1}
    \arrow{s,l}{-d_1}
    \node{\mathscr{Y}^+_1T^3A}
    \arrow{w,t}{\mathscr{Y}^+_2\theta_2}
    \arrow{s,l}{d_1}
    \node{\cdots}
    \arrow{w,t}{ }
    \\
    \node{\mathscr{Y}^+_0TA}
    \node{\mathscr{Y}^+_0T^2A}
    \arrow{w,t}{\mathscr{Y}^+_0\theta_1}
    \node{\mathscr{Y}^+_0T^3A}
    \arrow{w,t}{\mathscr{Y}^+_1\theta_2}
    \node{\cdots}
    \arrow{w,t}{}
  \end{diagram}
\end{equation}

The maps $\theta_1, \theta_2, \ldots$ are defined by formula~(\ref{eq.theta_n}).

Consider a second double complex, $\mathscr{A}^+_{*,*}$, the analogue of 
complex~(\ref{eq.scriptA}) for $\Delta S_+$.  It consists
of the complex
$\mathscr{Y}^+_*A$ as the $0^{th}$ column, and $0$ in all positive columns.
\begin{equation}\label{eq.scriptA-plus}
  \begin{diagram}
    \node{\vdots}
    \arrow{s,l}{ }
    \node{\vdots}
    \arrow{s,l}{ }
    \node{\vdots}
    \arrow{s,l}{ }
    \\
    \node{\mathscr{Y}^+_2A}
    \arrow{s,l}{d_2}
    \node{0}
    \arrow{w}
    \arrow{s}
    \node{0}
    \arrow{w}
    \arrow{s}
    \node{\cdots}
    \arrow{w,t}{ }
    \\
    \node{\mathscr{Y}^+_1A}
    \arrow{s,l}{d_1}
    \node{0}
    \arrow{w}
    \arrow{s}
    \node{0}
    \arrow{w}
    \arrow{s}
    \node{\cdots}
    \arrow{w,t}{ }
    \\
    \node{\mathscr{Y}^+_0A}
    \node{0}
    \arrow{w}
    \node{0}
    \arrow{w}
    \node{\cdots}
    \arrow{w,t}{}
  \end{diagram}
\end{equation}

We may think of each double complex construction as a functor:
\[
  \begin{array}{l}
    A \mapsto \mathscr{A}_{*,*}(A)\\
    A \mapsto \mathscr{T}_{*,*}(A)\\
    A \mapsto \mathscr{A}^+_{*,*}(A)\\
    A \mapsto \mathscr{T}^+_{*,*}(A)
  \end{array}
\]
Each functor takes unital morphisms of algebras to maps of double complexes in the
obvious way -- for example if $f : A \to B$, then the induced map 
$\mathscr{T}_{*,*}(A) \to \mathscr{T}_{*,*}(B)$ is defined on the $(p,q)$-component
by the map $\mathscr{Y}_qT^{p+1}f$.  The induced map commutes with vertical
differentials of $\mathscr{A}_{*,*}$ and $\mathscr{T}_{*,*}$ (resp.,
$\mathscr{A}^+_{*,*}$ and $\mathscr{T}^+_{*,*}$)
by naturality of $\mathscr{Y}_*$ (resp. $\mathscr{Y}_*^+$), and it commutes
with the horizontal differentials of $\mathscr{T}_{*,*}$ and $\mathscr{T}^+_{*,*}$
by naturality of $\theta_n$ (see Prop.~\ref{prop.naturality-of-theta_n}).

The map $J$ induces a natural transformation 
\mbox{$J_{*,*} : \mathscr{A}_{*,*} \to \mathscr{A}^+_{*,*}$}, defined by
\[
  J_{p,*}(A) = \left\{\begin{array}{ll}
                  J_A : \mathscr{Y}_*A \to \mathscr{Y}^+_*A, & p = 0\\
                  0, & p > 0 \end{array}\right.
\]
Define a map of bigraded modules, \mbox{$K_{*,*}(A) : \mathscr{T}_{*,*}(A)
\to \mathscr{T}^+_{*,*}(A)$} by:
\[
  K_{p,*}(A) = J_{T^{p+1}A} : \mathscr{Y}_*T^{p+1}A \to \mathscr{Y}^+_*T^{p+1}A
\]
Now, $K_{*,*}(A)$ commutes with the vertical differentials because each 
$J_{T^{p+1}A}$ is a chain map.  $K_{*,*}(A)$ commutes with the horizontal
differentials because of naturality of $J$.  Finally, $K_{*,*}$ defines
a natural transformation $\mathscr{T}_{*,*} \to \mathscr{T}^+_{*,*}$,
again by naturality of $J$.
\multiply \dgARROWLENGTH by2
\[
  \begin{diagram}
    \node{A}
    \arrow{s,l}{f}
    \node{\mathscr{Y}_qT^{p+1}A}
    \arrow{s,l}{\mathscr{Y}_qT^{p+1}f}
    \arrow{e,tb}{K_{p,q}(A)}{= J_{T^{p+1}A}}
    \node{\mathscr{Y}^+_qT^{p+1}A}
    \arrow{s,r}{\mathscr{Y}^+_qT^{p+1}f}\\
    \node{B}
    \node{\mathscr{Y}_qT^{p+1}A}
    \arrow{e,tb}{K_{p,q}(B)}{= J_{T^{p+1}B}}
    \node{\mathscr{Y}^+_qT^{p+1}A}
  \end{diagram}
\]
\divide \dgARROWLENGTH by2

Recall by Thm~\ref{thm.doublecomplexiso}, there is a map of double complexes, 
\[
  \Theta_{*,*}(A) : \mathscr{T}_{*,*}(A) \longrightarrow \mathscr{A}_{*,*}(A)
\]
inducing an isomorphism in homology of the total complexes.  Observe that
$\Theta_{*,*}$ provides a natural transformation $\mathscr{T}_{*,*}
\to \mathscr{A}_{*,*}$, since $\Theta_{*,*}(A)$ is defined in terms of
$\theta_A$, which is natural in $A$.  We shall need the analogous statement for
the double complexes $\mathscr{T}^+_{*,*}$ and $\mathscr{A}^+_{*,*}$.

\begin{theorem}\label{thm.doublecomplexiso-plus}
  For any unital associative algebra, $A$,
  there is a map of double complexes, $\Theta^+_{*,*}(A) : \mathscr{T}^+_{*,*}(A) \to
  \mathscr{A}^+_{*,*}(A)$ inducing isomorphism in homology
  \[
    H_*\left( \mathrm{Tot}(\mathscr{T}^+(A));\,k\right) \to
    H_*\left( \mathrm{Tot}(\mathscr{A}^+(A));\, k\right)
  \]
  Moreover, $\Theta^+_{*,*}$ provides a natural transformation $\mathscr{T}^+_{*,*}
  \to \mathscr{A}^+_{*,*}$.
\end{theorem}
\begin{proof}
  The map $\Theta^+_{*,*}(A)$ is defined as:
  \[
    \Theta^+_{p,q}(A) := \left\{\begin{array}{ll}
                             \mathscr{Y}^+_q\theta_A, & p = 0 \\
                             0, & p > 0
                           \end{array}\right.
  \]
  This map is a map of double complexes by functoriality of $\mathscr{Y}^+_*$,
  and the isomorphism follows from the exactness of the
  sequence~(\ref{eq.ex-seq-TA}).  Naturality of $\Theta_{*,*}$ follows from naturality
  of $\theta$.
\end{proof}

\begin{lemma}\label{lem.comm-diag-functors}
  The following diagram of functors and transformations is commutative.
  \begin{equation}\label{eq.comm-diag-functors}
    \begin{diagram}
      \node{\mathscr{T}_{*,*}}
      \arrow{s,l}{K_{*,*}}
      \arrow{e,t}{\Theta_{*,*}}
      \node{\mathscr{A}_{*,*}}
      \arrow{s,r}{J_{*,*}}\\
      \node{\mathscr{T}^+_{*,*}}
      \arrow{e,t}{\Theta^+_{*,*}}
      \node{\mathscr{A}^+_{*,*}}
    \end{diagram}
  \end{equation}
\end{lemma}
\begin{proof}
  It suffices to fix an algebra $A$ and examine only the $(p,q)$-components.
  \begin{equation}\label{eq.comm-diag-Apq}
    \begin{diagram}
      \node{\mathscr{T}_{p,q}(A)}
      \arrow{s,l}{K_{p,q}(A)}
      \arrow{e,t}{\Theta_{p,q}(A)}
      \node{\mathscr{A}_{p,q}(A)}
      \arrow{s,r}{J_{p,q}(A)}\\
      \node{\mathscr{T}^+_{p,q}(A)}
      \arrow{e,t}{\Theta^+_{p,q}(A)}
      \node{\mathscr{A}^+_{p,q}(A)}
    \end{diagram}
  \end{equation}
  If $p > 0$, then the right hand side of~(\ref{eq.comm-diag-Apq}) is trivial, so
  we may assume $p = 0$.  In this case, diagram~(\ref{eq.comm-diag-Apq}) becomes:
  \begin{equation}\label{eq.comm-diag-A0q}
    \begin{diagram}
      \node{\mathscr{Y}_qTA}
      \arrow{s,l}{(J_{TA})_q}
      \arrow{e,t}{\mathscr{Y}_q\theta_A}
      \node{\mathscr{Y}_qA}
      \arrow{s,r}{(J_A)_q}\\
      \node{\mathscr{Y}^+_qTA}
      \arrow{e,t}{\mathscr{Y}^+_q\theta_A}
      \node{\mathscr{Y}^+_qA}
    \end{diagram}
  \end{equation}
  This diagram commutes because of naturality of $J$.  
\end{proof}

To any double complex $\mathscr{B}_{*,*}$ over $k$, we may associate two spectral sequences:
$(E_{I}\mathscr{B})_{*,*}$, obtained by first taking vertical homology, then
horizontal; and $(E_{II}\mathscr{B})_{*,*}$, obtained by first taking horizontal homology, 
then vertical.  In the case that $\mathscr{B}_{*,*}$ lies entirely within the first quadrant,
both spectral sequences converge to $H_*\left( \mathrm{Tot}(\mathscr{B});\,k
\right)$ (See~\cite{Mc}, Section~2.4).
Maps of double complexes induce maps of spectral sequences, $E_{I}$ and $E_{II}$,
respectively.

Fix the algebra $A$, and consider the commutative diagram of spectral sequences
induced by diagram~(\ref{eq.comm-diag-functors}).
The induced maps will be indicated by an overline, and explicit mention of
the algebra $A$ is suppressed for brevity of notation.
\begin{equation}\label{eq.comm-diag-SpSeqII}
  \begin{diagram}
    \node{E_{II} \mathscr{T}}
    \arrow{s,l}{\overline{K}}
    \arrow{e,t}{\overline{\Theta}}
    \node{E_{II} \mathscr{A}}
    \arrow{s,r}{\overline{J}}\\
    \node{E_{II} \mathscr{T}^+}
    \arrow{e,t}{\overline{\Theta^+}}
    \node{E_{II} \mathscr{A}^+}
  \end{diagram}
\end{equation}

Now, by Thm.~\ref{thm.doublecomplexiso} and Thm.~\ref{thm.doublecomplexiso-plus}, we
know that $\Theta_{*,*}$ and $\Theta^+_{*,*}$ induce isomorphisms on total homology, so
$\overline{\Theta}$ and $\overline{\Theta^+}$ also induce isomorphisms on the limit
term of the spectral sequences.  In fact, both $\overline{\Theta}^r$ and 
$\overline{\Theta^+}^r$ are isomorphisms $(E_{II} \mathscr{T})^r \to
(E_{II} \mathscr{A})^r$ for $r \geq 1$.  This is because taking horizontal homology
of $\mathscr{T}_{*,*}$ (resp. $\mathscr{T}^+_{*,*}$) kills all components in positive 
columns, leaving only
the $0^{th}$ column, which is chain-isomorphic to the $0^{th}$ column of
$\mathscr{A}_{*,*}$ (resp. $\mathscr{A}^+_{*,*}$).  On the other hand,
taking horizontal homology of $\mathscr{A}_{*,*}$ (resp. $\mathscr{A}^+_{*,*}$) does not
change that double complex.

Consider a second diagram of spectral sequences, with induced maps indicated by
a hat.
\begin{equation}\label{eq.comm-diag-SpSeqI}
  \begin{diagram}
    \node{E_{I} \mathscr{T}}
    \arrow{s,l}{\widehat{K}}
    \arrow{e,t}{\widehat{\Theta}}
    \node{E_{I} \mathscr{A}}
    \arrow{s,r}{\widehat{J}}\\
    \node{E_{I} \mathscr{T}^+}
    \arrow{e,t}{\widehat{\Theta^+}}
    \node{E_{I} \mathscr{A}^+}
  \end{diagram}
\end{equation}
Now the map $\widehat{K}$ induces an isomorphism on the limit terms of the sequences
$E_{I} \mathscr{T}$ and $E_{I} \mathscr{T}^+$ as a result of
Lemma~\ref{lem.SymHom_plusComplex-tensalg}.  As before, $\widehat{K}^r$ is
an isomorphism for $r \geq 1$.

Now, since $H_*\left(\mathrm{Tot}(\mathscr{A});\,k\right) = H_*\left(\mathscr{Y}_*A;\,k
\right)$ 
and $H_*\left(\mathrm{Tot}(\mathscr{A}^+);\,k\right) = H_*\left( \mathscr{Y}_*^+A;\,k \right)$,
we can put together a chain of isomorphisms
\[
  \begin{diagram}
    \node{H_*\left(\mathscr{Y}_*A;\,k\right) \cong \left(E_{II} \mathscr{A}\right)^{\infty}_*}
    \node{\left(E_{II} \mathscr{T}\right)^{\infty}_*
      \cong \left(E_{I} \mathscr{T}\right)^{\infty}_*}
    \arrow{w,tb}{\overline{\Theta}^{\infty}}{\cong}
    \arrow{e,tb}{\widehat{K}^{\infty}_*}{\cong}
    \node{\left(E_{I} \mathscr{T}^+\right)^{\infty}_*}
  \end{diagram}
\]
\begin{equation}\label{eq.long-iso-YtoYplus}
  \begin{diagram}
    \node{\cong \left(E_{II} \mathscr{T}^+\right)^{\infty}_*}
    \arrow{e,tb}{(\overline{\Theta^+})^{\infty}_*}{\cong}
    \node{\left(E_{II} \mathscr{A}^+\right)^{\infty}_*
      \cong H_*\left( \mathscr{Y}_*^+A;\,k \right)}
  \end{diagram}
\end{equation}
Commutativity of Diagram~(\ref{eq.comm-diag-functors}) ensures that
the composition of maps in Diagram~\ref{eq.long-iso-YtoYplus} is the map
induced by $J_A$, hence proving Thm.~\ref{thm.J-iso}.

As a direct consequence, $HS_*(A) \cong H_*\left( \mathscr{Y}_*^+A;\,k \right)$,
proving Thm.~\ref{thm.SymHom_plusComplex}.

\section{The Category $\mathrm{Epi}\Delta S$ and a Smaller Resolution}\label{sec.epideltas} %

The complex~(\ref{symhomcomplex}) is an extremely large and unwieldy for computation.
Fortunately, when the algebra $A$ is equipped with an augmentation,
$\epsilon : A \to k$,~(\ref{symhomcomplex}) is 
homotopic to a much smaller subcomplex.  Let 
$I$ be the augmentation ideal, and let $\eta \,:\, k \to A$ be the unit map.
Since $\epsilon \eta = \mathrm{id}_k$, the following  short exact
sequence splits (as $k$-modules):
\[
  0 \to I \to A \stackrel{\epsilon}{\to} k \to 0,
\]
and every $x \in A$ can be written uniquely as $x = a + \eta(\lambda)$ for some $a \in I$,
$\lambda \in k$.  This property will allow $B_n^{sym_+}A$ to be decomposed in a
useful way.

\begin{definition}\label{def.B_JA}
  Suppose $J \subset [n]$.  Define
  \[
    B_{n,J}A := B_0 \otimes B_1 \otimes \ldots \otimes B_n, \quad \textrm{where}\;\;
    B_j = \left\{\begin{array}{ll}
                  I & \textrm{if $j \in J$}\\
                  \eta(k) & \textrm{if $j \notin J$}
                 \end{array}\right.
  \]
  Define $B_{-1,\emptyset}A = k$.
\end{definition}

\begin{lemma}\label{lem.I-decomp}
  For each $n \geq -1$, there is a direct sum decomposition of $k$-modules
  \[
    B_n^{sym_+}A \cong \bigoplus_{J \subset [n]} B_{n,J}A.
  \]
\end{lemma}
\begin{proof}
  The splitting of the unit map $\eta$ implies that $A \cong \eta(k) \oplus I$
  as $k$-modules.  So, for $n \geq 0$,
  \[
    B_n^{sym_+}A = (\eta(k) \oplus I)^{\otimes(n+1)}
    \cong \bigoplus_{J \subset [n]} B_{n,J}A.
  \]
  For $n = -1$, $B_{-1}^{sym_+}A = k = B_{n,\emptyset}A$.  (Recall, $[-1] = \emptyset$).
\end{proof}

\begin{definition}\label{def.basictensors}
  A {\it basic tensor} is any tensor $w_0 \otimes w_1 \otimes \ldots \otimes w_n$,
  where each $w_j$ is in $I$ or is equal to the unit of $A$.  Call a tensor factor
  $w_j$ \textit{trivial} if it is the unit of $A$; otherwise,
  the factor is called \textit{non-trivial}.  If all factors of a basic tensor
  are trivial, then the tensor is called trivial, otherwise non-trivial.
\end{definition}

It will become convenient to include the object $[-1]$ in $\Delta$.  Let $\Delta_+$
be the category with objects $[-1], [0], [1], [2], \ldots$, and morphisms are all
those of $\Delta$ together with $\iota_n : [-1] \to [n]$ for $n \geq -1$.  $\Delta_+$
may be thought of as the subcategory of $\Delta S_+$ consisting of all non-decreasing
set maps.

For a basic tensor $Y \in B_n^{sym_+}A$, we shall define a map 
$\delta_Y \in \mathrm{Mor}\Delta_+$ as follows:
If $Y$ is trivial, let $\delta_Y = \iota_n$.  Otherwise, 
$Y$ has $\overline{n} + 1$ non-trivial factors for some 
$\overline{n} \geq 0$.  Define $\delta_Y : [\overline{n}] \to [n]$ to be the 
unique injective map 
that sends each point $0, 1, \ldots, \overline{n}$ to a point $p \in [n]$ such that $Y$ is
non-trivial at \mbox{factor $p$}.  
Let $\overline{Y}$ be the tensor obtained from
$Y$ by omitting all trivial factors if such exist, or $\overline{Y} := 1$ if $Y$ is
trivial.  Note, $\overline{Y}$ is the unique basic tensor
such that $(\delta_Y)_*(\overline{Y}) = Y$.

\begin{prop}\label{prop.BsymI}
  Any chain $[q]\to [q_0] \to \ldots \to [q_n] \otimes Y \;\in\;
  k[N(-\setminus \Delta S_+)] \otimes_{\Delta S_+} B_*^{sym_+}A$, where $Y$ is a basic 
  tensor, is equivalent to a chain 
  $[\overline{q}] \to [q_0] \to \ldots \to [q_n] \otimes \overline{Y}$,
  where either $\overline{Y}$ has no trivial factors or 
  $\overline{Y} = 1$ and $\overline{q} = -1$.
\end{prop}
\begin{proof}
  Let $\delta_Y$ and $\overline{Y}$ be defined as above, and let 
  $[\overline{q}]$ be the source of $\delta_Y$.
  \[  
    [q] \stackrel{\phi}{\to} [q_0] \to \ldots \to [q_n] \otimes Y \;=\;
    [q] \stackrel{\phi}{\to} [q_0] \to \ldots \to [q_n] \otimes (\delta_Y)_*(\overline{Y})
  \]
  \[
    \approx \; [\overline{q}] \stackrel{\phi\delta_Y}{\to} [q_0] \to \ldots \to [q_n] 
    \otimes \overline{Y}
  \]
\end{proof}
  
Next, we turn our attention to the morphisms in the chain.  Our goal is to reduce to
those chains that involve only epimorphisms.

\begin{definition}
  Let $\mathscr{C}$ be a category.
  The category $\mathrm{Epi}\mathscr{C}$ (resp. $\mathrm{Mono}\mathscr{C}$) is the 
  subcategory 
  of $\mathscr{C}$ consisting of
  the same objects as $\mathscr{C}$ and only those morphisms $f \in \mathrm{Mor}
  \mathscr{C}$
  that are epimorphisms (resp. monomorphisms).
\end{definition}

Note, a morphism $\alpha = (\phi, g) \in \mathrm{Mor}\Delta S_+$ is epic, resp. monic, if
and only if $\phi$ is epic, resp. monic, as morphism in $\Delta_+$.

\begin{prop}\label{prop.decomp}
  Any morphism $\alpha \in \mathrm{Mor}\Delta S_+$ decomposes 
  uniquely as $(\eta, \mathrm{id}) \circ \gamma$, where
  $\gamma \in \mathrm{Mor}(\mathrm{Epi}\Delta S_+)$ and $\eta \in \mathrm{Mor}
  (\mathrm{Mono}\Delta_+)$.
\end{prop}
\begin{proof}
  Suppose $\alpha$ has source $[-1]$ and target $[n]$.  Then $\alpha = \iota_n$ is
  the only possibility, and this decomposes as $\iota_n \circ \mathrm{id}_{[-1]}$.
  Now suppose the source of $\alpha$ is $[p]$ for some $p \geq 0$.
  Write $\alpha = (\phi, g)$, with $\phi \in \mathrm{Mor}\Delta$ and 
  $g \in \Sigma_{p+1}^{\mathrm{op}}$.  We shall decompose
  $\phi$ as follows:
  For $\phi : [p] {\to} [r]$, suppose $\phi$ hits $q+1$ distinct points in $[r]$.
  Then $\pi : [p] \to [q]$ is induced by $\phi$ by maintaining the order of the points hit.
  $\eta$ is the obvious order-preserving monomorphism $[q] \to [r]$ so that 
  $\eta \pi = \phi$
  as morphisms in $\Delta$.  To get the required decomposition in $\Delta S$, use:
  $\alpha = (\eta, \mathrm{id}) \circ (\pi, g)$.
  
  Now, if $(\xi, \mathrm{id})\circ (\psi, h)$ is also a decomposition, with $\xi$
  monic and $\psi$ epic, then
  \[
    (\xi, \mathrm{id})\circ (\psi, h) = (\eta, \mathrm{id}) \circ (\pi, g)
  \]
  \[
    (\xi, \mathrm{id}) \circ (\psi, g^{-1}h) = (\eta, \mathrm{id}) \circ(\phi, \mathrm{id})
  \]
  \[
    (\xi \psi, g^{-1}h) = (\eta\phi, \mathrm{id}),
  \]
  proving $g = h$.  Uniqueness will then follow from uniqueness of such decompositions 
  entirely within the category $\Delta$.  The latter follows from Theorem B.2 
  of~\cite{L}, since any monomorphism (resp. epimorphism) of $\Delta$ can be built
  up (uniquely) as compositions of $\delta_{i}$ (resp. $\sigma_{i}$).
\end{proof}

Explicitly, if $\alpha = X_0 \otimes X_1 \otimes \ldots \otimes X_m : [n] \to [m]$, with
$X_i \neq 1$ for $i = j_0, j_1, \ldots j_k$, then $\mathrm{im}(\alpha)$ is 
isomorphic to the object $[k]$.  The surjection onto $[k]$ is 
$X_{j_0} \otimes \ldots \otimes X_{j_k}$.
The $\Delta$ injection $[k] \hookrightarrow [m]$ is $Z_0 \otimes \ldots \otimes Z_m$, 
where $Z_i = 1$ if
$i \neq j_0, j_1, \ldots j_k$ and for $i = j_0, \ldots, j_k$, the monomials
$Z_i$ are the symbols $x_0, x_1, \ldots, x_k$, in that order.  For example,
\[
  x_2 x_3 \otimes 1 \otimes x_1 \otimes 1 \otimes x_0 =
  x_0 \otimes 1 \otimes x_1 \otimes 1 \otimes x_2 \;\cdot\; x_2 x_3 \otimes x_1 \otimes x_0
\]
When morphisms are not labelled, 
we shall write:
\[
  [p] \to [r] \;=\; [p] \twoheadrightarrow \mathrm{im}([p] \to [r]) \hookrightarrow [r].
\]

For any $p \geq -1$, if $[p] \stackrel{\beta}{\to} [r_1] \stackrel{\alpha}{\to} [r_2]$, 
then there is an induced 
map 
\[
  \mathrm{im}([p] \to [r_1]) \stackrel{\overline{\alpha}}{\to}
  \mathrm{im}([p] \to [r_2])
\]
making the diagram commute:
\begin{equation}\label{eq.epidiagram}
  \begin{diagram}
  \node{ [r_1] }
  \arrow[2]{e,t}{ \alpha }
  \node[2]{ [r_2] }
  \\
  \node[2]{ [p] }
  \arrow{nw,b}{ \beta }
  \arrow{ne,b}{ \alpha\beta }
  \arrow{sw,t,A}{ \pi_1 }
  \arrow{se,t,A}{ \pi_2 }
  \\
  \node{ \mathrm{im}([p] \to [r_1]) }
  \arrow[2]{n,l,J}{ \eta_1 }
  \arrow[2]{e,b,A}{ \overline{\alpha} }  
  \node[2]{ \mathrm{im}([p] \to [r_2]) }
  \arrow[2]{n,r,J}{ \eta_2 }
  \end{diagram}
\end{equation}
$\overline{\alpha}$ is the epimorphism induced from the map $\alpha \eta_1$.  Furthermore,
for morphisms 
\[
  [p] \to [r_1] \stackrel{\alpha_1}{\to} [r_2] \stackrel{\alpha_2}{\to} [r_3],
\]
we have:
\[
  \overline{\alpha_2 \alpha_1} = \overline{\alpha_2} \circ \overline{\alpha_1},
\]
\textit{i.e.}, the epimorphism construction is a functor $( [p] \setminus \Delta S_+ ) \to
( [p] \setminus \mathrm{Epi}\Delta S_+ )$.

Define a variant of the symmetric bar construction:
\begin{definition}
  $B_*^{sym_+}I : \mathrm{Epi}\Delta S_+  \to k$-\textbf{Mod} is the functor defined
  by:
  \[
    \left\{
    \begin{array}{lll}
      B_n^{sym_+}I &:=& I^{\otimes n+1}, \quad n \geq 0, \\
      B_{-1}^{sym_+}I &:=& k,
    \end{array}
    \right.
  \]
  \[  
    B_*^{sym_+}I(\alpha) : (a_0 \otimes a_1 \otimes \ldots \otimes a_n) \mapsto
      \alpha(a_0, \ldots, a_n), \;\textrm{for $\alpha \in 
      \mathrm{Mor}(\mathrm{Epi}\Delta S_+$)}
  \]
\end{definition}
This definition makes sense, since the only epimorphism with source $[-1]$ is
$\iota_{-1} = \mathrm{id}_{[-1]}$, sending $B_{-1}^{sym_+}I = k$ identically 
to itself.  Since $\alpha \in \mathrm{Mor}_\mathrm{Epi}\Delta S_+([p],[q])$ 
for $p \geq 0$, there is no need for a unit element in $I$.
Furthermore, any product of elements of the ideal $I$ must also be in $I$.

Consider the simplicial $k$-module:
\begin{equation}\label{epiDeltaS_complex}
  \mathscr{Y}^{epi}_*A := 
  k[ N(- \setminus \mathrm{Epi}\Delta S_+) ] \otimes_{\mathrm{Epi}\Delta S_+} B_*^{sym_+}I
\end{equation}

There is an obvious inclusion,
\[
  f : \mathscr{Y}^{epi}_*A \longrightarrow \mathscr{Y}^+_*A
\]

Define a chain map, $g$, in the opposite direction as follows.  
First, by prop.~\ref{prop.BsymI} and observations above, we only need to
define $g$ on the chains $[q] \to [q_0] \to \ldots \to [q_n] \otimes Y$
where $Y \in B_*^{sym_+}I$ already.  In this case, define:
\[
  g( [q] \to [q_0] \to \ldots \to [q_n] \otimes Y )
\]
\[
  =
  \left\{\begin{array}{ll}
     {[-1] \twoheadrightarrow [-1] \twoheadrightarrow \ldots 
     \twoheadrightarrow [-1] \otimes 1,} & \quad \textrm{$Y$ trivial},\\
     {[q] \twoheadrightarrow \mathrm{im}([q] \to [q_0]) 
     \twoheadrightarrow \mathrm{im}([q] \to [q_{1}]) \twoheadrightarrow \ldots  
     \twoheadrightarrow \mathrm{im}([q] \to [q_n]) \otimes Y,} &
     \quad \textrm{$Y$ non-trivial}
  \end{array}\right.
\]

I claim $g$ is well-defined.  Indeed, if $Y\in B_q^{sym_+}I$ is trivial, 
then $q$ must be $-1$.  If $\psi : [-1] \to [q']$ is any
morphism of $\Delta S_+$, then we know $\psi = \iota_{q'}$, and $(\iota_{q'})_*(Y)$
is still a trivial tensor.  We have equivalent chains:
\[
  [-1] \to [q_0] \to \ldots \to [q_n] \otimes 1 \approx
  [q'] \to [q_0] \to \ldots \to [q_n] \otimes 1^{\otimes(q' + 1)}
\]
Applying $g$ to the chain on the left results in a chain of identity
maps,
\[
  [-1] \twoheadrightarrow [-1] \twoheadrightarrow \ldots \twoheadrightarrow [-1]
  \otimes Y.
\]
In order to apply $g$ to the chain
on the right, it we must put it into the correct form.  Since $1^{\otimes(q' + 1)}$
is trivial, we must use $\delta = \iota_{q'}$ to rewrite the chain.  But what results
is exactly the chain on the left, so $g$ is well-defined in this case.

Suppose now that $q \geq 0$ and $Y \in B_q^{sym_+}I$. Let $\psi : [q] \to [q']$
be any morphism of $\Delta S_+$.  Since $q \geq 0$, $\psi \in \mathrm{Mor}\Delta S$.
We have equivalent chains:
\[
  [q] \stackrel{\phi \psi}{\to} [q_0] \stackrel{\alpha_1}{\to} \ldots 
  \stackrel{\alpha_n}{\to} [q_n]   
   \otimes Y \;\approx\;
  [q'] \stackrel{\phi}{\to} [q_0] \stackrel{\alpha_1}{\to} \ldots 
  \stackrel{\alpha_n}{\to} [q_n] \otimes \psi_*(Y).
\]
Applying $g$ on the left hand side yields
\begin{equation}\label{eq.lhs}
  [q] \stackrel{\overline{\phi\psi}}{\twoheadrightarrow} \mathrm{im}(\phi\psi) 
  \twoheadrightarrow \ldots \twoheadrightarrow
  \mathrm{im}([q] \to [q_n]) 
  \otimes Y,
\end{equation}

Consider the chain on the right hand.  If $\psi$ happens to be an epimorphism,
then $\psi_*(Y) \in B_{q'}^{sym_+}I$, and we may apply $g$ directly to this chain.
In this case, we get:
\begin{equation}\label{eq.rhs}
  [q'] \stackrel{\overline{\phi}}{\twoheadrightarrow} \mathrm{im}(\phi)
  \twoheadrightarrow
  \ldots \twoheadrightarrow \mathrm{im}([q'] \to [q_n]) 
  \otimes \psi_*(Y)  
\end{equation}
Now, since $\psi$ is epic, $\mathrm{im}(\phi\psi) = \mathrm{im}(\phi)$.  Moreover,
$\mathrm{im}(\alpha_k\ldots\alpha_1\phi\psi) =\mathrm{im}(\alpha_k\ldots\alpha_1\phi)$
for each $k = 1, 2, \ldots, n$, and the induced morphisms are equal:
\[
  \left(\mathrm{im}([q'] \to [q_k]) \twoheadrightarrow 
  \mathrm{im}([q'] \to [q_{k+1}])\right) =
  \left(\mathrm{im}([q] \to [q_k]) \twoheadrightarrow \mathrm{im}([q] \to [q_{k+1}])\right)
\]
Hence, the chain~(\ref{eq.rhs}) is equal to:
\[
  [q'] \stackrel{\overline{\phi}}{\twoheadrightarrow} \mathrm{im}([q] \to [q_0]) 
  \twoheadrightarrow 
  \ldots \twoheadrightarrow \mathrm{im}([q] \to [q_n]) 
  \otimes \psi_*(Y)
\]
\[
  \approx [q] \stackrel{\overline{\phi}\circ \psi}{\twoheadrightarrow}
  \mathrm{im}([q] \to [q_0]) \twoheadrightarrow 
  \ldots \to \mathrm{im}([q] \twoheadrightarrow [q_n]) 
  \otimes Y
\]
Thus, since $\psi = \overline{\psi}$ and $\overline{\phi}\overline{\psi} =
\overline{\phi} \circ \overline{\psi}$, the chains~(\ref{eq.lhs}) and~(\ref{eq.rhs}) 
are equivalent.

Suppose now that $\psi$ is not epimorphic.  Use Prop.~\ref{prop.decomp} to write
$\psi = \pi \eta$ for $\pi \in \mathrm{Epi}\Delta S_+$ and $\eta \in \mathrm{Mono}\Delta_+$.
By the previous, it is clear that we may assume $\pi = \mathrm{id}$, so that
$\psi$ is a monomorphism of $\Delta_+$.  In this case, $\psi_*(Y)$ may have trivial
tensor factors.  Now, $g$ is only defined for chains in which the factor 
$Y \in B_n^{sym_+}A$ is a basic
tensor having no trivial factors, so we must use Prop.~\ref{prop.BsymI}
to rewrite the chain as:
\[
  [\overline{q}] \to [q_0] \to \ldots \to [q_n] \otimes \overline{\psi_*(Y)}.
\]
Since $Y$ is in $B_q^{sym_+}I$ and $\psi$ is a monomorphism of $\Delta$, we have 
$\overline{\psi_*(Y)} = Y$, and $\delta_{\psi_*(Y)} = \psi$, by uniqueness of the 
decomposition.  Thus, when we apply $g$ to this chain,
we must apply it to the equivalent chain:
\[
  [q] \stackrel{\phi\psi}{\to} [q_0] \to \ldots \to [q_n] \otimes Y.
\]

This shows that $g$ is well-defined.

Now, $gf = \mathrm{id}$, since if $[p] \to [r]$ is in $\mathrm{Mor}(\mathrm{Epi}\Delta S$),
then the epimorphism construction $[p] \twoheadrightarrow \mathrm{im}([p] \to [r])$ is
just the original morphism.
\begin{prop}
  $fg \simeq \mathrm{id}.$
\end{prop}
\begin{proof}
  In what follows, we assume $Y$ is a basic tensor in $B_q^{sym}I$.
  Define a presimplicial homotopy $h$ from $fg$ to $\mathrm{id}$ as follows:
  \[
    h_j^{(n)}([q] \to [q_0] \to \ldots \to [q_n] \otimes Y) \;:=
  \]
  \[
    [q] \twoheadrightarrow \mathrm{im}([q] \to [q_0]) \twoheadrightarrow \ldots 
    \twoheadrightarrow \mathrm{im}([q] \to [q_j]) \hookrightarrow [q_j] \to 
    \ldots \to [q_n] 
    \otimes Y.
  \]

  $h_j$ is well-defined when by the functorial properties of the
  epimorphism construction.

  Suppose $0 \leq i < j \leq n$.  We have $d_i h_j = h_{j-1} d_i$, since $d_i$ on the
  right hand side reduces the number of nodes to the left of $[q_j]$ by one.  We also
  use the functoriality of the epimorphism construction here.

  Suppose $1 \leq j+1 < i \leq n+1$.  $d_i h_j = h_j d_{i-1}$, since $h_j$ on the left
  hand side shifts all nodes to the right of (and including) $[q_j]$ to the right by one.

  Suppose $0 < i \leq n$.  First apply $d_ih_i$ to an arbitrary chain.
  \[
    \begin{diagram}
      \node{ [q] \to [q_0] \to \ldots \to [q_n] \otimes Y }
      \arrow{s,l,T}{ h_i }
      \\
      \node{ [q] \twoheadrightarrow \mathrm{im}([q] \to [q_0]) \twoheadrightarrow 
        \ldots \twoheadrightarrow \mathrm{im}([q] \to [q_{i-1}]) \twoheadrightarrow 
        \mathrm{im}([q] \to [q_i]) \hookrightarrow [q_i] \to \ldots \to [q_n] 
        \otimes Y }
      \arrow{s,l,T}{ d_i }
      \\
      \node{ [q] \twoheadrightarrow \mathrm{im}([q] \to [q_0]) \twoheadrightarrow \ldots 
        \twoheadrightarrow \mathrm{im}([q] \to [q_{i-1}]) \to [q_i] \to \ldots \to [q_n] 
        \otimes Y }
    \end{diagram}
  \]
  Apply $d_ih_{i-1}$ to the same chain.
  \[
    \begin{diagram}
      \node{ [q] \to [q_0] \to \ldots \to [q_n] \otimes Y }
      \arrow{s,l,T}{ h_{i-1} }
      \\
      \node{ [q] \twoheadrightarrow \mathrm{im}([q] \to [q_0]) \twoheadrightarrow 
        \ldots \twoheadrightarrow \mathrm{im}([q] \to [q_{i-1}]) \hookrightarrow [q_{i-1}] 
        \to [q_i] \to \ldots \to [q_n] \otimes Y }
      \arrow{s,l,T}{ d_i }
      \\
      \node{ [q] \twoheadrightarrow \mathrm{im}([q] \to [q_0]) \twoheadrightarrow \ldots 
        \twoheadrightarrow \mathrm{im}([q] \to [q_{i-1}]) \to [q_i] \to \ldots \to [q_n] 
        \otimes Y }
    \end{diagram}
  \]
  The fact that the composition of 
  $\mathrm{im}([q] \to [q_{i-1}]) \to [q_{i-1}] \to [q_i]$ is equal to the composition 
  of $\mathrm{im}([q] \to [q_{i-1}]) \to \mathrm{im}([q] \to [q_i]) \to [q_i] $ follows 
  from the commutativity of the
  outside square of diagram~\ref{eq.epidiagram}.
  Thus, $d_i h_i = d_i h_{i-1}$.
  
  Finally, 
  \[
    d_0 h_0 ([q] \to [q_0] \to \ldots \to [q_n] \otimes Y) =
    d_0\big([q]\twoheadrightarrow \mathrm{im}([q] \to [q_0]) \hookrightarrow [q_0] \to 
    \ldots \to [q_n] \otimes Y \big)
  \]
  \[
    = [q] \to [q_0] \to \ldots \to [q_n] \otimes Y,
  \]
  and
  \[
    d_{n+1} h_n ([q] \to [q_0] \to \ldots \to [q_n] \otimes Y)
  \]
  \[
    = d_{n+1}\big([q] \twoheadrightarrow \mathrm{im}([q] \to [q_0]) \twoheadrightarrow 
    \ldots \twoheadrightarrow \mathrm{im}([q] \to [q_n]) \hookrightarrow [q_n] 
    \otimes Y \big)
  \]
  \[
    = [q] \twoheadrightarrow \mathrm{im}([q] \to [q_0]) \twoheadrightarrow 
    \ldots \twoheadrightarrow \mathrm{im}([q] \to [q_n]) \otimes Y
  \]
  \[
    = g([q] \to [q_0] \to \ldots \to [q_n] \otimes Y)
  \]
  Hence, $\mathrm{id} \simeq fg$, as required.
\end{proof}

\begin{prop}\label{prop.epi}
  If $A$ has augmentation ideal $I$, then
  \[
    HS_*(A) = H_*\left(\mathscr{Y}^{epi}_*A;\,k\right) =
    H_*\left(k[ N(- \setminus \mathrm{Epi}\Delta S_+) ] \otimes_{\mathrm{Epi}\Delta S_+} 
    B_*^{sym_+}I;\,k \right).
  \]
\end{prop}
\begin{proof}
  The complex~(\ref{epiDeltaS_complex}) has been shown to be chain homotopy equivalent to
  the complex $\mathscr{Y}^+_*A$, which by Thm.~\ref{thm.SymHom_plusComplex},
  computes $HS_*(A)$.
\end{proof}

\begin{rmk}
  The condition that $A$ have an augmentation ideal may be lifted (as Richter conjectures),
  if it can be shown that $N(\mathrm{Epi}\Delta S)$ is contractible.  As partial
  progress along these lines, it can be shown that $N(\mathrm{Epi}\Delta S)$ is
  simply-connected.
\end{rmk}

\chapter{A SPECTRAL SEQUENCE FOR $HS_*(A)$}\label{chap.specseq}

\section{Filtering by Number of Strict Epimorphisms}\label{sec.specseq}

In this chapter, fix a unital associative algebra $A$ over commutative
ground ring $k$.  We also assume $A$ comes equipped with an augmentation,
and denote the augmentation ideal by $I$.
Let $\mathscr{Y}^{epi}_*A$ be the complex~(\ref{epiDeltaS_complex}).
Since $A$ is fixed, it will suffice to use 
the notation $\mathscr{Y}^{epi}_*$ in place of $\mathscr{Y}^{epi}_*A$.
As we have seen above in Section~\ref{sec.epideltas}, $H_*(\mathscr{Y}^{epi}) = HS_*(A)$.

Consider a 
filtration of $\mathscr{Y}^{epi}_*$
by number of strict epimorphisms, or \textit{jumps}:
\[
  \mathscr{F}_p\mathscr{Y}^{epi}_q \qquad \textrm{is generated by}
\]
\[
  \left\{[m_0] \to [m_{1}] \to \ldots \to [m_q] \otimes Y,\, \textrm{where 
  $m_{i-1} > m_{i}$ for no more than $p$ distinct values of $i$}\right\}.
\]
The face maps of $\mathscr{Y}^{epi}_*$ only delete morphisms or compose morphisms,
so this filtration is compatible with the differential of $\mathscr{Y}^{epi}_*$.
The filtration quotients are easily described:
\[
  E^0_{p,q} := \mathscr{F}_p\mathscr{Y}^{epi}_q / 
  \mathscr{F}_{p-1}\mathscr{Y}^{epi}_q \qquad
  \textrm{is generated by}
\]
\[
  \left\{[m_0]\to [m_1] \to \ldots \to [m_q] \otimes Y,\, \textrm{where 
  $m_{i-1} > m_{i}$ for exactly $p$ distinct values of $i$}\right\}.
\]
The induced differential on $E^0_{p,q}$ is of bidegree $(0,-1)$, so we may form
a spectral sequence with $E^1_{p,q} = H_{p+q}( E^0_{p,*} )$ (cf. \cite{Mc},
\cite{S}).

\begin{lemma}\label{lem.E1_term}
  There are chain maps (one for each $p$):
  \[
    E^0_{p,*} \to \bigoplus_{m_0 > \ldots > m_p} \bigg(
    I^{\otimes(m_0+1)} \otimes
    k\Big[ \prod_{i=1}^p \mathrm{Epi}_{\Delta_+}\big([m_{i-1}], [m_i]\big)\Big]
    \otimes_{k\Sigma_{m_p + 1}}
    E_*\Sigma_{m_p + 1}
    \bigg),
  \]
  inducing isomorphisms in homology:
  \[
    E^1_{p,q} \cong \bigoplus_{m_0 > \ldots > m_p} H_q\bigg( \Sigma_{m_p+1}^{\mathrm{op}}
      \; ; \; I^{\otimes(m_0+1)} \otimes 
      k\Big[ \prod_{i=1}^p \mathrm{Epi}_{\Delta_+}\big([m_{i-1}], [m_i]\big)\Big]
     \bigg).
  \]    
  Here, we use the convention that $I^{\otimes 0} = k$, and $\Sigma_0 \cong 1$, the
  trivial group.
\end{lemma}

We will begin by defining two related chain complexes:

Denote by $\mathscr{B}_*^{(m_0, \ldots, m_p)}$, the chain complex:
\[
  \bigoplus \Big(  
  [m_0] \stackrel{\cong}{\to} \ldots \stackrel{\cong}{\to} [m_{0}] 
  \searrow [m_{1}] \stackrel{\cong}{\to} \ldots 
  \stackrel{\cong}{\to} [m_{p-1}] \searrow [m_p] \otimes
  I^{\otimes(m_0+1)}
  \Big),
\]
where the sum extends over all such chains that begin with $0$ or more isomorphisms
of $[m_0]$, followed by a strict epimorphism
\mbox{$[m_{0}] \twoheadrightarrow [m_{1}]$}, followed by $0$ or more isomorphisms of
$[m_{1}]$, followed by a strict epimorphism \mbox{$[m_{1}] \twoheadrightarrow [m_{2}]$},
etc., and the last morphism must be a strict epimorphism \mbox{$[m_{p-1}] 
\twoheadrightarrow [m_{p}]$}.
$\mathscr{B}_*^{(m_0, \ldots, m_p)}$ is a subcomplex of $E^0_{p,*}$ with the
same induced differential, and there is a $\Sigma^{\mathrm{op}}_{m_p+1}$-action given by
postcomposition of $g$, regarded as an automorphism
of $[m_p]$.

Denote by $\mathscr{M}_*^{(m_0,\ldots, m_p)}$, the
chain complex consisting of $0$ in degrees different from $p$, and
\[
  \mathscr{M}_p^{(m_0, \ldots, m_p)} \;:=\; I^{\otimes(m_0+1)} \otimes
  k\Big[ \prod_{i=1}^p \mathrm{Epi}_{\Delta_+}\big([m_{i-1}], [m_i]\big)\Big],
\]
the coefficient group that shows up in Lemma~\ref{lem.E1_term}.  This complex
has trivial differential.

Now, $\mathscr{B}_p^{(m_0, \ldots, m_p)}$ is generated by elements of the form
\[
  [m_0] \searrow [m_{1}] \searrow \ldots \searrow [m_p] \otimes Y,
\]
where the chain consists entirely of strict epimorphisms of $\Delta S$.  Observe that
\[
  \mathscr{B}_p^{(m_0, \ldots, m_p)} =  
  k\big[\mathrm{Epi}_{\Delta S_+}( [m_{0}], [m_1] ) \big] \otimes \ldots \otimes 
  k\big[\mathrm{Epi}_{\Delta S_+}( [m_{p-1}], [m_p] ) \big]
  \otimes
  I^{\otimes(m_0+1)}
\]
\begin{equation}\label{eq.B_p}
  \cong
  I^{\otimes(m_0+1)} \otimes
  k\big[\mathrm{Epi}_{\Delta S_+}( [m_{0}], [m_1] ) \big] \otimes \ldots \otimes 
  k\big[\mathrm{Epi}_{\Delta S_+}( [m_{p-1}], [m_p] ) \big]
\end{equation}
as \mbox{$k$-module}.  Now, each $k\big[\mathrm{Epi}_{\Delta S_+}
( [m], [n] ) \big]$ is a 
\mbox{$(k\Sigma^\mathrm{op}_{n+1})$-$(k\Sigma^\mathrm{op}_{m+1})$-bimodule}.
View $\sigma \in \Sigma_{n+1}^\mathrm{op} = \mathrm{Aut}_{\Delta S_+}([n])$  and
$\tau \in \Sigma_{m+1}^\mathrm{op} = \mathrm{Aut}_{\Delta S_+}([m])$ as automorphisms.  Then
the action of $\sigma$, resp., $\tau$, is by postcomposition, resp., precomposition.  The 
bimodule structure, 
$(\sigma . \phi) . \tau = \sigma . (\phi . \tau)$, 
follows easily from associativity of composition in $\Delta S_+$, 
\mbox{$(\sigma \circ \phi)\circ \tau = \sigma \circ (\phi \circ \tau)$}.
We shall use the equivalent interpretation of $k\big[\mathrm{Epi}_{\Delta S_+}
( [m], [n] ) \big]$ as \mbox{$(k\Sigma_{m+1})$-$(k\Sigma_{n+1})$-bimodule}.
Explicitly, an element of $\mathrm{Epi}_{\Delta S_+}([m],[n])$ is a pair $(\psi, g)$, with
$\psi \in \mathrm{Epi}_{\Delta_+}([m],[n])$ and $g \in \Sigma_{m+1}^\mathrm{op}$, so for
$\tau \in \Sigma_{m+1}$ and $\sigma \in \Sigma_{n+1}$,
\[
  (\psi, g).\sigma \;=\; (\mathrm{id}, \sigma) \circ (\psi, g) \;=\;
  (\psi^\sigma, \sigma^\psi \circ g) \;=\; (\psi^\sigma, g\sigma^\psi),
\]
\[
  \tau.(\psi, g) \;=\; (\psi, g) \circ (\mathrm{id}, \tau) \;=\;
  (\psi, g \circ \tau) \;=\; (\psi, \tau g).
\]
Also, since $B_*^{sym_+}I$ is a \mbox{$\Delta S_+$-module}, we may view it as a right 
\mbox{$\Delta S_+^\mathrm{op}$-module}, hence $B_{m_0}^{sym_+}I = I^{\otimes(m_0+1)}$ is a 
right \mbox{$k\Sigma_{m_0+1}$-module}.

With this in mind,~(\ref{eq.B_p}) becomes a \mbox{$k\Sigma^\mathrm{op}_{m_p + 1}$-module},
where the action is the right action of $k\Sigma_{m_p+1}$ on the last tensor
factor by postcomposition, and the isomorphism given above respects this action.

Consider the $k$-module:
\begin{equation}\label{eq.M_alt}
  M \;:=\; I^{\otimes(m_0+1)}  \otimes_{kG_0}
  k\big[\mathrm{Epi}_{\Delta S_+}( [m_{0}], [m_1] ) \big]
  \otimes_{kG_{1}} \ldots \otimes_{kG_{p-1}}
  k\big[\mathrm{Epi}_{\Delta S_+}( [m_{p-1}], [m_p] ) \big],
\end{equation}
where $G_i$ is the group $\Sigma_{m_i + 1}$.  I claim that $M$ is isomorphic
to $\mathscr{M}_p^{(m_0, \ldots, m_p)}$ as $k$-module.  Indeed, any element
\[
  Y \otimes (\psi_p, g_p)\otimes \ldots \otimes (\psi_1, g_1)
\]
in $M$ is equivalent to one in which all $g_i$ are identities by writing $(\psi_p, g_p) =
g_p.(\psi_p, \mathrm{id})$ then commuting $g_p$ over the tensor to the left and iterating
this process to the leftmost tensor factor.  Thus, we may write the element uniquely as
\[
  Z \otimes \phi_1 \otimes \ldots \otimes \phi_p,
\]
where all tensors are now over $k$, and all morphisms are in $\mathrm{Epi}\Delta S_+$.

This isomorphism also allows us to view $\mathscr{M}_p^{(m_0, \ldots, m_p)}$ as a 
\mbox{$\Sigma^\mathrm{op}_{m_p+1}$-module}.
The action is defined as the right action of $\Sigma_{m_p+1}$ on the 
tensor factor $k\big[\mathrm{Epi}_{\Delta S_+}( [m_{p-1}], [m_p] )\big]$.  We 
then use the isomorphism to express this action in terms of 
$\mathscr{M}_p^{(m_0, \ldots, m_p)}$.

Let $\gamma_*$ be a chain map $\mathscr{B}_*^{(m_0, \ldots, m_p)} \to 
\mathscr{M}_*^{(m_0, \ldots, m_p)}$ defined as the zero map 
in degrees different from $p$,
and the canonical map
\[
  I^{\otimes(m_0+1)} \otimes k\big[\mathrm{Epi}_{\Delta S_+}( [m_0], [m_1] ) \big] 
  \otimes \ldots
  \otimes k\big[\mathrm{Epi}_{\Delta S_+}( [m_{p-1}], [m_p] ) \big]
  \longrightarrow
\]
\[
  I^{\otimes(m_0+1)} \otimes_{kG_0}
  k\big[\mathrm{Epi}_{\Delta S_+}( [m_0], [m_1] ) \big] \otimes_{kG_1} \ldots
  \otimes_{kG_{p-1}} k\big[\mathrm{Epi}_{\Delta S_+}( [m_{p-1}], [m_p] ) \big],
\]
in degree $p$.  $\gamma_*$ is $\Sigma_{m_p + 1}^\mathrm{op}$-equivariant due to
an elementary property of bimodules:
\begin{prop}
  Suppose $R$ and $S$ are $k$-algebras, $A$ is a right \mbox{$S$-module}, and $B$ is an
  \mbox{$S$-$R$-bimodule}, then the canonical map $A \otimes_k B \to A \otimes_S B$ is a map
  of right $R$-modules.
\end{prop}

Our aim now is to show that $\gamma_*$ induces an isomorphism in homology.
\begin{prop}\label{prop.gamma_iso}
  $\gamma_*$ induces an isomorphism
  \[
    H_*\big( \mathscr{B}_*^{(m_0,\ldots, m_p)} \big) \longrightarrow
    H_*\big( \mathscr{M}_*^{(m_0,\ldots, m_p)} \big)
  \]
\end{prop}
\begin{proof}
  We shall prove this by induction on $p$.  
  
  Suppose $p=0$.  Observe that
  \[
    \mathscr{B}_n^{(m_0)} = \left\{
    \begin{array}{ll}
      I^{\otimes(m_0+1)}, &\qquad n = 0 \\
      0, &\qquad n > 0.
    \end{array}\right.
  \]
  Moreover, $\gamma_*$ is the identity $\mathscr{B}_*^{(m_0)} \to \mathscr{M}_*^{(m_0)}$.

  Next, for the induction step, we assume $\gamma_* : \mathscr{B}_*^{(m_0, \ldots, m_{p-1})}
  \to \mathscr{M}_*^{(m_0, \ldots, m_{p-1})}$ induces an isomorphism in homology for any
  string of $p$ numbers $m_0 > m_1 > \ldots > m_{p-1}$.  Now assume $m_p < m_{p-1}$.  
  
  Let $G = \Sigma_{m_{p-1}+1}$.  As graded $k$-module, there is a
  degree-preserving isomorphism:
  \begin{equation}\label{eq.B_tensor_iso}
    \theta_* \;:\; \mathscr{B}_*^{(m_0, \ldots, m_{p-1})} \otimes_{kG} E_*G \otimes
    k\big[G \times \mathrm{Epi}_{\Delta_+}([m_{p-1}], [m_p])\big]
    \longrightarrow
    \mathscr{B}_*^{(m_0, \ldots, m_{p-1}, m_p)} 
  \end{equation}
  where the degree of an element $u \otimes (g_0, \ldots, g_n) \otimes (g, \phi)$
  is defined recursively (Note, all elements of $\mathscr{B}_n^{(m_0)}$ 
  are of degree $0$):
  \[
    deg\left(u \otimes (g_0, \ldots, g_n) \otimes (g, \phi)\right) := deg(u) + n + 1.
  \]
  Here, we are using the resolution $E_*G$ of $k$ as $kG$-module 
  defined by $E_nG = k\big[ \prod^{n+1}G \big]$, 
  with $G$-action $g . (g_0, g_1, \ldots, g_n) = (gg_0, g_1, \ldots, g_n)$, and face
  maps 
  \[
    \partial_i(g_0, g_1, \ldots, g_n) = \left\{\begin{array}{ll}
                  (g_0, \ldots, g_ig_{i+1}, \ldots, g_n), & 0 \leq i < n\\
                  (g_0, g_1, \ldots, g_{n-1}), & i = n
                 \end{array}\right.
  \]  
  $\theta_*$ is defined on generators by:
  \[
    \theta_* \;:\; u \otimes (g_0, g_1, \ldots, g_n) \otimes (g, \phi) \mapsto
  \]
  \[
    (u.g_0) \,\ast\, [m_{p-1}] \stackrel{ g_1 }{\longrightarrow} [m_{p-1}]
    \stackrel{ g_2 }{\longrightarrow} \ldots \stackrel{ g_n }
    {\longrightarrow} [m_{p-1}] \stackrel { (\phi, g) }{\longrightarrow} [m_p],
  \]
  where $u . g_0$ is the right action defined above for 
  $\mathscr{B}_*^{(m_0, \ldots, m_p)}$, and   
  $v \ast [n] \to \ldots \to [m]$ is the concatenation of chains ($v$ must have final
  target $[n]$).  $\theta_*$ is well-defined since for $h \in G$,
  \[
    u . h \otimes (g_0, \ldots, g_n) \otimes (g, \phi) \stackrel{\theta_*}{\mapsto}
  \]
  \[
    \big((u.h).g_0\big) \,\ast\, [m_{p-1}] \stackrel{ g_1 }{\longrightarrow} [m_{p-1}]
    \stackrel{ g_2 }{\longrightarrow} \ldots \stackrel{ g_n }
    {\longrightarrow} [m_{p-1}] \stackrel { (\phi, g) }{\longrightarrow} [m_p],
  \]
  while on the other hand,
  \[
    u \otimes h .(g_0, g_1, \ldots, g_n) \otimes (g, \phi) \;=\;
    u \otimes (hg_0, g_1, \ldots, g_n) \otimes (g, \phi) \stackrel{\theta_*}{\mapsto}
  \]
  \[
    \left(u.(hg_0)\right) \,\ast\, [m_{p-1}] \stackrel{ g_1 }{\longrightarrow} [m_{p-1}]
    \stackrel{ g_2 }{\longrightarrow} \ldots \stackrel{ g_n }
    {\longrightarrow} [m_{p-1}] \stackrel { (\phi, g) }{\longrightarrow} [m_p],
  \]
  If we define a right action of $\Sigma_{m_p+1}$ on 
  $k\big[G \times \mathrm{Epi}_{\Delta_+}([m_{p-1}], [m_p])\big]$ via
  \[
    (g, \phi) . h \;\mapsto \; \big( gh^{\phi}, \phi^h \big),
  \]
  then $\theta_*$ is a map of right $k\Sigma_{m_p+1}$-modules, since the action defined above
  simply amounts to post-composition of the morphism $(\phi, g)$ with $h$.
  
  $\theta_*$ has a two-sided $\Sigma_{m_p+1}^{\mathrm{op}}$-equivariant inverse, defined by:
  \[
    u \ast [m_{p-1}] \stackrel{g_1}{\to} [m_{p-1}] \stackrel{g_2}{\to} \ldots
    \stackrel{g_n}{\to} [m_{p-1}] \stackrel{ (\phi, g) }{\to} [m_p]
  \]
  \[
    \mapsto \; u \otimes (\mathrm{id}, g_1, g_2, \ldots, g_n) 
    \otimes (g, \phi).
  \]
  
  Observe that $\mathscr{B}_*^{(m_0, \ldots, m_{p-1})} \otimes_{kG} E_*G \otimes
  k\big[G \times \mathrm{Epi}_{\Delta_+}([m_{p-1}], [m_p])\big]$ is a tensor product of
  two chain complexes, and thus a chain complex in its own right.  The differential
  is given by:
  \[
    \partial\big( u \otimes (g_0, \ldots, g_n) \otimes (g, \phi) \big)
    \;=\; \partial(u) \otimes (g_0, \ldots, g_n) \otimes (g, \phi) +
    (-1)^{deg(u)} u \otimes \partial\big((g_0, \ldots, g_n) \otimes (g, \phi)\big).
  \]
  Note, the $n^{th}$ face map of $E_nG \otimes
  k\big[G \times \mathrm{Epi}_{\Delta_+}([m_{p-1}], [m_p])\big]$ is defined by:
  \[
    \partial_n\big( (g_0, \ldots, g_n) \otimes (g, \phi) \big)
    = (g_0, \ldots, g_{n-1}) \otimes ( g_ng, \phi).
  \]
  We shall verify that $\theta_*$ is a chain map with respect to this differential.
  
  Let $u \otimes (g_0, \ldots, g_n) \otimes (g, \phi)$ be a chain
  with $deg(u) = p$.  Denote by $\partial_i$, $(0\leq i \leq p+n+1)$, the $i^{th}$ 
  face map in either chain complex.  If $i < p$, then clearly $\partial_i \theta_* =
  \theta_* \partial_i$, since this face map acts only on $u$.
  
  Suppose now that $i = p$, and let $(\psi, h)$ be the final morphism in the chain $u$.
  \[
    \big(\ldots \to [m_{p-2}] \stackrel{(\psi, h)}{\longrightarrow} [m_{p-1}]\big) \otimes
    (g_0, g_1, \ldots, g_n) \otimes (g, \phi)
  \]
  \[
    \stackrel{\partial_p}{\mapsto}\quad \big(\ldots \to [m_{p-2}] \stackrel{(\psi, h)}
    {\longrightarrow} [m_{p-1}]\big) \otimes ( g_0g_1, g_2, \ldots, g_n) \otimes (g, \phi)
  \]
  \multiply \dgARROWLENGTH by3
  \divide \dgARROWLENGTH by2
  \[
    \stackrel{\theta_*}{\mapsto}\quad  \left(\ldots \to 
    \begin{diagram}
      \node{ [m_{p-2}] }
      \arrow{e,t}{ (\psi, h).(g_0g_1) }
      \node{ [m_{p-1}] }
    \end{diagram}
    \stackrel{ g_2 }{\longrightarrow} \ldots \stackrel{ g_n }
    {\longrightarrow} [m_{p-1}] \stackrel{ (\phi, g) }{\longrightarrow} [m_p]\right).
  \]
  On the other hand,
  \[
    \big(\ldots \to [m_{p-2}] \stackrel{(\psi, h)}{\longrightarrow} [m_{p-1}]\big) \otimes
    (g_0, g_1, \ldots, g_n) \otimes (g, \phi)
  \]
  \[
    \stackrel{\theta_*}{\mapsto}\quad \left(\ldots \to 
    \begin{diagram}
      \node{ [m_{p-2}] }
      \arrow{e,t}{ (\psi, h).g_0 }
      \node{ [m_{p-1}] }
    \end{diagram}
    \stackrel{ g_1 }{\longrightarrow} \ldots
    \stackrel{ g_n }{\longrightarrow} [m_{p-1}] \stackrel{(\phi, g)}
    {\longrightarrow} [m_p]\right)
  \]
  \[
    \stackrel{\partial_p}{\mapsto}\quad \left( \ldots \to 
    \begin{diagram}
      \node{ [m_{p-2}] }
      \arrow{e,t}{ \big((\psi, h).g_0\big).g_1 }
      \node{ [m_{p-1}] }
    \end{diagram}
    \stackrel{ g_2 }{\longrightarrow} 
    \ldots
    \stackrel{ g_n }{\longrightarrow} [m_{p-1}] \stackrel{(\phi, g)}
    {\longrightarrow} [m_p]\right).
  \]
  \multiply \dgARROWLENGTH by2
  \divide \dgARROWLENGTH by3
  
  Next, suppose $i = p + j$ for some $1 \leq j < n$.  In this case, $\partial_i$ has
  the effect of combining $g_i$ and $g_{i+1}$ into $g_ig_{i+1}$, for either chain, so
  clearly $\theta_*\partial_i = \partial_i\theta_*$.

  Finally, for $i = p+n$,
  \[
    \theta_*\partial_{p+n}\big( u \otimes (g_0, \ldots, g_n) \otimes (g, \phi) \big)
  \]
  \[
    = \theta_*\big( u \otimes (g_0, \ldots, g_{n-1}) \otimes (g_ng,\phi) \big)
  \]
  \[
    = (u.g_0) \ast [m_{p-1}] \stackrel{g_1}{\to}
      \ldots \stackrel{ g_{n-1} }{\longrightarrow} [m_{p-1}] 
      \stackrel{ (\phi,\, g_ng) }{\longrightarrow} [m_p],
  \]
  while
  \[
    \partial_{p+n}\theta_*\big( u \otimes (g_0, \ldots, g_n) \otimes (g, \phi) \big)
  \]
  \[
    = \partial_{p+n}\big( (u.g_0) \ast [m_{p-1}] \stackrel{g_1}{\to} 
    \ldots \stackrel{ g_n }{\to} [m_{p-1}] \stackrel{ (\phi,\, g) }{\longrightarrow} [m_p] \big)
  \]
  \[
    = (u.g_0) \ast [m_{p-1}] \stackrel{g_1}{\to} 
    \ldots \stackrel{ g_{n-1} }{\to} [m_{p-1}] 
    \stackrel{ g_n.(\phi,\, g) }{\longrightarrow} [m_p]
  \]
  \[
    = (u.g_0) \ast [m_{p-1}] \stackrel{g_1}{\to} 
    \ldots \stackrel{ g_{n-1} }{\to} [m_{p-1}] 
    \stackrel{ (\phi,\, g_ng) }{\longrightarrow} [m_p]
  \]
  Hence, the map $\theta_*$ is a chain isomorphism.
  
  The next step in this proof is to prove a chain homotopy equivalence,
  \[
    \mathscr{B}_*^{(m_0, \ldots, m_{p-1})} \otimes_{kG} E_*G \otimes
    k\big[G \times \mathrm{Epi}_{\Delta_+}([m_{p-1}], [m_p])\big]
  \]
  \[  
    \stackrel{\simeq}{\longrightarrow} \;
    \mathscr{B}_*^{(m_0, \ldots, m_{p-1})} \otimes 
    k\big[\mathrm{Epi}_{\Delta_+}([m_{p-1}], [m_p])\big]
  \]
  To that end, we shall define chain maps $F_*$ and $G_*$ between the two complexes.  Let
  \[
    \mathscr{U}_* \;:=\; \mathscr{B}_*^{(m_0, \ldots, m_{p-1})}, \;\textrm{and} \qquad
    S \;:=\; \mathrm{Epi}_{\Delta_+}([m_{p-1}], [m_p]).
  \]
  Define
  \[
    F_* \;:\; \mathscr{U}_* \otimes_{kG} E_*G \otimes k[ G \times S ]
    \longrightarrow \mathscr{U}_* \otimes k[S],
  \]
  \[
    F_*\big( u \otimes (g_0) \otimes (g, \phi) \big) \;:=\;
    u.(g_0g) \otimes \phi,
  \]
  \[
    F_*\big( u \otimes (g_0, \ldots, g_n) \otimes (g, \phi) \big) \;:=\; 0,
    \qquad \textrm{if $n > 0$}.
  \]
  The fact that $F_*$ is well-defined is trivial to verify (we only need to check
  for $n=0$, since otherwise $F_* = 0$):
  \[
    u.h \otimes (g_0) \otimes (g, \phi) \mapsto (u.h).(g_0g) \otimes \phi =
    u.(hg_0g) \otimes \phi,
  \]
  while
  \[
    u \otimes h.(g_0) \otimes (g, \phi) = u \otimes (hg_0) \otimes (g, \phi)
    \mapsto u.(hg_0g) \otimes \phi.
  \]
  Next, let
  \[
    G_* \;:\; \mathscr{U}_* \otimes k[S] \to \mathscr{U}_* \otimes_{kG} E_*G \otimes
    k[ G \times S ]
  \]
  be the composite
  \[
    \mathscr{U}_* \otimes k[S] \,\stackrel{\cong}{\to}\, \mathscr{U}_* \otimes_{kG} G
    \otimes k[S] \,=\, \mathscr{U}_* \otimes_{kG} E_0G \otimes k[S]
  \]
  \[
    \stackrel{j}{\rightarrow}\,\mathscr{U}_* \otimes_{kG} E_0G \otimes k[ G \times S ]
    \,\stackrel{inc}{\longrightarrow}\,
    \mathscr{U}_* \otimes_{kG} E_*G \otimes k[ G \times S ],
  \]
  where $j$ is induced by the map sending a generator $\phi \in S$ to 
  $(\mathrm{id}, \phi) \in G \times S$, and
  $inc$ is induced by the inclusion $E_0G \hookrightarrow E_*G$.
  Observe,
  \[
    F_*G_*( u \otimes \phi) \;=\; F_*\big( u \otimes (\mathrm{id}) \otimes 
    (\mathrm{id}, \phi)\big) \;=\; (u.\mathrm{id}) \otimes \phi.
  \]
  Thus, $F_*G_*$ is the identity.  I claim $G_*F_* \simeq \mathrm{id}$.  
  
  The desired homotopy will be given by:
  \[
    h_* \;:\; u \otimes (g_0, \ldots, g_n) \otimes (g, \phi)
    \mapsto (-1)^{deg(u) + n}u \otimes (g_0, \ldots, g_n, g)\otimes 
    (\mathrm{id}, \phi).
  \]
  
  First, observe:
  \[
    G_*F_*\big( u \otimes (g_0) \otimes (g, \phi) \big) \,=\,
      G_*( u.(g_0g) \otimes \phi ) \,=\, u.(g_0g) \otimes (\mathrm{id}) \otimes
      (\mathrm{id}, \phi),
  \]
  \[
    G_*F_*\big( u \otimes (g_0, \ldots, g_n) \otimes (g, \phi) \big) \,=\, G_*(0)
    \,=\, 0, \qquad \textrm{for $n>0$}.
  \]
  Hence, there are two cases we must explore.  For $n=0$, 
  \[
    h\partial\big( u \otimes (g_0) \otimes (g, \phi) \big)
    \;=\; h\big( \partial u \otimes (g_0) \otimes (g, \phi) \big)
  \]
  \[
    =\; (-1)^{deg(u)-1} \partial u \otimes (g_0, g) \otimes (\mathrm{id}, \phi)
  \]
  On the other hand,
  \[
    \partial h\big( u \otimes (g_0) \otimes (g, \phi) \big)
    \;=\; \partial \big( (-1)^{deg(u)}u \otimes (g_0, g) \otimes (\mathrm{id}, \phi)\big)
  \]
  \[
    =\; (-1)^{deg(u)} \partial u \otimes (g_0, g) \otimes (\mathrm{id}, \phi)
    + (-1)^{2 deg(u)} u \otimes (g_0g) \otimes (\mathrm{id}, \phi)
    - (-1)^{2 deg(u)} u \otimes (g_0) \otimes (g, \phi)
  \]
  \[
    =\; (-1)^{deg(u)} \partial u \otimes (g_0, g) \otimes (\mathrm{id}, \phi)
    + u \otimes (g_0g) \otimes (\mathrm{id}, \phi)
    - u \otimes (g_0) \otimes (g, \phi)
  \]
  So, 
  \[
    \big(h\partial + \partial h\big)\big( u \otimes (g_0) \otimes (g, \phi) \big)
    = u.(g_0g) \otimes (\mathrm{id}) \otimes (\mathrm{id}, \phi) -
    u \otimes (g_0) \otimes (g, \phi)
  \]
  \[
    = \big(G_*F_* - \mathrm{id}\big)\big( u \otimes (g_0) \otimes (g, \phi) \big).
  \]
  The case $n>0$ is handled similarly:
  \[
    u \otimes (g_0, \ldots, g_n) \otimes (g, \phi) \quad\stackrel{\partial}{\mapsto}
  \]
  
  \[
    \partial u \otimes (g_0, \ldots, g_n) \otimes (g, \phi) \,+
  \]
  \[
    (-1)^{deg(u)}\Big[\sum_{j=0}^{n - 1} \Big((-1)^j u \otimes (g_0, \ldots,
    g_jg_{j+1}, \ldots, g_n) \otimes (g, \phi)\Big) \,+
  \]
  \[
    (-1)^n u \otimes (g_0, \ldots, g_{n-1}) \otimes (g_n g, \phi)\Big]
  \]
  
  \[
    \stackrel{h}{\mapsto}\quad (-1)^{deg(u) + n - 1} 
    \partial u \otimes (g_0, \ldots, g_n, g) \otimes (\mathrm{id}, \phi) \,+
  \]
  \[
    (-1)^{2deg(u) + n - 1}\sum_{j=0}^{n-1} \Big((-1)^j u \otimes (g_0, \ldots,
    g_jg_{j+1}, \ldots, g_n, g) \otimes (\mathrm{id}, \phi)\Big) \,+
  \]
  \[
    (-1)^{2deg(u) + 2n - 1} u \otimes (g_0, \ldots, g_{n-1}, 
    g_n g) \otimes (\mathrm{id},
    \phi)
  \]
  
  \[
    = -(-1)^{deg(u) + n}
    \partial u \otimes (g_0, \ldots, g_n, g) \otimes (\mathrm{id}, \phi) \,+
  \]
  \[
    \sum_{j=0}^{n-1} \Big((-1)^{j+n-1} u \otimes (g_0, \ldots,
    g_jg_{j+1}, \ldots, g_n, g) \otimes (\mathrm{id}, \phi)\Big) \,+
  \]
  \begin{equation}\label{eq.lhs_hom_tensor}
    (-1)u \otimes (g_0, \ldots, g_{n-1}, g_n g) \otimes (\mathrm{id}, \phi).
  \end{equation}
  
  On the other hand,
  \[
    u \otimes (g_0, \ldots, g_n) \otimes (g, \phi) \;\stackrel{h}{\mapsto}\;
    (-1)^{deg(u) + n}u \otimes (g_0, \ldots, g_n , g) \otimes (\mathrm{id}, \phi)
  \]
  
  \[
    \stackrel{\partial}{\mapsto} \quad(-1)^{deg(u) + n}\Big[\partial u \otimes
    (g_0, \ldots, g_n, g) \otimes (\mathrm{id}, \phi) \,+
  \]
  \[
    \sum_{j=0}^{n-1} \Big( (-1)^{deg(u) + j} u \otimes 
    (g_0, \ldots, g_jg_{j+1}, \ldots, g_n,
    g) \otimes (\mathrm{id}, \phi)\Big)\,+  
  \]
  \[
    (-1)^{deg(u)+n}u \otimes (g_0, \ldots, g_{n-1}, g_ng) \otimes (\mathrm{id}, \phi) 
    \,+\, (-1)^{deg(u)+n+1} u \otimes (g_0, \ldots, g_n) \otimes (g, \phi)\Big]
  \]
  
  \[
    =\;(-1)^{deg(u) + n}\partial u \otimes
    (g_0, \ldots, g_n, g) \otimes (\mathrm{id}, \phi) \,+
  \]
  \[
    \sum_{j=0}^{n-1} \Big( (-1)^{j+n} u \otimes 
    (g_0, \ldots, g_jg_{j+1}, \ldots, g_n, g) \otimes (\mathrm{id}, \phi)\Big)\,+ 
  \]
  \begin{equation}\label{eq.rhs_hom_tensor}
    u \otimes (g_0, \ldots, g_{n-1}, g_ng) \otimes (\mathrm{id}, \phi) 
    \,-\, u \otimes (g_0, \ldots, g_n) \otimes (g, \phi).
  \end{equation}
  
  Now, adding eq.~(\ref{eq.rhs_hom_tensor}) to eq.~(\ref{eq.lhs_hom_tensor}) yields
  $(-1)u \otimes (g_0, \ldots, g_n) \otimes (g, \phi)$, proving the relation
  \[
    h\partial + \partial h = G_*F_* - \mathrm{id}, \qquad \textrm{for $n>0$.}
  \]
  
  To complete the proof, simply observe that every map in the following is either a
  chain isomorphism or a homotopy equivalence (each of which is also $\Sigma_{m_p+1}
  ^\mathrm{op}$-equivariant):
  
  \begin{equation}\label{eq.chain_of_equiv}
    \begin{diagram}
      \node{ \mathscr{B}_*^{(m_0, \ldots, m_p)} }
      \arrow{s,lr}{\cong}{\theta_*^{-1}}\\
      \node{ \mathscr{B}_*^{(m_0, \ldots, m_{p-1})} \otimes_{kG} E_*G \otimes
             k\big[G \times \mathrm{Epi}_{\Delta_+}([m_{p-1}], [m_p])\big] }
      \arrow{s,lr}{\simeq}{F_*}\\
      \node{ \mathscr{B}_*^{(m_0, \ldots, m_{p-1})} \otimes 
             k\big[\mathrm{Epi}_{\Delta_+}([m_{p-1}], [m_p])\big] }
      \arrow{s,lr}{\simeq}{\gamma_*, \;\textrm{by inductive hypothesis}}\\
      \node{ \mathscr{M}_*^{(m_0, \ldots, m_{p-1})} \otimes
             k\big[\mathrm{Epi}_{\Delta_+}([m_{p-1}], [m_p])\big] }
      \arrow{s,l}{=}\\
      \node{ I^{\otimes (m_0+1)} \otimes k\Big[ \prod_{i=1}^{p-1} 
             \mathrm{Epi}_{\Delta_+}\big([m_{i-1}], [m_i]\big)\Big]
             \otimes
             k\big[\mathrm{Epi}_{\Delta_+}([m_{p-1}], [m_p])\big] }
      \arrow{s,l}{\cong}\\
      \node{ I^{\otimes (m_0+1)} \otimes k\Big[ \prod_{i=1}^{p} 
             \mathrm{Epi}_{\Delta_+}\big([m_{i-1}], [m_i]\big)\Big] }
      \arrow{s,l}{=}\\
      \node{ \mathscr{M}_*^{(m_0, \ldots, m_p)} }
    \end{diagram}
  \end{equation}
  
  We must verify that this composition is indeed the map $\gamma_*$.  Denote by
  $\gamma'_*$, the composition defined by (\ref{eq.chain_of_equiv}).  
  If $u \in
  \mathscr{B}_*^{(m_0, \ldots, m_p)}$ has degree greater than $p$, then there is some
  isomorphism $g$ showing up in the chain.  If $g$ is an isomorphism of any $[m_i]$ for
  $i < p-1$, then $\gamma'_*(u) = 0$ since then the tensor factor of $u$ in
  $\mathscr{B}_*^{(m_0, \ldots, m_{p-1})}$ would have degree greater than $p-1$, hence
  $\gamma_*$ would send this element to $0$ in $\mathscr{M}_*^{(m_0, \ldots, m_{p-1})}$.
  If, on the other hand, $g$ is an isomorphism of $[m_{p-1}]$, then $F_*\theta_*^{-1}(u)
  = 0$, since there would be a factor in $E_*G$ of degree greater than $0$.  Thus,
  $\gamma'_*(u) = 0$ for any $u$ of degree different from $p$.
  
  Now, if $u$ is of degree $p$,
  \[
    u \;=\; Y \otimes (\psi_1, g_1) \otimes (\psi_2, g_2) \otimes \ldots \otimes
    (\psi_p, g_p)
  \]
  \[
    \stackrel{\theta_*^{-1}}{\mapsto} \quad
    \big[Y \otimes (\psi_1, g_1) \otimes \ldots \otimes (\psi_{p-1}, g_{p-1}) \big] \otimes
    (\mathrm{id}) \otimes (g_p, \psi_p)
  \]
  \[
    \stackrel{F_*}{\mapsto} \quad
    \big[Y \otimes (\psi_1, g_1) \otimes \ldots \otimes (\psi_{p-1}, g_{p-1}).g_p \big]
    \otimes \psi_p
  \]
  It should be clear that applying $\gamma_*$ to the $\mathscr{B}_*^{(m_0, \ldots,
  m_{p-1})}$-factor of the tensor product would have the same effect as $\gamma_*$
  on the original chain, $u$.
\end{proof}

Now, we may prove Lemma~\ref{lem.E1_term}.  Let $G = \Sigma_{m_p + 1}$.  Observe,
\[
  E^0_{p,q} \cong \bigoplus_{m_0 > \ldots > m_p} 
  \bigoplus_{s+t = q} \mathscr{B}_s^{(m_0, \ldots, m_p)}
  \otimes_{kG} E_tG,
\]
with differential corresponding exactly to the vertical differential defined for $E^0$.
Note, the outer direct sum respects the differential $d^0$, so the $E^1$ term given by:
\begin{equation}\label{eq.E1_expression}
  E^1_{p,q} = H_{p+q}(E^0_{p,*}) \cong
  \bigoplus_{m_0 > \ldots > m_p} H_{p+q}\big(
  \mathscr{B}_*^{(m_0, \ldots, m_p)}
  \otimes_{kG} E_*G\big),
\end{equation}
where we view $\mathscr{B}_*^{(m_0, \ldots, m_p)} \otimes_{kG} E_*G$ as a double complex.
In what follows, let $(m_0, \ldots, m_p)$ be fixed.
In order to take the homology of the double complex, we set up another spectral
sequence.  From the discussion above, the total differential is given by
\[
  \partial_{total} = d^{v} + d^{h}, \quad \textrm{where}
\]
\[
  d^{v}\big(u \otimes (g_0, \ldots, g_t)\big)
  \;:=\; \partial_{B}(u) \otimes (g_0, \ldots, g_t), \quad \textrm{and}
\]
\[
  d^{h}\big(u \otimes (g_0, \ldots, g_t)\big)
  \;:=\; (-1)^{deg(u)} u \otimes \partial_{E}(g_0, \ldots, g_t),
\]
where $\partial_B$ and $\partial_E$ are the differentials previously mentioned for
$\mathscr{B}_*^{(m_0, \ldots, m_p)}$ and $E_*G$, respectively.
Thus, there is a spectral sequence $\{\overline{E}^r_{*,*}, d^r\}$ with
\[
  \overline{E}^2 \cong H_{*,*}\Big( H\big( \mathscr{B}_*^{(m_0, \ldots, m_p)}
   \otimes_{kG} E_*G,\, d^h \big),\,
  d^v \Big),
\]
in the notation of McCleary (see~\cite{Mc}, Thm 3.10).  Since this is a
first quadrant spectral sequence,
it must converge to $H_*\big(\mathscr{B}_*^
{(m_0, \ldots, m_p)}\otimes_{kG} E_*G\big)$.  Let us examine what happens after taking
the horizontal differential.  Let $t$ be fixed:
\[
  \overline{E}^1_{*,t} = H_*\big( \mathscr{B}_*^{(m_0, \ldots, m_p)}
   \otimes_{kG} E_tG,\, d^h \big)
\]
\[
  \cong H_*( \mathscr{B}_*^{(m_0, \ldots, m_p)})\otimes_{kG} E_tG,
\]
since $E_tG$ is flat as left $kG$-module (in fact, $E_tG$ is free).  Then, 
by Prop.~\ref{prop.gamma_iso},
\[
  \overline{E}^1_{*,t} \cong H_*( \mathscr{M}_*^{(m_0, \ldots, m_p)})\otimes_{kG} E_tG,
\]
\[
  = \left\{\begin{array}{ll}
      I^{\otimes (m_0+1)} \otimes
      k\Big[ \prod_{i=1}^p \mathrm{Epi}_{\Delta_+}\big([m_{i-1}], [m_i]\big)\Big] 
      \otimes_{kG} E_tG, &
      \textrm{in degree $p$}\\
      0, & \textrm{in degrees different from $p$}
    \end{array}\right.
\]
So, the only groups that survive are concentrated in column $p$.  Taking the vertical
differential now amounts to obtaining the $G^\mathrm{op}$-equivariant homology of
\[
  I^{\otimes (m_0+1)} \otimes k\Big[ \prod_{i=1}^p \mathrm{Epi}_{\Delta_+}\big([m_{i-1}], 
  [m_i]\big)\Big],
\]
so
\[
  \overline{E}^2_{p,t} \cong H_t\bigg( G^{\mathrm{op}}
  \; ; \; I^{\otimes (m_0+1)} \otimes
  k\Big[ \prod_{i=1}^p \mathrm{Epi}_{\Delta_+}\big([m_{i-1}], [m_i]\big)\Big] \bigg).
\]
Since $\overline{E}^2_{s,t} = 0$ for $s \neq p$, the sequence collapses here.  Thus,
\[
  H_{p+q}\big( \mathscr{B}_*^{(m_0, \ldots, m_p)} \otimes_{kG} E_*G\big) \cong
  H_q\bigg( G^{\mathrm{op}} \; ; \; I^{\otimes (m_0+1)} \otimes
  k\Big[ \prod_{i=1}^p \mathrm{Epi}_{\Delta_+}\big([m_{i-1}], [m_i]\big)\Big] \bigg)
\]
Putting this information back into eq.~(\ref{eq.E1_expression}), we obtain the desired
isomorphism:
\[
  E^1_{p,q} \cong 
  \bigoplus_{m_0 > \ldots > m_p} H_q\bigg( G^{\mathrm{op}} \; ; \; I^{\otimes (m_0+1)} 
  \otimes
  k\Big[ \prod_{i=1}^p \mathrm{Epi}_{\Delta_+}\big([m_{i-1}], [m_i]\big)\Big] \bigg).
\]

A final piece of information needed in order to use Lemma~\ref{lem.E1_term} for
computation is a
description of the horizontal differential $d^1_{p,q}$ on $E^1_{p,q}$.  This map
is induced from the differential $d$ on $\mathscr{Y}_*$, and reduces the filtration
degree by $1$.  Thus, it is the sum of face maps that combine strict epimorphisms.

Let
\[
  [u] \in \bigoplus H_q\bigg( \Sigma_{m_p+1}^{\mathrm{op}}
      \; ; \; I^{\otimes (m_0+1)} \otimes
    k\Big[ \prod_{i=1}^p \mathrm{Epi}_{\Delta_+}\big([m_{i-1}], [m_i]\big)\Big] \bigg)
\]
be represented by a chain:
\[
  u = Y \otimes (\phi_1, \phi_2, \ldots, \phi_p) \otimes (g_0, \ldots, g_q).
\]
Then, the face maps are defined by:
\[
  \partial_0(u) = (\phi_1)_*(Y) \otimes (\phi_2, \ldots, \phi_p) \otimes (g_0, \ldots, g_q),
\]
\[
  \partial_i(u) = Y \otimes (\phi_1, \ldots, \phi_{i+1}\phi_i, \ldots, \phi_p)
  \otimes (g_0, \ldots, g_q), \quad \textrm{for $0 < i < p$},
\]
The last face map has the effect of removing the morphism $\phi_p$ by iteratively commuting 
it past any group elements to
the right of it.
\[
  \partial_p(u) =  Y \otimes (\phi_1, \ldots, \phi_{p-1}) \otimes (g'_0, \ldots, g'_q),
\]
where
\[
  g'_i = g_i^{\phi_p^{g_0g_1\ldots g_{i-1}}}.
\]

Note that $\partial_p$ involves a change of group from $\Sigma_{m_p}$ to $\Sigma_{m_{p-1}}$.

\begin{prop}
  The spectral sequence $E^r_{p,q}$ above collapses at $r = 2$.
\end{prop}
\begin{proof}
  This proof relies on the fact that the differential $d$ on $\mathscr{Y}_*$ cannot
  reduce the filtration degree by more than $1$.  Explicitly, we shall show that
  $d^r \,:\, E^r_{p,q} \to E^r_{p-r,q+r-1}$ is trivial for $r \geq 2$.
  
  $d^r$ is induced by $d$ in the following way.  Let $Z_p^r = \{ x \in \mathscr{F}_p
  \mathscr{Y}_* \,|\, d(x) \in \mathscr{F}_{p-r}\mathscr{Y}_* \}$.  Then
  $E^r_{p,*} = Z_p^r/(Z_{p-1}^{r-1} + dZ_{p+r-1}^{r-1})$.  Now, $d$ maps
  \[
    Z_p^r \to Z_{p-r}^r
  \]
  and 
  \[
    Z_{p-1}^{r-1} + dZ_{p+r-1}^{r-1} \to dZ_{p-1}^{r-1}.
  \]
  Hence, there is an induced map $\overline{d}$ making the square below commute.  $d^r$
  is obtained as the composition of $\overline{d}$ with a projection onto
  $E^r_{p-r,*}$.
  \[
    \begin{diagram}
      \node{Z_p^r}
      \arrow{r,t}{d}
      \arrow{s,l,A}{\pi_1}
      \node{Z_{p-r}^r}
      \arrow{s,r,A}{\pi_2}\\
      \node{Z_p^r/(Z_{p-1}^{r-1} + dZ_{p+r-1}^{r-1})}
      \arrow{s,=}
      \arrow{e,t}{\overline{d}}
      \node{Z_{p-r}^r/dZ_{p-1}^{r-1}}
      \arrow{s,r,A}{\pi'}\\
      \node{E_{p,*}^r}
      \arrow{se,t}{d^r}
      \node{Z_{p-r}^r/(Z_{p-r-1}^{r-1} + dZ_{p-1}^{r-1})}
      \arrow{s,=}\\
      \node[2]{E_{p-r,*}^r}
    \end{diagram}
  \]
  In our case, $x \in Z_p^r$ is a sum of the form
  \[
    x = \sum_{q \geq 0} a_i\left( Y \otimes (f_1, f_2, \ldots, f_q) \right),
  \]
  where $a_i \neq 0$ for only finitely many $i$, and the sum extends over all symbols
  $Y \otimes (f_1, f_2, \ldots, f_q)$ with $Y \in B_*^{sym}I$, $f_j \in
  \mathrm{Epi\Delta S_+}$ composable maps, and at most $p$ of the $f_j$ maps are
  strict epimorphisms.  The image of $x$ under $\pi_1$ looks like
  \[
    \pi_1(x) = \sum_{q \geq 0} a_i\left[ Y \otimes (f_1, f_2, \ldots, f_q) \right],
  \]
  where exactly $p$ of the $f_j$ maps are strictly epic.  There are, of course,
  other relations present as well -- those arising from modding out by $dZ_{p+r-1}^{r-1}$.
  Consider, $\overline{d}\pi_1(x)$.  This should be the result of lifting $\pi_1(x)$
  to a representative in $Z_p^r$, then applying $\pi_2 \circ d$.  One such representative
  is:
  \[
    y = \sum_{q \geq 0} a_i\left( Y \otimes (f_1, f_2, \ldots, f_q) \right),
  \]
  in which each symbol $Y \otimes (f_1, f_2, \ldots, f_q)$ has exactly $p$
  strict epimorphisms.  Now, $d(y)$ is a sum
  \[
    d(y) = \sum_{q \geq 0} b_i\left( Z \otimes ( g_1, g_2, \ldots, g_{q-1}) \right),
  \]
  where each symbol $Z \otimes ( g_1, g_2, \ldots, g_{q-1})$ has either $p$ or $p-1$
  strict epimorphisms, since $d$ only combines two morphisms at a time.  Thus, if
  $r \geq 2$, then $d(y) \in Z_{p-r}^r \Rightarrow d(y) = 0$.  But then,
  $\overline{d}\pi_1(x) = \pi_2d(y) = 0$, and $d^r = \pi'\overline{d}$ is the
  zero map.
\end{proof}

\section{Implications in Characteristic 0}\label{sec.char0}

In this section, we shall assume that $k$ is a field of characteristic 0.
Then for any finite group $G$ and $kG$-module $M$, $H_q( G, M ) = 0$
for all $q > 0$ (see~\cite{B}, for example).  Thus, by Lemma~\ref{lem.E1_term},
the $E^1$ term of the spectral sequence is concentrated in row $0$, and
\[
  E^1_{p,0} \cong \bigoplus_{m_0 > \ldots > m_p} \bigg( I^{\otimes (m_0+1)} \otimes
    k\Big[ \prod_{i=1}^p \mathrm{Epi}_{\Delta_+}\big([m_{i-1}], [m_i]\big)\Big] \bigg)/
   \Sigma_{m_p+1}^{\mathrm{op}},
\]
that is, the group of co-invariants of the coefficient group, under the right-action
of $\Sigma_{m_p+1}$.

Since $E^1$ is concentrated on a row, the spectral sequence collapses at this term.
Hence for the $k$-algebra $A$, with augmentation ideal $I$,
\begin{equation}\label{eq.symhom-char0}
  HS_*(A) =
  H_*\bigg( \bigoplus_{p\geq 0}
    \bigoplus_{m_0 > \ldots > m_p}\Big(I^{\otimes (m_0+1)} \otimes
    k\Big[ \prod_{i=1}^p \mathrm{Epi}_{\Delta_+}\big([m_{i-1}], [m_i]\big)\Big]\Big) /
    \Sigma_{m_p+1}^{\mathrm{op}},\; d^1\bigg).
\end{equation}

This complex is still rather unwieldy as the $E^1$ term is infinitely generated in each
degree. In the next chapter, we shall see another spectral sequence that is more
computationally useful.

\chapter{A SECOND SPECTRAL SEQUENCE}\label{chap.spec_seq2}

\section{Filtering by Degree}\label{sec.filtdeg} %

Again, we shall assume $A$ is a $k$-algebra equipped with an augmentation,
and whose augmentation ideal is $I$.  Assume further that $I$ is free as
$k$-module, with countable basis $X$.  Let $\mathscr{Y}_*^{epi}$ be the 
complex~\ref{epiDeltaS_complex}, with differential $d = \sum (-1)^i \partial_i$.
It will become convenient to use a \textit{reduced} version of the $B_*^{sym_+}I$ 
functor, induced by the inclusion $\Delta S \hookrightarrow \Delta S_+$.
\[
  B_n^{sym}I \;:= \; I^{\otimes_{n+1}}, \quad \textrm{for $n \geq 0$.}
\]
Let
\begin{equation}\label{eq.reduced_Y}
  \widetilde{\mathscr{Y}}_* = \bigoplus_{q \geq 0} \;\bigoplus_{m_0 \geq
  \ldots \geq m_q}
  k\big[ [m_0] \twoheadrightarrow [m_1] \twoheadrightarrow \ldots 
  \twoheadrightarrow [m_q] \big] \otimes B_{m_0}^{sym}I.
\end{equation}
Observe that there is a splitting of $\mathscr{Y}^{epi}_*$ as:
\[
  \mathscr{Y}^{epi}_* \cong \widetilde{\mathscr{Y}}_* \oplus k[ N(\ast) ],
\]
where $\ast$ is the trivial subcategory of $\mathrm{Epi}_{\Delta S_+}$ consisting
of the object $[-1]$ and morphism $\mathrm{id}_{[-1]}$.  The fact that $I$ is an 
ideal ensures that this splitting passes to homology.  Hence, we have:
\[
  HS_*(A) \;\cong\; H_*(\widetilde{\mathscr{Y}}_*) \oplus k_0,
\]
where $k_0$ is the graded $k$-module consisting of $k$ concentrated in degree $0$.

Now, since $I = k[X]$ as \mbox{$k$-module}, $B_n^{sym}I = 
k[X]^{\otimes(n+1)}$.
Thus, as \mbox{$k$-module}, $\widetilde{\mathscr{Y}}_*$ is generated by elements of
the form:
\[
  \left( [m_0] \twoheadrightarrow [m_1] \twoheadrightarrow \ldots 
  \twoheadrightarrow [m_p] \right) \otimes (x_0 \otimes x_1 \otimes \ldots
  \otimes x_{m_0}), \quad x_i \in X.
\]

Using an isomorphism analogous
to that of eq.~\ref{eq.B_p}, we may write:
\[
  \widetilde{\mathscr{Y}}_q = \bigoplus_{m_0 \geq 0} \;\bigoplus_{m_0 \geq
  \ldots \geq m_q}  B_{m_0}^{sym}I \otimes
  k\Big[\prod_{i=1}^q \mathrm{Epi}_{\Delta S}\big([m_{i-1}], [m_i]\big)\Big].
\]
\[
  \cong \; \bigoplus_{m_0 \geq 0} \;\bigoplus_{m_0 \geq
  \ldots \geq m_q} k[X]^{\otimes(m_0+1)} \otimes
  k\Big[\prod_{i=1}^q \mathrm{Epi}_{\Delta S}\big([m_{i-1}], [m_i]\big)\Big].
\]

The face maps are given explicitly below:
\[
  \partial_0( Y \otimes f_1 \otimes f_2 \otimes \ldots \otimes f_q )
  = f_1(Y) \otimes f_2 \otimes \ldots \otimes f_q,
\]
\[
  \partial_i( Y \otimes f_1 \otimes \ldots \otimes f_q )
  = Y \otimes f_1 \otimes \ldots \otimes (f_{i+1}f_i) \otimes \ldots \otimes f_q,
\]
\[
  \partial_{q}(Y \otimes f_1 \otimes \ldots \otimes f_{q-1} \otimes f_q )
  = Y \otimes f_1 \otimes \ldots \otimes f_{q-1}.
\]

Consider a filtration 
$\mathscr{G}_*$ of
$\widetilde{\mathscr{Y}}_*$ by degree of $Y \in B^{sym}_*I$:
\[
  \mathscr{G}_p\widetilde{\mathscr{Y}}_q = \bigoplus_{p \geq m_0 \geq 0} 
  \;\bigoplus_{m_0 \geq
  \ldots \geq m_q} k[X]^{\otimes(m_0+1)} \otimes
  k\Big[\prod_{i=1}^q \mathrm{Epi}_{\Delta S}\big([m_{i-1}], [m_i]\big)\Big].
\]

The face maps $\partial_i$ for $i > 0$ do not affect the degree of $Y \in
k[X]^{\otimes(m_0+1)}$. Only $\partial_0$ needs to be checked.  Since all morphisms are epic,
$\partial_0$ can only reduce the degree of $Y$.  Thus, $\mathscr{G}_*$ is compatible 
with the differential $d$.  The filtration quotients are:
\[
  E^0_{p,q} = 
  \;\bigoplus_{p = m_0 \geq
  \ldots \geq m_q} k[X]^{\otimes(p+1)} \otimes
  k\Big[\prod_{i=1}^q \mathrm{Epi}_{\Delta S}\big([m_{i-1}], [m_i]\big)\Big].
\]
The induced differential, $d^0$ on $E^0$ differs from $d$ only when $m_0 > m_1$. Indeed,
\[
  d^0 = \left\{\begin{array}{ll}
                 d, \quad & m_0 = m_1,\\
                 d - \partial_0, \quad & m_0 > m_1.
               \end{array}\right.
\]

$E^0$ splits into a direct sum based on the product of $x_i$'s in
$(x_0, \ldots, x_p) \in X^{p+1}$.  For $u \in X^{p+1}$, let
$C_u$ be the set of all distinct permutations of $u$.  Then,
\[
  E^0_{p,q} = 
  \;
  \bigoplus_{u \in X^{p+1}/\Sigma_{p+1}} \bigg(
  \bigoplus_{p = m_0 \geq \ldots \geq m_q} 
  \bigoplus_{w \in C_u}
  w \otimes k\Big[
  \prod_{i=1}^q \mathrm{Epi}_{\Delta S}\big([m_{i-1}], [m_i]\big)\Big] \bigg).
\]

Before proceeding with the main theorem of this chapter, we must define four related
categories
$\widetilde{\mathcal{S}}_p$, $\widetilde{\mathcal{S}}'_p$, $\mathcal{S}_p$, and 
$\mathcal{S}'_p$.  In the definitions
that follow, let $\{z_0, z_1, z_2, \ldots x_p\}$ be a set of independent
indeterminates over $k$.

\begin{definition}
  $\widetilde{\mathcal{S}}_p$ is the category with objects formal tensor products
  $Z_0 \otimes \ldots \otimes Z_s$, where each $Z_i$ is a non-empty product of $z_i$'s,
  and every one of $z_0, z_1, \ldots, z_p$ occurs exactly once in the tensor product.
  There is a unique morphism
  \[
    Z_0 \otimes \ldots \otimes Z_s \to Z'_0 \otimes \ldots \otimes Z'_t,
  \]
  if and only if the tensor factors of the latter are products of the factors of
  the former in some order.  In such a case,  
  there is a unique $\beta \in \mathrm{Epi}\Delta S$ so that
  $\beta_*(Z_0 \otimes \ldots \otimes Z_s) = Z'_0 \otimes \ldots \otimes Z'_t$.
\end{definition}

$\widetilde{\mathcal{S}}_p$ has initial objects $\sigma(z_0 \otimes z_1
\otimes \ldots \otimes z_p)$, for $\sigma \in \Sigma^{\mathrm{op}}_{p+1}$,
so $N\widetilde{\mathcal{S}}_p$ is a contractible complex.  Let $\widetilde{\mathcal{S}}'_p$
be the full subcategory of $\widetilde{\mathcal{S}}_p$ with all objects
$\sigma(z_0 \otimes \ldots \otimes z_p)$ deleted.

Let $\mathcal{S}_p$ be a skeletal category equivalent to $\widetilde{\mathcal{S}}_p$.  
In fact, we make make $\mathcal{S}_p$ the quotient category, identifying each
object \mbox{$Z_0 \otimes \ldots \otimes Z_s$} with any permutation of its tensor
factors, and identifying morphisms $\phi$ and $\psi$ if their source and target
are equivalent.  This
category has nerve $N\mathcal{S}_p$ homotopy-equivalent to $N\widetilde{\mathcal{S}}_p$
(see Prop.2.1 in~\cite{Se}, for example).  Now, $\mathcal{S}_p$ is a poset with
unique initial object, $z_0 \otimes \ldots \otimes z_p$.  Let $\mathcal{S}'_p$
be the full subcategory (subposet) of $\mathcal{S}_p$ obtained by deleting
the object $z_0 \otimes \ldots \otimes z_p$.  Clearly, $\mathcal{S}'_p$ is a
skeletal category equivalent to $\widetilde{\mathcal{S}}'_p$.

\begin{theorem}\label{thm.E1_NS}
There is spectral sequence converging weakly to $\widetilde{H}S_*(A)$ with
\[
  E^1_{p,q} \cong \bigoplus_{u\in X^{p+1}/\Sigma_{p+1}}
  H_{p+q}(EG_{u}\ltimes_{G_u}
  |N\mathcal{S}_p/N\mathcal{S}'_p|;k),
\]
where $G_{u}$ is the isotropy subgroup of the orbit 
$u \in X^{p+1}/\Sigma_{p+1}$.
\end{theorem}

Recall, for a group $G$, right $G$-space X, and left $G$-space $Y$, 
$X \ltimes_G Y$ denotes the \textit{equivariant half-smash product}.  If 
$\ast$ is a chosen basepoint for
$Y$ having trivial $G$-action, then
\[
  X \ltimes_G Y := (X \times_G Y)/(X \times_G \ast)   = X \times Y/\approx,
\]
with equivalence relation defined by $(x.g , y) \approx (x, g.y)$ and
$(x, \ast) \approx (x', \ast)$ for all $x, x' \in X$, $y \in Y$ and $g \in G$
(cf.~\cite{M4}).
In our case, $X$ is of the form $EG$, with canonical underlying complex $E_*G$.  
In chapter~\ref{chap.specseq}, we took $E_*G$ to have
a \textit{left} \mbox{$G$-action}, 
$g.(g_0, g_1, \ldots, g_n) = (gg_0, g_1, \ldots, g_n)$, but
$E_*G$ also has a \textit{right} \mbox{$G$-action}, $(g_0, g_1, \ldots, g_n).g =
(g_0, g_1, \ldots, g_ng)$.  It is this latter action that we shall use in the
definitions of \mbox{$EG_u \ltimes_{G_u} |N\mathcal{S}_p/
N\mathcal{S}'_p|$} and \mbox{$EG_u \ltimes_{G_u} |N\widetilde{\mathcal{S}}_p/
N\widetilde{\mathcal{S}}'_p|$}.

Observe, both $N\widetilde{\mathcal{S}}_p$ and $N\widetilde{\mathcal{S}}'_p$ carry
a left $\Sigma_{p+1}$-action (hence also a $G_u$-action) given by viewing $\sigma \in 
\Sigma_{p+1}$ as the $\Delta S$-isomorphism
$\overline{\sigma} \in \Sigma_{p+1}^\mathrm{op}$, then pre-composing with the 
$Z_i$'s (regarded as
morphisms written in tensor notation).  The result of the composition should be expressed
in tensor notation in order to be consistent.
\[
  \sigma.(Z_0 \stackrel{\phi_1}{\to} Z_1 \stackrel{\phi_2}{\to} \ldots
  \stackrel{\phi_q}{\to} Z_q) :=
  Z_0\overline{\sigma} \stackrel{\phi_1}{\to} Z_1\overline{\sigma} \stackrel{\phi_2}{\to} 
  \ldots \stackrel{\phi_q}{\to} Z_q\overline{\sigma}.
\]

Define for each $u \in X^{p+1}/\Sigma_{p+1}$, the following subcomplex of $E^0_{p,q}$:
\[
  \mathscr{M}_u \;:=\;\bigoplus_{p = m_0 \geq \ldots \geq m_q} 
  \bigoplus_{w \in C_u}
  w \otimes k\Big[
  \prod_{i=1}^q \mathrm{Epi}_{\Delta S}\big([m_{i-1}], [m_i]\big)\Big]
\]

\begin{lemma}\label{lem.G_u-identification}
There is a chain-isomorphism
\[
  \big(N\widetilde{\mathcal{S}}_p/N\widetilde{\mathcal{S}}'_p\big)/G_u
  \stackrel{\cong}{\to} \mathscr{M}_u
\]
\end{lemma}
\begin{proof}
  The forward map is given on generators by:
  \[
    (Z_0 \stackrel{\phi_1}{\to} \ldots
    \stackrel{\phi_q}{\to} Z_q) \;\mapsto\;
    Z_0(u) \otimes \phi_1 \otimes \ldots \otimes \phi_q.
  \]
  This map is well-defined, since if $g \in G_u$, then
  \[
    g.(Z_0 \stackrel{\phi_1}{\to} \ldots
    \stackrel{\phi_q}{\to} Z_q) =
    (Z_0\overline{g} \stackrel{\phi_1}{\to} \ldots
    \stackrel{\phi_q}{\to} Z_q\overline{g})
    \;\mapsto
  \]
  \[
    Z_0\overline{g}(u) \otimes \phi_1 \otimes \ldots \otimes \phi_q
    = Z_0(u) \otimes \phi_1 \otimes \ldots \otimes \phi_q
  \]
    
  For the opposite direction, we begin with a generator of the form
  \begin{equation}\label{eq.w_element}
    w \otimes \phi_1 \otimes \ldots \otimes \phi_q, \quad w \in C_u.
  \end{equation}
  Let $\tau \in \Sigma_{p+1}$ so that $w = \overline{\tau}(u)$.  Then the image 
  of~\ref{eq.w_element} is defined to be
  \begin{equation}\label{eq.image-w_element}
    \overline{\tau} \stackrel{\phi}{\to} \phi_1\overline{\tau} 
    \stackrel{\phi_2}{\to} \ldots \stackrel{\phi_q}{\to}
    \phi_q\ldots\phi_1\overline{\tau}.
  \end{equation}
  We must check that this definition does not depend on choice of $\tau$.  Indeed, if
  $w = \overline{\sigma}(u)$ also, then $u = \overline{\sigma}^{-1}\overline{\tau}(u)$,
  hence $\tau\sigma^{-1} \in G_u$.  Thus,
  \[
    \overline{\sigma} \stackrel{\phi}{\to} \phi_1\overline{\sigma} 
    \stackrel{\phi_2}{\to} \ldots \stackrel{\phi_q}{\to}
    \phi_q\ldots\phi_1\overline{\sigma}
    \;\approx\;\tau\sigma^{-1} . (\overline{\sigma} \stackrel{\phi}{\to} 
    \phi_1\overline{\sigma} 
    \stackrel{\phi_2}{\to} \ldots \stackrel{\phi_q}{\to}
    \phi_q\ldots\phi_1\overline{\sigma})
  \]
  \[
    =\; \overline{\sigma}(\overline{\sigma}^{-1}\overline{\tau}) \stackrel{\phi}{\to} 
    \phi_1\overline{\sigma}(\overline{\sigma}^{-1}\overline{\tau})
    \stackrel{\phi_2}{\to} \ldots \stackrel{\phi_q}{\to}
    \phi_q\ldots\phi_1\overline{\sigma}(\overline{\sigma}^{-1}\overline{\tau})
    \;=\;\overline{\tau} \stackrel{\phi}{\to} \phi_1\overline{\tau} 
    \stackrel{\phi_2}{\to} \ldots \stackrel{\phi_q}{\to}
    \phi_q\ldots\phi_1\overline{\tau}.
  \]
  The maps are clearly inverse to one another.  All that remains is to verify that
  these are chain maps.  For $i>0$, the face maps $\partial_i$ simply compose the
  maps $\phi_{i+1}$ and $\phi_i$ in either chain complex, so only the zeroth face
  map needs to be checked.  First, for the forward map, assume $\phi_1$ is an
  isomorphism.
  \[
    (Z_0 \stackrel{\phi_1}{\to} \ldots
    \stackrel{\phi_q}{\to} Z_q) \;\mapsto\;
    Z_0(u) \otimes \phi_1 \otimes \ldots \otimes \phi_q \;\stackrel{\partial_0}{\mapsto}\;
    \phi_1Z_0(u) \otimes \phi_2 \otimes \ldots \otimes \phi_q,
  \]
  while
  \[
    (Z_0 \stackrel{\phi_1}{\to} \ldots
    \stackrel{\phi_q}{\to} Z_q)\;\stackrel{\partial_0}{\mapsto}\;
    (Z_1 \stackrel{\phi_2}{\to} \ldots
    \stackrel{\phi_q}{\to} Z_q)\;\mapsto\;
    Z_1(u) \otimes \phi_2 \otimes \ldots \otimes \phi_q.
  \]
  The two results agree since $\phi_1Z_0 = Z_1$.  If $\phi_1$ is not an isomorphism,
  then it must be a strict epimorphism, and so $Z_0(u) \otimes \phi_1 \otimes \ldots
  \otimes \phi_q = 0$ in $E^0$.  On the other hand,  
  the chain $Z_1 \to \ldots \to Z_q$ is in
  $N\widetilde{\mathcal{S}}'_p$, hence should also be identified with $0$.  
  
  In the reverse direction, assume
  $w = \overline{\tau}(u)$ as above, and let $\phi_1$ be an isomorphism.
  \[
    w \otimes \phi_1 \otimes \ldots \otimes \phi_q \;\mapsto\;
    (\overline{\tau} \stackrel{\phi}{\to} \phi_1\overline{\tau} 
    \stackrel{\phi_2}{\to} \ldots \stackrel{\phi_q}{\to}
    \phi_q\ldots\phi_1\overline{\tau}) \;\stackrel{\partial_0}{\mapsto}\;
    (\phi_1\overline{\tau} \stackrel{\phi_2}{\to} \phi_2\phi_1\overline{\tau}
    \stackrel{\phi_3}{\to}
    \ldots \stackrel{\phi_q}{\to} \phi_q\ldots\phi_1\overline{\tau}),
  \]
  while
  \[
    w \otimes \phi_1 \otimes \ldots \otimes \phi_q \;\stackrel{\partial_0}{\mapsto}\;
    \phi_1(w) \otimes \phi_2 \ldots \otimes \phi_q \;\mapsto\;
    (\phi_1\overline{\tau} \stackrel{\phi_2}{\to} \phi_2\phi_1\overline{\tau}
    \stackrel{\phi_3}{\to}
    \ldots \stackrel{\phi_q}{\to} \phi_q\ldots\phi_1\overline{\tau}).
  \]
  The rightmost expression results from the observation that if $w = \overline{\tau}(u)$,
  then \mbox{$\phi_1(w) = \phi_1\overline{\tau}(u)$}.  Now, if $\phi_1$ is a
  strict epimorphism, then both results are $0$ for similar reasons as above.
\end{proof}

Using this lemma, we identify $\mathscr{M}_u$ with the orbit complex 
$\big(N\widetilde{\mathcal{S}}_p/N\widetilde{\mathcal{S}}'_p\big)/G_u$.  Now,
the complex $N\widetilde{\mathcal{S}}_p/N\widetilde{\mathcal{S}}'_p$ is a 
free $G_u$-complex, so we have an isomorphism:
\[
  H_*\Big(\big(N\widetilde{\mathcal{S}}_p/N\widetilde{\mathcal{S}}'_p\big)/G_u\Big)
  \cong H^{G_u}_*\big(N\widetilde{\mathcal{S}}_p/N\widetilde{\mathcal{S}}'_p\big),
\]
(\textit{i.e.}, $G_u$-equivariant homology.  See~\cite{B} for details).
Then, by definition,
\[
  H^{G_u}_*\big(N\widetilde{\mathcal{S}}_p/N\widetilde{\mathcal{S}}'_p\big) =
  H_*(G_u, N\widetilde{\mathcal{S}}_p/N\widetilde{\mathcal{S}}'_p),
\]
which may be computed using the free resolution, $E_*G_u$ of $k$ as right
$G_u$-module.  The resulting complex $k[E_*G_u] \otimes_{kG_u} 
k[N\widetilde{\mathcal{S}}_p]/k[N\widetilde{\mathcal{S}}'_p]$ is a double complex
isomorphic to the quotient of two double complexes, namely:
\[
  \big(k[E_*G_u] \otimes_{kG_u} k[N\widetilde{\mathcal{S}}_p]\big)/
  \big(k[E_*G_u] \otimes_{kG_u} k[N\widetilde{\mathcal{S}}'_p]\big)
\]
\[
  \cong k\Big[ \big(E_*G_u \times_{G_u} N\widetilde{\mathcal{S}}_p\big)/
  \big(E_*G_u \times_{G_u} N\widetilde{\mathcal{S}}'_p\big)\Big].
\]
This last complex may be identified 
with the simplicial complex of the space
\[
  \big(EG_u \times_{G_u} |N\widetilde{\mathcal{S}}_p|\big)/
  \big(EG_u \times_{G_u} |N\widetilde{\mathcal{S}}'_p|\big)
\]
\[
  \cong EG_u \ltimes_{G_u} |N\widetilde{\mathcal{S}}_p/N\widetilde{\mathcal{S}}'_p|.
\]

The last piece of the puzzle involves simplifying the spaces 
$|N\widetilde{\mathcal{S}}_p/N\widetilde{\mathcal{S}}'_p|$.  Since $\mathcal{S}$
is a skeletal subcategory of $\widetilde{\mathcal{S}}$, there is an
equivalence of categories $\widetilde{\mathcal{S}} \simeq \mathcal{S}$, inducing
a homotopy equivalence of complexes (hence also of spaces)
$|N\widetilde{\mathcal{S}}| \simeq |N\mathcal{S}|$.  Note that $N\mathcal{S}$
inherits a $G_u$-action from $N\widetilde{\mathcal{S}}$, and the map
$\widetilde{\mathcal{S}} \to \mathcal{S}$ is $G_u$-equivariant.  Consider the fibration
\[
  X \to EG \times_{G} X \to BG
\]
associated to a group $G$ and path-connected $G$-space $X$.  The resulting homotopy 
sequence breaks up
into isomorphisms \mbox{$0 \to \pi_i(X) 
\stackrel{\cong}{\to} \pi_i(EG \times_G X) \to 0$} for $i \geq 2$ and a
short exact sequence \mbox{$0 \to \pi_1(X) \to \pi_1(EG \times_G X) \to G \to 0$}.  If 
there is
a $G$-equivariant \mbox{$f : X \to Y$} for a path-connected $G$-space $Y$, then for 
$i \geq 2$, we have
isomorphisms 
\[
  \pi_i(EG \times_{G} X) \gets \pi_i(X) \stackrel{f_*}{\to} 
  \pi_i(Y) \to \pi_i(EG \times_G Y),
\]
and a diagram corresponding to $i = 1$:
\[
  \begin{diagram}
    \node{0}
    \arrow{e}
    \arrow{s,=}
    \node{\pi_1(X)}
    \arrow{s,l}{f_*}
    \arrow{e}
    \node{\pi_1(EG \times_G X)}
    \arrow{s,r}{(\mathrm{id} \times f)_*}
    \arrow{e}
    \node{G}
    \arrow{s,=}
    \arrow{e}
    \node{0}
    \arrow{s,=}
    \\
    \node{0}
    \arrow{e}
    \node{\pi_1(Y)}
    \arrow{e}
    \node{\pi_1(EG \times_G Y)}
    \arrow{e}
    \node{G}
    \arrow{e}
    \node{0}
  \end{diagram}
\]
Thus, there is a weak equivalence $EG \times_G X \to EG \times_G Y$.  In our case, we
wish to obtain weak equivalences:
\[
  EG_u \times_{G_u} |N\widetilde{\mathcal{S}}_p| \to EG_u \times_{G_u} |N\mathcal{S}_p|
\]
and
\[
  EG_u \times_{G_u} |N\widetilde{\mathcal{S}}'_p| \to EG_u \times_{G_u} |N\mathcal{S}'_p|,
\]
inducing a weak equivalence
\[
  EG_u \ltimes_{G_u} |N\widetilde{\mathcal{S}}_p/N\widetilde{\mathcal{S}}'_p|
  \to
  EG_u \ltimes_{G_u} |N\mathcal{S}_p/N\mathcal{S}'_p|.
\]

This will follow as long as the spaces $|N\widetilde{\mathcal{S}}'_p|$ and
$|N\mathcal{S}'_p|$ are path-connected.  (Note, $|N\widetilde{\mathcal{S}}_p|$ and
$|N\mathcal{S}_p|$ are path-connected
because they are contractible).  In fact, we need only check $|N\mathcal{S}'_p|$,
since this space is homotopy-equivalent to $|N\widetilde{\mathcal{S}}'_p|$.
\begin{lemma}\label{lemma.NS-path-connected}
  For $p > 2$, $|N\mathcal{S}'_p|$ is path-connected.
\end{lemma}
\begin{proof}
  Assume $p > 2$ and
  let $W_0 := z_0z_1 \otimes z_2 \otimes \ldots \otimes z_p$.  This represents
  a vertex of $N\mathcal{S}'_p$.  Suppose
  $W = Z_0 \otimes \ldots \otimes Z_i'z_0z_1Z_i'' \otimes \ldots \otimes Z_s$.
  Then there is a morphism $W_0 \to W$, hence an edge between $W_0$ and $W$.
  Next, suppose $W = Z_0 \otimes \ldots \otimes Z_i'z_0Z_i''z_1Z_i''' \otimes \ldots 
  \otimes Z_s$.  There is a path:
  \[
    \begin{diagram}
      \node{Z_0 \otimes \ldots \otimes Z_i'z_0Z_i''z_1Z_i''' \otimes \ldots 
        \otimes Z_s}
      \arrow{s}\\
      \node{Z_0Z_1 \ldots Z_i'z_0Z_i''z_1Z_i''' \ldots Z_s}\\
      \node{Z_0Z_1 \ldots Z_i' \otimes z_0 \otimes Z_i''z_1Z_i''' \ldots Z_s}
      \arrow{n}
      \arrow{s}\\
      \node{z_0 \otimes Z_0Z_1 \ldots Z_i'Z_i''z_1Z_i''' \ldots Z_s}\\
      \node{z_0 \otimes Z_0Z_1 \ldots Z_i'Z_i'' \otimes z_1Z_i''' \ldots Z_s}
      \arrow{n}
      \arrow{s}\\
      \node{z_0z_1Z_i''' \ldots Z_s \otimes Z_0Z_1 \ldots Z_i'Z_i''}\\
      \node{W_0}
      \arrow{n}
    \end{diagram}
  \]
  Similarly, if $W = Z_0 \otimes \ldots \otimes Z_i'z_1Z_i''z_0Z_i''' \otimes \ldots 
  \otimes Z_s$, there is a path to $W_0$.  Finally, if
  $W = Z_0 \otimes \ldots \otimes Z_s$ with $z_0$ occurring in $Z_i$ and
  $z_1$ occurring in $Z_j$ for $i \neq j$, there is an edge to some 
  $W'$ in which $Z_iZ_j$ occurs, and thus a path to $W_0$.
\end{proof}

The above discussion coupled with Lemma~\ref{lem.G_u-identification} produces the
required isomorphism in homology, hence proving Thm.~\ref{thm.E1_NS} for $p > 2$:
\[
  E^1_{p,q} \;=\; \bigoplus_{u \in X^{p+1}/\Sigma_{p+1}} H_{p+q}(\mathscr{M}_u) 
  \;\cong\; \bigoplus_{u \in X^{p+1}/\Sigma_{p+1}}H_{p+q}
  \left( EG_u \ltimes_{G_u}
  |N\widetilde{\mathcal{S}}_p / N\widetilde{\mathcal{S}}'_p |; k \right)
\]
\[
  \cong\;
  \bigoplus_{u \in X^{p+1}/\Sigma_{p+1}} H_{p+q}\left(
  EG_u \ltimes_{G_u} |N\mathcal{S}_p/N\mathcal{S}'_p|, k\right).
\]

The cases $p = 0, 1$ and $2$ are handled individually:

Observe that $|N\widetilde{\mathcal{S}}'_0|$ and 
$|N\mathcal{S}'_0|$ are empty spaces, since $\widetilde{\mathcal{S}}'_0$
has no objects.  
\[
  EG_u \times_{G_u} |N\widetilde{\mathcal{S}}'_0|
  = EG_u \times_{G_u} |N\mathcal{S}'_0| = \emptyset.
\]
Furthermore, any group $G_u$ must be trivial.  Thus, 
\[
  H_q\left( EG_u \ltimes_{G_u}|N\widetilde{\mathcal{S}}_0 / 
  N\widetilde{\mathcal{S}}'_0 |; k \right)
  = H_q\left( |N\widetilde{\mathcal{S}}_0|; k \right)
\]
\[
  \cong
  H_q\left( |N\mathcal{S}_0|; k \right)
  = H_q\left(EG_u \ltimes_{G_u} |N\mathcal{S}_p/N\mathcal{S}'_p|, k\right),
\]
completing the theorem for $p=0$.

Next, since $|N\widetilde{\mathcal{S}}'_1|$ is homeomorphic to
$|N\mathcal{S}'_1|$, each space consisting of the two discrete points $z_0z_1$ and
$z_1z_0$, 
the theorem is true for $p = 1$ as well.  

For $p=2$, observe that $|N\widetilde{\mathcal{S}}'_2|$ has two connected 
components, $\widetilde{U}_1$ and $\widetilde{U}_2$ that are interchanged by any 
odd permutation $\sigma \in \Sigma_3$.  Similarly, $|N\mathcal{S}'_2|$ consists
of two connected components, $U_1$ and $U_2$, interchanged by any odd permutation
of $\Sigma_3$.  Now, resticted to the alternating group, $A\Sigma_3$, we certianly
have weak equivalences for any subgroup $H_u \subset A\Sigma_3$,
\[
  EH_u \times_{H_u} \widetilde{U}_1 \stackrel{\simeq}{\longrightarrow}
  EH_u \times_{H_u} U_1,
\]
\[
  EH_u \times_{H_u} \widetilde{U}_2 \stackrel{\simeq}{\longrightarrow}
  EH_u \times_{H_u} U_2.
\]
The action of an odd permutation induces equivariant homeomorphisms
\[
  \widetilde{U}_1 \stackrel{\cong}{\longrightarrow} \widetilde{U}_2,
\]
\[
  U_1 \stackrel{\cong}{\longrightarrow} U_2,
\]
and so if we have a subgroup $G_u \subset \Sigma_3$ generated by $H_u \subset
A\Sigma_3$ and a transposition, then the two connected components are
identified in an \mbox{$A\Sigma_3$-equivariant} manner.  Thus, if $G_u$ contains 
a transposition,
\[
  EG_u \times_{G_u} |N\widetilde{\mathcal{S}}'_2| \cong
  EH_u \times_{H_u} \widetilde{U}_1 \simeq EH_u \times_{H_u} U_1
  \cong EG_u \times_{G_u} |N\mathcal{S}'_2|.
\]
This completes the case $p=2$ and the proof of Thm.~\ref{thm.E1_NS}.

\begin{cor}\label{cor.square-zero}
If the augmentation ideal of $A$ satisfies $I^2 = 0$, then 
\[
  HS_n(A) \;\cong\; \bigoplus_{p \geq 0} \bigoplus_{u \in X^{p+1}/\Sigma_{p+1}}
  H_{n}(EG_u \ltimes_{G_u} N\mathcal{S}_p/N\mathcal{S}'_p; k).
\]
\end{cor}
\begin{proof}
  This follows from consideration of the original $E^0$ term of the spectral
  sequence.
  $E^0$ is generated by chains $Y \otimes \phi_0 \otimes \ldots \otimes \phi_n$,
  with induced differential $d^0$, agreeing with the differential $d$ of 
  $\widetilde{\mathscr{Y}}_*$ when $\phi_0$ is an isomorphism.  When $\phi_0$ is a strict
  epimorphism, however, $d^0 = d - \partial_0$.  But if $\phi_0$ is strictly epic,
  then
  \[
    \partial_0( Y \otimes \phi_0 \otimes \ldots \otimes \phi_n)
    = (\phi_0)_*(Y) \otimes \phi_1 \otimes \ldots \otimes \phi_n
    \;=\; 0,
  \]
  since $(\phi_0)_*(Y)$ would have at least one tensor factor that is the product of
  two or more elements of $I$, hence, $d^0$ agrees with $d$ in the case that
  $\phi_0$ is strictly epic.  But if $d^0$ agrees with $d$ for all elements, the
  spectral sequence collapses at level 1.
\end{proof}

\section{The complex $Sym_*^{(p)}$}\label{sec.sym-p}              %

Note, for $p > 0$, there are homotopy equivalences
\[
  |N\mathcal{S}_p/N\mathcal{S}'_p| \simeq |N\mathcal{S}_p| 
  \vee S|N\mathcal{S}'_p| \simeq S|N\mathcal{S}'_p|,
\]
since $|N\mathcal{S}_p|$ is contractible.  $|N\mathcal{S}_p|$ is
a disjoint union of $(p+1)!$ $p$-cubes, identified along certain faces.
Geometric analysis of 
$S|N\mathcal{S}'_p|$, however, seems quite difficult.  Fortunately, there is an even
smaller chain complex that computes the homology of $|N\mathcal{S}_p/N\mathcal{S}'_p|$.

\begin{definition}\label{def.sym_complex}
  Let $p \geq 0$ and impose an equivalence relation on $k\big[\mathrm{Epi}_{\Delta S}
  ([p], [q])\big]$ generated by:
  \[
    Z_0 \otimes \ldots \otimes Z_i \otimes Z_{i+1} \otimes \ldots \otimes Z_q
    \approx
    (-1)^{ab} Z_0 \otimes \ldots \otimes Z_{i+1} \otimes Z_{i} \otimes \ldots \otimes Z_q,
  \]
  where $Z_0 \otimes \ldots \otimes Z_q$ is a morphism expressed in tensor notation, and
  $a = deg(Z_i) := |Z_i| - 1$, $b = deg(Z_{i+1}) := |Z_{i+1}| - 1$.  Here, $deg(Z)$ is one less than
  the number of factors of the monomial $Z$.  Indeed, 
  if $Z = z_{i_0}z_{i_1}   \ldots   z_{i_s}$, then $deg(Z) = s$.
    
  The complex $Sym_*^{(p)}$ is then defined by:
  \begin{equation}\label{eq.def_Sym}
    Sym_i^{(p)} \;:=\; k\big[\mathrm{Epi}_{\Delta S}([p], [p-i])\big]/\approx,
  \end{equation}
  The face maps will be defined recursively.  On monomials,
  \begin{equation}\label{eq.Sym-differential}
    \partial_i(z_{j_0} \ldots z_{j_s}) = \left\{\begin{array}{ll}
       0, & i < 0,\\
       z_{j_0}   \ldots   z_{j_i} \otimes z_{j_{i+1}}   \ldots   z_{j_s},
       \quad,& 0 \leq i < s,\\
       0, & i \geq s.
       \end{array}\right.
  \end{equation}
  Then, extend $\partial_i$ to tensor products via:
  \begin{equation}\label{eq.face_i-tensor}
    \partial_i(W \otimes V) = \partial_i(W) \otimes V + W \otimes \partial_{i-deg(W)}(V).
  \end{equation}
  In the above formula, $W$ and $V$ are formal tensors in 
  $k\big[\mathrm{Epi}_{\Delta S}([p], [q])\big]$, and
  \[
    deg(W) = deg(W_0 \otimes \ldots \otimes W_t) \;:=\; \sum_{k=0}^t deg(W_k).
  \]
  The boundary map $Sym_n^{(p)} \to Sym_{n-1}^{(p)}$ 
  is then
  \[
    d_n = \sum_{i=0}^n (-1)^i \partial_i = \sum_{i=0}^{n-1} (-1)^i \partial_i
  \]
\end{definition}
\begin{rmk}
  The result of applying $\partial_i$ on any
  formal tensor will result in only a single formal tensor, since in eq.~\ref{eq.face_i-tensor},
  at most one of the two terms will be non-zero).
\end{rmk}
\begin{rmk}\label{rmk.action}
  There is an action $\Sigma_{p+1} \times Sym_i^{(p)} \to Sym_i^{(p)}$, given by permuting the
  formal indeterminates $z_i$.  Furthermore, this action is compatible with the differential.
\end{rmk}
\begin{lemma}\label{lem.NS-Sym-homotopy}
   $Sym_*^{(p)}$ is chain-homotopy equivalent to $k[N\mathcal{S}_p]/k[N\mathcal{S}'_p]$.
\end{lemma}
\begin{proof}
  Let $v_0$ represent the common initial vertex of the $p$-cubes making up $N\mathcal{S}_p$.
  Then, as cell-complex, $N\mathcal{S}_p$ consists of $v_0$ together with all corners of
  the various $p$-cubes, together with $i$-cells for each $i$-face of the cubes.  Thus,
  $N\mathcal{S}_p$ consists of $(p+1)!$ $p$-cells with attaching maps 
  \mbox{$\partial I^p \to (N\mathcal{S}_p)^{p-1}$} defined according to the face maps for
  $N\mathcal{S}_p$ given above.  Presently, I shall provide an explicit construction.
  
  Label each top-dimensional cell with the permutation
  induced on $\{0, 1, \ldots, p\}$ by the final vertex, $z_{i_0}z_{i_1}\ldots z_{i_p}$.
  On a given $p$-cell, for each vertex $Z_0 \otimes \ldots \otimes Z_s$, there is an 
  ordering of the tensor
  factors so that \mbox{$Z_0 \otimes \ldots \otimes Z_s \to z_{i_0}z_{i_1}\ldots z_{i_p}$}
  preserves the order of formal indeterminates $z_i$.  Rewrite each vertex of this $p$-cell
  in this order.  Now, any chain
  \[
    (z_{i_0} \otimes z_{i_1} \otimes \ldots \otimes z_{i_p}) \to \ldots \to
    z_{i_0}z_{i_1}\ldots z_{i_p}
  \]
  is obtained by choosing the order in which to combine the factors.  In fact, 
  the $p$-chains
  are in bijection with the elements of $S_{p}$.  A given permutation
  \mbox{$\{1,2, \ldots, p\}\mapsto \{ j_1, j_2, \ldots, j_p \}$} will represent the
  chain obtained by first combining $z_{j_0} \otimes z_{j_1}$ into 
  $z_{j_0}z_{j_1}$, then combining $z_{j_1} \otimes z_{j_2}$ into 
  $z_{j_1}z_{j_2}$.  In effect, we ``erase'' the tensor product symbol between
  $z_{j_{r-1}}$ and $z_{j_r}$ for each $j_r$ in the list above.
  
  We shall declare that
  the {\it natural} order of combining the factors will be 
  the one that always combines the last two:
  \[
    (z_{i_0} \otimes \ldots \otimes z_{i_{p-1}} \otimes z_{i_p}) \to
    (z_{i_0} \otimes \ldots \otimes z_{i_{p-1}}z_{i_p}) \to
    (z_{i_0} \otimes z_{i_{p-2}}z_{i_{p-1}}z_{i_p}) \to 
    \ldots.
  \]
  This corresponds to a permutation $\rho := \{1,\ldots, p\} \mapsto
  \{p, p-1, \ldots, 2, 1\}$, and this chain will be regarded as {\it positive}.
  A chain $C_\sigma$, corresponding to another permutation, $\sigma$, will be
  regarded as positive or negative depending on the sign of the permutation
  $\sigma\rho^{-1}$.  Finally, the entire $p$-cell should be identified with the sum
  \[
    \sum_{\sigma \in S_p} sign(\sigma\rho^{-1}) C_\sigma
  \]
  It is this sign convention that
  permits the inner faces of the cube to cancel appropriately in the boundary maps.
  Thus we have a map on the top-dimensional chains:
  \[
    \theta_p : Sym_p^{(p)} \to \big(k[N\mathcal{S}_p]/k[N\mathcal{S}'_p]\big)_p.
  \]
  Define $\theta_*$ for arbitrary $k$-cells by sending $Z_0 \otimes \ldots
  \otimes Z_{p-k}$ to the sum of $k$-length chains with source $z_0 \otimes
  \ldots \otimes z_p$ and target $Z_0 \otimes \ldots \otimes Z_{p-k}$ with signs determined
  by the natural order of erasing tensor product symbols of $z_0 \otimes \ldots \otimes z_p$,
  excluding those tensor product symbols that never get erased.  
  
  The following example
  should clarify the point.  $W = z_3z_0 \otimes z_1 \otimes z_2z_4$ is a $2$-cell of
  $Sym_*^{(4)}$.  $W$ is obtained from $z_0 \otimes z_1 \otimes z_2 \otimes z_3
  \otimes z_4 = z_3 \otimes z_0 \otimes z_1 \otimes z_2 \otimes z_4$ by combining
  factors in some order.  There are only $2$ erasable tensor product symbols in
  this example.  The natural order (last to first) corresponds to the chain:
  \[
    z_3 \otimes z_0 \otimes z_1 \otimes z_2 \otimes z_4
    \to z_3 \otimes z_0 \otimes z_1 \otimes z_2z_4
    \to z_3z_0 \otimes z_1 \otimes z_2z_4.
  \]
  So, this chain shows up in $\theta_*(W)$ with positive sign, whereas the chain
  \[
    z_3 \otimes z_0 \otimes z_1 \otimes z_2 \otimes z_4
    \to z_3z_0 \otimes z_1 \otimes z_2 \otimes z_4
    \to z_3z_0 \otimes z_1 \otimes z_2z_4
  \]
  shows up with a negative sign.
  
  Now, $\theta_*$ is easily seen to be a chain map $Sym_*^{(p)} \to
  k[N\mathcal{S}_p]/k[N\mathcal{S}'_p]$.  Geometrically, $\theta_*$ has
  the effect of subdividing a cell-complex (defined with cubical cells) into
  a simplicial space.
\end{proof}  
    
As an example, consider $|N\mathcal{S}_2|$. There are $6$ \mbox{$2$-cells},
each represented by a copy of $I^2$.  The \mbox{$2$-cell} labelled by the permutation 
\mbox{$\{0,1,2\} \mapsto \{1,0,2\}$}
consists of the chains 
\[
  z_1 \otimes z_0 \otimes z_2 \to z_1 \otimes z_0z_2 \to z_1z_0z_2
\]
and
\[
  -(z_1 \otimes z_0 \otimes z_2 \to z_1z_0 \otimes z_2 \to z_1z_0z_2).
\]
Hence, the boundary is the sum of \mbox{$1$-chains}:
\[
  (z_1 \otimes z_0z_2 \to z_1z_0z_2) - 
  (z_1 \otimes z_0 \otimes z_2 \to z_1z_0z_2)
  + (z_1 \otimes z_0 \otimes z_2 \to z_1 \otimes z_0z_2)\qquad\qquad\qquad 
\]
\[
  \qquad\qquad\qquad-(z_1z_0 \otimes z_2 \to z_1z_0z_2) +  
  (z_1 \otimes z_0 \otimes z_2 \to z_1z_0z_2)
  - (z_1 \otimes z_0 \otimes z_2 \to z_1z_0 \otimes z_2)
\]
\[
  = (z_1 \otimes z_0z_2 \to z_1z_0z_2) + (z_1 \otimes z_0 \otimes z_2 \to z_1 \otimes z_0z_2)
  \qquad\qquad\qquad\qquad\qquad\qquad
\]
\[
  \qquad\qquad\qquad
  - (z_1z_0 \otimes z_2 \to z_1z_0z_2) - (z_1 \otimes z_0 \otimes z_2 \to z_1z_0 \otimes z_2).
\]
These \mbox{$1$-chains} correspond to the $4$ edges of the square.
  
Thus, in our example this $2$-cell
of $|N\mathcal{S}_p|$ will correspond to \mbox{$z_1z_0z_2 \in Sym_2^{(2)}$}, and
its boundary in $|N\mathcal{S}_p/N\mathcal{S}'_p|$ will consist of the two edges adjacent
to $z_0 \otimes z_1 \otimes z_2$ with appropriate signs:
\[
  (z_0 \otimes z_1 \otimes z_2 \to z_1 \otimes z_0z_2)
  - (z_0 \otimes z_1 \otimes z_2 \to z_1z_0 \otimes z_2).
\]
The corresponding boundary in $Sym_1^{(2)}$ will be: 
\mbox{$(z_1 \otimes z_0z_2) - (z_1z_0 \otimes z_2)$}, matching with the differential 
already defined on $Sym_*^{(p)}$.  
See Figs.~4.1 and~4.2.

\begin{figure}[ht]
  \psset{unit=.75in}
  \begin{pspicture}(6,3.5)
  
  \psdots[linecolor=black, dotsize=4pt]
  (2, 2)(2, 3.5)(3.3, 2.75)(3.3, 1.25)(2, .5)
  (.7, 1.25)(.7, 2.75)
  
  \psline[linewidth=1pt, linecolor=black](2, 2)(2, 3.5)
  \psline[linewidth=1pt, linecolor=black](2, 2)(3.3, 1.25)
  \psline[linewidth=1pt, linecolor=black](2, 2)(.7, 1.25)
  \psline[linewidth=1pt, linecolor=black](2, 3.5)(3.3, 2.75)
  \psline[linewidth=1pt, linecolor=black](3.3, 2.75)(3.3, 1.25)
  \psline[linewidth=1pt, linecolor=black](3.3, 1.25)(2, .5)
  \psline[linewidth=1pt, linecolor=black](2, .5)(.7, 1.25)
  \psline[linewidth=1pt, linecolor=black](.7, 1.25)(.7, 2.75)
  \psline[linewidth=1pt, linecolor=black](.7, 2.75)(2, 3.5)
  
  \rput(1.9, 2.12){$z_0 \otimes z_1 \otimes z_2$}
  \rput(2, 3.62){$z_0z_1 \otimes z_2$}
  \rput(3.28, 1.16){$z_0 \otimes z_1z_2$}
  \rput(.9, 1.16){$z_1 \otimes z_2z_0$}
  \rput(3.5, 2.87){$z_0z_1z_2$}
  \rput(2, .38){$z_1z_2z_0$}
  \rput(.6, 2.9){$z_2z_0z_1$}
  
  \psdots[linecolor=black, dotsize=4pt]
  (6, 2)(6, 3.5)(7.3, 2.75)(7.3, 1.25)(6, .5)
  (4.7, 1.25)(4.7, 2.75)
  
  \psline[linewidth=1pt, linecolor=black](6, 2)(6, 3.5)
  \psline[linewidth=1pt, linecolor=black](6, 2)(7.3, 1.25)
  \psline[linewidth=1pt, linecolor=black](6, 2)(4.7, 1.25)
  \psline[linewidth=1pt, linecolor=black](6, 3.5)(7.3, 2.75)
  \psline[linewidth=1pt, linecolor=black](7.3, 2.75)(7.3, 1.25)
  \psline[linewidth=1pt, linecolor=black](7.3, 1.25)(6, .5)
  \psline[linewidth=1pt, linecolor=black](6, .5)(4.7, 1.25)
  \psline[linewidth=1pt, linecolor=black](4.7, 1.25)(4.7, 2.75)
  \psline[linewidth=1pt, linecolor=black](4.7, 2.75)(6, 3.5)

  \rput(5.9, 2.12){$z_0 \otimes z_1 \otimes z_2$}
  \rput(6, 3.62){$z_1z_0 \otimes z_2$}
  \rput(7.28, 1.16){$z_1 \otimes z_0z_2$}
  \rput(4.9, 1.16){$z_0 \otimes z_2z_1$}
  \rput(7.5, 2.87){$z_1z_0z_2$}
  \rput(6, .38){$z_0z_2z_1$}
  \rput(4.6, 2.9){$z_2z_1z_0$}

  \end{pspicture}

  \caption[$|N\mathcal{S}_2|$]{$|N\mathcal{S}_2|$ consists of six squares, grouped
  into two hexagons that share a common center vertex}
  \label{fig.NS_2}
\end{figure}
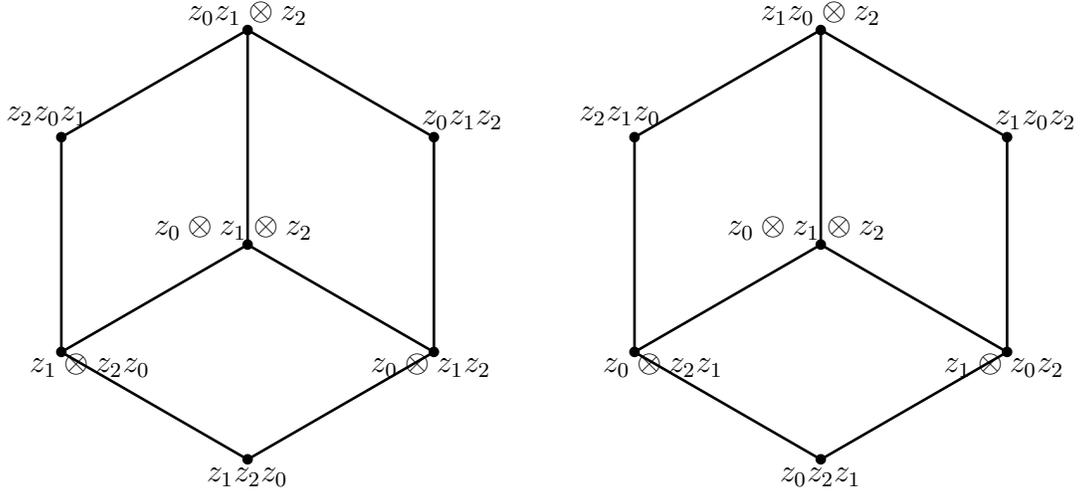

\begin{figure}[ht]
  \psset{unit=.75in}
  \begin{pspicture}(6,3.5)
  
  \psdots[linecolor=gray, dotsize=4pt]
  (2, 3.5)(3.3, 2.75)(3.3, 1.25)(2, .5)
  (.7, 1.25)(.7, 2.75)
  
  \psline[linewidth=1pt, linecolor=gray](2, 2)(2, 3.5)
  \psline[linewidth=1pt, linecolor=gray](2, 2)(3.3, 1.25)
  \psline[linewidth=1pt, linecolor=gray](2, 2)(.7, 1.25)
  \psline[linewidth=1pt, linecolor=gray](2, 3.5)(3.3, 2.75)
  \psline[linewidth=1pt, linecolor=gray](3.3, 2.75)(3.3, 1.25)
  \psline[linewidth=1pt, linecolor=gray](3.3, 1.25)(2, .5)
  \psline[linewidth=1pt, linecolor=gray](2, .5)(.7, 1.25)
  \psline[linewidth=1pt, linecolor=gray](.7, 1.25)(.7, 2.75)
  \psline[linewidth=1pt, linecolor=gray](.7, 2.75)(2, 3.5)
  
  \psdots[linecolor=black, dotsize=4pt]
  (2, 2)(2, 2.75)(2.65, 1.625)(1.35, 1.625)
  (2, 1.25)(1.35, 2.375)(2.65, 2.375)
  
  \psline[linewidth=1pt, linecolor=black]{->}(2, 1.25)(2.65, 1.625)
  \psline[linewidth=1pt, linecolor=black]{->}(2, 1.25)(1.35, 1.625)
  \psline[linewidth=1pt, linecolor=black]{->}(1.35, 2.375)(1.35, 1.625)
  \psline[linewidth=1pt, linecolor=black]{->}(1.35, 2.375)(2, 2.75)
  \psline[linewidth=1pt, linecolor=black]{->}(2.65, 2.375)(2, 2.75)
  \psline[linewidth=1pt, linecolor=black]{->}(2.65, 2.375)(2.65, 1.625)
  \psline[linewidth=1pt, linecolor=black]{->}(2.65, 2.375)(2, 2)
  \psline[linewidth=1pt, linecolor=black]{->}(2, 1.25)(2, 2)
  \psline[linewidth=1pt, linecolor=black]{->}(1.35, 2.375)(2, 2)
  \psline[linewidth=1pt, linecolor=black]{->}(2, 2.75)(2, 2)
  \psline[linewidth=1pt, linecolor=black]{->}(2.65, 1.625)(2, 2)
  \psline[linewidth=1pt, linecolor=black]{->}(1.35, 1.625)(2, 2)
    
  \rput(2, 1.1){$z_1z_2z_0$}
  \rput(1.1, 2.5){$z_2z_0z_1$}
  \rput(2.8, 2.5){$z_0z_1z_2$}
  \rput(3, 1.5){$z_0 \otimes z_1z_2$}
  \rput(1.1, 1.5){$z_1 \otimes z_2z_0$}
  \rput(2, 2.9){$z_0z_1 \otimes z_2$}
  
  \psdots[linecolor=gray, dotsize=4pt]
  (6, 3.5)(7.3, 2.75)(7.3, 1.25)(6, .5)
  (4.7, 1.25)(4.7, 2.75)
  
  \psline[linewidth=1pt, linecolor=gray](6, 2)(6, 3.5)
  \psline[linewidth=1pt, linecolor=gray](6, 2)(7.3, 1.25)
  \psline[linewidth=1pt, linecolor=gray](6, 2)(4.7, 1.25)
  \psline[linewidth=1pt, linecolor=gray](6, 3.5)(7.3, 2.75)
  \psline[linewidth=1pt, linecolor=gray](7.3, 2.75)(7.3, 1.25)
  \psline[linewidth=1pt, linecolor=gray](7.3, 1.25)(6, .5)
  \psline[linewidth=1pt, linecolor=gray](6, .5)(4.7, 1.25)
  \psline[linewidth=1pt, linecolor=gray](4.7, 1.25)(4.7, 2.75)
  \psline[linewidth=1pt, linecolor=gray](4.7, 2.75)(6, 3.5)

  \psdots[linecolor=black, dotsize=4pt]
  (6, 2)(6, 2.75)(6.65, 1.625)(5.35, 1.625)
  (6, 1.25)(5.35, 2.375)(6.65, 2.375)
  
  \psline[linewidth=1pt, linecolor=black]{->}(6, 1.25)(6.65, 1.625)
  \psline[linewidth=1pt, linecolor=black]{->}(6, 1.25)(5.35, 1.625)
  \psline[linewidth=1pt, linecolor=black]{->}(5.35, 2.375)(5.35, 1.625)
  \psline[linewidth=1pt, linecolor=black]{->}(5.35, 2.375)(6, 2.75)
  \psline[linewidth=1pt, linecolor=black]{->}(6.65, 2.375)(6, 2.75)
  \psline[linewidth=1pt, linecolor=black]{->}(6.65, 2.375)(6.65, 1.625)
  \psline[linewidth=1pt, linecolor=black]{->}(6.65, 2.375)(6, 2)
  \psline[linewidth=1pt, linecolor=black]{->}(6, 1.25)(6, 2)
  \psline[linewidth=1pt, linecolor=black]{->}(5.35, 2.375)(6, 2)
  \psline[linewidth=1pt, linecolor=black]{->}(6, 2.75)(6, 2)
  \psline[linewidth=1pt, linecolor=black]{->}(6.65, 1.625)(6, 2)
  \psline[linewidth=1pt, linecolor=black]{->}(5.35, 1.625)(6, 2)

  \rput(6, 1.1){$z_0z_2z_1$}
  \rput(5.1, 2.5){$z_2z_1z_0$}
  \rput(6.8, 2.5){$z_1z_0z_2$}
  \rput(7, 1.5){$z_1 \otimes z_0z_2$}
  \rput(5.1, 1.5){$z_0 \otimes z_2z_1$}
  \rput(6, 2.9){$z_1z_0 \otimes z_2$}

  \end{pspicture}

  \caption[$Sym^{(2)} \simeq N\mathcal{S}_2/N\mathcal{S}'_2$]
  {$Sym^{(2)} \simeq N\mathcal{S}_2/N\mathcal{S}'_2$.  The center of each hexagon
  is $z_0\otimes z_1 \otimes z_2$.}
  \label{fig.sym2}
\end{figure}
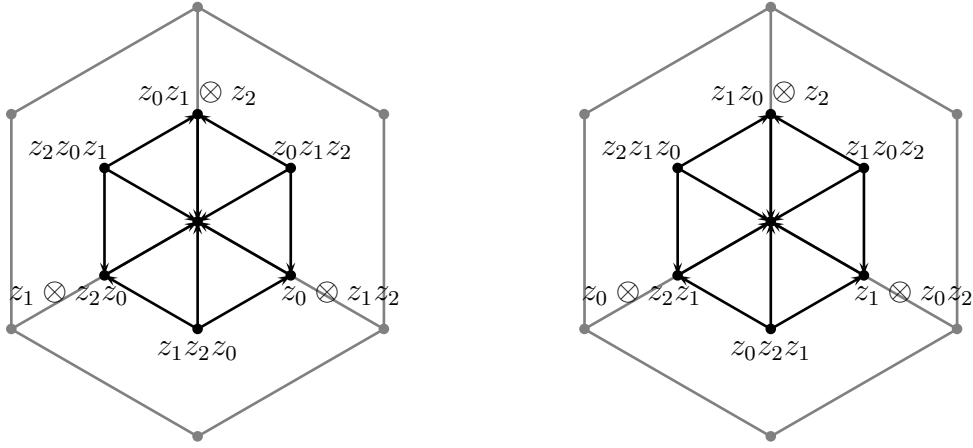

Now, with one piece of new notation, we may re-interpret Thm.~\ref{thm.E1_NS}.
\begin{definition}\label{def.ltimescirc}
  Let $G$ be a group.  Let $k_0$ be the chain complex consisting of $k$ concentrated 
  in degree $0$, with trivial $G$-action.
  If $X_*$ is a right $G$-complex, $Y_*$ is a left $G$-complex with $k_0 \hookrightarrow
  Y_*$ as a $G$-subcomplex, then define the \textit{equivariant half-smash tensor product}
  of the two complexes:
  \[
    X_* \ltimescirc_G Y_* := \left(X_* \otimes_{G} 
    Y_*\right)/\left(X_* \otimes_{G} k_0\right)
  \]
\end{definition}

\begin{cor}\label{cor.E1_Sym}
There is spectral sequence converging weakly to $\widetilde{H}S_*(A)$ with
\[
  E^1_{p,q} \cong 
  \bigoplus_{u \in X^{p+1}/\Sigma_{p+1}}
  H_{p+q}\left(E_*G_u \ltimescirc_{G_u} Sym_*^{(p)}; k\right),
\]
where $G_{u}$ is the isotropy subgroup of the orbit 
$u \in X^{p+1}/\Sigma_{p+1}$.
\end{cor}

\section{Algebra Structure of $Sym_*$}\label{sec.sym_alg}         %

We may consider $Sym_* := \bigoplus_{p \geq 0} Sym_*^{(p)}$ as a bigraded 
differential algebra, where
$bideg(W) = (p+1, i)$ for $W \in Sym_i^{(p)}$.  The product
\[
  \boxtimes : Sym_i^{(p)} \otimes Sym_j^{(q)} \to Sym_{i+j}^{(p+q+1)}
\]
is defined by:
\[
  W \boxtimes V := W \otimes V',
\]
where $V'$ is obtained from $V$ by replacing each formal indeterminate $z_r$ by $z_{r+p+1}$ for
\mbox{$0 \leq r \leq q$}.  Eq.~\ref{eq.face_i-tensor} then implies:
\begin{equation}\label{eq.partialboxtimes}
  d( W \boxtimes V ) = d(W) \boxtimes V + (-1)^{bideg(W)_2}W \boxtimes d(V),
\end{equation}
where $bideg(W)_2$ is the second component of $bideg(W)$.

\begin{prop}\label{prop.boxtimes}
  The product $\boxtimes$ is defined on the level of homology.  Furthermore, this product (on both
  the chain level and homology level) is skew commutative in a twisted sense:
  \[
    W \boxtimes V = (-1)^{ij}\tau(V \boxtimes W),
  \]  
  where $bideg(W) = (p+1, i)$, $bideg(V) = (q+1, j)$, and $\tau$ is the permutation sending
  \[
    \{0, 1, \ldots, q, q+1, q+2, \ldots, p+q, p+q+1\} \mapsto \{p+1, p+2, \ldots, p+q+1, 0,
    1, \ldots, p-1, p \}
  \]
  In fact, $\tau$ is nothing more than the block transformation $\beta_{q,p}$ defined in 
  section~\ref{sec.deltas}.
\end{prop}
\begin{proof}
  By Eq.~\ref{eq.partialboxtimes}, the product of two cycles is again a cycle, and the 
  product of a
  cycle with a boundary is a boundary, hence there is an induced product on homology.
  
  Now, suppose $W \in Sym_i^{(p)}$ and $V \in Sym_j^{(p)}$.  So,
  \[
    W = Y_0 \otimes Y_1 \otimes \ldots \otimes Y_{p-i},
  \]
  \[
    V = Z_0 \otimes Z_1 \otimes \ldots \otimes Z_{q-j}.
  \]
  \begin{equation}\label{eq.boxtimesformula}
    V \boxtimes W = V \otimes W' = (-1)^\alpha W' \otimes V,
  \end{equation}
  where $W'$ is related to $W$ by replacing each $z_r$ by $z_{r+q+1}$.  The exponent $\alpha$ is
  determined by the relations in $Sym_{i+j}^{(p+q+1)}$:
  \[
    \alpha = \big[deg(Z_0) + \ldots + deg(Z_{q-j})\big]\big[deg(Y_0) + \ldots + deg(Y_{p-i})\big]
  \]
  \[
    = deg(V)deg(W).
  \]
  Observe that $deg(W)$ is exactly $i$ and $deg(V) = j$.  Indeed, 
  $W \in \mathrm{Epi}_{\Delta S}([p], [p-i])$, so there are exactly $i$ distinct
  positions where one could insert a tensor product symbol.  That is, there are $i$
  {\it cut points} in $W$.  Since $deg(W) = \sum deg(Y_k)$, and $deg(Y_k)$ is one less
  than the number of factors in $Y_k$, it follows that $deg(Y_k)$ is exactly the
  number of cut points in $Y_k$.  Hence, $deg(W) = i$.
  
  Next, apply the block transformation $\tau$ to eq.~\ref{eq.boxtimesformula} to obtain
  \[
    \tau(V \boxtimes W) = (-1)^\alpha \tau(W' \otimes V) = (-1)^\alpha W \otimes V' = (-1)^\alpha W
    \boxtimes V,
  \]
  where $V'$ is obtained by replacing $z_r$ by $z_{r+p+1}$ in $V$.  
\end{proof}

\section{Computer Calculations}\label{sec.comp_cal}               %

In principle, the homology of $Sym_*^{(p)}$ may be found by using a computer.  In fact,
we have the following results up to $p = 7$:
\begin{theorem}\label{thm.poincare_sym_complex}
  For $0 \leq p \leq 7$, the groups $H_*(Sym_*^{(p)})$ are free abelian and have
  Poincar\'e polynomials $P_p(t) := P\left(H_*(Sym_*^{(p)}); t\right)$:
  \[
    P_0(t) = 1,
  \]
  \[
    P_1(t) = t,
  \]
  \[
    P_2(t) = t + 2t^2,
  \]
  \[
    P_3(t) = 7t^2 + 6t^3,
  \]
  \[
    P_4(t) = 43t^3 + 24t^4,
  \]
  \[
    P_5(t) = t^3 + 272t^4 + 120t^5,
  \]
  \[
    P_6(t) = 36t^4 + 1847t^5 + 720t^6,
  \]
  \[
    P_7(t) = 829t^5 + 13710t^6 + 5040t^7.
  \]
\end{theorem}
\begin{proof}
  These computations were performed using scripts written for the computer algebra
  systems \verb|GAP|~\cite{GAP4} and \verb|Octave|~\cite{E}.  See Appendix~\ref{app.comp}
  for more detail about the scripts used to calculate these polynomials.
\end{proof}

\section{Representation Theory of $H_*(Sym_*^{(p)})$}\label{sec.rep_sym}%

By remark~\ref{rmk.action}, the groups $H_i(Sym_*^{(p)}; k)$ are 
\mbox{$k\Sigma_{p+1}$-modules},
so it seems natural to investigate the irreducible representations comprising these
modules.
\begin{prop}
  Let $C_{p+1} \hookrightarrow \Sigma_{p+1}$ be the cyclic group of order $p+1$, embedded
  into the symmetric group as the subgroup generated by the permutation 
  $\tau_p := (0, p, p-1, \ldots, 1)$.
  Then there is a $\Sigma_{p+1}$-isomorphism:
  \[
    H_p(Sym_*^{(p)}) \cong AC_{p+1} \uparrow \Sigma_{p+1},
  \]
  {\it i.e.}, the alternating representation of the cyclic group, induced up to
  the symmetric group.  Note, for $p$ even, $AC_{p+1}$ coincides with the trivial
  representation $IC_{p+1}$.
  
  Moreover, $H_p(Sym_*^{(p)})$ is
  generated by the elements $\sigma(b_p)$, for the distinct cosets $\sigma C_{p+1}$,
  where $b_p$ is the element:
  \[
    b_p := \sum_{j = 0}^p (-1)^{jp} \tau_p^j(z_0z_1 \ldots z_p).
  \]
\end{prop}
\begin{proof}
  Let $w$ be a general element of $Sym_p^{(p)}$.
  \[
    w = \sum_{\sigma \in \Sigma_{p+1}} c_\sigma \sigma(z_0z_1 \ldots z_p),
  \]
  where $c_\sigma$ are constants in $k$.  $H_p(Sym_*^{(p)})$ consists of
  those $w$ such that $d(w) = 0$.  That is,
  \[
    0 = \sum_{\sigma \in \Sigma_{p+1}} c_\sigma \sigma\sum_{i=0}^{p-1}
    (-1)^i (z_0 \ldots z_i \otimes z_{i+1} \ldots z_p)
  \]
  \begin{equation}\label{eq.sum_sigma_zero}
    = \sum_{\sigma \in \Sigma_{p+1}} \sum_{i=0}^{p-1}
    (-1)^i c_\sigma \sigma(z_0 \ldots z_i \otimes z_{i+1} \ldots z_p).
  \end{equation}
  Now for fixed $\sigma$, the terms corresponding to 
  $\sigma(z_0 \ldots z_i \otimes z_{i+1} \ldots z_p)$
  occur in pairs in the above formula.  The obvious term of the pair is
  \[
    (-1)^i c_{\sigma}  \sigma(z_0 \ldots z_i \otimes z_{i+1} \ldots z_p).
  \]
  Not so obviously, the second term of the pair is
  \[
    (-1)^{(p-i-1)i}(-1)^{p-i-1} c_{\rho} \rho
    (z_0 \ldots z_{p-i-1} \otimes z_{p-i} \ldots z_p),
  \]
  where $\rho = \sigma \tau_p^{p-i}$.  Thus, if $d(w) = 0$, then
  \[
    (-1)^i c_\sigma + (-1)^{(p-i-1)(i+1)}c_\rho = 0,
  \]
  \[
    c_\rho = (-1)^{(p-i-1)(i+1) + (i+1)}c_\sigma = (-1)^{(p-i)(i+1)}c_\sigma.
  \]
  Set $j = p-i$, so that
  \[
    c_\rho = (-1)^{j(p-j+1)}c_\sigma = (-1)^{jp -j^2 +j}c_\sigma = (-1)^{jp}c_\sigma.
  \]
  This proves that the only restrictions on the coefficients $c_\sigma$ are that
  the absolute values of coefficients corresponding to $\sigma, \sigma \tau_p,
  \sigma \tau_p^2, \ldots$ must be the same, and their corresponding signs in $w$
  alternate if and only if $p$ is odd; otherwise, they have the same signs.  Clearly,
  the elements $\sigma(b_p)$ for distinct cosets $\sigma C_{p+1}$ represents an
  independent set of generators over $k$ for $H_p(Sym_*^{(p)})$.

  Observe that $b_p$ is invariant under the action of $sign(\tau_p)\tau_p$, and so
  $b_p$ generates an alternating representation $A C_{p+1}$ over $k$.  
  Induced up to
  $\Sigma_{p+1}$, we obtain the representation $AC_{p+1} \uparrow \Sigma_{p+1}$ 
  of dimension $(p+1)!/(p+1) = p!$,
  generated by the elements $\sigma(b_p)$ as in the proposition.
\end{proof}

\begin{definition}
  For a given proper partition $\lambda = [\lambda_0, \lambda_1, \lambda_2, \ldots, \lambda_s]$ 
  of the $p+1$ integers $\{0, 1, \ldots, p\}$,
  an element $W$ of $Sym_*^{(p)}$ will designated as {\it type $\lambda$} if it equivalent
  to $\pm(Y_0 \otimes Y_1 \otimes Y_2 \otimes \ldots \otimes Y_s)$ with
  $deg(Y_i) = \lambda_i - 1$.  That is, each $Y_i$ has $\lambda_i$ factors.
  
  The notation $Sym_\lambda^{(p)}$ or $Sym_\lambda$ will denote the $k$-submodule of 
  $Sym_{p-s}^{(p)}$
  generated by all elements of type $\lambda$.
\end{definition}

In what follows, $|\lambda|$ will refer to the number of components of $\lambda$.
The action of $\Sigma_{p+1}$ leaves $Sym_\lambda$ invariant for any given $\lambda$, so the
there is a decomposition
\[
  Sym_{p-s}^{(p)} = \bigoplus_{\lambda \vdash (p+1), |\lambda| = s+1} Sym_{\lambda}
\]
as $k\Sigma_{p+1}$-module.

\begin{prop}
  For a given proper partition $\lambda \vdash (p+1)$, 
  
  {\it (a)} $Sym_\lambda$ contains exactly one
  alternating representation $A\Sigma_{p+1}$ iff $\lambda$ contains no repeated components.
  
  {\it (b)} $Sym_\lambda$ contains exactly one trivial representation
  $I\Sigma_{p+1}$ iff $\lambda$ contains no repeated even components.
\end{prop}
\begin{proof}
  $Sym_\lambda$ is a quotient of the regular representation, since it is the image of the
  $\Sigma_{p+1}$-map
  \[
    \pi_\lambda \;:\; k\Sigma_{p+1} \to Sym_\lambda
  \]
  \[
    \sigma \mapsto \psi_\lambda \overline{\sigma},
  \]
  where $\overline{\sigma} \in \Sigma_{p+1}^{\mathrm{op}}$ is the $\Delta S$
  automorphism of $[p]$ corresponding to $\sigma$ and $\psi_\lambda$ is a $\Delta$
  morphism $[p] \to [ |\lambda| ]$ that sends the points $0, \ldots, \lambda_0-1$
  to $0$, the points $\lambda_0, \ldots, \lambda_0 + \lambda_1 -1$ to $1$, and
  so on.  Hence, there can be at most $1$ copy of $A\Sigma_{p+1}$ and at most
  $1$ copy of $I\Sigma_{p+1}$ in $Sym_\lambda$.
  
  Let $W$ be the ``standard'' element of $Sym_\lambda$.  That is, the indeterminates
  $z_i$ occur in $W$ in numerical order.
  $A\Sigma_{p+1}$ exists in $Sym_\lambda$ iff the element
  \[
    V = \sum_{\sigma \in \Sigma_{p+1}} sign(\sigma)\sigma(W)
  \]
  is non-zero.
  
  Suppose that some component of $\lambda$ is repeated, say $\lambda_i = \lambda_{i+1} = \ell$.
  If $W = Y_0 \otimes Y_1 \otimes \ldots \otimes Y_s$, then $deg(Y_i) = deg(Y_{i+1}) = \ell-1$.
  Now, we know that
  \[
    W = (-1)^{deg(Y_i)deg(Y_{i+1})} Y_0 \otimes \ldots \otimes Y_{i+1} \otimes Y_i \otimes
    \ldots Y_s
  \]
  \[
    = (-1)^{\ell^2 - 2\ell + 1}\alpha(W)
  \]
  \[
    = -(-1)^{\ell} \alpha(W),
  \]
  for the permutation $\alpha \in \Sigma_{p+1}$ that exchanges the indices of indeterminates in
  $Y_i$ with those in $Y_{i+1}$ in an order-preserving way.  In $V$, the term $\alpha(W)$ shows 
  up with sign $sign(\alpha) = (-1)^{\ell^2} = (-1)^\ell$, thus cancelling with $W$.  Hence,
  $V = 0$, and no alternating representation exists.
  
  If, on the other hand, no component of $\lambda$ is repeated, then no term $W$ can be
  equivalent to $\pm \alpha(W)$ for $\alpha \neq \mathrm{id}$, so $V$ survives as the generator
  of $A\Sigma_{p+1}$ in $Sym_\lambda$.
  
  A similar analysis applies for trivial representations.  This time, we examine
  \[
    U = \sum_{\sigma \in \Sigma_{p+1}} \sigma(W),
  \]
  which would be a generator for $I\Sigma_{p+1}$ if it were non-zero.
  
  As before, if there is a repeated component, $\lambda_i = \lambda_{i+1} = \ell$, then
  \[
    W = (-1)^{\ell - 1} \alpha(W).
  \]
  However, this time, $W$ cancels with $\alpha(W)$ only if $\ell - 1$ is odd.  That is,
  $|\lambda_i| = |\lambda_{i+1}|$ is even.  If $\ell - 1$ is even, or if all $\lambda_i$
  are distinct, then the element $U$ must be non-zero.
\end{proof}
  
\begin{prop}\label{prop.alternating_reps}
  $H_i(Sym_*^{(p)})$ contains an alternating representation for each partition 
  $\lambda \vdash (p+1)$ with \mbox{$|\lambda| = p-i$} such that no component of
  $\lambda$ is repeated.
\end{prop}
\begin{proof}
  This proposition will follow from the fact that $d(V) = 0$ for any generator
  $V$ of an alternating representation in $Sym_\lambda$.  Then, by Schur's Lemma,
  the alternating representation must survive at the homology level.
  
  Let $V$ be the generator mentioned above,
  \[
    V = \sum_{\sigma \in \Sigma_{p+1}} sign(\sigma)\sigma(W).
  \]
  $d(V)$ consists of terms $\partial_j(\sigma(W)) = \sigma(\partial_j(W))$ along with 
  appropriate signs.
  
  For a given, $j$, write
  \begin{equation}\label{eq.d_jW}
    \partial_j(W) = (-1)^{a + \ell} Y_0 \otimes \ldots \otimes Y_i\{0, \ldots, \ell\}
    \otimes Y_i\{\ell+1, \ldots, m\} \otimes \ldots \otimes Y_s,
  \end{equation}
  where if $Y = z_{i_0}z_{i_1} \ldots z_{i_r}$, then the notation $Y\{s, \ldots, t\}$ refers
  to the monomial $z_{i_t}z_{i_{t+1}} \ldots z_{i_s}$, assuming $0 \leq s \leq t \leq r$.
  In the above expression,
  $a = deg(Y_0) + \ldots + deg(Y_{i-1})$.
  
  Now, we may use the relations in $Sym_*$ to rewrite eq.~\ref{eq.d_jW} as
  \begin{equation}
    (-1)^{(a + \ell) + \ell(m - \ell - 1)}Y_0 \otimes \ldots \otimes Y_i\{\ell+1, \ldots, m\}
    \otimes Y_i\{0, \ldots, \ell\} \otimes \ldots \otimes Y_s.
  \end{equation}
  Let $\alpha$ be the block permutation that relabels indices thus:
  \begin{equation}\label{eq.d_jW-alpha}
    (-1)^{a + m\ell - \ell^2}\alpha\big(Y_0 \otimes \ldots \otimes 
    Y_i\{0, \ldots, m-\ell-1\}
    \otimes Y_i\{m-\ell, \ldots, m\} \otimes \ldots \otimes Y_s\big)
  \end{equation}
  
  Now, The above tensor product also occurs in $\partial_{j'}\big(sign(\alpha)\alpha(W)\big)$
  for some $j'$.
  This term looks like:
  \begin{equation}
    sign(\alpha)(-1)^{a + m-\ell-1}
    \alpha\big(Y_0 \otimes \ldots \otimes 
    Y_i\{0, \ldots, m-\ell-1\}
    \otimes Y_i\{m-\ell, \ldots, m\} \otimes \ldots \otimes Y_s\big)
  \end{equation}
  \begin{equation}
    = (-1)^{(m-\ell)(\ell + 1)+a + m-\ell-1}
    \alpha\big(Y_0 \otimes \ldots \otimes 
    Y_i\{0, \ldots, m-\ell-1\}
    \otimes Y_i\{m-\ell, \ldots, m\} \otimes \ldots \otimes Y_s\big)
  \end{equation}
  \begin{equation}\label{eq.alpha-d_jprime}
    = (-1)^{m\ell - \ell^2 + a - 1}
    \alpha\big(Y_0 \otimes \ldots \otimes 
    Y_i\{0, \ldots, m-\ell-1\}
    \otimes Y_i\{m-\ell, \ldots, m\} \otimes \ldots \otimes Y_s\big)
  \end{equation}
  
  Compare the sign of eq.~\ref{eq.alpha-d_jprime} with that of eq.~\ref{eq.d_jW-alpha}.
  I claim the signs are opposite, which would mean the two terms cancel each other
  in $d(V)$.  Indeed, all that we must show is that
  \[
    (a + m\ell - \ell^2) + (m\ell - \ell^2 + a - 1) \equiv 1\; \mathrm{mod}\; 2,
  \]
  which is obviously true.
\end{proof}

By Proposition~\ref{prop.alternating_reps}, it is clear that if $p+1$ is a triangular
number -- {\it i.e.}, $p+1$ is of the form $r(r+1)/2$ for some positive integer $r$,
then the lowest dimension in which an alternating representation may occur 
is $p + 1 - r$, corresponding
to the partition $\lambda = [r, r-1, \ldots, 2, 1]$.  A
little algebra yields the following statement for any $p$:

\begin{cor}\label{cor.lowest_alternating_reps}
  $H_i(Sym_*^{(p)})$ contains an alternating representation in degree $p+1-r$, where
  \[
    r = \lfloor \sqrt{2p + 9/4} - 1/2 \rfloor.
  \]
  Moreover, there are no alternating representations present for $i \leq p-r$.
\end{cor}
\begin{proof}
  Simply solve $p+1 = r(r+1)/2$ for $r$, and note that the increase in $r$ occurs
  exactly when $p$ hits the next triangular number.
\end{proof}

There is not much known about the other irreducible representations occurring in 
the homology groups of $Sym_*^{(p)}$, however computational evidence shows that 
$H_i(Sym_*^{(p)})$ contains no trivial representation,
$I\Sigma_{p+1}$, for \mbox{$i \leq p-r$} ($r$ as in the conjecture above) up to
$p = 50$.  

\section{Connectivity of $Sym_*^{(p)}$}\label{sec.connectivity_sym}%

Quite recently, Vre\'cica and \v{Z}ivaljevi\'c~\cite{VZ} observed that the
complex $Sym_*^{(p)}$ is isomorphic to 
the suspension of the cycle-free chessboard complex $\Omega_{p+1}$ (in fact,
the isomorphism takes the form $k\left[S\Omega_{p+1}^+\right] \to Sym_*^{(p)}$, where 
$\Omega_{p+1}^+$ is the augmented complex).

The $m$-chains of the complex $\Omega_n$ are generated by lists
\[
  L = \{ (i_0, j_0), (i_1, j_1), \ldots, (i_m, j_m) \},
\]
where $1 \leq i_0 < i_1 < \ldots < i_m \leq n$, all $1 \leq j_s \leq n$ are distinct
integers, and the list $L$ is {\it cycle-free}.  It may be easier to say what it means for
$L$ not to be cycle free: $L$ is not
cycle-free if there exists a subset $L_c \subset L$ and re-ordering of $L_c$ so that
\[
  L_c = \{ (\ell_0, \ell_1), (\ell_1, \ell_2), \ldots, (\ell_{t-1}, \ell_t),
  (\ell_t, \ell_0) \}.
\]
The differential of $\Omega_n$ is defined on generators by:
\[
  d\big( \{ (i_0, j_0), \ldots, (i_m, j_m) \} \big)
  := \sum_{s = 0}^{m} (-1)^s \{ (i_0, j_0), \ldots, (i_{s-1}, j_{s-1}),
  (i_{s+1}, j_{s+1}), \ldots, (i_m, j_m) \}.
\]
For completeness, an explicit isomorphism shall be provided:

\begin{prop}\label{prop.iso_omega_sym}
  Let $\Omega^+_n$ denote the augmented cycle-free $(n \times n)$-chessboard complex,
  where the unique $(-1)$-chain is represented by the empty $n \times n$
  chessboard, and the boundary map on $0$-chains takes a vertex to the
  unique $(-1)$-chain.
  For each $p \geq 0$, there is a chain isomorphism
  \[
    \omega_* : k\left[S\Omega^+_{p+1}\right] \to Sym_*^{(p)}
  \]
\end{prop}
\begin{proof}
  Note that we may define $m$-chains of $\Omega_{p+1}$ as cycle-free lists
  \[
    L = \{ (i_0, j_0), (i_1, j_1), \ldots, (i_m, j_m) \},
  \]
  with no requirement on the order of $\{ i_0, i_1, \ldots, i_m \}$, under the
  equivalence relation:
  \[
    \{ (i_{\sigma^{-1}(0)}, j_{\sigma^{-1}(0)}), \ldots,
       (i_{\sigma^{-1}(m)}, j_{\sigma^{-1}(m)}) \} \approx
    sign(\sigma)\{ (i_0, j_0), \ldots, (i_m, j_m) \},
  \]
  for $\sigma \in \Sigma_{m+1}$.  
  
  Suppose $L$ is an $(m+1)$-chain of $S\Omega^+_{p+1}$ ({\it i.e.} an $m$-chain
  of $\Omega^+_{p+1}$).
  Call a subset $L' \subset L$ a {\it queue} if
  there is a reordering of $L'$ such that
  \[
    L' = \{(\ell_0, \ell_1), (\ell_1, \ell_2), \ldots, (\ell_{t-1}, \ell_t)\},
  \]
  and $L'$ is called a {\it maximal queue} if it is not properly contained in
  any other queue.  
  Since $L$ is supposed to be cycle-free, we can partition
  $L$ into some number of maximal queues, $L_1', L_2', \ldots, L_q'$.  Let
  $\sigma$ be a permutation representing the re-ordering of $L$ into 
  maximal ordered queues.
  \[
    L \approx sign(\sigma)\{ (\ell^{(1)}_0, \ell^{(1)}_1), 
                               (\ell^{(1)}_1, \ell^{(1)}_2), \ldots,  
                               (\ell^{(1)}_{t_1-1}, \ell^{(1)}_{t_1}),
                             \ldots,
                             (\ell^{(q)}_0, \ell^{(q)}_1),
                               (\ell^{(q)}_1, \ell^{(q)}_2), \ldots,  
                               (\ell^{(q)}_{t_q-1}, \ell^{(q)}_{t_q}) \}
  \]
  Each maximal ordered queue will correspond to a monomial of formal
  indeterminates $z_i$.  The correspondence is as follows:
  \begin{equation}\label{eq.monomial_correspondence}
    \{ (\ell_0, \ell_1), (\ell_1, \ell_2), \ldots,  
    (\ell_{t-1}, \ell_{t}) \} \mapsto  z_{\ell_0-1}z_{\ell_1-1}\cdots
    z_{\ell_{t}-1}.
  \end{equation}
  For each maximal ordered queue, $L'_s$,
  denote the monomial obtained by formula~(\ref{eq.monomial_correspondence}) by $Z_s$.
  
  Let $k_1, k_2, \ldots, k_u$ be the numbers in $\{0, 1, 2, \ldots, p\}$ 
  such that $k_{r} + 1$ does not appear in any pair $(i_s, j_s) \in L$.
  
  Now we may define $\omega_*$ on $L = L'_1 \cup L'_2 \cup \ldots \cup L'_q$.
  \[
    \omega_{m+1}(L) := Z_1 \otimes Z_2 \otimes \ldots \otimes Z_q \otimes
    z_{k_1} \otimes z_{k_2} \otimes \ldots \otimes z_{k_u}.
  \]
  Observe, if $L = \emptyset$ is the $(-1)$-chain of $\Omega^+_{p+1}$, then there are no 
  maximal queues in $L$, and so
  \[
    \omega_0(\emptyset) = z_0 \otimes z_1 \otimes \ldots \otimes z_p.
  \]
  $\omega_*$ is a (well-defined) chain map due to the equivalence relations present
  in $Sym_*^{(p)}$ 
  (See formulas~(\ref{eq.def_Sym}),~(\ref{eq.Sym-differential}),
  and~(\ref{eq.face_i-tensor})).  To see that $\omega_*$ is an isomorphism,
  it suffices to exhibit an
  inverse.  To each monomial 
  $Z = z_{i_0}z_{i_1}\cdots z_{i_t}$ with $t > 0$, there is an associated
  ordered queue $L' = \{ (i_0 + 1, i_1 + 1), (i_1 + 1, i_2 + 1), \ldots
  (i_{t-1} +1, i_t + 1) \}$.  If the monomial is a singleton, $Z = z_{i_0}$, the
  associated ordered queue will be the empty set.
  Now, given a generator $Z_1 \otimes Z_2 \otimes
  \ldots \otimes Z_q \in Sym_*^{(p)}$, map it to the list 
  $L := L'_1 \cup L'_2 \cup \ldots \cup L'_q$, preserving the original order of indices.
\end{proof}

\begin{theorem}\label{thm.connectivity}
  $Sym_*^{(p)}$ is $\lfloor\frac{2}{3}(p-1)\rfloor$-connected.
\end{theorem}
This remarkable fact yields the following useful corollaries:

\begin{cor}\label{cor.finitely-generated}
  The spectral sequences of Thm.~\ref{thm.E1_NS} and Cor.~\ref{cor.E1_Sym} converge 
  strongly to $\widetilde{H}S_*(A)$.
\end{cor}
\begin{proof}
  This relies on the fact that the connectivity of the complexes $Sym_*^{(p)}$ is
  a non-decreasing function of $p$.
  Fix $n \geq 0$, and consider the component of $E^1$ residing at position $p, q$ for
  $p + q = n$,
  \[
    \bigoplus_{u}H_n(E_*G_u \ltimescirc_{G_u} Sym_*^{(p)}).
  \]
  A priori, the induced differentials whose sources are $E^1_{p,q}, E^2_{p,q}, 
  E^3_{p,q}, \ldots$ 
  will have as targets
  certain subquotients of 
  \[
    \bigoplus_{u}H_{n-1}(E_*G_u \ltimescirc_{G_u} Sym_*^{(p+1)}),\;
    \bigoplus_{u}H_{n-1}(E_*G_u \ltimescirc_{G_u} Sym_*^{(p+2)}),
  \]
  \[
    \bigoplus_{u}H_{n-1}(E_*G_u \ltimescirc_{G_u} Sym_*^{(p+3)}),\ldots
  \]
  Now, if $n-1 < \lfloor (2/3)(p+k-1)\rfloor$ for some $k \geq 0$, then for $K > k$,
  we have
  \[
    H_{n-1}(Sym_*^{(p + K)}) = 0,
  \]
  hence also,
  \[
    H_{n-1}(E_*G_{u} \ltimescirc_{G_u} Sym_*^{(p+K)}) = 0,
  \]
  using the fibration mentioned in the proof of Thm.~\ref{thm.E1_NS} and
  the Hurewicz Theorem.
  Thus, the induced differential $d^k$ is zero for all $k \geq K$.
  
  On the other hand, the induced differentials whose targets are $E^1_{p,q},
  E^2_{p,q}, E^3_{p,q}, \ldots$ must be zero after stage $p$, since there are
  no non-zero components with $p < 0$.
\end{proof}

\begin{cor}\label{cor.trunc-isomorphism}
  For each $i \geq 0$, there is a positive integer $N_i$ so that if 
  $p \geq N_i$, there is an isomorphism
  \[
    H_i(\mathscr{G}_p\widetilde{\mathscr{Y}}_*) \cong \widetilde{H}S_i(A).
  \]
\end{cor}

\begin{cor}\label{cor.fin-gen}
  If $A$ is finitely-generated over a Noetherian ground ring $k$, then
  $HS_*(A)$ is finitely-generated over $k$ in each degree.
\end{cor}
\begin{proof}
  Examination
  of the $E^1$ term shows that the $n^{th}$ reduced symmetric homology group of $A$ is a
  subquotient of:
  \[
    \bigoplus_{p \geq 0} \bigoplus_{u\in X^{p+1}/\Sigma_{p+1}}
    H_n(E_*G_{u}\ltimescirc_{G_u} Sym_*^{(p)} ; k)
  \]
  Each $H_n(E_*G_{u}\ltimescirc_{G_u} Sym_*^{(p)} ; k)$ is a finite-dimensional
  $k$-module.  The inner sum is finite as long as $X$ is finite.  
  Thm.~\ref{thm.connectivity} shows the
  outer sum is finite as well.
\end{proof}

The bounds on connectivity are conjectured to be tight.  This is certainly true for
$p \equiv 1$ (mod $3$), based on Thm.~16 of~\cite{VZ}.  
Corollary~12 of the same paper establishes the following result:
\[
  \mathrm{Either}\qquad H_{2k}(Sym_*^{(3k-1)}) \neq 0 \quad \mathrm{or}
  \quad H_{2k}(Sym_*^{(3k)}) \neq 0.
\]
For $k \leq 2$, both statements are true.
When the latter condition is true, this gives a tight bound on connectivity for
$p \equiv 0$ (mod $3$).  When the former is true, there is not enough
information for a tight bound, since we are more interested in proving that
$H_{2k-1}(Sym_*^{(3k-1)})$ is non-zero, since for $k = 1, 2$, we have computed
the integral homology:
\[
  H_1(Sym_*^{(2)}) = \mathbb{Z} \quad \mathrm{and} \quad H_3(Sym_*^{(5)}) = \mathbb{Z}.
\]

\section{Filtering $Sym_*^{(p)}$ by partition types}\label{sec.filterSym}%

In section~\ref{sec.rep_sym}, we saw that $Sym_*^{(n)}$ decomposes over $k\Sigma_{n+1}$
as a direct sum of the submodules $Sym_\lambda$ for partitions $\lambda \vdash (n+1)$.
Filter $Sym_*^{(n)}$ by the size of the largest component of the partition.
\[
  \mathscr{F}_pSym_q^{(n)} := \bigoplus_{\lambda \vdash (n+1), 
  |\lambda| = n+1-(p+q), \lambda_0 \leq p+1} 
  Sym_{\lambda},
\]
where $\lambda = [\lambda_0, \lambda_1, \ldots, \lambda_{n-q}]$, is written so that 
$\lambda_0 \geq
\lambda_1 \geq \ldots \geq \lambda_{n-q}$.  The differential of $Sym_*^{(n)}$ respects 
this filtering, since it can only
reduce the size of partition components.  With respect to this filtering, we have
an $E^0$ term for a spectral sequence:
\[
  E_{p,q}^0 \cong \bigoplus_{\lambda \vdash (n+1), |\lambda| = n+1-(p+q),\lambda_0 = p+1} 
  Sym_{\lambda}.
\]
The vertical differential $d^0$ is induced from $d$ by keeping only those
terms of $d(W)$ that share largest component with $W$.

\chapter{$E_{\infty}$-STRUCTURE AND HOMOLOGY OPERATIONS}\label{chap.prod}

\section{Definitions}\label{sec.homopdefs}                        %

In this chapter, homology operations are defined for $HS_*(A)$,
following May~\cite{M2}.  Let $\mathscr{Y}_*^+A$ be
the simplicial $k$-module of section~\ref{sec.deltas_plus}. 
The key is to show that $\mathscr{Y}_*^+A$ admits the structure of
$E_\infty$-algebra.
This will be accomplished using various guises
of the Barratt-Eccles operad.  While our final goal is to produce an 
action on the level of $k$-complexes, we must induce the structure from the 
level of categories and through simplicial $k$-modules.  Finally, we may 
use Definitions~2.1 and~2.2 of~\cite{M2} to define homology operations at
the level $k$-complexes.

\begin{rmk}
  Throughout this chapter, we shall fix $S_n$ to be the symmetric group
  on the letters $\{1, 2, \ldots, n\}$, given by permutations $\sigma$
  that act on the {\it left} of lists of size $n$.  {\it i.e.,} 
  \[
    \sigma.(i_1, i_2, \ldots i_n) = \left(i_{\sigma^{-1}(1)}, i_{\sigma^{-1}(2)},
    \ldots, i_{\sigma^{-1}(n)}\right),
  \]
  so that $(\sigma \tau).L = \sigma.(\tau.L)$ for all $\sigma, \tau \in S_n$,
  and an $n$-element list, $L$.
\end{rmk}

In this chapter, we make use of operad structures in various categories, 
so for completeness, formal definitions of {\it symmetric monoidal category}
as well as {\it operad}, {\it operad-algebra}, and {\it operad-module}
will be given below.  

\begin{definition}\label{def.symmoncat}
  A category $\mathscr{C}$ is symmetric monoidal if there is a bifunctor
  $\odot : \mathscr{C} \times \mathscr{C} \to \mathscr{C}$ together with:

  1. A natural isomorphism,
  \[
    a : \odot ( \mathrm{id}_\mathscr{C} \times \odot ) \to \odot( \odot \times 
    \mathrm{id}_\mathscr{C} )
  \]
  satisfying the MacLane pentagon condition (commutativity of the following
  diagram for all objects $A$, $B$, $C$, $D$).
  \[
    \begin{diagram}
      \node[2]{(A \odot B) \odot (C \odot D)}
      \arrow{se,t}{a_{A \odot B, C, D} }\\
      \node{A\odot\left(B\odot(C\odot D)\right)}
      \arrow{ne,t}{a_{A, B, C\odot D} }
      \arrow{s,l}{\mathrm{id} \odot a_{B, C, D} }
      \node[2]{\left( (A \odot B) \odot C \right) \odot D}\\
      \node{A \odot \left( (B \odot C) \odot D\right)}
      \arrow[2]{e,t}{a_{A, B \odot C, D} }
      \node[2]{\left( A \odot (B \odot C)\right)\odot D}
      \arrow{n,r}{a_{A, B, C} \odot\mathrm{id}  }
    \end{diagram}
  \]
    
  2. A \textit{unit} object $e \in \mathrm{Obj}_\mathscr{C}$, together with
  natural isomorphisms
  \[
    \ell : e \odot \mathrm{id}_{\mathscr{C}} \to \mathrm{id}_\mathscr{C}
  \]
  \[
    r : \mathrm{id}_{\mathscr{C}} \odot e \to \mathrm{id}_\mathscr{C}
  \]
  making the following diagram commute for all objects $A$ and $B$:
  \[
    \begin{diagram}
      \node{A \odot (e \odot B)}
      \arrow[2]{e,t}{a_{A,e,B}}
      \arrow{se,b}{\mathrm{id} \odot \ell}
      \node[2]{(A \odot e) \odot B}
      \arrow{sw,b}{r \odot \mathrm{id}}\\
      \node[2]{A \odot B}
    \end{diagram}
  \]

  3. A natural transformation \mbox{$s : \odot \to \odot T$}, where 
  \mbox{$T : \mathscr{C} \times \mathscr{C} \to \mathscr{C}$} is the transposition functor 
  \mbox{$(A,B) \mapsto (B,A)$}, such that $s^2 = \mathrm{id}$ and  
  $s$ satisfies the {\it hexagon identity} (commutativity of the following
  diagram for all objects $A$, $B$, $C$).
  \[
    \begin{diagram}
      \node[2]{A \odot (B \odot C)}
      \arrow{sw,t}{\mathrm{id} \odot s_{B,C} }
      \arrow{se,t}{a_{A,B,C}}\\
      \node{A \odot (C \odot B)}
      \arrow{s,l}{a_{A, C, B}}
      \node[2]{(A \odot B) \odot C}
      \arrow{s,r}{s_{A\odot B, C} }\\
      \node{(A \odot C) \odot B}
      \arrow{se,b}{s_{A,C} \odot \mathrm{id}}
      \node[2]{C \odot (A \odot B)}
      \arrow{sw,b}{a_{C, A, B} }\\
      \node[2]{(C \odot A) \odot B}
    \end{diagram}
  \]
  
  Observe that if $a_{A,B,C}$, $\ell_A$ and $r_A$ are identity morphisms, then
  this definition reduces to that of a permutative category.  
\end{definition}

Let $\mathbf{S}$ denote the
{\it symmetric groupoid} as category.  For our purposes, we may label the objects
by $\underline{0}, \underline{1}, \underline{2}, \ldots$, and the only morphisms
of $\mathbf{S}$ are automorphisms, $\mathrm{Aut}_{\mathbf{S}}(\underline{n}) :=
S_n$, the symmetric group on $n$ letters.  

\begin{definition}\label{def.operad}
  Suppose $\mathscr{C}$ is a symmetric monoidal category, with unit $e$.
  An {\it operad} $\mathscr{P}$ in the category $\mathscr{C}$ is a functor
  \[
    \mathscr{P} : \mathbf{S}^\mathrm{op} \to \mathscr{C},
  \]
  with $\mathscr{P}(0) = e$, together with the following data:
  
  1. Morphisms $\gamma_{k, j_1, \ldots, j_k} : \mathscr{P}(k) \odot \mathscr{P}(j_1)
  \odot \ldots \odot \mathscr{P}(j_k) \to \mathscr{P}(j)$, where
  $j = \sum j_s$.  For brevity, we denote these morphisms simply by $\gamma$.  
  The morphisms $\gamma$ should satisfy the following associativity condition.
  The diagram below is commutative for
  all $k \geq 0$, $j_s \geq 0$, $i_r \geq 0$.  
  Here, $T$ is a map that permutes the components of the product in the specified way, 
  using the symmetric transformation $s$ of $\mathscr{C}$.  Coherence of $s$
  guarantees that this is a well-defined map.

  {\bf Associativity:}
  \[
    \begin{diagram}
      \node{\mathscr{P}(k) \odot \bigodot_{s = 1}^{k} \mathscr{P}(j_s)
        \odot \bigodot_{r = 1}^{j} \mathscr{P}(i_r)}
      \arrow{e,t}{T}
      \arrow{s,l}{\gamma \odot \mathrm{id}^{\odot j}}
      \node{\mathscr{P}(k) \odot \bigodot_{s=1}^{k}
        \left( \mathscr{P}(j_s) \odot \bigodot_{r = j_1 + \ldots + j_{s-1} + 1}^{j_1 + \ldots
        + j_s} \mathscr{P}(i_r) \right)}
      \arrow{s,r}{\mathrm{id} \odot \gamma^{\odot k}}\\
      \node{\mathscr{P}(j) \odot \bigodot_{r = 1}^{j} \mathscr{P}(i_r) }
      \arrow{s,l}{\gamma}
      \node{\mathscr{P}(k) \odot \bigodot_{s=1}^{k} \mathscr{P}\left( \sum_{r=j_1 + \ldots
        + j_{s-1} + 1}^{j_1 + \ldots + j_s} i_r \right) }
      \arrow{s,r}{\gamma}\\
      \node{\mathscr{P}\left(\sum_{r=1}^{j} i_r \right) }
      \arrow{e,=}
      \node{\mathscr{P}\left(\sum_{r=1}^{j_1 + \ldots + j_k} i_r \right) }
    \end{diagram}
  \]
  
  2. A {\it Unit} morphism \mbox{$\eta : e \to \mathscr{P}(1)$} making the following 
  diagrams commute:
  
  {\bf Left Unit Condition:}
  \[
    \begin{diagram}
      \node{e \odot \mathscr{P}(j)}
      \arrow{e,t}{\eta \odot \mathrm{id}}
      \arrow{se,b}{\ell}
      \node{\mathscr{P}(1) \odot \mathscr{P}(j)}
      \arrow{s,r}{\gamma}\\
      \node[2]{\mathscr{P}(j)}      
    \end{diagram}
  \]

  {\bf Right Unit Condition:}
  \[
    \begin{diagram}
      \node{\mathscr{P}(j) \odot e^{\odot j}}
      \arrow{e,t}{\mathrm{id} \odot \eta^{\odot j}}
      \arrow{se,b}{r^j}
      \node{\mathscr{P}(j) \odot \mathscr{P}(1)^{\odot j}}
      \arrow{s,r}{\gamma}
      \\
      \node[2]{\mathscr{P}(j)}
    \end{diagram}
  \]
  
  Here, $r^j$ is the {\it iterated} right unit map defined recursively (for an 
  object $A$ of $\mathscr{C}$):
  \[
    r_A^j := \left\{\begin{array}{ll}
                      r_A, & j=1\\
                      r_A^{j-1}\left(r_A \odot \mathrm{id}^{\odot(j-1)}\right),
                      & j > 1
                    \end{array}\right.
  \]
      
  3.  The right action of $S_n$ on $\mathscr{P}(n)$ for each $n$
  must satisfy the following {\it equivariance
  conditions}.  Both diagrams below are commutative for
  all \mbox{$k \geq 0$}, \mbox{$j_s \geq 0$}, \mbox{$(j = \sum j_s)$}, 
  \mbox{$\sigma \in S_k^\mathrm{op}$},
  and \mbox{$\tau_s \in S_{j_s}^\mathrm{op}$}. 
  Here, $T_{\sigma}$ is
  a morphism that permutes the components of the product in the specified way,
  using the symmetric transformation $s$.
  $\sigma\{j_1, \ldots, j_k\}$  denotes the permutation of $j$ letters which permutes
  the $k$ blocks of letters (of sizes $j_1$, $j_2$, \ldots $j_k$) according to
  $\sigma$, and \mbox{$\tau_1 \oplus \ldots \oplus \tau_k$} denotes the image of
  \mbox{$(\tau_1, \ldots, \tau_k)$} under the evident inclusion
  \mbox{$S_{j_1}^\mathrm{op} \times \ldots \times S_{j_k}^\mathrm{op} \hookrightarrow
  S_j^\mathrm{op}$}.
  
  {\bf Equivariance Condition A:}
  \[
    \begin{diagram}
      \node{\mathscr{P}(k) \odot \bigodot_{s=1}^k \mathscr{P}(j_s)}
      \arrow{e,t}{\mathrm{id} \odot T_{\sigma}}
      \arrow{s,l}{\sigma \odot \mathrm{id}^{\odot k}}
      \node{\mathscr{P}(k) \odot \bigodot_{s=1}^k \mathscr{P}\left(j_{\sigma^{-1}(s)}\right)}
      \arrow{s,r}{\gamma}\\
      \node{\mathscr{P}(k) \odot \bigodot_{s=1}^k \mathscr{P}(j_s)}
      \arrow{se,l}{\gamma}
      \node{\mathscr{P}(j)}
      \arrow{s,b}{\sigma\{j_1, \ldots, j_k\}}\\
      \node[2]{\mathscr{P}(j)}
    \end{diagram}
  \]

  {\bf Equivariance Condition B:}
  \[
    \begin{diagram}
      \node{\mathscr{P}(k) \odot \bigodot_{s=1}^k \mathscr{P}(j_s)}
      \arrow{e,t}{\gamma}
      \arrow{s,l}{\mathrm{id} \odot \tau_1 \odot \ldots \odot \tau_k}
      \node{\mathscr{P}(j)}
      \arrow{s,r}{\tau_1 \oplus \ldots \oplus \tau_k}\\
      \node{\mathscr{P}(k) \odot \bigodot_{s=1}^k \mathscr{P}(j_s)}
      \arrow{e,t}{\gamma}
      \node{\mathscr{P}(j)}
    \end{diagram}
  \]  
\end{definition}

\begin{definition}\label{def.operad-algebra}
  For a symmetric monoidal category $\mathscr{C}$ with product $\odot$
  and an operad $\mathscr{P}$ over $\mathscr{C}$, a 
  {\it $\mathscr{P}$-algebra}
  structure on an object $X$ in $\mathscr{C}$ is defined by a 
  family of maps
  \[
    \chi : \mathscr{P}(n) \odot_{S_n} X^{\odot n} \to X,
  \]
  which are compatible with the multiplication, unit maps, and equivariance
  conditions of $\mathscr{P}$.  Note, the symbol $\odot_{S_n}$ denotes
  an internal equivariance condition:
  \[
    \chi(\pi.\sigma \odot x_1 \odot \ldots \odot x_n)
    = \chi(\pi \odot x_{\sigma^{-1}(1)} \odot \ldots \odot x_{\sigma^{-1}(n)})
  \]
  If $X$ is a $\mathscr{P}$-algebra, we will say that $\mathscr{P}$ acts on $X$.
\end{definition}  

\begin{definition}\label{def.operad-module}
  Let $\mathscr{P}$ be an operad over the symmetric monoidal category $\mathscr{C}$, 
  and let $\mathscr{M}$ be a functor \mbox{$\mathbf{S}^\mathrm{op} \to \mathscr{C}$}.
  A {\it (left) \mbox{$\mathscr{P}$-module}} structure on
  $\mathscr{M}$ is a collection of structure maps, 
  \[
    \mu : \mathscr{P}(n) \odot \mathscr{M}(j_1) \odot \ldots \odot \mathscr{M}(j_n)
    \longrightarrow \mathscr{M}(j_1 + \ldots + j_n),
  \]
  satisfying the evident compatibility relations with the operad multiplication of
  $\mathscr{P}$.  For the precise definition, see~\cite{KM}.
\end{definition}

In the course of this chapter, it shall become necessary to induce structures
up from small categories to simplicial sets, then to simplicial $k$-modules, and
finally to $k$-complexes.  Each of these categories is symmetric monoidal.
For notational convenience, all operads, operad-algebras,
and operad-modules will carry a subscript denoting the ambient category over which 
the structure is defined:
\begin{center}
\begin{tabular}{|ll|l|l|}
\hline
Category &  & Sym. Mon. Product & Notation\\
\hline
Small categories & $\mathbf{Cat}$ & $\times$ (product of categories) &
  $\mathscr{P}_\mathrm{cat}$ \\
Simplicial sets & $\mathbf{SimpSet}$ & $\times$ 
  (degree-wise set product) & $\mathscr{P}_\mathrm{ss}$ \\
Simplicial $k$-modules & $k$-$\mathbf{SimpMod}$ & $\widehat{\otimes}$ (degree-wise
  tensor product) & $\mathscr{P}_\mathrm{sm}$  \\
$k$-complexes & $k$-$\mathbf{Complexes}$ & $\otimes$ (tensor prod. of chain complexes)
  & $\mathscr{P}_\mathrm{ch}$ \\
\hline
\end{tabular}
\end{center}

\begin{rmk}
  The notation $\widehat{\otimes}$, appearing in Richter~\cite{R}, is useful for 
  indicating degree-wise tensoring of graded modules:
  \[
    \left(A_* \,\widehat{\otimes}\, B_*\right)_n :=
    A_n \otimes_k B_n,
  \]
  as opposed to the standard tensor product (over $k$) of complexes:
  \[
    \left(\mathscr{A}_* \otimes \mathscr{B}_*\right)_n := \bigoplus_{p+q=n} 
    \mathscr{A}_p \otimes_k \mathscr{B}_q
  \]
\end{rmk}

Furthermore, we are interested in certain functors from one category to the next in
the list.  These functors preserve the symmetric monoidal structure in a sense we
will make precise in Section~\ref{sec.op-mod-structure} -- hence, it will follow
that these functors send operads to operads, operad-modules to operad-modules,
and operad-algebras to operad-algebras.
\[
  \begin{diagram}
    \node{ (\mathbf{Cat}, \times) }
    \arrow{s,r}{N \quad \textrm{(Nerve of categories)}}
    \\
    \node{ (\mathbf{SimpSet}, \times) }
    \arrow{s,r}{k[ - ] \quad \textrm{($k$-linearization)}}
    \\
    \node{ (\textrm{$k$-$\mathbf{SimpMod}$}, \widehat{\otimes} ) }
    \arrow{s,r}{\mathcal{N} \quad \textrm{(Normalization functor)}}
    \\
    \node{ (\textrm{$k$-$\mathbf{Complexes}$}, \otimes )}
  \end{diagram}
\]

\begin{rmk}
  Note, the normalization functor $\mathcal{N}$ is one direction of the
  Dold-Kan correspondence between simplicial modules and complexes.  
\end{rmk}

\begin{rmk}
  The ultimate goal of this chapter is to construct an $E_\infty$ structure on 
  the chain complex associated with $\mathscr{Y}_*^+A$, {\it i.e.} an
  action by an $E_\infty$-operad.
  While we could define the notion of $E_{\infty}$-operad over general
  categories, it would require extra structure on the ambient symmetric
  monoidal category $(\mathscr{C}, \odot)$ -- which the examples above possess.
  To avoid needless technicalities, we shall instead define versions of the
  Barratt-Eccles operad over each of our ambient categories, and take for
  granted that they are all $E_{\infty}$-operads.
  For a more general discussion of $E_\infty$-operads and algebras in
  the category of chain complexes, see~\cite{MSS}.
\end{rmk}
 
\begin{definition}\label{def.operadD}
  $\mathscr{D}_{\mathrm{cat}}$ is the operad $\{\mathscr{D}_{\mathrm{cat}}(m)\}$ in 
  the category \textbf{Cat},
  where $\mathscr{D}_{\mathrm{cat}}(0) = \ast$, 
  $\mathscr{D}_{\mathrm{cat}}(m)$ is the category whose objects are the elements of 
  $S_m$, and for each pair of
  objects, $\sigma, \tau$, we have $\mathrm{Mor}(\sigma, \tau) = \{ \tau\sigma^{-1} \}$.
  The structure map (multiplication) $\delta$ in $\mathscr{D}_{\mathrm{cat}}$ is a 
  functor defined on objects by:
  \[
    \delta \;:\; \mathscr{D}_{\mathrm{cat}}(m) \times \mathscr{D}_{\mathrm{cat}}(k_1) 
    \times \ldots \times \mathscr{D}_{\mathrm{cat}}(k_m) \longrightarrow 
    \mathscr{D}_{\mathrm{cat}}(k), \;\;\textit{where $k = \sum k_i$}
  \]
  \[
    (\sigma, \tau_1, \ldots, \tau_m) \mapsto \sigma\{k_1, \ldots, k_m\} (\tau_1
    \oplus \ldots \oplus \tau_m)
  \]  
  Here, $\tau_1 \oplus \ldots \oplus \tau_m \in S_k$ and
  $\sigma\{k_1, \ldots, k_m\} \in S_k$ are defined as in Definition~\ref{def.operad}.
  The functor takes the unique morphism 
  \[
    (\sigma, \tau_1, \ldots, \tau_m) \to (\rho, \psi_1, \ldots, \psi_m)
  \]
  to the unique morphism
  \[
    \sigma\{k_1, \ldots, k_m\}(\tau_1 \oplus \ldots \oplus \tau_m) \to 
    \rho\{k_1, \ldots, k_m\}(\psi_1 \oplus\ldots \oplus \psi_m).
  \]
  The action of $S_m^{\mathrm{op}}$ on objects of 
  $\mathscr{D}_{\mathrm{cat}}(m)$ is given by right
  multiplication.
\end{definition}

\begin{rmk}\label{rmk.Barratt-Eccles}
  We are following the notation of May for our operad
  $\mathscr{D}_{\mathrm{cat}}$ (See~\cite{M3}, Lemmas~4.3, 4.8).  May's
  notation for $\mathscr{D}_{\mathrm{cat}}(m)$ is $\widetilde{\Sigma}_m$,
  and he defines the related operad $\mathscr{D}$ over the category of spaces, as
  the geometric realization of the nerve of $\widetilde{\Sigma}$.
  The nerve of $\mathscr{D}_{\mathrm{cat}}$ is generally known 
  in the literature as the {\it Barratt-Eccles operad} (See~\cite{BE}, where
  the notation for $N\mathscr{D}_{\mathrm{cat}}$ is $\Gamma$).
\end{rmk}

\section{Operad-Module Structure}\label{sec.op-mod-structure}     %

In order to best define the $E_\infty$ structure of $\mathscr{Y}_*^+A$,
we will begin with an {\it operad-module} structure over the category 
of small categories,
then induce this structure up to the category of \mbox{$k$-complexes}.

\begin{definition}
  Define for each $m \geq 0$, a category,
  \[
    \mathscr{K}_{\mathrm{cat}}(m) := [m-1] \setminus \Delta S_+ = \underline{m} \setminus
    \mathcal{F}(as),
  \]
  (See Section~\ref{sec.deltas} for a definition of $\mathcal{F}(as)$.)
\end{definition}

Identifying $\Delta S_+$ with $\mathcal{F}(as)$, we see
that the morphism $(\phi, g)$ of $\Delta S_+$ consists of the
set map $\phi$, precomposed with $g^{-1}$ in order to indicate the total
ordering on all preimage sets.  Thus, precomposition with symmetric
group elements defines a {\it right} $S_m$-action on objects
of $\underline{m} \setminus \mathcal{F}(as)$.  When writing morphisms
of $\mathcal{F}(as)$, we may avoid confusion by writing the automorphisms
as elements of the symmetric group as opposed to its opposite group, with
the understanding that $g \in \mathrm{Mor}\mathcal{F}(as)$ corresponds
to $g^{-1} \in \mathrm{Mor}\Delta S_+$.
\[
  \mathscr{K}_{\mathrm{cat}}(m) \times S_m \to \mathscr{K}_{\mathrm{cat}}(m)
\]
\[
  (\phi, g).h := (\phi, gh)
\]
Let $m, j_1, j_2, \ldots, j_m \geq 0$, and let $j = \sum j_s$.  
We shall define a family of functors,
\begin{equation}\label{eq.mu}
  \mu = \mu_{m,j_1,\ldots,j_m} : \mathscr{D}_{\mathrm{cat}}(m) \times \prod_{s=1}^m 
  \mathscr{K}_{\mathrm{cat}}(j_s)  \longrightarrow   \mathscr{K}_{\mathrm{cat}}(j).
\end{equation}
Assume that the pairs of morphisms $f_i$, $g_i$ ($1 \leq i \leq m$) have specified
sources and targets:
\[
  \underline{j_i} \stackrel{f_i}{\to} \underline{p_i} \stackrel{g_i}{\to} 
  \underline{q_i}.
\]
$\mu$ is defined on objects by:
\begin{equation}\label{eq.mu-objects}
  \mu( \sigma, f_1, f_2, \ldots, f_m ) := 
   {\sigma\{p_1, p_2, \ldots, p_m\}}
  (f_1 \odot f_2 \odot \ldots \odot f_m).
\end{equation}
For compactness of notation, denote 
\[
  \underline{m} := \{1, \ldots, m\},
\]
as ordered list, along with a left $S_m$ action, 
\[
  \tau \underline{m} := \{\tau^{-1}(1), \ldots, \tau^{-1}(m)\}.
\]
Then for any permutation, $\tau \underline{m}$, of the ordered list 
$\underline{m}$, and any list of $m$ numbers, $\{j_1, \ldots, j_m\}$,
denote
\[
  j_{\tau \underline{m}} := \{j_{\tau^{-1}(1)}, \ldots, j_{\tau^{-1}(m)}\}.
\]
Furthermore, if $f_1, f_2, \ldots, f_m \in \mathrm{Mor}\mathcal{F}(as)$, denote
\[
  f_{\tau \underline{m}}^\odot := f_{\tau^{-1}(1)} \odot \ldots \odot f_{\tau^{-1}(m)}
\]
so in particular, we may write:
\[
  \mu( \sigma, f_1, f_2, \ldots, f_m ) = 
  \sigma\{p_{\underline{m}}\}f^{\odot}_{\underline{m}}
\]  
Using this notation, define $\mu$ on morphisms by:
\begin{equation}\label{eq.mu-morphisms}
  \begin{diagram}
    \node{ (\sigma, f_1, \ldots, f_m) }
     \arrow{s,r}{ \tau\sigma^{-1} \times g_1 \times \ldots \times g_m }
    \arrow{e,t}{ \mu }
    \node{ \sigma\{p_{\underline{m}}\}f_{\underline{m}}^\odot }
    \arrow{s,r}{ (\tau\sigma^{-1})\{q_{\sigma \underline{m}} \}
      g_{\sigma \underline{m}}^\odot } 
    \\
    \node{ (\tau, g_1f_1, \ldots, g_mf_m) }
    \arrow{e,t}{ \mu }
    \node{ \tau\{q_{\underline{m}}\}
      (g_{\underline{m}}f_{\underline{m}})^\odot }
  \end{diagram}
\end{equation}
It is useful to note three properties of the block permutations and symmetric
monoid product in the category $\coprod_{m \geq 0} \mathscr{K}_{\mathrm{cat}}(m)$.
\begin{prop}\label{prop.properties-blockperm-odot}
  1. For $\sigma, \tau \in S_m$, and non-negative $p_1, p_2, \ldots, p_m$,
  \[
    (\sigma\tau)\{ p_{\underline{m}} \} = \sigma\{p_{\tau \underline{m}}\}
    \tau\{p_{\underline{m}}\}.
  \]
  
  2. For $\sigma \in S_m$, and morphisms $g_i \in \mathrm{Mor}_{\mathcal{F}(as)}
  (\underline{p_i}, \underline{q_i})$, $(1 \leq i \leq m)$,
  \[
    \sigma\{q_{\underline{m}}\}g_{\underline{m}}^\odot = 
    g_{\sigma \underline{m}}^\odot \sigma\{p_{\underline{m}}\},
  \]

  3. For morphisms $\underline{j_i} \stackrel{f_i}{\to} \underline{p_i}
  \stackrel{g_i}{\to} \underline{q_i}$,
  \[
    (g_{\underline{m}}f_{\underline{m}})^\odot = g_{\underline{m}}^\odot 
    f_{\underline{m}}^\odot
  \]
\end{prop}
\begin{proof}
  Property 1 is a standard composition property of block permutations.  
  See~\cite{MSS}, p. 41, for example.  Property 2 expresses the fact that
  it does not matter whether we apply the morphisms $g_i$ to blocks first, then
  permute those blocks, or permute the blocks first, then apply $g_i$ to
  the corresponding permuted block.  Finally, property 3 is a result of
  functoriality of the product $\odot$.
\end{proof}
Using these properties, it is straightforward to verify that $\mu$ is a
functor.  We shall show that $\mu$ defines a 
$\mathscr{D}_{\mathrm{cat}}$-module structure on
$\mathscr{K}_{\mathrm{cat}}$.

First observe that if $(\mathscr{B}, \odot)$ is a permutative category,
then $\mathscr{B}$ admits the structure of $E_\infty$-algebra.
We may express this structure using the $E_\infty$-operad, 
$\mathscr{D}_{\mathrm{cat}}$.  The structure map is
a family of functors 
\[
  \theta : \mathscr{D}_{\mathrm{cat}}(m) \times \mathscr{B}^m \to \mathscr{B},
\]
given on objects $C_i \in \mathrm{Obj}\mathscr{B}$ by:
\[
  \theta(\sigma, C_1, \ldots, C_m) := C_{\sigma^{-1}(1)} \odot \ldots \odot 
  C_{\sigma^{-1}(m)},
\]
and on morphisms $(\sigma, C_1, \ldots, C_m) \to (\tau, D_1, \dots, D_m)$ by:
\[
  \begin{diagram}
  \node{ (\sigma, C_1, \ldots, C_m) }
  \arrow{e,t}{ \theta }
  \arrow[2]{s,l}{ \tau \sigma^{-1} \times f_1 \times \ldots \times f_m }
  \node{ C_{\sigma^{-1}(1)} \odot \ldots \odot C_{\sigma^{-1}(m)} }
  \arrow{s,r}{ f_{\sigma^{-1}(1)} \odot \ldots \odot f_{\sigma^{-1}(m)} }
  \\
  \node[2]{ D_{\sigma^{-1}(1)} \odot \ldots \odot D_{\sigma^{-1}(m)} }
  \arrow{s,lr}{ \cong }{ T_{\tau\sigma^{-1}} }
  \\
  \node{ (\tau, D_1, \dots, D_m) }
  \arrow{e,t}{ \theta }  
  \node{ D_{\tau^{-1}(1)} \odot \ldots \odot D_{\tau^{-1}(m)} }
  \end{diagram}
\]

We just need to verify that $\theta$ satisfies the expected equivariance 
condition.  Recall, the required condition is
\[
  \theta(\sigma \tau, C_1, \ldots, C_m) = \theta\left(\sigma, C_{\tau^{-1}(1)},
  \ldots C_{\tau_{-1}(m)}\right)
\]
This is easily verified on objects, as the left-hand side evaluates as:
\[
  \theta(\sigma \tau, C_1, \ldots, C_m) =
  C_{(\sigma\tau)^{-1}(1)} \odot \ldots \odot C_{(\sigma\tau)^{-1}(m)},
\]
while the right-hand side evaluates as:
\[
  \theta\left(\sigma, C_{\tau^{-1}(1)},
  \ldots C_{\tau_{-1}(m)}\right) = C_{\tau^{-1}\left(\sigma^{-1}(1)\right)} \odot
  \ldots C_{\tau^{-1}\left(\sigma^{-1}(m)\right)}
\]
\[
  = C_{(\sigma\tau)^{-1}(1)} \odot \ldots \odot C_{(\sigma\tau)^{-1}(m)}.
\]

Now, let $M$ be a monoid with unit, $1$.  Let \mbox{$\mathscr{X}^+_*M := 
N(- \setminus \Delta S_+) \times_{\Delta S_+} B^{sym_+}_*M$}.  This is analogous to
the construction $\mathscr{X}_*$ of section~\ref{sec.symhommonoid}.  
\begin{prop}\label{prop.X-perm-cat}
 $\mathscr{X}^+_*M$ is the nerve of a permutative category.
\end{prop}
\begin{proof}
  Consider a category $\mathcal{T}M$ whose objects are the elements of 
  $\coprod_{n \geq 0} M^n$,
  where $M^0$ is understood to be the set consisting of the empty tuple, $\{()\}$.
  Morphisms of $\mathcal{T}M$ consist of the morphisms of $\Delta S_+$, 
  reinterpreted as
  follows:  A morphism $f : [p] \to [q]$ in $\Delta S$ will be considered a morphism
  $(m_0, m_1, \ldots , m_p) \to f(m_0, m_1, \ldots, m_p) \in M^{q+1}$.  The unique morphism
  $\iota_n$ will be considered a morphism $() \to (1, 1, \ldots, 1) \in M^{n+1}$.
  The nerve of $\mathcal{T}M$ consists of chains,
  \[
    (m_0, \ldots, m_n)
    \stackrel{ f_1 }{\to}
    f_1(m_0, \ldots, m_n)
    \stackrel{ f_2 }{\to}
    \ldots
    \stackrel{ f_i }{\to}
    f_i \ldots f_1(m_0, \ldots, m_n)
  \]
    
  This chain can be rewritten uniquely as an element of $M^n$ together with a chain 
  in $N\Delta S$.
  \[
    \left( [n] \stackrel{f_1}{\to} [n_1] \stackrel{f_2}{\to} \ldots
    \stackrel{f_i}{\to} [n_i]\,,\,(m_0, \ldots, m_n) \right),
  \]
  which in turn is uniquely identified with an element of $\mathscr{X}^+_iM$:
  \[
    \left( [n] \stackrel{\mathrm{id}}{\to} [n]\stackrel{f_1}{\to} 
    [n_1] \stackrel{f_2}{\to} \ldots
    \stackrel{f_i}{\to} [n_i]\,,\,(m_0, \ldots, m_n) \right)
  \]
  Clearly, since any element of $\mathscr{X}^+_*M$ may be written so that the
  first morphism of the chain component is the identity,  
  $N(\mathcal{T}M)$ can be identified with $\mathscr{X}^+_*M$.
  
  Now, we show that $\mathcal{T}M$ is permutative.  Define
  the product on objects:
  \[
    (m_0, \ldots, m_p) \odot (n_0, \ldots, n_q) := (m_0, \ldots, m_p, n_0, \ldots, n_q),
  \]
  and for morphisms $f, g \in \mathrm{Mor}\mathcal{T}M$,
  simply use $f \odot g$
  as defined for $\Delta S_+$ in section~\ref{sec.deltas}.  Associativity is strict, since
  it is induced by the associativity of $\odot$ in $\Delta S_+$.
  There is also a strict unit, the empty tuple, $()$.  
  
  The natural transposition ({\it i.e.}, $\gamma : \odot \to  \odot T$ of the
  definition given in Prop~\ref{prop.deltaSpermutative}) is defined on objects by:
  \[
    \gamma : (m_0, \ldots, m_p) \odot (n_0, \ldots, n_q) \to (n_0, \ldots, n_q)
     \odot (m_0, \ldots, m_p)
  \]
  \[
    \gamma = \beta_{p,q} \qquad \textrm{(the block transformation morphism of 
      section~\ref{sec.deltas}).}
  \]
  Suppose we have a morphism in the product category,
  \[
    (f,g) : \left((m_0, \ldots, m_p), (n_0, \ldots, n_q)\right)
    \to \left( (m'_0, \ldots, m'_s), (n'_0, \ldots, n'_t)\right)
  \]
  Then it is easy to verify that the map $\gamma$ defined above makes the following
  diagram commutative (showing $\gamma$ is a natural transformation).
  \[
    \begin{diagram}
      \node{ (m_0, \ldots, m_p, n_0, \ldots, n_q) }
      \arrow{e,t}{\gamma = \beta_{p,q}}
      \arrow{s,r}{f \odot g}
      \node{ (n_0, \ldots, n_q, m_0, \ldots, m_p) }
      \arrow{s,r}{g \odot f}
      \\
      \node{ (m'_0, \ldots, m'_s, n'_0, \ldots, n'_t) }
      \arrow{e,t}{\gamma  = \beta_{s,t}}
      \node{ (n'_0, \ldots, n'_t, m'_0, \ldots, m'_s) }
    \end{diagram}  
  \]
  The coherence conditions are satisfied for $\gamma$ in the same way as in 
  $\Delta S_+$.  
\end{proof}

In particular, given any monoid $M$, the category $\mathcal{T}M$
constructed in the proof of Prop.~\ref{prop.X-perm-cat} has the structure
of $E_\infty$-algebra.

\begin{lemma}\label{lem.D-module_str_on_K}
  $\mathscr{K}_{\mathrm{cat}}$ has the structure of a 
  \mbox{$\mathscr{D}_{\mathrm{cat}}$-module}.
\end{lemma}
\begin{proof}
  Let $X = \{x_i\}_{i \geq 1}$ be a countable set of formal independent 
  indeterminates, and $J(X_+)$ the free monoid on the set
  $X_+ := X \cup \{1\}$.  Since the category
  $\mathcal{T}J(X_+)$ is permutative, there is an
  $E_\infty$-algebra structure
  \[
    \theta : \mathscr{D}_{\mathrm{cat}}(m) \times 
    \left[\mathcal{T}J(X_+)\right]^m
    \to \mathcal{T}J(X_+)
  \]
  We can identify:
  \[
    \coprod_{m \geq 0} \left(\mathscr{K}_{\mathrm{cat}}(m) \times_{S_m} X^m \right)
    = \mathcal{T}J(X_+),
  \]
  via the map
  \[
    (f, x_{i_1}, x_{i_2}, \ldots, x_{i_m}) \mapsto f(x_{i_1}, x_{i_2}, \ldots,
    x_{i_m}).
  \]
  The fact that $J(X_+)$ is the free monoid on $X_+$ ensures that there is
  a map in the inverse direction, well-defined up to action of the symmetric 
  group $S_m$.  Note, the action of $S_m$ on $X^m$ is by permutation of the 
  components $\{x_{i_1}, x_{i_2}, \ldots, x_{i_m} \}$.

  Next, let $m, j_1, j_2, \ldots, j_m \geq 0$, and let $j = \sum j_s$.
  Furthermore, let
  \[
    X_s = (x_{j_1 + j_2 + \ldots + j_{s-1} + 1} , \ldots, x_{j_1 + j_2 + \ldots + j_s}).
  \]
  Define a functor $\alpha = \alpha_{m,j_1,\ldots,j_m}$:
  \[
    \alpha : \mathscr{D}_{\mathrm{cat}}(m) \times \prod_{s=1}^m 
    \mathscr{K}_{\mathrm{cat}}(j_s) \longrightarrow
    \mathscr{D}_{\mathrm{cat}}(m) \times \prod_{s=1}^m 
    \left(\mathscr{K}_{\mathrm{cat}}(j_s) \times_{S_{j_s}} X^{j_s} \right)
  \]
  \[
    \alpha\left( \sigma, f_1, f_2, \ldots f_m \right)
    = \left(\sigma,\; \prod_{s=1}^m 
    (f_s, X_s) \right).
  \]
  $\alpha$ takes a morphism 
  \[
    \tau\sigma^{-1} \times g_1 \times \ldots \times g_m \in 
    \mathrm{Mor}\left(\mathscr{D}_{\mathrm{cat}}(m) \times \prod_{s=1}^m 
    \mathscr{K}_{\mathrm{cat}}(j_s)\right)
  \]
  to the morphism
  \[
    \tau\sigma^{-1} \times (g_1 \times \mathrm{id}^{j_1}) \times \ldots \times
    (g_m \times \mathrm{id}^{j_m}) \in
    \mathrm{Mor}\left(\mathscr{D}_{\mathrm{cat}}(m) \times \prod_{s=1}^m 
    \left(\mathscr{K}_{\mathrm{cat}}(j_s) \times_{S_{j_s}} X^{j_s} \right)\right).
  \]
  Let $inc$ be the inclusion of categories:
  \[
    inc : \mathscr{D}_{\mathrm{cat}}(m) \times \prod_{s=1}^m 
    \left(\mathscr{K}_{\mathrm{cat}}(j_s) \times_{S_{j_s}} X^{j_s} \right)
    \longrightarrow
    \mathscr{D}_{\mathrm{cat}}(m) \times \left[
    \coprod_{i \geq 0} \left(\mathscr{K}_{\mathrm{cat}}(i) \times_{S_i} X^i \right)
    \right]^m,
  \]
  induced by the evident inclusion of for each $s$:
  \[
    \mathscr{K}_{\mathrm{cat}}(j_s) \times_{S_{j_s}} X^{j_s} \hookrightarrow
    \coprod_{i \geq 0} \left(\mathscr{K}_{\mathrm{cat}}(i) \times_{S_i}
    X^i \right)
  \]
  Let $\alpha_0$ be the functor
  \[
    \alpha_0 : \mathscr{K}_{\mathrm{cat}}(j) \longrightarrow
    \mathscr{K}_{\mathrm{cat}}(j) \times_{S_{j}} X^{j}
  \]
  \[
    \alpha_0(f) = (f, x_1, x_2, \ldots, x_j),
  \]
  and $inc_0$ be the inclusion $\mathscr{K}_{\mathrm{cat}}(j) \times_{S_{j}} X^{j}
  \hookrightarrow \ds{\coprod_{i \geq 0} \left(\mathscr{K}_{\mathrm{cat}}(i) \times_{S_i}
  X^i \right)}$.
  
  Next, consider the following diagram.  The top row is the map $\mu$ of
  (\ref{eq.mu}), and the bottom row is the 
  operad-algebra structure map for $\mathcal{T}J(X_+)$.
  \begin{equation}\label{eq.operad-module-comm-diag}
    \begin{diagram}
      \node{ \mathscr{D}_{\mathrm{cat}}(m) \times \prod_{s=1}^m 
        \mathscr{K}_{\mathrm{cat}}(j_s) }
      \arrow{e,t}{\mu}
      \arrow{s,l}{\alpha}
      \node{ \mathscr{K}_{\mathrm{cat}}(j) }
      \arrow{s,r}{\alpha_0}
      \\
      \node{ \mathscr{D}_{\mathrm{cat}}(m) \times \prod_{s=1}^m 
        \left(\mathscr{K}_{\mathrm{cat}}(j_s) \times_{S_{j_s}} X^{j_s} \right) }
      \arrow{s,l}{inc}
      \node{ \mathscr{K}_{\mathrm{cat}}(j) \times_{S_{j}} X^{j} }
      \arrow{s,r}{inc_0}
      \\
      \node{ \mathscr{D}_{\mathrm{cat}}(m) \times \left[
        \coprod_{i \geq 0} \left(\mathscr{K}_{\mathrm{cat}}(i) \times_{S_i} X^i \right)
        \right]^m }
      \arrow{e,t}{\theta}
      \node{ \coprod_{i \geq 0} \left(\mathscr{K}_{\mathrm{cat}}(i) \times_{S_i} 
        X^i\right) }
    \end{diagram}
  \end{equation}
  I claim that this diagram commutes.  Let $w := (\sigma, f_1, \ldots, f_m) \in 
  \mathscr{D}_{\mathrm{cat}}(m) \times \prod_{s=1}^m 
  \mathscr{K}_{\mathrm{cat}}(j_s)$ be arbitrary.  Following the left-hand column
  of the diagram, we obtain the element
  \begin{equation}\label{eq.alphaw}
    \alpha(w) = \left(\sigma,\; \prod_{s=1}^m 
    (f_s, X_s) \right).
  \end{equation}
  It is important to note that the list of all $x_r$ that occur in
  expression~(\ref{eq.alphaw}) is exactly $\{x_1, x_2, \ldots, x_j\}$
  with no repeats, up to permutations by $S_{j_1} \times \ldots \times S_{j_m}$.
  Thus, after applying $\theta$, the result is an element of the form:
  \begin{equation}\label{eq.theta_alphaw}
    \theta\alpha(w) = (F, x_1, x_2, \ldots, x_j),
  \end{equation}
  up to permutations by $S_j$.  Thus, $\theta\alpha(w)$ is in the image of
  $\alpha_0$, say $\theta\alpha(w) = \alpha_0(v)$.  All that remains is to show that
  $\mu(w) = v$.  Let us examine closely what the morphism $F$ in
  formula~(\ref{eq.theta_alphaw}) must be.  $\alpha(w)$ is
  identified with the element
  \[
    \left(\sigma,\; \prod_{s=1}^m 
    f_s(X_s) \right)
  \]
  of $\ds{\mathscr{D}_{\mathrm{cat}}(m) \times \left[\mathcal{T}J(X_+)\right]^m}$,
  and $\theta$ sends this element to
  \begin{equation}\label{eq.theta_alphaw2}
    \theta\alpha(w) = \bigodot_{s=1}^m f_{\sigma^{-1}(s)}(X_{\sigma^{-1}(s)}).
  \end{equation}
  $\theta\alpha(w)$ is interpreted in $\mathscr{K}_{\mathrm{cat}}(j) \times_{S_{j}} 
  X^{j}$ as follows:  Begin with the tuple
  $(x_1, x_2, \ldots, x_j)$.  This tuple is divided into blocks, $(X_1, \ldots, X_m)$.
  Apply $f_1 \odot \ldots \odot f_m$ to obtain
  \[
    f_1(X_1) \odot \ldots \odot
    f_m(X_m).
  \]
  
  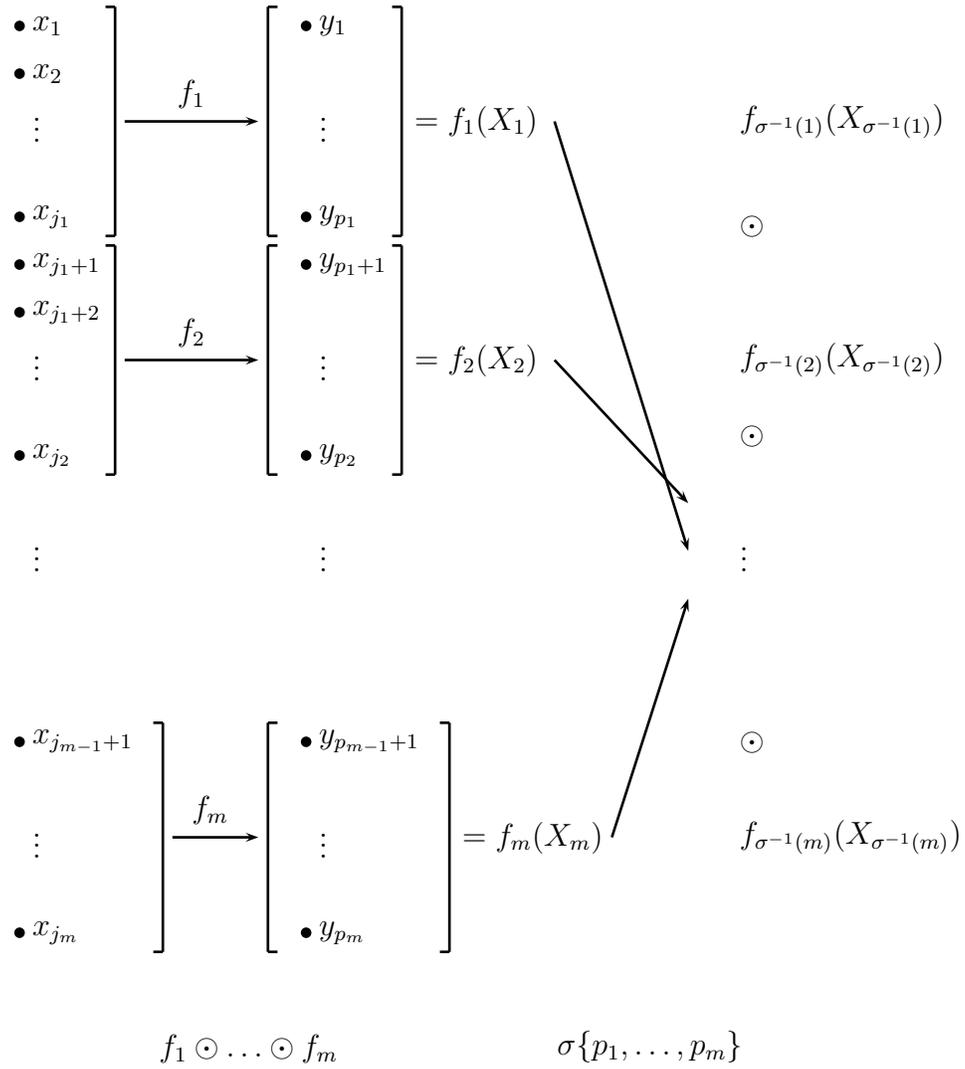
\begin{figure}[htp]
    \psset{unit=1in}
    \begin{pspicture}(5,6)
  
    \psdots[linecolor=black, dotsize=4pt]
    (0.5, 1)(0.5, 2)(0.5, 3.5)(.5, 4.25)(.5, 4.5)(.5, 4.75)(.5, 5.5)(.5, 5.75)
    (2, 1)(2, 2)(2, 3.5)(2, 4.5)(2, 4.75)(2, 5.75)
  
    \psline[linewidth=1pt, linecolor=black](1.25, .9)(1.25, 2.1)
    \psline[linewidth=1pt, linecolor=black](1.25, .9)(1.2, .9)
    \psline[linewidth=1pt, linecolor=black](1.25, 2.1)(1.2, 2.1)
    \psline[linewidth=1pt, linecolor=black](1, 3.4)(1, 4.6)
    \psline[linewidth=1pt, linecolor=black](1, 3.4)(.95, 3.4)
    \psline[linewidth=1pt, linecolor=black](1, 4.6)(.95, 4.6)
    \psline[linewidth=1pt, linecolor=black](1, 4.65)(1, 5.85)
    \psline[linewidth=1pt, linecolor=black](1, 4.65)(.95, 4.65)
    \psline[linewidth=1pt, linecolor=black](1, 5.85)(.95, 5.85)

    \psline[linewidth=1pt, linecolor=black](1.8, .9)(1.8, 2.1)
    \psline[linewidth=1pt, linecolor=black](1.8, .9)(1.85, .9)
    \psline[linewidth=1pt, linecolor=black](1.8, 2.1)(1.85, 2.1)
    \psline[linewidth=1pt, linecolor=black](1.8, 3.4)(1.8, 4.6)
    \psline[linewidth=1pt, linecolor=black](1.8, 3.4)(1.85, 3.4)
    \psline[linewidth=1pt, linecolor=black](1.8, 4.6)(1.85, 4.6)
    \psline[linewidth=1pt, linecolor=black](1.8, 4.65)(1.8, 5.85)
    \psline[linewidth=1pt, linecolor=black](1.8, 4.65)(1.85, 4.65)
    \psline[linewidth=1pt, linecolor=black](1.8, 5.85)(1.85, 5.85)
  
    \psline[linewidth=1pt, linecolor=black](2.75, .9)(2.75, 2.1)
    \psline[linewidth=1pt, linecolor=black](2.75, .9)(2.7, .9)
    \psline[linewidth=1pt, linecolor=black](2.75, 2.1)(2.7, 2.1)
    \psline[linewidth=1pt, linecolor=black](2.5, 3.4)(2.5, 4.6)
    \psline[linewidth=1pt, linecolor=black](2.5, 3.4)(2.45, 3.4)
    \psline[linewidth=1pt, linecolor=black](2.5, 4.6)(2.45, 4.6)
    \psline[linewidth=1pt, linecolor=black](2.5, 4.65)(2.5, 5.85)
    \psline[linewidth=1pt, linecolor=black](2.5, 4.65)(2.45, 4.65)
    \psline[linewidth=1pt, linecolor=black](2.5, 5.85)(2.45, 5.85)
  
    \psline[linewidth=1pt, linecolor=black]{->}(1.05, 5.25)(1.75, 5.25)
    \psline[linewidth=1pt, linecolor=black]{->}(1.05, 4)(1.75, 4)
    \psline[linewidth=1pt, linecolor=black]{->}(1.3, 1.5)(1.75, 1.5)
    \uput[u](1.4, 5.25){$f_1$}
    \uput[u](1.4, 4){$f_2$}
    \uput[u](1.5, 1.5){$f_m$}

    \psline[linewidth=1pt, linecolor=black]{->}(3.3, 5.25)(4, 3)
    \psline[linewidth=1pt, linecolor=black]{->}(3.3, 4)(4, 3.25)
    \psline[linewidth=1pt, linecolor=black]{->}(3.6, 1.5)(4, 2.75)
 
    \uput[r](.5, 1){$x_{j_m}$}
    \uput[r](.5, 1.5){$\vdots$}
    \uput[r](.5, 2){$x_{j_{m-1}+1}$}
    \uput[r](.5, 3){$\vdots$}
    \uput[r](.5, 3.5){$x_{j_2}$}
    \uput[r](.5, 4){$\vdots$}
    \uput[r](.5, 4.25){$x_{j_1 + 2}$}
    \uput[r](.5, 4.5){$x_{j_1+1}$}
    \uput[r](.5, 4.75){$x_{j_1}$}
    \uput[r](.5, 5.25){$\vdots$}
    \uput[r](.5, 5.5){$x_{2}$}
    \uput[r](.5, 5.75){$x_1$}
    \uput[r](2, 1){$y_{p_m}$}
    \uput[r](2, 1.5){$\vdots$}
    \uput[r](2, 2){$y_{p_{m-1}+1}$}
    \uput[r](2, 3){$\vdots$}
    \uput[r](2, 3.5){$y_{p_2}$}
    \uput[r](2, 4){$\vdots$}
    \uput[r](2, 4.5){$y_{p_1+1}$}
    \uput[r](2, 4.75){$y_{p_1}$}
    \uput[r](2, 5.25){$\vdots$}
    \uput[r](2, 5.75){$y_1$}
  
    \uput[r](2.5, 5.25){$= f_1(X_1)$}
    \uput[r](2.5, 4){$= f_2(X_2)$}
    \uput[r](2.75, 1.5){$= f_m(X_m)$}

    \uput[r](4.2, 5.25){$f_{\sigma^{-1}(1)}(X_{\sigma^{-1}(1)})$}
    \uput[r](4.2, 4.7){$\odot$}
    \uput[r](4.2, 4){$f_{\sigma^{-1}(2)}(X_{\sigma^{-1}(2)})$}
    \uput[r](4.2, 3.6){$\odot$}
    \uput[r](4.2, 3){$\vdots$}
    \uput[r](4.2, 2){$\odot$}
    \uput[r](4.2, 1.5){$f_{\sigma^{-1}(m)}(X_{\sigma^{-1}(m)})$}

    \uput[u](1.7, .25){$f_1 \odot \ldots \odot f_m$}
    \uput[u](3.8, .25){$\sigma\{p_1, \ldots, p_m\}$}

    \end{pspicture}

    \caption[Operad-Module Structure]{$\theta\alpha(\sigma, f_1, \ldots, f_m)$, interpreted
      as an object of $\mathscr{K}_{\mathrm{cat}}(j) \times_{S_{j}} 
      X^{j}$}
    \label{diag.theta_alphaw}
  \end{figure}

  Finally, apply the block permutation $\sigma\{p_1, \ldots, p_m\}$ to get the
  correct order in the result (see Fig.~\ref{diag.theta_alphaw}).  This shows that 
  $F =  {\sigma\{p_1, \ldots, p_m\}}(f_1 \odot \ldots \odot f_m)$, as required.
  
  Now that we have the diagram~(\ref{eq.operad-module-comm-diag}), it is
  straightforward to show that $\mu$ satisfies the associativity condition
  for an operad-module structure map.  Essentially, associativity is induced
  by the associativity condition of the algebra structure map $\theta$.  
  All that remains is to verify the unit and equivariance conditions.
  
  It is trivial to verify the unit condition on the level of objects.  The
  unit object of $\mathscr{D}_{\mathrm{cat}}(1)$ is the identity of $S_1$.
  According to formula~(\ref{eq.mu-morphisms}), we obtain the following
  diagram for morphisms $f : \underline{j} \to \underline{p}$
  and $g : \underline{p} \to \underline{q}$.  Clearly, the right-hand
  column is identical to the morphism $f \stackrel{g}{\to} gf$.
  \[
    \begin{diagram}
      \node{ (\mathrm{id}_{S_1}, f) }
      \arrow{s,r}{ \mathrm{id}_{S_1} \times g}
      \arrow{e,t}{\mu}
      \node{  ({\mathrm{id}_{S_p}}) f }
      \arrow{s,r}{ ({\mathrm{id}_{S_q}}) g }
      \\
      \node{ (\mathrm{id}_{S_1}, gf) }
      \arrow{e,t}{\mu}
      \node{  ({\mathrm{id}_{S_q}}) gf }
    \end{diagram}
  \]
  
  (Note, there is no corresponding right unit condition in an operad-module
  structure.)
  
  Now, specify
  the right-action of $\rho \in S_j$ on $\mathscr{K}_{\mathrm{cat}}(j)$ as 
  precomposition by $ {\rho}$.  That is, $f . \rho := 
  f {\rho}$ for
  $f \in \underline{j} \setminus \mathcal{F}(as)$.  A routine check verifies the 
  equivariance on the level of objects.  
  Let $f_i \in \mathscr{K}_{\mathrm{cat}}(j_i)$ have 
  specified source and target:
  \[
    \underline{j_i} \stackrel{f_i}{\to} \underline{p_i}.
  \]
  
  Equivariance A:
  \[
    \begin{diagram}
      \node{ (\sigma, f_1, \ldots, f_m) }
      \arrow[2]{e,t,T}{ \mathrm{id} \times T_{\tau} }
      \arrow{s,l,T}{ \tau \times \mathrm{id}^{m} }
      \node[2]{(\sigma, f_{\tau^{-1}(1)}, \ldots, f_{\tau^{-1}(m)}) }
      \arrow{s,r,T}{ \mu }
      \\
      \node{ (\sigma\tau, f_1, \ldots, f_m) }
      \arrow{s,l,T}{ \mu }
      \node[2]{ \sigma\{p_{\tau \underline{m}} \}
        f_{\tau \underline{m}}^\odot }
      \arrow{s,b,T}{ \tau\{ j_{\underline{m}} \} }
      \\
      \node{ (\sigma\tau)\{p_{\underline{m}}\} f_{\underline{m}}^\odot }
      \arrow{e,=}
      \node{ \sigma\{ p_{\tau \underline{m}} \}\tau\{ p_{\underline{m}} \}
        f_{\underline{m}}^\odot }
      \arrow{e,=}
      \node{ \sigma\{ p_{\tau \underline{m}} \} f_{\tau\underline{m}}^\odot
        \tau\{ j_{\underline{m}} \} }
    \end{diagram}
  \]
  Equivariance B:
  \[
    \begin{diagram}
      \node{(\sigma, f_1, \ldots, f_m)}
      \arrow[2]{e,t,T}{\mu}
      \arrow{s,l,T}{\mathrm{id} \times \tau_1 \times \ldots \times \tau_m}
      \node[2]{\sigma\{p_{\underline{m}}\}f_{\underline{m}}^\odot}
      \arrow{s,r,T}{\tau_1 \oplus \ldots \oplus \tau_m}\\
      \node{(\sigma, f_1\tau_1, \ldots, f_m\tau_m)}
      \arrow{e,t,T}{\mu}
      \node{\sigma\{p_{\underline{m}}\} (f_{\underline{m}}
        \tau_{\underline{m}})^\odot }
      \arrow{e,=}
      \node{ \sigma\{p_{\underline{m}}\} f_{\underline{m}}^\odot
        \tau_{\underline{m}}^\odot }
    \end{diagram}
  \]
\end{proof}

\begin{rmk}\label{rmk.K-pseuo-operad}
  It turns out that $\mathscr{K}_{\mathrm{cat}}$ is in fact a pseudo-operad.
  Recall from~\cite{MSS} that a pseudo-operad is a `non-unitary' operad.  That is,
  there are multiplication maps that satisfy operad associativity, and actions by
  the symmetric groups that satisfy operad equivariance conditions, but there is
  no requirement concerning a left or right unit map.
  The multiplication is defined as the composition:
  \[
    \mathscr{K}_{\mathrm{cat}}(m) \times \prod_{s=1}^m \mathscr{K}_{\mathrm{cat}}(j_s)
    \stackrel{\pi \times \mathrm{id}^j}{\longrightarrow}
    \mathscr{D}_{\mathrm{cat}}(m) \times \prod_{s=1}^m \mathscr{K}_{\mathrm{cat}}(j_s)
    \stackrel{\mu}{\longrightarrow}
    \mathscr{K}_{\mathrm{cat}}(j_1 + \ldots + j_m),
  \]
  where $\pi : \mathscr{K}_{\mathrm{cat}}(m) \to \mathscr{D}_{\mathrm{cat}}(m)$
  is the projection functor defined as isolating the group element (automorphism)
  of an $\mathcal{F}(as)$ morphism:
  \[
    (\phi, g) \stackrel{\pi}{\mapsto} g
  \]
  Indeed, $\pi$ defines a covariant isomorphism
  of the subcategory 
  $\mathrm{Aut}\left(\underline{m} \setminus \mathcal{F}(as)\right)$ onto
  $\mathscr{D}_{\mathrm{cat}}(m)$.  
\end{rmk}

Now that we have a $\mathscr{D}_{\mathrm{cat}}$-module structure on 
$\mathscr{K}_{\mathrm{cat}}$, we
shall proceed in steps to induce this structure to an analogous operad-module
structure on the level of $k$-complexes.
First, we shall require the definition
of {\it lax symmetric monoidal functor}.  The following appears in~\cite{T}, as
well as~\cite{R2}.

\begin{definition}
  Let $\mathscr{C}$, resp. $\mathscr{C}'$, be a symmetric monoidal category with
  multiplication $\odot$, resp. $\boxdot$.  Denote the associativity
  maps in $\mathscr{C}$, resp. $\mathscr{C}'$ by $a$, resp. $a'$, and the commutation
  maps by $s$, resp. $s'$.  A functor $F : \mathscr{C} \to
  \mathscr{C}'$ is a lax symmetric monoidal functor if there are natural maps
  \[
    f : FA \,\boxdot\, FB \to F(A \odot B)
  \]
  such that the following diagrams are commutative:
  \begin{equation}\label{eq.lax-assoc}
    \begin{diagram}
      \node{FA \,\boxdot\, (FB \,\boxdot\, FC)}
      \arrow{e,t}{\mathrm{id} \,\boxdot\, f}
      \arrow{s,l}{a'}
      \node{FA \,\boxdot\, F(B\odot C)}
      \arrow{e,t}{f}
      \node{F\left(A \odot (B \odot C)\right)}
      \arrow{s,r}{Fa}
      \\
      \node{(FA \,\boxdot\, FB) \,\boxdot\, FC}
      \arrow{e,t}{f \,\boxdot\, \mathrm{id}}
      \node{F(A \odot B) \,\boxdot\, FC}
      \arrow{e,t}{f}
      \node{F\left((A \odot B) \odot C\right)}
    \end{diagram}
  \end{equation}
  \begin{equation}\label{eq.lax-comm}
    \begin{diagram}
      \node{FA \,\boxdot\, FB}
      \arrow{e,t}{f}
      \arrow{s,l}{s'}
      \node{F(A \odot B)}
      \arrow{s,r}{Fs}
      \\
      \node{FB \,\boxdot\, FA}
      \arrow{e,t}{f}
      \node{F(B \odot A)}
    \end{diagram}
  \end{equation}
  If the transformation $f$ is a natural isomorphism, then the functor $F$ is
  called {\it strong symmetric monoidal}.
\end{definition}

Observe that the functor \mbox{$N : \mathbf{Cat} \to \mathbf{Simpset}$}
is strong symmetric monoidal with associated natural map, $S_*$:
\[
  S_* : N \mathscr{A}_{\mathrm{cat}} \times N \mathscr{B}_{\mathrm{cat}}
  \longrightarrow N\left(\mathscr{A}_{\mathrm{cat}} \times 
  \mathscr{B}_{\mathrm{cat}}\right),
\]
defined on $n$-chains:
\[
  \left( A_0 \stackrel{f_1}{\to} A_1 \stackrel{f_2}{\to} \ldots
  \stackrel{f_n}{\to} A_n, \;
  B_0 \stackrel{g_1}{\to} B_1 \stackrel{g_2}{\to} \ldots
  \stackrel{g_n}{\to} B_n \right)
  \stackrel{S_n}{\mapsto}
  (A_0, B_0) \stackrel{f_1 \times g_1}{\longrightarrow} \ldots
  \stackrel{f_n \times g_n}{\longrightarrow}(A_n, B_n) 
\]

The $k$-linearization functor, \mbox{$k[ - ] : \mathbf{SimpSet} \to
\textrm{$k$-\textbf{SimpMod}}$} is also strong symmetric monoidal,
with associated natural map,
\[
  k[ \mathscr{A}_{\mathrm{ss}} ] \,\widehat{\otimes}\, k[ \mathscr{B}_{\mathrm{ss}} ]
  \longrightarrow k[ \mathscr{A}_{\mathrm{ss}} \times 
  \mathscr{B}_{\mathrm{ss}} ],
\]
defined degree-wise:
\[
  k[ (\mathscr{A}_{\mathrm{ss}})_n ] \otimes k[ (\mathscr{B}_{\mathrm{ss}})_n ]
  \stackrel{\cong}{\longrightarrow} 
  k[ (\mathscr{A}_{\mathrm{ss}})_n \times (\mathscr{B}_{\mathrm{ss}})_n ]
  =
  k[ (\mathscr{A}_{\mathrm{ss}} \times  \mathscr{B}_{\mathrm{ss}})_n ].
\]

Finally, the normalization functor, \mbox{$\mathcal{N} : \textrm{$k$-\textbf{SimpMod}}
\to \textrm{$k$-$\mathbf{Complexes}$}$} is lax symmetric monoidal,
with associated natural map $f$ being the Eilenberg-Zilber shuffle map (see~\cite{R}).
\[ 
  Sh : \mathcal{N}\mathscr{A}_{\mathrm{sm}} \otimes \mathcal{N}\mathscr{B}_{\mathrm{sm}}
  \longrightarrow \mathcal{N}\left(\mathscr{A}_{\mathrm{sm}} \,\widehat{\otimes}\,
  \mathscr{B}_{\mathrm{sm}}\right)
\]

Now, define the versions of $\mathscr{D}_{\mathrm{cat}}$ and 
$\mathscr{K}_{\mathrm{cat}}$ over the various symmetric monoidal categories we 
are considering:
\[
  \begin{array}{lcl}
  \mathscr{D}_{\mathrm{ss}} &=& N\mathscr{D}_{\mathrm{cat}}\\
  \mathscr{D}_{\mathrm{sm}} &=& k\left[\mathscr{D}_{\mathrm{ss}}\right]\\
  \mathscr{D}_{\mathrm{ch}} &=& \mathcal{N}\mathscr{D}_{\mathrm{sm}}
  \end{array}
  \qquad\qquad\qquad
  \begin{array}{lcllcl}
  \mathscr{K}_{\mathrm{ss}} &=& N\mathscr{K}_{\mathrm{cat}}\\
  \mathscr{K}_{\mathrm{sm}} &=& k\left[\mathscr{K}_{\mathrm{ss}}\right]\\
  \mathscr{K}_{\mathrm{ch}} &=& \mathcal{N}\mathscr{K}_{\mathrm{sm}}
  \end{array}
\]

\begin{lemma}\label{lem.lax-sym-mon-functor}
  Let $(\mathscr{C}, \odot, e)$ and $(\mathscr{C}', \boxdot, e')$ be symmetric monoidal 
  categories, and $F : \mathscr{C} \to \mathscr{C}'$ a
  lax symmetric monoidal functor with associated natural transformation $f$ such
  that $e' = Fe$.
  
  1.  If $\mathscr{P}$ is an operad over $\mathscr{C}$, then
  $F\mathscr{P}$ is an operad over $\mathscr{C}'$.
  
  2.  If $\mathscr{P}$ is an operad and $\mathscr{M}$ is a $\mathscr{P}$-module
  over $\mathscr{C}$, then $F\mathscr{M}$ is an $F\mathscr{P}$-module over
  $\mathscr{C}'$.
  
  3.  If $\mathscr{P}$ is an operad over $\mathscr{C}$ and $Z \in 
  \mathrm{Obj}\mathscr{C}$ is a $\mathscr{P}$-algebra, then
  $FZ$ is an $F\mathscr{P}$-algebra over $\mathscr{C}'$.
\end{lemma}
\begin{proof}
  Note that properties~(\ref{eq.lax-assoc}) and~(\ref{eq.lax-comm}) imply that
  the associativity map $a'$ and symmetry map $s'$ of $\mathscr{C}'$ may
  be viewed as induced by the associativity map $a$ and symmetry 
  map $s$ of $\mathscr{C}$.  That is, all symmetric monoidal structure is
  carried by $F$ from $\mathscr{C}$ to $\mathscr{C}'$.
  
  Denote by $f^m$ the
  natural transformation induced by $f$ on $m+1$ components:
  \[
    f^m : FA_0 \boxdot FA_1 \boxdot \ldots \boxdot FA_m \to
    F(A_0 \odot A_1 \odot \ldots \odot A_m).
  \]
  Technically, we should write parentheses to
  indicate associativity in the source and target of $f^m$, but
  property~(\ref{eq.lax-assoc}) of the functor and the MacLane Pentagon
  diagram of Def.~\ref{def.symmoncat} makes this unnecessary.
  
  Let $\mathscr{P}$ have structure map $\gamma$.  Define the structure
  map $\gamma'$ for $F\mathscr{P}$:
  \[
    \gamma' := F\gamma \circ f^{m} :
    F\mathscr{P}(m) \boxdot F\mathscr{P}(j_1) \boxdot \ldots
    \boxdot F\mathscr{P}(j_m) \longrightarrow
    F\mathscr{P}(j_1 + \ldots + j_m).
  \]
  Let $e$ (resp. $e'$) be the unit object of $\mathscr{C}$ (resp. $\mathscr{C}'$).
  Then $e' = Fe$.  If $\mathscr{P}$ has unit map $\eta : e \to \mathscr{P}(1)$,
  then define the unit map of $F\mathscr{P}$ by
  \[
    \eta' := F\eta : e' \to F\mathscr{P}(1).
  \]
  The action of $S_n$ on $F\mathscr{P}$ is defined by viewing $F\mathscr{P}$
  as a functor $\mathbf{S}^{\mathrm{op}} \to \mathscr{C}'$:
  \[
    \mathbf{S}^{\mathrm{op}} \stackrel{\mathscr{P}}{\longrightarrow}
    \mathscr{C} \stackrel{F}{\longrightarrow} \mathscr{C}'
  \]
  To verify that the proposed structure on $F\mathscr{P}$ defines an
  operad is straightforward, but the required commutative diagrams do not fit
  on one page!
  
  Assertions {\it 2} and {\it 3} are proved similarly.
\end{proof}
\begin{cor}\label{cor.K_ch}
  $\mu$ induces a multiplication map $\widetilde{\mu}_*$ on the
  level of chain complexes, making $\mathscr{K}_{\mathrm{ch}}$ into a
  $\mathscr{D}_{\mathrm{ch}}$-module.
  \[
    \widetilde{\mu}_* :  \mathscr{D}_{\mathrm{ch}}(m) \otimes 
    \mathscr{K}_{\mathrm{ch}}(j_1)
    \otimes \ldots \otimes \mathscr{K}_{\mathrm{ch}}(j_m) \to 
    \mathscr{K}_{\mathrm{ch}}(j_1 + \ldots + j_m).
  \]  
\end{cor}
\begin{proof}
  Since $\mathscr{K}_{\mathrm{cat}}$ is a
  $\mathscr{D}_{\mathrm{cat}}$-module, and
  the functors $N$, $k[ - ]$ and $\mathcal{N}$ are symmetric monoidal
  (the first two in the strong sense, the third in the lax sense), this
  result follows immediately from Lemma~\ref{lem.lax-sym-mon-functor}.
\end{proof}

\section{$E_{\infty}$-Algebra Structure}\label{sec.alg-structure} %

In this section we use the operad-module structure defined in the
previous section to induce a related operad-algebra structure.

\begin{definition}\label{def.distributive}
  Suppose $(\mathscr{C}, \odot)$ is a cocomplete symmetric monoidal category.
  We say $\mathscr{C}$ {\it $\odot$ 
  distributes over colimits}, or is {\it distributive}, if the natural map
  \[
    \colim_{i \in \mathscr{I}} (B \odot C_i) \longrightarrow
    B \odot \colim_{i \in \mathscr{I}} C_i
  \]
  is an isomorphism.
\end{definition}

\begin{rmk}
  All of the ambient categories we consider in this chapter, $\mathbf{Cat}$, 
  $\mathbf{SimpSet}$,
  $k$-$\mathbf{SimpMod}$, and
  $k$-$\mathbf{Complexes}$, are cocomplete and distributive.
\end{rmk}

\begin{lemma}\label{lem.operad-algebra}
  Suppose $(\mathscr{C}, \odot)$ is a cocomplete distributive symmetric monoidal 
  category,
  $\mathscr{P}$ is an operad over $\mathscr{C}$, $\mathscr{L}$ is
  a left $\mathscr{P}$-module, and $Z \in \mathrm{Obj}\mathscr{C}$.  Then
  \[
    \mathscr{L} \langle Z \rangle := 
    \coprod_{m \geq 0} \mathscr{L}(m) \odot_{S_m} Z^{\odot m}
  \]
  admits the structure of a $\mathscr{P}$-algebra.
\end{lemma}
\begin{rmk}
  The notation $\mathscr{L} \langle Z \rangle$ appears in Kapranov and 
  Manin~\cite{KM}, and is also present in~\cite{MSS} as the {\it Schur
  functor} of an operad (\cite{MSS}, Def~1.24), 
  $\mathcal{S}_{\mathscr{L}}(Z)$.
\end{rmk}
\begin{proof}
  What we are looking for is an equivariant map
  \[
    \mathscr{P}\left\langle \mathscr{L} \langle Z \rangle \right\rangle
    \to \mathscr{L} \langle Z \rangle.
  \]
  That is, a map
  \[
    \coprod_{n \geq 0} \mathscr{P}(n) \odot_{S_n} \left[
      \coprod_{m \geq 0} \mathscr{L}(m) \odot_{S_m} Z^{\odot m} \right]^{\odot n}
    \longrightarrow \coprod_{m \geq 0} \mathscr{L}(m) \odot_{S_m} Z^{\odot m},
  \]
  satisfying the required associativity conditions for an operad-algebra
  structure.

  Observe that the equivariant product $\odot_{S_m}$ may be constructed
  as the coequalizer of the maps corresponding to the $S_m$-action.  That is,
  \[
    \mathscr{L}(m) \odot_{S_m} Z^{\odot m} = 
    \coeq_{\sigma \in S_m} \left\{ \sigma^{-1} \odot \sigma \right\},
  \]
  where \mbox{$\sigma^{-1} \odot \sigma : \mathscr{L}(m) \odot Z^{\odot m}
  \to \mathscr{L}(m) \odot Z^{\odot m}$} is given by right action of
  $\sigma^{-1}$ on $\mathscr{L}(m)$ and by permutation of the factors of
  $Z^{\odot m}$ by $\sigma$ (See~\cite{MSS}, formula~(1.11)).  Thus,
  $\mathscr{L}\langle Z \rangle$ may be expressed as a (small) colimit.
  Since we presuppose that $\odot$ distributes over all small colimits, it suffices
  to fix an integer $n$ as well as $n$ integers $m_1$, $m_2$, \ldots, $m_n$, and
  examine the following diagram:
  \[
    \begin{diagram}
      \node{ \mathscr{P}(n) \odot \left( \left[\mathscr{L}(m_1) \odot 
        Z^{\odot m_1}\right] \odot \ldots \odot \left[\mathscr{L}(m_n) 
        \odot Z^{\odot m_n}\right]\right) }
      \arrow{s,r}{T}
      \\
      \node{ \mathscr{P}(n) \odot \left(\mathscr{L}(m_1) \odot
        \ldots \odot \mathscr{L}(m_n)\right) \odot \left( 
        Z^{\odot m_1} \odot \ldots \odot Z^{\odot m_n}\right) }
      \arrow{s,r}{ \mu \odot \mathrm{id} }
      \\
      \node{ \mathscr{L}(m_1 + \ldots + m_n) \odot \left( 
        Z^{\odot m_1} \odot \ldots \odot Z^{\odot m_n}\right) }
      \arrow{s,r}{a}
      \\
      \node{ \mathscr{L}(m_1 + \ldots + m_n) \odot Z^{\odot (m_1 + \ldots + m_n)} }
    \end{diagram}
  \]
  In this diagram, $T$ is the evident shuffling of components so that the
  components $\mathscr{L}(m_i)$ are grouped together, $\mu$ is the 
  operad-module structure map
  for $\mathscr{L}$, and $a$ stands for the various associativity maps that are
  required to obtain the final form.  This composition defines a family of maps 
  \[
    \eta : \mathscr{P}(n) \odot \bigodot_{i=1}^{n}
    \left[\mathscr{L}(m_i) \odot Z^{\odot m_i}\right]
    \longrightarrow \mathscr{L}(m_1 + \ldots + m_n) \odot 
    Z^{\odot (m_1 + \ldots + m_n)}.
  \]
  The maps $\eta$ pass to $S_{m_i}$-equivalence classes, producing a family of maps:
  \[
    \overline{\eta} : \mathscr{P}(n) \odot \bigodot_{i=1}^{n}
    \left[\mathscr{L}(m_i) \odot_{S_{m_i}} Z^{\odot m_i}\right]
    \longrightarrow \mathscr{L}(m_1 + \ldots + m_n) \odot_{S_{m_1} \times \ldots 
    \times S_{m_n}}
    Z^{\odot (m_1 + \ldots + m_n)}.
  \]
  Let $p$ be the evident projection map for a right
  \mbox{$S_{m_1 + \ldots + m_n}$-object} $M$ and left 
  \mbox{$S_{m_1 + \ldots + m_n}$-object} $N$,
  \[
    M \odot_{S_{m_1} \times \ldots \times S_{m_n}} N
    \stackrel{p}{\longrightarrow}
    M \odot_{S_{m_1 + \ldots + m_n}} N.
  \]
  Define the family of maps, $\chi := p \circ \overline{\eta}$,
  \[
    \chi : \mathscr{P}(n) \odot \bigodot_{i=1}^{n}
    \left[\mathscr{L}(m_i) \odot_{S_{m_i}} Z^{\odot m_i}\right]
    \longrightarrow \mathscr{L}(m_1 + \ldots + m_n) \odot_{S_{m_1 + \ldots 
    + m_n}}
    Z^{\odot (m_1 + \ldots + m_n)}.
  \]
  That is, we have a structure map, $\chi$, defined for each $n$:
  \[
    \chi : \mathscr{P}(n) \odot \left[\mathscr{L} \langle Z \rangle \right]^{\odot n}
    \to \mathscr{L} \langle Z \rangle.
  \]
  Since $\mathscr{L}$ is a left $\mathscr{P}$-module, $\chi$ is compatible
  with the multiplication maps of $\mathscr{P}$.
  Equivariance follows from external
  equivariance conditions on $\mathscr{L}$ as $\mathscr{P}$-module, together
  with the internal equivariance relations present in
  \[
    \mathscr{L}(m_1 + \ldots + m_n) \odot_{S_{m_1 + \ldots + m_n}}
    Z^{\odot (m_1 + \ldots + m_n)},
  \]
  inducing the required operad-algebra structure map
  \[
    \chi : \mathscr{P}(n) \odot_{S_n} \left[\mathscr{L} \langle Z \rangle \right]^{\odot n}
    \to \mathscr{L} \langle Z \rangle.
  \]
\end{proof}

Let $A$ be an associative, unital algebra over $k$.  
Let \mbox{$\mathscr{Y}^+_*A = k[N(- \setminus \Delta S_+)]\otimes_{\Delta S_+} 
B^{sym_+}_*A $} be the complex from section~\ref{sec.deltas_plus}, regarded
as simplicial $k$-module.  Observe, we may identify:
\[
  k[N(- \setminus \Delta S_+)] \otimes_{\mathrm{Aut}\Delta S_+} B^{sym_+}_*A
  = \bigoplus_{n \geq 0} 
  \mathscr{K}_{\mathrm{sm}}(n) \,\widehat{\otimes}_{S_n} A^{\widehat{\otimes}\, n}
  = \mathscr{K}_{\mathrm{sm}} \langle A \rangle.
\]
(Note, $A$ is regarded as a trivial simplicial object, with all faces and
degeneracies being identities.)

\begin{lemma}\label{lem.hatK-operadalgebra}
  $\mathscr{K}_{\mathrm{sm}} \langle A \rangle$ has the structure of an 
  $E_\infty$-algebra over the category of simplicial $k$-modules, 
  \[
    \chi : \mathscr{D}_{\mathrm{sm}} \left\langle \mathscr{K}_{\mathrm{sm}} 
    \langle A \rangle  \right\rangle \longrightarrow
    \mathscr{K}_{\mathrm{sm}} \langle A \rangle
  \]
\end{lemma}
\begin{proof}
  This follows immediately from Lemma~\ref{lem.operad-algebra} and the fact that
  $k$-$\mathbf{SimpMod}$ is cocomplete and distributive.  
\end{proof}
\begin{rmk}
  The fact that $\mathscr{K}_{\mathrm{cat}}$ is a pseudo-operad 
  (See Remark~\ref{rmk.K-pseuo-operad}) implies that
  $\mathscr{K}_{\mathrm{ss}}$ and $\mathscr{K}_{\mathrm{sm}}$ are also
  pseudo-operads (c.f.~Lemma~\ref{lem.lax-sym-mon-functor}).  Now, the properties of 
  the Schur functor do not depend on existence of a right unit map
  for $\mathscr{K}_{\mathrm{sm}}$, so we could conclude immediately that 
  $\mathcal{S}_{\mathscr{K}_{\mathrm{sm}}}(A) = 
  \mathscr{K}_{\mathrm{sm}} \langle A \rangle$
  is a `pseudo-operad'-algebra over $\mathscr{K}_{\mathrm{sm}}$, however the preceding 
  proof requires a bit less machinery.
\end{rmk}  

\begin{lemma}\label{lem.Y-operadalgebra}
  The $\mathscr{D}_{\mathrm{sm}}$-algebra structure on $\mathscr{K}_{\mathrm{sm}} 
  \langle A \rangle$
  induces a quotient $\mathscr{D}_{\mathrm{sm}}$-algebra structure on $\mathscr{Y}_*^+A$,
  \[
    \overline{\chi} : \mathscr{D}_{\mathrm{sm}} 
    \left\langle \mathscr{Y}_*^+A \right\rangle \longrightarrow
    \mathscr{Y}_*^+A
  \]
  That is, $\mathscr{Y}_*^+A$ is an $E_{\infty}$-algebra over the category of
  simplicial $k$-modules.
\end{lemma}
\begin{proof}
  We must verify that the structure map $\chi$ from Lemma~\ref{lem.hatK-operadalgebra}
  is well-defined on equivalence classes in $\mathscr{Y}_*^+A$.  
  Let $\overline{\chi}$ be defined 
  by applying $\chi$ to a representative, so that we obtain a map
  \[
    \overline{\chi} : \mathscr{D}_{\mathrm{sm}}(n) \,\widehat{\otimes}_{S_n}
    \left(\mathscr{Y}_*^+A\right)^{\widehat{\otimes}\, n}
    \to \mathscr{Y}_*^+A.
  \]
  It suffices to check $\overline{\chi}$ is well-defined on $0$-chains.  
  Let $f_i$, $g_i$, $1 \leq i \leq n$,
  be morphisms of $\mathcal{F}(as)$ with specified sources and targets:
  \[
    \underline{m_i} \stackrel{f_i}{\to} \underline{p_i} 
    \stackrel{g_i}{\to} \underline{q_i}.
  \]
  Let $V_i$ be a simple tensor of $A^{\otimes m_i}$, that is, $V_i := a_1 \otimes
  a_2 \otimes \ldots \otimes a_{m_i}$ for some $a_s \in A$.  Let $\sigma \in S_n$.
  Consider the $0$-chain of 
  $\mathscr{D}_{\mathrm{sm}}(n) \,\widehat{\otimes}_{S_n}
    \left(\mathscr{Y}_*^+A\right)^{\widehat{\otimes}\, n}$:
  \begin{equation}\label{eq.D-algebra-invarianceL}
    \sigma \,\widehat{\otimes}\, \left( g_1f_1 \otimes V_1 \right) \,\widehat{\otimes}\, \ldots 
    \,\widehat{\otimes}\,
    \left( g_nf_n \otimes V_n \right).
  \end{equation}
  The map $\overline{\chi}$ sends this to:
  \[
    \sigma\{q_{\underline{n}}\}(g_{\underline{n}}f_{\underline{n}})^\odot 
    \otimes V_{\underline{n}}^\otimes,
  \]
  where $V_{\underline{n}}^\otimes := V_1 \otimes \ldots \otimes V_n \in 
  A^{\otimes(m_1 + \ldots + m_n)}$.
  
  On the other hand, the element~(\ref{eq.D-algebra-invarianceL}) is equivalent under
  $\mathrm{Mor}\left(\mathcal{F}(as)\right)$-equivariance to:
  \begin{equation}\label{eq.D-algebra-invarianceR}
    \sigma \,\widehat{\otimes}\, \left( g_1 \otimes f_1(V_1) \right) \,\widehat{\otimes}\, \ldots 
    \,\widehat{\otimes}\,
    \left( g_n \otimes f_n(V_n) \right),
  \end{equation}
  and $\overline{\chi}$ sends this to:
  \[
    \sigma\{q_{\underline{n}}\}g_{\underline{n}}^\odot \otimes 
    \left[f_{\underline{n}}(V_{\underline{n}})\right]^\otimes 
  \]
  \[
    = \sigma\{q_{\underline{n}}\}g_{\underline{n}}^\odot \otimes 
    \left[(f_{\underline{n}})^\odot\right](V_{\underline{n}}^\otimes)
  \]
  \[
    \approx \sigma\{q_{\underline{n}}\}g_{\underline{n}}^\odot 
    f_{\underline{n}}^\odot \otimes V_{\underline{n}}^\otimes
  \]
  \[
    = \sigma\{q_{\underline{n}}\}(g_{\underline{n}}f_{\underline{n}})^\odot 
    \otimes V_{\underline{n}}^\otimes.
  \]
  This proves $\overline{\chi}$ is well defined, and so $\mathscr{Y}_*^+A$
  admits the structure of an $E_\infty$-algebra over the category of 
  simplicial $k$-modules.  
\end{proof}

\begin{theorem}\label{thm.Y-operadalgebra-complexes}
  The $\mathscr{D}_{\mathrm{sm}}$-algebra structure on $\mathscr{Y}_*^+A$ induces
  a $\mathscr{D}_{\mathrm{ch}}$-algebra structure on $\mathcal{N}\mathscr{Y}_*^+A$ 
  (as $k$-complex).
\end{theorem}
\begin{proof}
  Again since the normalization functor $\mathcal{N}$ is lax symmetric monoidal,
  the operad-algebra structure map of $\mathscr{Y}_*^+A$ induces
  an operad algebra structure map over $k$-$\mathbf{complexes}$ (by 
  Lemma~\ref{lem.lax-sym-mon-functor}):
  \[
    \widetilde{\chi} : \mathscr{D}_{\mathrm{ch}}(n) \otimes_{S_n}
    \left(\mathcal{N}\mathscr{Y}_*^+A\right)^{\otimes n}
    \to \mathcal{N}\mathscr{Y}_*^+A.
  \]
\end{proof}

\begin{cor}\label{cor.pontryagin}
  $HS_*(A)$ admits a Pontryagin product, giving it the structure of graded
  commutative and associative algebra.
\end{cor}
\begin{proof}
  The product is induced on the chain level by first choosing any $0$-chain, $c$,
  in $\mathscr{D}_{\mathrm{ch}}(2)$ corresponding to the generator
  of $1 \in H_0\left(\mathscr{D}_{\mathrm{ch}}(2)\right) = k$, and then taking
  the composite,
  \[
    k \otimes \mathcal{N}\mathscr{Y}_*^+A \otimes \mathcal{N}\mathscr{Y}_*^+A
    \longrightarrow \mathscr{D}_{\mathrm{ch}}(2) \otimes 
    \mathcal{N}\mathscr{Y}_*^+A \otimes \mathcal{N}\mathscr{Y}_*^+A
  \]
  \[
    \longrightarrow \mathscr{D}_{\mathrm{ch}}(2) \otimes_{k[S_2]}
    \mathcal{N}\mathscr{Y}_*^+A \otimes \mathcal{N}\mathscr{Y}_*^+A
    \stackrel{\widetilde{\chi}}{\longrightarrow} \mathcal{N}\mathscr{Y}_*^+A
  \]
  That is, for homology classes $[x]$ and $[y]$ in $HS_*(A)$, the product
  is defined by
  \begin{equation}\label{eq.pontryagin-product}
    [x] \cdot [y] := \left[\widetilde{\chi}(c \otimes x \otimes y)\right]
  \end{equation}
  Now, the choice of $c$ does not matter, since $\mathscr{D}_{\mathrm{ch}}(2)$
  is contractible.  Indeed, since each $\mathscr{D}_{\mathrm{ch}}(p)$ is
  contractible, Thm.~\ref{thm.Y-operadalgebra-complexes} shows that
  $\mathcal{N}\mathscr{Y}_*^+A$ is a {\it homotopy-everything} complex, analogous
  to the homotopy-everything spaces of Boardman and Vogt~\cite{BV}.
  Thus, the product~(\ref{eq.pontryagin-product}) is associative and
  commutative in the graded sense on the level of homology (see also
  May~\cite{M}, p.~3).
\end{proof}

\begin{rmk}
  The Pontryagin product of Cor.~\ref{cor.pontryagin} is directly
  related to the algebra structure on the complexes $Sym_*^{(p)}$ of
  Chapter~\ref{chap.spec_seq2}.  Indeed,
  if $A$ has augmentation ideal $I$ which is free and has countable rank over
  $k$, and $I^2 = 0$, then by Cor.~\ref{cor.square-zero},
  the spectral sequence would collapse at the $E_1$ stage, giving:
  \[
    HS_n(A) \;\cong\; \bigoplus_{p \geq 0} \bigoplus_{u \in X^{p+1}/\Sigma_{p+1}}
    H_{n}(EG_u \ltimes_{G_u} Sym_*^{(p)}; k),
  \]
  and the product structure of $Sym_*^{(p)}$ may be viewed as a restriction
  of the algebra structure of $HS_*(A)$ to the free orbits.
\end{rmk}

\section{Homology Operations}\label{sec.homology-operations}      %

Recall, for a commutative ring, $k$, and a cyclic group $\pi$ of order $p$, there is a 
periodic resolution of $k$ by free $k\pi$-modules (cf~\cite{M2},~\cite{B}):
\begin{definition}\label{def.W-complex}
  Let $\tau$ be a generator of $\pi$.  Let $N = 1 + \tau + \ldots + \tau^{p-1}$.  Define
  a $k\pi$-complex, $(W, d)$ by:
  
  $W_i$ is a the free $k\pi$-module
  on the generator $e_i$, for each $i \geq 0$, with differential:
  \[
    \left\{
    \begin{array}{lll}
      d(e_{2i+1}) &=& (\tau - 1)e_{2i} \\
      d(e_{2i}) &=& Ne_{2i-1}
    \end{array}
    \right.
  \]
  $W$ also has the structure of a $k\pi$-coalgebra, with augmentation $\epsilon$ and
  coproduct $\psi$:
  \[
    \epsilon(\tau^j e_0) = 1
  \]
  \[
    \left\{
    \begin{array}{lll}
      \psi(e_{2i+1}) &=& \ds{\sum_{j+k = i} e_{2j} \otimes e_{2k+1} \;+\;
                        \sum_{j+k = i} e_{2j+1} \otimes \tau e_{2k}} \\
      \psi(e_{2i}) &=& \ds{\sum_{j+k = i} e_{2j} \otimes e_{2k} \;+\;
                      \sum_{j+k = i-1} \left(\sum_{0 \leq r < s < p} \tau^r e_{2j+1} \otimes
                        \tau^s e_{2k+1}\right)}
    \end{array}
    \right.
  \]
\end{definition}

In what follows, we shall specialize $p$ prime and to $k = \Z/p\Z$ (as a ring).  
Let $\pi = C_p$ (as group), and denote by 
$W$ the standard resolution of $k$ by $k\pi$-modules, as in definition~\ref{def.W-complex}.
Recall, $\mathscr{D}_{\mathrm{ch}}(p)$ is a contractible $k$-complex on which
$S_p$ acts freely.  Embed $\pi \hookrightarrow S_p$ by $\tau \mapsto (1, p, p-1, \ldots, 2)$.
Clearly $\pi$ acts freely on $\mathscr{D}_{\mathrm{ch}}(p)$ as well.  Thus,
there exists a homotopy equivalence $\xi : W \to \mathscr{D}_{\mathrm{ch}}(p)$. 

Observe that the complex $\mathcal{N}\mathscr{Y}_*^+A$ computes $HS_*(A)$, since it
is defined as the quotient of $\mathscr{Y}_*^+A$ by degeneracies.  By
Thm.~\ref{thm.Y-operadalgebra-complexes},
$\mathcal{N}\mathscr{Y}_*^+A$, has the structure of $E_{\infty}$-algebra,
so by results of May, if \mbox{$x \in
H_*(\mathcal{N}\mathscr{Y}_*^+A) = HS_*(A)$}, then $e_i \otimes x^{\otimes p}$ is a 
well-defined
element of $H_*\left(W \otimes_{k\pi} (\mathcal{N}\mathscr{Y}_*^+A)^{\otimes p}\right)$, 
where $e_i$ is the distinguished generator of $W_i$.  We then use the homotopy equivalence
$\xi : W \to \mathscr{D}_{\mathrm{ch}}(p)$ and $\mathscr{D}_{\mathrm{ch}}$-algebra structure 
of $\mathcal{N}\mathscr{Y}_*^+A$ to produce the required map.
Define $\kappa$ as the composition:
\[
  \begin{diagram}
  \node{ H_*\left(W \otimes_{k\pi} (\mathcal{N}\mathscr{Y}_*^+A)^{\otimes p}\right) }
  \arrow{e,tb}{ H(\xi \otimes \mathrm{id}^p) }{ \cong }
  \node{ H_*\left(\mathscr{D}_{\mathrm{ch}}(p) \otimes_{k\pi}
    (\mathcal{N}\mathscr{Y}_*^+A)^{\otimes p}\right) }
  \arrow{e,t}{ H(\widetilde{\chi}) }
  \node{ H_*(\mathcal{N}\mathscr{Y}_*^+A)}
  \end{diagram}
\]

This gives a way of defining homology (Steenrod) operations
on $HS_*(A)$.  Following definition 2.2
of~\cite{M2}, first define the maps $D_i$.  For $x \in HS_q(A)$ and
$i \geq 0$, define
\[
  D_i(x) := \kappa(e_i \otimes x^{\otimes p}) \in HS_{pq + i}(A).
\]

\begin{definition}\label{def.homology-operations}
  If $p=2$, define:
  \[
    P_s : HS_q(A) \to HS_{q+s}(A)
  \]
  \[
    P_s(x) = \left\{
    \begin{array}{ll}
      0 \; &\textrm{if $s < q$}\\
      D_{s-q}(x) \; &\textrm{if $s \geq q$}
    \end{array}
    \right.
  \]
  
  If $p > 2$ ({\it i.e.}, an odd prime), let
  \[
    \nu(q) = (-1)^{s+\frac{q(q-1)(p-1)}{4}}\Big[\Big( \frac{p-1}{2} \Big)!\Big]^q,
  \]
  and define:
  \[
    P_s : HS_q(A) \to HS_{q+2s(p-1)}(A)
  \]
  \[
    P_s(x) = \left\{
    \begin{array}{ll}
      0 \; &\textrm{if $2s < q$}\\
      \nu(q) D_{(2s-q)(p-1)}(x) \; &\textrm{if $2s \geq q$}
    \end{array}
    \right.
  \]
  \[
    \beta P_s : HS_q(A) \to HS_{q+2s(p-1)-1}(A)
  \]
  \[
    \beta P_s(x) = \left\{
    \begin{array}{ll}
      0 \; &\textrm{if $2s \leq q$}\\
      \nu(q) D_{(2s-q)(p-1)-1}(x) \; &\textrm{if $2s > q$}
    \end{array}
    \right.
  \]
\end{definition}

Note, the definition of $\nu(q)$ given here differs from that given in~\cite{M2} by
the sign $(-1)^s$ in order that all constants be collected into the term $\nu(q)$.

\chapter{LOW-DEGREE SYMMETRIC HOMOLOGY}\label{chap.ldsymhom}

\section{Partial Resolution}\label{sec.partres}                     %

As before, $k$ is a commutative ground ring.
In this chapter, we find an explicit partial resolution of $\underline{k}$ by 
projective $\Delta S^\mathrm{op}$-modules, allowing the computation of $HS_0(A)$ and
$HS_1(A)$ for a unital associative $k$-algebra $A$.

\begin{theorem}\label{thm.partial_resolution}
  $HS_i(A)$ for $i=0,1$ is the homology of the following partial chain
  complex
  \[
    0\longleftarrow A \stackrel{\partial_1}{\longleftarrow} A\otimes A\otimes A
    \stackrel{\partial_2}{\longleftarrow}(A\otimes A\otimes A\otimes A)\oplus A,
  \]
  where 
  \[
    \partial_1 : a\otimes b\otimes c \mapsto abc - cba,
  \]
  \[
    \partial_2 : \left\{
                 \begin{array}{lll}
                   a\otimes b\otimes c\otimes d &\mapsto& ab\otimes c\otimes d + 
                   d\otimes ca\otimes b + bca\otimes 1\otimes d + d\otimes bc\otimes a,\\
                   a &\mapsto& 1\otimes a\otimes 1.
                 \end{array}
                 \right.
  \]
\end{theorem}

The proof will proceed in stages from the lemmas below.
\begin{lemma}\label{lem.0-stage}
  For each $n \geq 0$,
  \[
    0 \gets k \stackrel{\epsilon}{\gets} k\big[\mathrm{Mor}_{\Delta S}([n], [0])\big]
    \stackrel{\rho}{\gets} k\big[\mathrm{Mor}_{\Delta S}([n], [2])\big]
  \]
  is exact, where $\epsilon$ is defined by $\epsilon(\phi) = 1$ for any morphism
  $\phi : [n] \to [0]$, and $\rho$ is defined by $\rho(\psi) = 
  (x_0x_1x_2)\circ\psi - (x_2x_1x_0)\circ\psi$
  for any morphism $\psi : [n] \to [2]$.  Note, $x_0x_1x_2$ and $x_2x_1x_0$ are
  $\Delta S$ morphisms $[2] \to [0]$ written in tensor notation (see 
  section~\ref{sec.deltas})
\end{lemma}
\begin{proof}
  Clearly, $\epsilon$ is surjective.  Now, $\epsilon\rho = 0$, since $\rho(\psi)$
  consists of two morphisms with opposite signs.  
  Let $\phi_0 = x_0x_1 \ldots x_n : [n] \to [0]$. 
  The kernel of $\epsilon$ is spanned by elements $\phi - \phi_0$ for $\phi \in
  \mathrm{Mor}_{\Delta S}([n],[0])$.  So, it suffices to show that the 
  sub-module of
  $k\big[\mathrm{Mor}_{\Delta S}([n], [0])\big]$ generated by
  $(x_0x_1x_2)\psi - (x_2x_1x_0)\psi$ (for $\psi : [n] \to [2]$) contains all of the
  elements $\phi - \phi_0$.  In other words, it suffices to find a sequence
  \[
    \phi =: \phi_k,\; \phi_{k-1},\;  \ldots,\; \phi_2,\; \phi_1,\; \phi_0
  \]
  so that each $\phi_i$ is obtained from $\phi_{i+1}$ by reversing the order of
  3 blocks, $XYZ \to ZYX$.  Note, $X$, $Y$, or $Z$ can be empty.
  Let $\phi = x_{i_0}x_{i_1}\ldots x_{i_n}$.  If $\phi = \phi_0$, we may
  stop here.  Otherwise, we produce a sequence ending in $\phi_0$ by way of
  two types of rearrangements.  
  
  Type I:
  \[
     x_{i_0}x_{i_1}\ldots x_{i_n} \leadsto x_{i_n}x_{i_0}x_{i_1} \ldots x_{i_{n-1}}.
  \]
  
  Type II:
  \[
     x_{i_0}x_{i_1} \ldots x_{i_{k-1}}x_{i_k}x_{i_{k+1}}\ldots x_{i_n}
     \leadsto x_{i_{k+1}}\ldots x_{i_n}x_{i_k}x_{i_0}x_{i_1} \ldots x_{i_{k-1}}.
  \]
  
  In fact, it will be sufficient to use a more specialized version of the Type II
  rearrangement.
  
  Type II$'$:
  \[
     x_{i_0}x_{i_1} \ldots x_{i_{k-1}}x_{i_k}x_{k+1}\ldots x_n
     \leadsto x_{k+1}\ldots x_n x_{i_k}x_{i_0}x_{i_1} \ldots x_{i_{k-1}},
  \]  
  where $i_k \neq k$.
  
  Beginning with $\phi$, perform Type I rearrangements until the final variable
  is $x_n$.  For convenience of notation, let this new monomial be
  $x_{j_0}x_{j_1} \ldots x_{j_n}$.  Of course, $j_n = n$.  If $j_k = k$ for
  all $k = 0, 1, \ldots, n$, then we are done.  Otherwise,  
  there will be a number $k$ such that $j_k \neq k$ but $j_{k+1} = k + 1, \ldots,
  j_n = n$.  Perform a Type II$'$ rearrangement with $j_{k}$ as pivot, followed by
  enough Type I rearrangements to make the final variable $x_n$ again.  The net
  result of such a combination is that the ending block $x_{k+1}x_{k+2}\ldots x_n$
  remains fixed while the beginning block $x_{j_0}x_{j_1}\ldots x_{j_{k}}$ 
  becomes $x_{j_k}x_{j_0} \ldots x_{j_{k-1}}$.  It is clear that applying this
  combination repeatedly will finally obtain a monomial
  $x_{\ell_0}x_{\ell_1}\ldots x_{\ell_{k-1}} x_k x_{k+1} \ldots x_n$.  After a finite
  number of steps, we finally obtain $\phi_0$.
\end{proof}
Let $\mathscr{B}_n = \{ x_{i_0}x_{i_1}\ldots x_{i_{k-1}} \otimes x_{i_k} \otimes
x_{k+1}x_{k+2} \ldots x_{n} \;:\; k \geq 1, i_k \neq k \}$.  $k[\mathscr{B}_n]$ is
a free submodule of $k\big[\mathrm{Mor}_{\Delta S}([n], [2])\big]$ of size
$(n+1)! - 1$.  This count is obtained by observing that 
$\{ x_{i_0} \ldots x_{i_{k-1}} \otimes x_{i_k} \otimes
x_{k+1} \ldots x_{n} \;:\; k = c, i_k \neq k \}$ has exactly $c \cdot c! = (c+1)! - c!$
elements, then adding the telescoping sum from $c = 1$ to $n$.
\begin{cor}\label{cor.B_n}
  When restricted to $k[\mathscr{B}_n]$, the map $\rho$ of Lemma~\ref{lem.0-stage} is
  surjective onto the kernel of $\epsilon$.
\end{cor}
\begin{proof}
  In the proof of Lemma~\ref{lem.0-stage}, the Type I rearrangements correspond to
  the image of elements $x_{i_0} \ldots x_{i_{n-1}} \otimes x_{i_n} \otimes 1$.
  Note, we did not need $i_n = n$ in any such rearrangement.
  The Type II$'$ rearrangements correspond to the image of elements
  $x_{i_0} \ldots x_{i_{k-1}} \otimes x_{i_k} \otimes x_{k+1} \ldots x_{n}$, for
  $k \geq 1$ and $i_k \neq k$.
\end{proof}
\begin{lemma}\label{lem.rank}
  $\#\mathrm{Mor}_{\Delta S}([n], [m]) = (m+n+1)!/m!$, so
  $k\big[\mathrm{Mor}_{\Delta S}([n], [m])\big]$ is a free $k$-module of
  rank $(m+n+1)!/m!$.
\end{lemma}
\begin{proof}
  A morphism $\phi : [n] \to [m]$ of $\Delta S$ is nothing more than an assignment
  of $n+1$ objects into $m+1$ compartments, along with a total ordering of the
  original $n+1$ objects, hence:
  \[
    \#\mathrm{Mor}_{\Delta S}([n], [m]) = \binom{m+n+1}{m}(n+1)! = \frac{(m+n+1)!}{m!}.
  \]
\end{proof}
\begin{lemma}\label{lem.rho-iso}
  $\rho|_{k[\mathscr{B}_n]}$ is an isomorphism $k[\mathscr{B}_n] \cong \mathrm{ker}
  \,\epsilon$.
\end{lemma}
\begin{proof}
  Since the rank of $k\big[\mathrm{Mor}_{\Delta S}([n], [m])\big]$ is $(n+1)!$, the
  rank of the kernel of $\epsilon$ is $(n+1)! - 1$.  The isomorphism then follows
  from Corollary~\ref{cor.B_n}.
\end{proof}
\begin{lemma}\label{lem.4-term-relation}
  The relations of the form:
  \begin{equation}\label{eq.4-term}
    XY \otimes Z \otimes W + W \otimes ZX \otimes Y + YZX \otimes 1 \otimes W
    + W \otimes YZ \otimes X \approx 0
  \end{equation}
  \begin{equation}\label{eq.1-term}
    \qquad \mathrm{and} \qquad
    1 \otimes X \otimes 1 \approx 0
  \end{equation}
  collapse $k\big[\mathrm{Mor}_{\Delta S}([n], [2])\big]$ onto $k[\mathscr{B}_n]$.
\end{lemma}
\begin{proof}
  This proof proceeds in multiple steps.
  
  {\bf Step 1.}
  \begin{equation}\label{eq.step1}
    X \otimes 1 \otimes 1 \approx 1 \otimes X \otimes 1 \approx 1 \otimes 1 \otimes X
    \approx 0.
  \end{equation}
  $1 \otimes X \otimes 1 \approx 0$ proceeds directly from Eq.~\ref{eq.1-term}.  Letting
  $X = Y = W = 1$ in Eq.~\ref{eq.4-term} yields
  \[
    3(1 \otimes Z \otimes 1) + Z \otimes 1 \otimes 1 \approx 0 \; \Rightarrow\;
    Z \otimes 1 \otimes 1 \approx 0.
  \]
  Then, $X = Z = W = 1$ in Eq.~\ref{eq.4-term} produces
  \[
    2(Y \otimes 1 \otimes 1) + 1 \otimes 1 \otimes Y + 1 \otimes Y \otimes 1
    \approx 0 \; \Rightarrow\; 1 \otimes 1 \otimes Y \approx 0.
  \]
  
  {\bf Step 2.}
  \begin{equation}\label{eq.step2}
    1 \otimes X \otimes Y +  1 \otimes Y \otimes X \approx 0.
  \end{equation}
  Let $Z = W = 1$ in Eq.~\ref{eq.4-term}.  Then
  \[
    XY \otimes 1 \otimes 1 + 1 \otimes X \otimes Y + YX \otimes 1 \otimes 1
    + 1 \otimes Y \otimes X \approx 0
  \]
  \[
    \Rightarrow\; 1 \otimes X \otimes Y + 1 \otimes Y \otimes X \approx 0,
    \qquad \textrm{by step 1.}
  \]
  
  {\bf Step 3. [Degeneracy Relations]}
  \begin{equation}\label{eq.step3}
    X \otimes Y \otimes 1 \approx X \otimes 1 \otimes Y \approx 1 \otimes X \otimes Y.
  \end{equation}
  Let $X = W = 1$ in Eq.~\ref{eq.4-term}.
  \[
    Y \otimes Z \otimes 1 + 1 \otimes Z \otimes Y + YZ \otimes 1 \otimes 1
    + 1 \otimes YZ \otimes 1 \approx 0
  \]
  \[
    \Rightarrow\; Y \otimes Z \otimes 1 + 1 \otimes Z \otimes Y \approx 0,
    \qquad \textrm{by step 1.}
  \]
  \begin{equation}\label{eq.step3a}
    \Rightarrow\; Y \otimes Z \otimes 1 - 1 \otimes Y \otimes Z \approx 0,
    \qquad \textrm{by step 2.}
  \end{equation}
  Next, let $X = Y = 1$ in Eq.~\ref{eq.4-term}.
  \[
    1 \otimes Z \otimes W + W \otimes Z \otimes 1 + Z \otimes 1 \otimes W
    + W \otimes Z \otimes 1 \approx 0
  \]
  \[
    \Rightarrow\; 1 \otimes Z \otimes W + 2(1 \otimes W \otimes Z)
    + Z \otimes 1 \otimes W \approx 0,
    \qquad \textrm{by Eq.~\ref{eq.step3a},}
  \]  
  \[
    \Rightarrow\; 1 \otimes Z \otimes W - 2(1 \otimes Z \otimes W)
    + Z \otimes 1 \otimes W \approx 0,
    \qquad \textrm{by step 2,}
  \]
  \[
    \Rightarrow\; Z \otimes 1 \otimes W - 1 \otimes Z \otimes W \approx 0.
  \]
  
  {\bf Step 4. [Sign Relation]}
  \begin{equation}\label{eq.step4}
    X \otimes Y \otimes Z + Z \otimes Y \otimes X \approx 0
  \end{equation}
  Let $Y = 1$ in Eq.~\ref{eq.4-term}.
  \[
    X \otimes Z \otimes W  + W \otimes ZX \otimes 1 + ZX \otimes 1 \otimes W
    + W \otimes Z \otimes X \approx 0,
  \]
  \[
    \Rightarrow\; X \otimes Z \otimes W + 1 \otimes W \otimes ZX + 1 \otimes ZX \otimes W
    + W \otimes Z \otimes X \approx 0,
    \qquad \textrm{by step 3,}
  \]
  \[
    \Rightarrow\; X \otimes Z \otimes W + W \otimes Z \otimes X \approx 0,
    \qquad \textrm{by step 2.}
  \]
  
  {\bf Step 5. [Hochschild Relation]}
  \begin{equation}\label{eq.step5}
    XY \otimes Z \otimes 1 - X \otimes YZ \otimes 1 + ZX \otimes Y \otimes 1
    \approx 0.
  \end{equation}
  Let $W = 1$ in Eq.~\ref{eq.4-term}.
  \[
    XY \otimes Z \otimes 1 + 1 \otimes ZX \otimes Y + YZX \otimes 1 \otimes 1
    + 1 \otimes YX \otimes X \approx 0,
  \]
  \[
    \Rightarrow\; XY \otimes Z \otimes 1 + ZX \otimes Y \otimes 1 + 0 
    - X \otimes YX \otimes 1 \approx 0,
    \qquad \textrm{by steps 1, 3 and 4}.
  \]
  
  {\bf Step 6. [Cyclic Relation]}
  \begin{equation}\label{eq.step6}
    \sum_{j = 0}^n \tau_n^j\big(x_{i_0}x_{i_1} \ldots
    x_{i_{n-1}} \otimes x_{i_n} \otimes 1\big) \approx 0,
  \end{equation}
  where $\tau_n \in \Sigma_{n+1}$ is the $(n+1)$-cycle $(0, n, n-1, \ldots, 2, 1)$,
  which acts by permuting the indices.
  For $n = 0$, there are no such relations (indeed, no relations at all).  For $n=1$,
  the cyclic relation takes the form $x_0 \otimes x_1 \otimes 1 + x_1 \otimes x_0
  \otimes 1 \approx 0$, which follows from degeneracy and sign relations.
  
  Assume now that $n \geq 2$.  For each $k = 1, 2, \ldots, n-1$, define:
  \[
    \left\{\begin{array}{ll}
      A_k := & x_{i_0}x_{i_1}\ldots x_{i_{k-1}},\\
      B_k := & x_{i_k},\\
      C_k := & x_{i_{k+1}} \ldots x_{i_n}.
    \end{array}\right.
  \]
  By the Hochschild relation,
  \[
    0 \approx \sum_{k=1}^{n-1} (A_kB_k \otimes C_k \otimes 1 -
    A_k \otimes B_kC_k \otimes 1 + C_kA_k \otimes B_k \otimes 1).
  \]
  But for $k \leq n-2$,
  \[
    A_kB_k \otimes C_k \otimes 1 = A_{k+1} \otimes B_{k+1}C_{k+1} \otimes 1
    = x_{i_0} \ldots x_{i_k} \otimes x_{i_{k+1}}\ldots x_{i_n} \otimes 1.
  \]
  Thus, after some cancellation:
  \[
    0 \approx - A_1 \otimes B_1 C_1 \otimes 1 + A_{n-1}B_{n-1} \otimes C_{n-1}
    \otimes 1 + \sum_{k=1}^{n-1} C_kA_k \otimes B_k \otimes 1
  \]
  \[
    = - (x_{i_0} \otimes x_{i_1} \ldots x_{i_n} \otimes 1)
    + (x_{i_0} \ldots x_{i_{n-1}} \otimes x_{i_n} \otimes 1)
    + \sum_{k=1}^{n-1} x_{i_{k+1}} \ldots x_{i_n}x_{i_0} \ldots x_{i_{k-1}} 
    \otimes x_{i_k} \otimes 1
  \]
  \[
    = (x_{i_0} \ldots x_{i_{n-1}} \otimes x_{i_n} \otimes 1)
    + (x_{i_1} \ldots x_{i_n} \otimes x_{i_0} \otimes 1)
    + \sum_{k=1}^{n-1} x_{i_{k+1}} \ldots x_{i_n}x_{i_0} \ldots x_{i_{k-1}} 
    \otimes x_{i_k} \otimes 1,
  \]
  by sign and degeneracy relations.  This last expression is precisely the
  relation Eq.~\ref{eq.step6}.
  
  {\bf Step 7.}
  
  Every element of the form $X \otimes Y \otimes 1$ is equivalent
  to a linear combination of elements of $\mathscr{B}_n$.
  
  To prove this, we shall induct on the size of $Y$.  Suppose $Y$ consists of
  a single variable.  That is, $X \otimes Y \otimes 1 =
  x_{i_0} \ldots x_{i_{n-1}} \otimes x_{i_n} \otimes 1$.  Now, if $i_n \neq n$,
  we are done.  Otherwise, we use the cyclic relation to write
  \[
    x_{i_0} \ldots x_{i_{n-1}} \otimes x_{i_n} \otimes 1 \approx
    -\sum_{j = 1}^n \tau_n^j\big(x_{i_0} \ldots
    x_{i_{n-1}} \otimes x_{i_n} \otimes 1\big).
  \]
  Now suppose $k \geq 1$ and any element $Z \otimes W \otimes 1$ with $|W| = k$
  is equivalent to an element of $k[\mathscr{B}_n]$.  Consider $X \otimes Y \otimes 1
  = x_{i_0} \ldots x_{i_{n-k-1}} \otimes x_{i_{n-k}} \ldots x_{i_n} \otimes 1$.
  Let
  \[
    \left\{\begin{array}{ll}
      A_k := & x_{i_0}x_{i_1}\ldots x_{i_{n-k-1}},\\
      B_k := & x_{i_{n-k}} \ldots x_{i_{n-1}},\\
      C_k := & x_{i_n}.
    \end{array}\right.
  \]
  Then, by the Hochschild relation,
  \[
    X \otimes Y \otimes 1 = A_k \otimes B_kC_k \otimes 1
    \approx A_kB_k \otimes C_k \otimes 1 + C_kA_k \otimes B_k \otimes 1.
  \]
  But since $|C_k| = 1$ and $|B_k| = k$, this last expression is equivalent to an element
  of $k[\mathscr{B}_n]$.
  
  {\bf Step 8. [Modified Cyclic Relation]}
  \begin{equation}\label{eq.step8}
    \sum_{j = 0}^k \tau_k^j\big(x_{i_0}x_{i_1} \ldots
    x_{i_{k-1}} \otimes x_{i_k} \otimes x_{i_{k+1}}\ldots x_{i_n}\big) \approx 0
    \pmod{k\big[\{A \otimes B \otimes 1\}\big]}.
  \end{equation}
  Note, the $(k+1)$-cycle $\tau_k$ permutes the indices $i_0, i_1, \ldots, i_k$, and
  fixes the rest.
  
  First, we show that $X \otimes Y \otimes W + Y \otimes X \otimes W \approx 0
  \pmod{k\big[\{A \otimes B \otimes 1\}\big]}$.  Indeed, if we let $Z = 1$ in 
  Eq.~\ref{eq.4-term}, then
  \[
    XY \otimes 1 \otimes W + W \otimes X \otimes Y + YX \otimes 1 \otimes W
    + W \otimes Y \otimes X \approx 0,
  \]
  \[
    \Rightarrow\; -(X \otimes Y \otimes W + Y \otimes X \otimes W)
    + XY \otimes 1 \otimes W + YX \otimes 1 \otimes W \approx 0,
    \qquad \textrm{by step 4},
  \]
  \begin{equation}\label{eq.step8a}
    \Rightarrow\; X \otimes Y \otimes W + Y \otimes X \otimes W
    \approx XY \otimes W \otimes 1 + YX \otimes W \otimes 1,
    \qquad \textrm{by step 3}.
  \end{equation}
  Now, we have $X \otimes Y \otimes W + Y \otimes X \otimes W
  \approx 0 \pmod{k\big[\{A \otimes B \otimes 1\}\big]}$.  Note that
  this last expression can be used to prove the modified cyclic relation for $k=1$.
  
  Next, rewrite
  Eq.~\ref{eq.4-term}:
  \[
    XY \otimes Z \otimes W + W \otimes ZX \otimes Y + YZX \otimes 1 \otimes W
    + W \otimes YZ \otimes X \approx 0,
  \]
  \[
    \Rightarrow\; XY \otimes Z \otimes W - Y \otimes ZX \otimes W
    + YZX \otimes W \otimes 1 - X \otimes YZ \otimes W \approx 0,
    \qquad \textrm{by steps 3 and 4},
  \]
  \[
    \Rightarrow\; XY \otimes Z \otimes W + ZX \otimes Y \otimes W
    + YZX \otimes W \otimes 1 - X \otimes YZ \otimes W  \approx 0\pmod{k\big[
    \{A \otimes B \otimes 1\}\big]},
  \]
  using the relation $X \otimes Y \otimes W + Y \otimes X \otimes W
  \approx 0 \pmod{k\big[\{A \otimes B \otimes 1\}\big]}$ proven above.
  \begin{equation}\label{eq.step8b}
    \Rightarrow\; XY \otimes Z \otimes W - X \otimes YZ \otimes W
    + ZX \otimes Y \otimes W \approx 0 \pmod{k\big[\{A \otimes B \otimes 1\}\big]}.
  \end{equation}
  Eq.~\ref{eq.step8b} is a modified Hochschild relation, and we can use it in the same
  way we used the Hochschild relation in step 6.  Assume $k \geq 2$, and define
  for $j = 1, 2, \ldots k-1$:
  \[
    \left\{\begin{array}{ll}
      A_j := & x_{i_0}x_{i_1}\ldots x_{i_{j-1}},\\
      B_j := & x_{i_j},\\
      C_j := & x_{i_{j+1}} \ldots x_{i_k}.
    \end{array}\right.
  \]
  Using the modified Hochschild relation together with the observation that
  for $j \leq k-2$,
  \[
    A_jB_j \otimes C_j \otimes W = A_{j+1} \otimes B_{j+1}C_{j+1} \otimes W,
  \]
  we finally arrive at the sum:
  \[
    0 \approx - A_1 \otimes B_1 C_1 \otimes W + A_{k-1}B_{k-1} \otimes C_{k-1}
    \otimes W + \sum_{j=1}^{k-1} C_jA_j \otimes B_j \otimes W
    \pmod{k\big[\{A \otimes B \otimes 1\}\big]}
  \]
  \[
    \equiv - (x_{i_0} \otimes x_{i_1} \ldots x_{i_n} \otimes W)
    + (x_{i_0} \ldots x_{i_{n-1}} \otimes x_{i_n} \otimes W)
    + \sum_{k=1}^{n-1} x_{i_{k+1}} \ldots x_{i_n}x_{i_0} \ldots x_{i_{k-1}} 
    \otimes x_{i_k} \otimes W
  \]
  \[
    \equiv (x_{i_1} \ldots x_{i_n} \otimes x_{i_0} \otimes W)
    + (x_{i_0} \ldots x_{i_{n-1}} \otimes x_{i_n} \otimes W)
    + \sum_{k=1}^{n-1} x_{i_{k+1}} \ldots x_{i_n}x_{i_0} \ldots x_{i_{k-1}} 
    \otimes x_{i_k} \otimes W.
  \]
  
  {\bf Step 9.}
  
  Every element of the form $X \otimes Y \otimes x_n$ is equivalent
  to an element of $k[\mathscr{B}_n]$.
  
  We shall use the modified cyclic relation and modified Hochschild relation in 
  a similar way as cyclic and Hochschild relations were used in step 7.  Again we
  induct on the size of $Y$.  If $|Y| = 1$, then
  \[
    X \otimes Y \otimes x_n = x_{i_0} \ldots x_{i_{n-2}} \otimes x_{i_{n-1}} \otimes
    x_{n}.
  \]
  If $i_{n-1} \neq n-1$, then we are done.  Otherwise, use the modified cyclic
  relation to re-express $X \otimes Y \otimes x_n$ as a sum of elements of
  $k[\mathscr{B}_n]$, modulo $k\big[\{A \otimes B \otimes 1\}\big]$.  Of course, by step
  7, all elements $A \otimes B \otimes 1$ are also in $k[\mathscr{B}_n]$.
  
  Next, suppose $k \geq 1$ and any element $Z \otimes W \otimes x_n$ with $|W| = k$
  is equivalent to an element of $k[\mathscr{B}_n]$.  Consider $X \otimes Y \otimes x_n
  = x_{i_0} \ldots x_{i_{n-k-2}} \otimes x_{i_{n-k-1}} \ldots x_{i_{n-1}} \otimes x_n$.
  Using the modified Hochschild relation with
  \[
    \left\{\begin{array}{ll}
      A_k := & x_{i_0}x_{i_1}\ldots x_{i_{n-k-2}},\\
      B_k := & x_{i_{n-k-1}} \ldots x_{i_{n-2}},\\
      C_k := & x_{i_{n-1}}.
    \end{array}\right.
  \]
  we obtain:
  \[
    X \otimes Y \otimes x_n = A_k \otimes B_kC_k \otimes x_n
    \approx A_kB_k \otimes C_k \otimes x_n + C_kA_k \otimes B_k \otimes x_n
    \pmod{k\big[\{A \otimes B \otimes 1\}\big]}.
  \]
  But since $|C_k| = 1$ and $|B_k| = k$, this last expression is equivalent to an element
  of $k[\mathscr{B}_n]$.

  {\bf Step 10.}

  Every element of $k\big[\mathrm{Mor}_{\Delta S}([n], [2])\big]$ is equivalent
  to a linear combination of elements from the set
  \begin{equation}\label{eq.step10}
    \mathscr{C}_n := \{X \otimes x_{i_n} \otimes 1 \;|\; i_n \neq n\} \cup 
    \{X \otimes x_{i_{n-1}} \otimes x_n \;|\; i_{n-1} \neq n-1\} \cup
    \{X \otimes Y \otimes Zx_n \;|\; |Z| \geq 1\}
  \end{equation}
  Note, the $k$-module generated by $\mathscr{C}_n$ contains $k[\mathscr{B}_n]$.
  
  Let $X \otimes Y \otimes Z$ be an arbitrary element of 
  $k\big[\mathrm{Mor}_{\Delta S}([n], [2])\big]$.  If $|X| = 0$, $|Y|=0$, or
  $|Z| = 0$, then the degeneracy relations imply that this element is
  equivalent to an element of the form $X' \otimes Y' \otimes 1$.  Step 7 implies 
  this element is equivalent to one in $k[\mathscr{B}_n]$, hence also
  in $k[\mathscr{C}_n]$.
  
  Suppose now that $|X|, |Y|, |Z| \geq 1$.  
  If $x_n$ occurs in $X$, use the relation $X \otimes Y \otimes W \approx 
  -Y \otimes X \otimes W \pmod{k\big[\{A \otimes B \otimes 1\}\big]}$
  to ensure that $x_n$ occurs in the middle factor.  If $x_n$ occurs
  in the $Z$, use the sign relation and the above relation to put $x_n$ into
  the middle factor.  In any case, it suffices to assume our element has
  the form:
  \[
    X \otimes Ux_nV \otimes Z.
  \]
  By the modified Hochschild relation,
  \[
    X \otimes Ux_nV \otimes Z
    \approx XUx_n \otimes V \otimes Z + VX \otimes Ux_n \otimes Z,
    \pmod{k\big[\{A \otimes B \otimes 1\}\big]},
  \]
  \[
    \approx -Z \otimes V \otimes XUx_n + Z \otimes VX \otimes Ux_n.
  \]
  The first term is certainly in $k[\mathscr{C}_n]$, since $|X| \geq 1$.
  If $|U| > 0$, the second term also lies in $k[\mathscr{C}_n]$.  If, 
  however, $|U| = 0$, then use step 9 to re-express $Z \otimes VX \otimes x_n$
  as an element of $k[\mathscr{B}_n]$.
  
  Observe that Step 10 proves Lemma~\ref{lem.4-term-relation} for $n = 0, 1, 2$,
  since in these cases, any elements that fall within the set
  $\{X \otimes Y \otimes Zx_n \;|\; |Z| \geq 1\}$ must have either $|X| = 0$
  or $|Y| = 0$, hence are equivalent via the degeneracy relation to elements of
  the form $A \otimes B \otimes 1$.
  In what follows, assume $n \geq 3$.
  
  {\bf Step 11.}
  \begin{equation}\label{eq.step11}
    W \otimes Z \otimes Xx_n \approx W \otimes x_nZ \otimes X
    \pmod{k\big[\{A \otimes B \otimes 1\}\cup\{A \otimes B \otimes x_n\}\big]}.
  \end{equation}
  
  Let $Y = x_n$ in Eq.~\ref{eq.4-term}.
  \[
    Xx_n \otimes Z \otimes W + W \otimes ZX \otimes x_n
    + x_nZX \otimes 1 \otimes W + W \otimes x_nZ \otimes X \approx 0,
  \]
  \[
    \Rightarrow\; W \otimes Z \otimes Xx_n \approx
    W \otimes ZX \otimes x_n + x_nZX \otimes W \otimes 1 + W \otimes x_nZ \otimes X,
  \]
  \[
    \approx W \otimes x_nZ \otimes X \pmod{k\big[\{A \otimes B \otimes 1\}\cup\{A \otimes B 
    \otimes x_n\}\big]}.
  \]
  
  {\bf Step 12.}
  
  Every element of $k\big[\mathrm{Mor}_{\Delta S}([n], [2])\big]$ is equivalent
  to a linear combination of elements from the set
  \begin{equation}\label{eq.step12}
    \mathscr{D}_n := \{X \otimes x_{i_{n-2}} \otimes x_{n-1}x_{n} \;|\; i_{n-2} \neq n-2\} 
    \cup 
    \{X \otimes Y \otimes Zx_{n-1}x_n \;|\; |Z| \geq 1\},
  \end{equation}
  modulo $k\big[\{A \otimes B \otimes 1\}\cup\{A \otimes B \otimes x_n\}\big]$.
  
  Let $X \otimes Y \otimes Z$ be an arbitrary element of 
  $k\big[\mathrm{Mor}_{\Delta S}([n], [2])\big]$.  Locate $x_{n-1}$ and use the
  techniques of Step 10 to re-express $X \otimes Y \otimes Z$ as a linear
  combination of terms of the form:
  \[
    W_j \otimes Z_j \otimes X_jx_{n-1},
  \]
  modulo $k\big[\{A \otimes B \otimes 1\}]$.
  Now, for each $j$, we will want to re-express
  $W_j \otimes Z_j \otimes X_j$ as linear combinations of vectors in which $x_n$ occurs
  only in the second factor.  If $x_n$ occurs in $W_j$, then we just use the
  modified cyclic relation:
  \[
    W_j \otimes Z_j \otimes X_jx_{n-1} \approx
    -Z_j \otimes W_j \otimes X_jx_{n-1}
    \pmod{k\big[\{A \otimes B \otimes 1\}]}.
  \]
  If $x_n$ occurs in $X_j$, then first use  
  Eq.~\ref{eq.4-term} with $Y = x_{n-1}$:
  \[
    Xx_{n-1} \otimes Z \otimes W + W \otimes ZX \otimes x_{n-1}
    + x_{n-1}ZX \otimes 1 \otimes W + W \otimes x_{n-1}Z \otimes X \approx 0,
  \]
  \[
    \Rightarrow\; W \otimes Z \otimes Xx_{n-1} \approx W \otimes ZX \otimes x_{n-1}
    + x_{n-1}ZX \otimes W \otimes 1 + W \otimes x_{n-1}Z \otimes X,
  \]
  \[
    \approx W \otimes ZX \otimes x_{n-1}
    + W \otimes x_{n-1}Z \otimes X \pmod{k\big[\{A \otimes B \otimes 1\}]},
  \]
  \[
    \approx W \otimes ZX \otimes x_{n-1}
    + Wx_{n-1} \otimes Z \otimes X + ZW \otimes x_{n-1} \otimes X
    \pmod{k\big[\{A \otimes B \otimes 1\}]},
  \]
  \[
    \approx W \otimes ZX \otimes x_{n-1}
    + Z \otimes X \otimes Wx_{n-1} - ZW \otimes X \otimes x_{n-1},
    \pmod{k\big[\{A \otimes B \otimes 1\}]}.
  \]
  Thus, we can express our original element $X \otimes Y \otimes Z$ as a linear
  combination of elements of the form:
  \[
    X' \otimes U'x_nV' \otimes Z'x_{n-1},
  \]
  modulo $k\big[\{A \otimes B \otimes 1\}]$.  Use the modified Hochschild relation
  to obtain
  \[
    X' \otimes U'x_nV' \otimes Z'x_{n-1} \approx
    X'U' \otimes x_nV' \otimes Z'x_{n-1} + x_nV'X' \otimes U' \otimes Z'x_{n-1}
    \pmod{k\big[\{A \otimes B \otimes 1\}]}.
  \]
  \[
    \approx
    X'U' \otimes x_nV' \otimes Z'x_{n-1} - U' \otimes x_nV'X' \otimes Z'x_{n-1}
    \pmod{k\big[\{A \otimes B \otimes 1\}]}.
  \]
  \[
    \approx
    X'U' \otimes V' \otimes Z'x_{n-1}x_n - U' \otimes V'X' \otimes Z'x_{n-1}x_n
    \pmod{k\big[\{A \otimes B \otimes 1\}\cup\{A \otimes B \otimes x_n\}\big]},
  \] 
  by step 11.  If $|Z'| \geq 1$, then we are done, otherwise, we have elements
  of the form $X'' \otimes Y'' \otimes x_{n-1}x_n$.  Use an induction 
  argument analogous to that in step 9 to re-express this type of element as
  a linear combination of elements of the form:
  \[
    U \otimes x_{i_{n-2}} \otimes x_{n-1}x_n, \quad i_{n-2} \neq n-2,
    \pmod{k\big[\{A \otimes B \otimes 1\}\big]}.
  \]
    
  {\bf Step 13.}
  
  Every element of $k\big[\mathrm{Mor}_{\Delta S}([n], [2])\big]$ is equivalent
  to an element of $k[\mathscr{B}_n]$.
  
  We shall use an iterative re-writing procedure.  First of all, define
  \[
    \mathscr{B}_n^{j} := \{ A \otimes x_{i_{n-j}} \otimes x_{n-j+1} \ldots
    x_n \;|\; i_{n-j} \neq n-j\},
  \]
  \[
    \mathscr{C}_n^{j} := \{ A \otimes B \otimes Cx_{n-j+1} \ldots
    x_n \;|\; |C| \geq 1\}.
  \]
  Now clearly, $\mathscr{B}_n  = \bigcup_{j=0}^{n-1} \mathscr{B}_n^{j}$.
  
  By steps 10 and 12, we can reduce any arbitrary element $X \otimes Y \otimes Z$
  to linear combinations of elements in $\mathscr{B}_n^0 \cup \mathscr{B}_n^1 \cup
  \mathscr{B}_n^2 \cup \mathscr{C}_n^2$.  Suppose now that we have reduced elements
  to linear combinations from $\mathscr{B}_n^0 \cup \mathscr{B}_n^1 \cup \ldots
  \cup \mathscr{B}_n^j \cup \mathscr{C}_n^j$, for some $j \geq 2$.  I claim any
  element of $\mathscr{C}_n^j$ can be re-expressed as a linear combination
  from $\mathscr{B}_n^0 \cup \mathscr{B}_n^1 \cup \ldots
  \cup \mathscr{B}_n^{j+1} \cup \mathscr{C}_n^{j+1}$.
  
  Let $X \otimes Y \otimes Zx_{n-j+1} \ldots x_n$, with $|Z| \geq 1$.  Let
  $w := x_{n-j+1} \ldots x_n$.  We may now think of $X \otimes Y \otimes Zw$
  as consisting of the variables $x_0, x_1, \ldots, x_{n-j}, w$, hence, by
  step 12, we may re-express this element as a linear combination of elements
  from the set
  \[
    \{ X \otimes x_{i_{n-j-1}} \otimes x_{n-j}w \;|\; i_{n-j-1} \neq n-j-1\}
    \cup \{ X \otimes Y \otimes Zx_{n-j}w \;|\; |Z| \geq 1 \},
  \]
  modulo $k\big[\{A \otimes B \otimes 1\}\cup\{A \otimes B \otimes w\}\big]$.
  Of course, this implies the element may written as a linear combination
  of elements from $\mathscr{B}_n^{j+1} \cup \mathscr{C}_n^{j+1}$, modulo
  $k\big[\{A \otimes B \otimes 1\}\cup\{A \otimes B \otimes x_{n-j+1}\ldots
  x_n\}\big]$.  Since $\{A \otimes B \otimes x_{n-j+1}x_{n-j+2}\ldots
  x_n\} \subset \mathscr{C}_n^{j-1}$, the inductive hypothesis ensures that
  $\{A \otimes B \otimes x_{n-j+1}\ldots
  x_n\} \subset \mathscr{B}_n^0 \cup \ldots \cup \mathscr{B}_n^{j}$.  This
  completes the inductive step.
  
  After a finite number of iterations, then, we can re-express any element
  $X \otimes Y \otimes Z$ as a linear combination from the set
  $\mathscr{B}_n^0 \cup \ldots \mathscr{B}_n^{n-1} \cup
  \mathscr{C}_n^{n-1}
  = \mathscr{B}_n \cup \mathscr{C}_n^{n-1}$.  But $\mathscr{C}_n^{n-1} = 
  \{ A \otimes B \otimes Cx_{2} \ldots
  x_n \;|\; |C| \geq 1\}$.  Any element from this set has either $|A| = 0$ or
  $|B| = 0$, therefore is equivalent to an element from $k[\mathscr{B}_n]$
  already.  
\end{proof}
\begin{cor}\label{cor.k-contains-one-half}
  If $\frac{1}{2} \in k$, then the four-term relation
  \begin{equation}\label{eq.4-term_cor}
    XY \otimes Z \otimes W + W \otimes ZX \otimes Y + YZX \otimes 1 \otimes W
    + W \otimes YZ \otimes X \approx 0
  \end{equation}
  is sufficient to collapse $k\big[\mathrm{Mor}_{\Delta S}([n], [2])\big]$ onto 
  $k[\mathscr{B}_n]$.
\end{cor}
\begin{proof}
  We only need to modify step 1 of the previous proof:
  
  {\bf Step 1$'$.}
  \[
    X \otimes 1 \otimes 1 \approx 1 \otimes X \otimes 1 \approx 
    1 \otimes 1 \otimes X \approx 0.
  \]
  Letting three variables at a time equal $1$ in Eq.~\ref{eq.4-term_cor},
  \begin{equation}\label{eq.step1_1}
    2(1 \otimes 1 \otimes W) + 2(W \otimes 1 \otimes 1) \approx 0,
    \quad \textrm{when $X = Y = Z = 1$.}
  \end{equation}
  \begin{equation}\label{eq.step1_2}
    3(1 \otimes Z \otimes 1) + Z \otimes 1 \otimes 1 \approx 0,
    \quad \textrm{when $X = Y = W = 1$.}
  \end{equation}
  \begin{equation}\label{eq.step1_3}
    2(Y \otimes 1 \otimes 1) + 1 \otimes Y \otimes 1 + 1 \otimes 1 \otimes Y
    \approx 0, \quad \textrm{when $X = Z = W = 1$.}
  \end{equation}
  Now, replace $W$ with $X$ in Eq.~\ref{eq.step1_1}, $Z$ with $X$ in
  Eq.~\ref{eq.step1_2}, and $Y$ with $X$ in Eq.~\ref{eq.step1_3}.  Then
  \[
    -3\big[2(1 \otimes 1 \otimes X) + 2(X \otimes 1 \otimes 1)\big]
    -2\big[3(1 \otimes X \otimes 1) + X \otimes 1 \otimes 1\big]
    +6\big[2(X \otimes 1 \otimes 1) + 1 \otimes X \otimes 1 + 1 \otimes 1 \otimes X\big]
  \]
  \[
    = 4(X \otimes 1 \otimes 1).
  \]
  \[
    \Rightarrow\; X \otimes 1 \otimes 1 \approx 0,
  \]
  as long as $2$ is invertible in $k$.  Next, by Eq.~\ref{eq.step1_1},
  $2(1 \otimes 1 \otimes X) \approx 0 \Rightarrow 1 \otimes 1 \otimes X
  \approx 0$.  Finally, Eq.~\ref{eq.step1_3} gives $1 \otimes X \otimes 1 \approx 0$.
\end{proof}
Lemma~\ref{lem.4-term-relation} together with Lemma~\ref{lem.rho-iso} and 
Lemma~\ref{lem.0-stage} show the following sequence is exact:
\[
   0 \gets k \stackrel{\epsilon}{\gets} k\big[\mathrm{Mor}_{\Delta S}([n], [0])\big]
     \stackrel{\rho}{\gets} k\big[\mathrm{Mor}_{\Delta S}([n], [2])\big]\qquad\qquad\qquad
\]
\begin{equation}\label{eq.part_res}
   \hspace{2.5in}  \stackrel{(\alpha, \beta)}{\longleftarrow}
     k\big[\mathrm{Mor}_{\Delta S}([n], [3])\big] \oplus
     k\big[\mathrm{Mor}_{\Delta S}([n], [0])\big],
\end{equation}
where $\alpha : k\big[\mathrm{Mor}_{\Delta S}([n], [3])\big]
\to k\big[\mathrm{Mor}_{\Delta S}([n], [2])\big]$ is induced by
\[
  x_0x_1 \otimes x_2 \otimes x_3 + x_3 \otimes x_2x_0 \otimes x_1
  + x_1x_2x_0 \otimes 1 \otimes x_3 + x_3 \otimes x_1x_2 \otimes x_0,
\]
and $\beta : k\big[\mathrm{Mor}_{\Delta S}([n], [0])\big]
\to k\big[\mathrm{Mor}_{\Delta S}([n], [2])\big]$ is induced by $1 \otimes x_0 \otimes 1$.
This holds for all $n \geq 0$, so the following is a partial resolution of $k$ by
projective \mbox{$\Delta S^\mathrm{op}$-modules}:
\[
  0 \gets k \stackrel{\epsilon}{\gets} k\big[\mathrm{Mor}_{\Delta S}(-, [0])\big]
     \stackrel{\rho^*}{\gets} k\big[\mathrm{Mor}_{\Delta S}(-, [2])\big]
     \stackrel{(\alpha^*, \beta^*)}{\longleftarrow}
     k\big[\mathrm{Mor}_{\Delta S}(-, [3])\big] \oplus
     k\big[\mathrm{Mor}_{\Delta S}(-, [0])\big]
\]
Hence, we may compute
$HS_0(A)$ and $HS_1(A)$ as the homology groups of the following complex:
\[
  0 \gets 
  k\big[\mathrm{Mor}_{\Delta S}(-, [0])\big]
  \otimes_{\Delta S} B_*^{sym}A \stackrel{\rho\otimes\mathrm{id}}{\longleftarrow}
  k\big[\mathrm{Mor}_{\Delta S}(-, [2])\big]
  \otimes_{\Delta S} B_*^{sym}A \stackrel{(\alpha, \beta)\otimes\mathrm{id}}
  {\longleftarrow}
\]
\[
  \Big(k\big[\mathrm{Mor}_{\Delta S}(-, [3])\big] \oplus
  k\big[\mathrm{Mor}_{\Delta S}(-, [0])\big]\Big)
  \otimes_{\Delta S} B_*^{sym}A.
\]
This complex is isomorphic to the one from Thm.~\ref{thm.partial_resolution},
via the evaluation map 
\[
  k\big[\mathrm{Mor}_{\Delta S}(-, [p])\big]
  \otimes_{\Delta S} B_*^{sym}A \stackrel{\cong}{\to} B_p^{sym}A.
\]

\section{Low-degree Computations of $HS_*(A)$}\label{sec.low-dim-comp} %

\begin{theorem}\label{thm.HS_0}
  For a unital associative algebra $A$ over commutative ground ring $k$,
  \[
    HS_0(A) \cong A/([A,A]),
  \]
  where $([A,A])$ is the ideal generated by the commutator submodule $[A,A]$.
\end{theorem}
\begin{proof}
  By Thm.~\ref{thm.partial_resolution}, $HS_0(A) \cong A/k\big[\{abc-cba\}\big]$
  as $k$-module.  But $k\big[\{abc-cba\}\big]$ is an ideal of $A$, since if
  $x \in A$, then
  \[
    xabc - xbca = (x)(a)(bc) - (bc)(a)(x) + (bca)(x) - (x)(bca) \in 
    k\big[\{abc-cba\}\big].
  \]
  Clearly $([A,A]) \subset k\big[\{abc-cba\}\big]$, and
  $k\big[\{abc-cba\}\big] \subset ([A,A])$ since
  \[
    abc - cba = a(bc-cb) + a(cb) - (cb)a.
  \]
\end{proof}
\begin{cor}
  If $A$ is commutative, then $HS_0(A) \cong A$.
\end{cor}

Note that Theorem~\ref{thm.HS_0} implies that symmetric homology does not
preserve Morita equivalence, since for $n>1$,
\[
  HS_0\left(M_n(A)\right) = M_n(A)/\left([M_n(A),M_n(A)]\right) = 0.
\]
This implies $HS_*\left(M_n(A)\right) = 0$, since for any
$x \in HS_q\left(M_n(A)\right)$, $x = 1 \cdot x = 0 \cdot x = 0$, via the
Pontryagin product of Cor.~\ref{cor.pontryagin}, while in general
$HS_0(A) = A/([A,A]) \neq 0$.

By working with the complex in Thm.~\ref{thm.partial_resolution}, an explicit
formula for the product $HS_0(A) \otimes HS_1(A) \to HS_1(A)$ can be
determined.
\begin{prop}\label{prop.module-structure}
  For a unital associative algebra $A$ over commutative ground ring $k$,
  $HS_1(A)$ is a left $HS_0(A)$-module, via
  \[
    A/([A,A]) \otimes HS_1(A) \longrightarrow HS_1(A)
  \]
  \[
    [a] \otimes [x\otimes y \otimes z] \mapsto
    [ax \otimes y \otimes z] - [x \otimes ya \otimes z] + [x \otimes y \otimes az]
  \]
  Here, elements of $HS_1(A)$ are represented as equivalence classes of elements 
  in $A \otimes A \otimes A$,
  via the complex from Thm.~\ref{thm.partial_resolution}.)
  
  Moreover, there is a right module structure
  \[
    HS_1(A) \otimes HS_0(A) \longrightarrow HS_1(A)
  \]
  \[
    [x\otimes y \otimes z] \otimes [a] \mapsto
    [xa \otimes y \otimes z] - [x \otimes ay \otimes z] + [x \otimes y \otimes za],
  \]
  and the two actions are equal.  
\end{prop}
\begin{proof}
  Define the products on the chain level.
  For $a \in A$ and $x \otimes y \otimes z \in A \otimes A \otimes A$, put
  \begin{equation}\label{eq.left-HS0-action-on-HS1}
    a . (x \otimes y \otimes z) := ax \otimes y \otimes z
                                   - x \otimes ya \otimes z
                                   + x \otimes y \otimes az
  \end{equation}
  \begin{equation}\label{eq.right-HS0-action-on-HS1}
    (x \otimes y \otimes z).a := xa \otimes y \otimes z
                                   - x \otimes ay \otimes z
                                   + x \otimes y \otimes za
  \end{equation}

  There are a few details to verify:
  
  1.  Formula~(\ref{eq.left-HS0-action-on-HS1}) gives an action
  \[
    A \otimes HS_1(A) \to HS_1(A)
  \]
  The product is unital, since
  \[
    1. [x\otimes y \otimes z] =
    [x \otimes y \otimes z] - [x \otimes y \otimes z] + [x \otimes y \otimes z]
    = [x \otimes y \otimes z].
  \]
  
  Similarly, Formula~(\ref{eq.right-HS0-action-on-HS1}) is a unital product.
  
  Associativity follows by examining the following expression (on the level of chains):
  \[
    (ab).(x\otimes y \otimes z) - a.\left(b.(x\otimes y \otimes z)\right)
  \]
  \[
    = (abx \otimes y \otimes z - x \otimes yab \otimes z + x \otimes y \otimes abz)
    - \left( \begin{array}{c}
         abx \otimes y \otimes z - bx \otimes ya \otimes z + bx \otimes y \otimes az\\
         -ax \otimes yb \otimes z + x \otimes yba \otimes z - x \otimes yb \otimes az\\
         +ax \otimes y \otimes bz - x \otimes ya \otimes bz + x \otimes y \otimes abz
           \end{array}\right)
  \]
  \[
    = - x \otimes yab \otimes z + bx \otimes ya \otimes z - bx \otimes y \otimes az
    +ax \otimes yb \otimes z - x \otimes yba \otimes z 
  \]
  \begin{equation}\label{eq.associator}
    \qquad\qquad+ x \otimes yb \otimes az
    -ax \otimes y \otimes bz + x \otimes ya \otimes bz.
  \end{equation}
  We may view the variables $a$, $b$, $x$, $y$ and $z$ 
  in Eq.~(\ref{eq.associator}) as formal variables,
  and hence the expression itself may be regarded as chain in
  $k\left[ \mathrm{Mor}_{\Delta S}([4],[2]) \right]$ of the complex~(\ref{eq.part_res})
  Now, the differential $\rho$ of complex~(\ref{eq.part_res}) applied to 
  Eq.~(\ref{eq.associator}) results in:
  \begin{equation}\label{eq.d_associator}
    \left(\begin{array}{c}
    -xyabz + bxyaz - bxyaz
    +axybz - xybaz + xybaz
    -axybz + xyabz\\
    +zyabx - zyabx + azybx
    -zybax + zybax - azybx
    +bzyax - bzyax
    \end{array}\right).
  \end{equation}
  All terms cancel, showing that Eq.~(\ref{eq.associator}) is in the kernel of $\rho$.
  Exactness of complex~(\ref{eq.part_res}) implies that Eq.~(\ref{eq.associator})
  is in the image of $(\alpha, \beta)$.  Thus, there is a $2$-chain
  $C \in A^{\otimes 4} \oplus A$
  such that $\partial_2(C) = (ab).(x\otimes y \otimes z) - 
  a.\left(b.(x\otimes y \otimes z)\right)$, and on the level of homology,
  \[
    (ab).[x\otimes y \otimes z] = a.\left(b.[x\otimes y \otimes z]\right)
  \]
  
  The associativity of Formula~(\ref{eq.right-HS0-action-on-HS1}) is proven in
  the same way.
  
  2. Formula~(\ref{eq.left-HS0-action-on-HS1}) induces an action
  \[
    HS_0(A) \otimes HS_1(A) \to HS_1(A)
  \]
  It is sufficient to show that if $u$ is a $1$-cycle, then $(ab-ba).u$
  is a boundary for any $a, b \in A$.  Consider the following element.
  \[
    w := (ab - ba).(x \otimes y \otimes z) - (a \otimes b \otimes 1).
    \partial_1(x \otimes y \otimes z)
    =(ab - ba).(x \otimes y \otimes z) - (a \otimes b \otimes 1).(xyz-zyx)
  \]
  After expanding and canceling a pair of terms,
  \[
    w = \left(\begin{array}{c}
      abx \otimes y \otimes z - x \otimes yab \otimes z + x \otimes y \otimes abz\\
      - bax \otimes y \otimes z + x \otimes yba \otimes z - x \otimes y \otimes baz\\
      -axyz \otimes b \otimes 1 + a \otimes xyzb \otimes 1
      +azyx \otimes b \otimes 1 - a \otimes zyxb \otimes 1
      \end{array}\right).
  \]
  Consider all variables as formal, so $w \in 
  k\left[ \mathrm{Mor}_{\Delta S}([4],[2]) \right]$.  A routine verification shows
  that $\rho(w) = 0$, hence by exactness, $w$ is a boundary.  Now, any $1$-cycle
  $u$ is a $k$-linear combination of elements of the form $x \otimes y \otimes z$,
  so the chain
  \[
    (ab - ba).u - (a \otimes b \otimes 1)\partial_1(u) = (ab- ba).u
  \]
  is a boundary. 
  
  We show that Formula~(\ref{eq.right-HS0-action-on-HS1}) induces an action
  \[
    HS_0(A) \otimes HS_1(A) \to HS_1(A)
  \]
  in the analogous way, using
  \[
    v := (x \otimes y \otimes z).(ab - ba) - 
    \partial_1(x \otimes y \otimes z).(a \otimes b \otimes 1)
  \]
  
  Lastly, we prove that the product structures are equal.
  Consider the following element.
  \[
    t := a.(x \otimes y \otimes z) - (x \otimes y \otimes z).a -
    a \otimes \partial_1(x \otimes y \otimes z) \otimes 1
  \]
  \[
    = a.(x \otimes y \otimes z) - (x \otimes y \otimes z).a -
    a \otimes xyz \otimes 1 + a \otimes zyx \otimes 1
  \]
  It can be verified that $\rho(t) = 0$, proving $t \in
  k\left[ \mathrm{Mor}_{\Delta S}([3],[2]) \right]$ is a boundary.  If $u$ is a
  $1$-cycle, then
  \[
    a.u - u.a - a \otimes \partial_1(u) \otimes 1 = a.u - u.a
  \]
  is a boundary, which shows that $a.[u] = [u].a$ in $HS_1(A)$.
\end{proof}

Note, we expect the product structure given above to agree with the Pontryagin
product, but this has not been verified yet due to time constraints.

Using \verb|GAP|, we have made the following explicit computations of degree 1
integral symmetric homology.  The $HS_0(A)$-module structure is also displayed.

\begin{center}
\begin{tabular}{l|l|l}
$A$ & $HS_1(A \;|\; \Z)$ & $HS_0(A)$-module structure\\
\hline
$\Z[t]/(t^2)$ & $\Z/2\Z \oplus \Z/2\Z$ & Generated by $u$ with $2u=0$\\
$\Z[t]/(t^3)$ & $\Z/2\Z \oplus \Z/2\Z$ & Generated by $u$ with $2u=0$ and $t^2u=0$\\
$\Z[t]/(t^4)$ & $(\Z/2\Z)^4$ & Generated by $u$ with $2u=0$\\
$\Z[t]/(t^5)$ & $(\Z/2\Z)^4$ & \\
$\Z[t]/(t^6)$ & $(\Z/2\Z)^6$ & \\
\hline
$\Z[C_2]$ & $\Z/2\Z \oplus \Z/2\Z$ & Generated by $u$ with $2u=0$\\
$\Z[C_3]$ & $0$ & \\
$\Z[C_4]$ & $(\Z/2\Z)^4$ & Generated by $u$ with $2u=0$\\
$\Z[C_5]$ & $0$ & \\
$\Z[C_6]$ & $(\Z/2\Z)^6$ & \\
\hline
\end{tabular}
\end{center}

Based on these calculations, we conjecture:
\begin{conj}
  \[
    HS_1\big(k[t]/(t^n)\big) = \left\{\begin{array}{ll}
                                 (k/2k)^n, & \textrm{if $n \geq 0$ is even.}\\
                                 (k/2k)^{n-1} & \textrm{if $n \geq 1$ is odd.}
                               \end{array}\right.
  \]
\end{conj}

\begin{rmk}
  The computations of $HS_1\big(\Z[C_n]\big)$ are consistent with those of
  Brown and Loday~\cite{BL}.  See section~\ref{sec.2-torsion} for a more detailed 
  treatment of $HS_1$ for group rings.
\end{rmk}
  
Additionally, $HS_1$ has been computed for the following examples.  These
computations were done using \verb|GAP| in some cases and in others,
\verb|Fermat|~\cite{LEW} computations on sparse matrices
were used in conjunction with the \verb|GAP| scripts. ({\it e.g.} when the 
algebra has dimension greater than $6$ over $\Z$).

\begin{center}
\begin{tabular}{l|l}
$A$ & $HS_1(A \;|\; \Z)$\\
\hline
$\Z[t,u]/(t^2, u^2)$ & $\Z \oplus (\Z/2\Z)^{11}$\\
$\Z[t,u]/(t^3, u^2)$ & $\Z^2 \oplus (\Z/2\Z)^{11} \oplus \Z/6\Z$\\
$\Z[t,u]/(t^3, u^2, t^2u)$ & $\Z^2 \oplus (\Z/2\Z)^{10}$\\
$\Z[t,u]/(t^3, u^3)$ & $\Z^4 \oplus (\Z/2\Z)^7 \oplus (\Z/6\Z)^5$\\
$\Z[t,u]/(t^2, u^4)$ & $\Z^3 \oplus (\Z/2\Z)^{20} \oplus \Z/4\Z$\\
$\Z[t,u,v]/(t^2, u^2, v^2)$ & $\Z^6 \oplus (\Z/2\Z)^{42}$\\
$\Z[t,u]/(t^4, u^3)$ & $\Z^6 \oplus (\Z/2\Z)^{19} \oplus \Z/6\Z \oplus (\Z/12\Z)^2$\\
$\Z[t,u,v]/(t^2, u^2, v^3)$ & $\Z^{11} \oplus (\Z/2\Z)^{45} \oplus (\Z/6\Z)^4$\\
$\Z[i,j,k], i^2=j^2=k^2=ijk=-1$ & $(\Z/2\Z)^8$\\
$\Z[C_2 \times C_2]$ & $(\Z/2\Z)^{12}$\\
$\Z[C_3 \times C_2]$ & $(\Z/2\Z)^{6}$\\
$\Z[C_3 \times C_3]$ & $(\Z/3\Z)^{9}$\\
$\Z[S_3]$ & $(\Z/2\Z)^2$\\
\hline
\end{tabular}
\end{center}

\section{Splittings of the Partial Resolution}\label{sec.splittings}        %

Under certain circumstances, the partial complex in Thm.\ref{thm.partial_resolution}
splits as a direct sum of smaller complexes.  This observation becomes increasingly
important as the dimension of the algebra increases.  Indeed, some of the computations
of the previous section were done using splittings.

\begin{definition}
  For a commutative $k$-algebra $A$ and $u \in A$, define the $k$-modules:
  \[
    \big(A^{\otimes n}\big)_u := \{ a_1 \otimes a_2 \otimes \ldots \otimes a_n
    \in A^{\otimes n} \;|\; a_1a_2\cdot \ldots \cdot a_n = u \}
  \]
\end{definition}

\begin{prop}
  If $A = k[M]$ for a commutative monoid $M$, then the complex in 
  Thm.\ref{thm.partial_resolution}
  splits as a direct sum of complexes
  \begin{equation}\label{eq.u-homology}
    0\longleftarrow (A)_u \stackrel{\partial_1}{\longleftarrow} 
    (A\otimes A\otimes A)_u
    \stackrel{\partial_2}{\longleftarrow}
    (A\otimes A\otimes A\otimes A)_u\oplus (A)_u,
  \end{equation}
  where $u$ ranges over the elements of $M$.  For each $u$, the
  homology groups of Eq.~(\ref{eq.u-homology}) will be called the $u$-layered
  symmetric homology of $A$, denoted $HS_i(A)_u$.
  Thus, for $i = 0,1$, we have:
  \[
    HS_i(A) \cong \bigoplus_{u \in M} HS_i(A)_u.
  \]  
\end{prop}
\begin{proof}
  Since $M$ is a commutative monoid, there are direct sum decompositions as
  \mbox{$k$-module}:
  \[
    A^{\otimes n} = \bigoplus_{u \in M} \big( A^{\otimes n}\big)_u.
  \]
  The maps $\partial_1$ and $\partial_2$ preserve the products of tensor factors,
  so the inclusions
  $\big(A^{\otimes n}\big)_u \hookrightarrow A^{\otimes n}$ induce maps of complexes,
  hence the complex itself splits as a direct sum.
\end{proof}

We may use layers to investigate the symmetric homology of $k[t]$.  This algebra
is monoidal, generated by the monoid $\{1, t, t^2, t^3, \ldots \}$.  Now, the $t^m$-layer
symmetric homology of $k[t]$ will be the same as the $t^m$-layer symmetric homology
of $k[M^{m+2}_{m+1}]$, where $M^p_q$ denotes the cyclic monoid generated by a variable
$s$ with the property that $s^p = s^q$.  Using this observation and subsequent
computation, we conjecture:
\begin{conj}\label{conj.HS_1freemonoid}
  \[
    HS_1\big(k[t]\big)_{t^m} = \left\{\begin{array}{ll}
                                 0 & m = 0, 1\\
                                 k/2k, & m \geq 2\\
                               \end{array}\right.
  \]
\end{conj}
This conjecture has been verified up to $m = 18$, in the case $k = \Z$.

\section{2-torsion in $HS_1$}\label{sec.2-torsion}         %

The occurrence of 2-torsion in $HS_1(A)$ for the examples considered 
in sections~\ref{sec.low-dim-comp} and \ref{sec.splittings} comes as no surprise,
based on Thm.~\ref{thm.HS_group}.  First consider the following chain
of isomorphisms:
\[
  \pi_{2}^s(B\Gamma) = \pi_2\big(\Omega^{\infty}S^{\infty}(B\Gamma)\big)
  \cong \pi_{1}\big(\Omega\Omega^{\infty}S^{\infty}(B\Gamma)\big) 
\]
\[
  \cong \pi_{1}\big(\Omega_0\Omega^{\infty}S^{\infty}(B\Gamma)\big) 
  \stackrel{h}{\to} H_1\big(\Omega_0\Omega^{\infty}S^{\infty}(B\Gamma)\big).
\]
Here, $\Omega_0\Omega^{\infty}S^{\infty}(B\Gamma)$ denotes the component of the
constant loop, and $h$ is the Hurewicz homomorphism,
which is an isomorphism since
$\Omega_0\Omega^{\infty}S^{\infty}(B\Gamma)$ is path-connected and $\pi_1$ is
abelian.  On the other hand, by Thm.~\ref{thm.HS_group},
\[
  HS_1(k[\Gamma]) \cong H_1\big(\Omega\Omega^{\infty}S^{\infty}(B\Gamma); k\big)
  \cong H_1\big(\Omega\Omega^{\infty}S^{\infty}(B\Gamma)\big) \otimes k.
\]
(All tensor products will be over
$\mathbb{Z}$ in this section.)  Now $\Omega\Omega^{\infty}S^{\infty}(B\Gamma)$
consists of isomorphic copies of $\Omega_0\Omega^{\infty}S^{\infty}(B\Gamma)$, one for
each element of $\Gamma/[\Gamma, \Gamma]$, so we may write
\[
  H_1\big(\Omega\Omega^{\infty}S^{\infty}(B\Gamma)\big) \otimes k
  \cong
  H_1\big(\Omega_0\Omega^{\infty}S^{\infty}(B\Gamma)\big) \otimes
  k\big[ \Gamma/[\Gamma, \Gamma] \big].
\]
Thus, we obtain the result:
\begin{cor}\label{cor.stablegrouphomotopy}
  If $\Gamma$ is a group, then
  \[
    HS_1(k[\Gamma]) \cong \pi_{2}^s(B\Gamma) \otimes 
    k\big[ \Gamma/[\Gamma, \Gamma] \big].
  \]
\end{cor}

Now, by results of Brown and Loday~\cite{BL}, if $\Gamma$ is abelian, then
$\pi_2^s(B\Gamma)$ is the reduced tensor square.  That is,
\[
  \pi_2^s(B\Gamma) = \Gamma\, \widetilde{\wedge} \,\Gamma = 
  \left(\Gamma \otimes \Gamma\right)/\approx,
\]
where $g \otimes h \approx - h \otimes g$ for all $g, h \in \Gamma$.  (This construction
is notated with multiplicative group action in~\cite{BL}, since they deal with the
more general case of non-abelian groups.) So in particular,
if $\Gamma = C_n$, the cyclic group of order $n$, we have
\[
  \pi_2^s(BC_n) = \left\{\begin{array}{ll}
                       \Z/2\Z & \textrm{$n$ even.}\\
                       0 & \textrm{$n$ odd.}
                  \end{array}\right.
\]
\begin{cor}\label{cor.HS_1-C_n}
  \[
    HS_1(k[C_n]) = \left\{\begin{array}{ll}
                       (\Z/2\Z)^n & \textrm{$n$ even.}\\
                       0 & \textrm{$n$ odd.}
                  \end{array}\right.
  \]
\end{cor}
\begin{proof}
  The result follows from Cor.~\ref{cor.stablegrouphomotopy}, as
  $k\big[ C_n/[C_n, C_n] \big] \cong k[C_n] \cong k^n$, as $k$-module.
\end{proof}


\section{Relations to Cyclic Homology}\label{sec.cyc-homology}  %

The relation between the symmetric bar construction and the cyclic bar
construction arising from the chain of inclusions~(\ref{eq.delta-C-S-chain})
gives rise to a natural map
\begin{equation}\label{eq.HCtoHS}
  HC_*(A) \to HS_*(A)
\end{equation}
Indeed, by remark~\ref{rmk.HC},
we may define cyclic homology via:
\[
  HC_*(A) = \mathrm{Tor}_*^{\Delta C}( \underline{k}, B_*^{sym}A ).
\]

Using the partial complex of Thm.~\ref{thm.partial_resolution}, and an
analogous one for computing cyclic homology (c.f.~\cite{L}, p. 59), the
map~(\ref{eq.HCtoHS}) for degrees $0$ and $1$ is induced by the following
partial chain map:
\[
  \begin{diagram}
    \node{ 0 }
    \node{ A }
    \arrow{w}
    \arrow{s,r}{ \gamma_0 = \mathrm{id} }
    \node{ A \otimes A }
    \arrow{w,t}{ \partial_1^C }
    \arrow{s,r}{ \gamma_1 }
    \node{ A^{\otimes 3} \oplus A }
    \arrow{w,t}{ \partial_2^C }
    \arrow{s,r}{ \gamma_2 }
    \\
    \node{ 0 }
    \node{ A }
    \arrow{w}
    \node{ A^{\otimes 3} }
    \arrow{w,t}{ \partial_1^S }
    \node{ A^{\otimes 4} \oplus A }
    \arrow{w,t}{ \partial_2^S }
  \end{diagram}
\] 
In this diagram, the boundary maps in the upper row are defined as follows:
\[
  \partial_1^C : a \otimes b \mapsto ab - ba
\]
\[
  \partial_2^C : \left\{\begin{array}{ll}
                   a \otimes b \otimes c &\mapsto ab \otimes c - a \otimes bc + ca \otimes b\\
                   a &\mapsto 1 \otimes a - a \otimes 1
                 \end{array}\right.
\]
The boundary maps in the lower row are defined as in Thm.~\ref{thm.partial_resolution}.
\[
  \partial_1^S : a \otimes b \otimes c \mapsto abc - cba
\]
\[
  \partial_2^S : \left\{\begin{array}{ll}
                   a \otimes b \otimes c \otimes d &\mapsto ab \otimes c \otimes d 
                     - d \otimes ca \otimes b + bca \otimes 1 \otimes d
                     + d \otimes bc \otimes a\\
                   a &\mapsto 1 \otimes a \otimes 1
                 \end{array}\right.
\]
The partial chain map is given in degree 1 by $\gamma_1(a \otimes b) := a \otimes b 
\otimes 1$.  In degree 2, $\gamma_2$ is defined on the summand $A^{\otimes 3}$ via
\[
  a \otimes b \otimes c \mapsto (a \otimes b \otimes c \otimes 1
                     - 1 \otimes a \otimes bc \otimes 1 
                     + 1 \otimes ca \otimes b \otimes 1
                     + 1 \otimes 1 \otimes abc \otimes 1
                     - b \otimes ca \otimes 1 \otimes 1)
                     - 2abc - cab,
\]
and on the summand $A$ via
\[
  a \mapsto (-1 \otimes 1 \otimes a \otimes 1) + (4a).
\]

\backmatter

\appendix

\chapter{SCRIPTS USED FOR COMPUTER CALCULATIONS}\label{app.comp}

The computer algebra systems \verb+GAP+, \verb+Octave+ and \verb+Fermat+
were used to verify proposed theorems and also to obtain some concrete
computations of symmetric homology for some small algebras.  Here is a link to 
a tar-file of the scripts that were created in the course of writing this
dissertation as well as the \LaTeX source of this dissertation.

\begin{center}
  \url{http://arxiv.org/e-print/0807.4521v1}
\end{center}

The tar-file contains the following files:
\begin{itemize}
  \item \verb+Basic.g+ \quad - Some elementary functions, necessary for some 
  functions in \verb+DeltaS.g+
  
  \item \verb+HomAlg.g+ \quad - Homological Algebra functions, such as computation
  of homology groups for chain complexes.
  
  \item \verb+Fermat.g+ \quad - Functions necessary to invoke \verb+Fermat+ 
  for fast sparse matrix computations.
  
  \item \verb+fermattogap+, \verb+gaptofermat+ \quad - Auxiliary text files for
  use when invoking \verb+Fermat+ from \verb+GAP+.

  \item \verb+DeltaS.g+ \quad - This is the main repository of scripts used to
  compute various quantities associated with the category $\Delta S$, including
  $HS_1(A)$ for finite-dimensional algebras $A$.
\end{itemize}
  
In order to use the functions of \verb+DeltaS.g+, simply copy the above files into
the working directory (such as \verb+~/gap/+), invoke \verb+GAP+, then read
in \verb+DeltaS.g+ at the prompt.  The dependent modules will automatically be
loaded (hence they must be present in the same directory as \verb+DeltaS.g+).
Note, most of the computations involving homology require substantial memory
to run.  I recommend calling \verb+GAP+ with the command line option 
``\verb+-o +{\it mem}'', where {\it mem} is the amount of memory to be allocated
to this instance of \verb+GAP+.  All computations done in this dissertation can
be accomplished by allocating 20 gigabytes of memory.  The following provides a
few examples of using the functions of \verb+DeltaS.g+

\begin{verbatim}
[ault@math gap]$ gap -o 20g

gap> Read("DeltaS.g");
gap> 
gap> ## Number of morphisms [6] --> [4]
gap> SizeDeltaS( 6, 4 );
1663200
gap> 
gap> ## Generate the set of morphisms of Delta S, [2] --> [2]
gap> EnumerateDeltaS( 2, 2 );
[ [ [ 0, 1, 2 ], [  ], [  ] ], [ [ 0, 2, 1 ], [  ], [  ] ], 
  [ [ 1, 0, 2 ], [  ], [  ] ], [ [ 1, 2, 0 ], [  ], [  ] ], 
  [ [ 2, 0, 1 ], [  ], [  ] ], [ [ 2, 1, 0 ], [  ], [  ] ], 
  [ [ 0, 1 ], [ 2 ], [  ] ], [ [ 0, 2 ], [ 1 ], [  ] ], 
  [ [ 1, 0 ], [ 2 ], [  ] ], [ [ 1, 2 ], [ 0 ], [  ] ], 
  [ [ 2, 0 ], [ 1 ], [  ] ], [ [ 2, 1 ], [ 0 ], [  ] ], 
  [ [ 0, 1 ], [  ], [ 2 ] ], [ [ 0, 2 ], [  ], [ 1 ] ], 
  [ [ 1, 0 ], [  ], [ 2 ] ], [ [ 1, 2 ], [  ], [ 0 ] ], 
  [ [ 2, 0 ], [  ], [ 1 ] ], [ [ 2, 1 ], [  ], [ 0 ] ], 
  [ [ 0 ], [ 1, 2 ], [  ] ], [ [ 0 ], [ 2, 1 ], [  ] ], 
  [ [ 1 ], [ 0, 2 ], [  ] ], [ [ 1 ], [ 2, 0 ], [  ] ], 
  [ [ 2 ], [ 0, 1 ], [  ] ], [ [ 2 ], [ 1, 0 ], [  ] ], 
  [ [ 0 ], [ 1 ], [ 2 ] ], [ [ 0 ], [ 2 ], [ 1 ] ], [ [ 1 ], [ 0 ], [ 2 ] ], 
  [ [ 1 ], [ 2 ], [ 0 ] ], [ [ 2 ], [ 0 ], [ 1 ] ], [ [ 2 ], [ 1 ], [ 0 ] ], 
  [ [ 0 ], [  ], [ 1, 2 ] ], [ [ 0 ], [  ], [ 2, 1 ] ], 
  [ [ 1 ], [  ], [ 0, 2 ] ], [ [ 1 ], [  ], [ 2, 0 ] ], 
  [ [ 2 ], [  ], [ 0, 1 ] ], [ [ 2 ], [  ], [ 1, 0 ] ], 
  [ [  ], [ 0, 1, 2 ], [  ] ], [ [  ], [ 0, 2, 1 ], [  ] ], 
  [ [  ], [ 1, 0, 2 ], [  ] ], [ [  ], [ 1, 2, 0 ], [  ] ], 
  [ [  ], [ 2, 0, 1 ], [  ] ], [ [  ], [ 2, 1, 0 ], [  ] ], 
  [ [  ], [ 0, 1 ], [ 2 ] ], [ [  ], [ 0, 2 ], [ 1 ] ], 
  [ [  ], [ 1, 0 ], [ 2 ] ], [ [  ], [ 1, 2 ], [ 0 ] ], 
  [ [  ], [ 2, 0 ], [ 1 ] ], [ [  ], [ 2, 1 ], [ 0 ] ], 
  [ [  ], [ 0 ], [ 1, 2 ] ], [ [  ], [ 0 ], [ 2, 1 ] ], 
  [ [  ], [ 1 ], [ 0, 2 ] ], [ [  ], [ 1 ], [ 2, 0 ] ], 
  [ [  ], [ 2 ], [ 0, 1 ] ], [ [  ], [ 2 ], [ 1, 0 ] ], 
  [ [  ], [  ], [ 0, 1, 2 ] ], [ [  ], [  ], [ 0, 2, 1 ] ], 
  [ [  ], [  ], [ 1, 0, 2 ] ], [ [  ], [  ], [ 1, 2, 0 ] ], 
  [ [  ], [  ], [ 2, 0, 1 ] ], [ [  ], [  ], [ 2, 1, 0 ] ] ]
gap> 
gap> ## Generate only the epimorphisms [2] -->> [2]
gap> EnumerateDeltaS( 2, 2 : epi ); 
[ [ [ 0 ], [ 1 ], [ 2 ] ], [ [ 0 ], [ 2 ], [ 1 ] ], 
  [ [ 1 ], [ 0 ], [ 2 ] ], [ [ 1 ], [ 2 ], [ 0 ] ], 
  [ [ 2 ], [ 0 ], [ 1 ] ], [ [ 2 ], [ 1 ], [ 0 ] ] ]
gap> 
gap> ## Compose two morphisms of Delta S.
gap> a := Random(EnumerateDeltaS(4,3));
[ [ 0 ], [ 2, 4, 1 ], [  ], [ 3 ] ]
gap> b := Random(EnumerateDeltaS(3,2));
[ [  ], [ 3, 0, 2 ], [ 1 ] ]
gap> MultDeltaS(b, a);
[ [  ], [ 3, 0 ], [ 2, 4, 1 ] ]
gap> MultDeltaS(a, b);
Maps incomposeable
[  ]
gap> 
gap> ## Examples of using morphisms of Delta S to act on simple tensors
gap> A := TruncPolyAlg([3,2]);
<algebra of dimension 6 over Rationals>
gap> ## TruncPolyAlg is defined in Basic.g
gap> ##  TruncPolyAlg([i_1, i_2, ..., i_n]) is generated by
gap> ##  x_1, x_2, ..., x_n, under the relation (x_j)^(i_j) = 0.
gap> g := GeneratorsOfLeftModule(A);
[ X^[ 0, 0 ], X^[ 0, 1 ], X^[ 1, 0 ], X^[ 1, 1 ], X^[ 2, 0 ], X^[ 2, 1 ] ]
gap> x := g[2]; y := g[3];
X^[ 0, 1 ]
X^[ 1, 0 ]
gap> v := [ x*y, 1, y^2 ];
gap> ## v represents the simple tensor  xy \otimes 1 \otimes y^2.
[ X^[ 1, 1 ], 1, X^[ 2, 0 ] ]
gap> ActByDeltaS( v, [[2], [], [0], [1]] );
[ X^[ 2, 0 ], 1, X^[ 1, 1 ], 1 ]
gap> ActByDeltaS( v, [[2], [0,1]] );       
[ X^[ 2, 0 ], X^[ 1, 1 ] ]
gap> ActByDeltaS( v, [[2,0], [1]] );  
[ 0*X^[ 0, 0 ], 1 ]
gap> 
gap> ## Symmetric monoidal product on DeltaS_+
gap> a := Random(EnumerateDeltaS(4,2));
[ [  ], [ 2, 1, 0 ], [ 3, 4 ] ]
gap> b := Random(EnumerateDeltaS(3,3));
[ [  ], [  ], [  ], [ 1, 3, 2, 0 ] ]
gap> MonoidProductDeltaS(a, b);
[ [  ], [ 2, 1, 0 ], [ 3, 4 ], [  ], [  ], [  ], [ 6, 8, 7, 5 ] ]
gap> MonoidProductDeltaS(b, a);
[ [  ], [  ], [  ], [ 1, 3, 2, 0 ], [  ], [ 6, 5, 4 ], [ 7, 8 ] ]
gap> MonoidProductDeltaS(a, []);
[ [  ], [ 2, 1, 0 ], [ 3, 4 ] ]
gap> 
gap> ## Symmetric Homology of the algebra A, in degrees 0 and 1.
gap> SymHomUnitalAlg(A);
[ [ 0, 0, 0, 0, 0, 0 ], [ 2, 2, 2, 2, 2, 2, 2, 2, 2, 2, 2, 6, 0, 0 ] ]
gap> ## '0' represents a factor of Z, while a non-zero p represents
gap> ## a factor of Z/pZ.
gap> 
gap> ## Using layers to compute symmetric homology
gap> C2 := CyclicGroup(2);                                       
<pc group of size 2 with 1 generators>
gap> A := GroupRing(Rationals, DirectProduct(C2, C2));
<algebra-with-one over Rationals, with 2 generators>
gap> ## First, a direct computation without layers:
gap> SymHomUnitalAlg(A);
[ [ 0, 0, 0, 0 ], [ 2, 2, 2, 2, 2, 2, 2, 2, 2, 2, 2, 2 ] ]
gap> ## Next, compute HS_0(A)_u and HS_1(A)_u for each generator u.
gap> g := GeneratorsOfLeftModule(A);
[ (1)*<identity> of ..., (1)*f2, (1)*f1, (1)*f1*f2 ]
gap> SymHomUnitalAlgLayered(A, g[1]);
[ [ 0 ], [ 2, 2, 2 ] ]
gap> SymHomUnitalAlgLayered(A, g[2]);
[ [ 0 ], [ 2, 2, 2 ] ]
gap> SymHomUnitalAlgLayered(A, g[3]);
[ [ 0 ], [ 2, 2, 2 ] ]
gap> SymHomUnitalAlgLayered(A, g[4]);
[ [ 0 ], [ 2, 2, 2 ] ]
gap> ## Computing HS_1( Z[t] ) by layers:
gap> SymHomFreeMonoid(0,10);
HS_1(k[t])_{t^0} :  [  ]
HS_1(k[t])_{t^1} :  [  ]
HS_1(k[t])_{t^2} :  [ 2 ]
HS_1(k[t])_{t^3} :  [ 2 ]
HS_1(k[t])_{t^4} :  [ 2 ]
HS_1(k[t])_{t^5} :  [ 2 ]
HS_1(k[t])_{t^6} :  [ 2 ]
HS_1(k[t])_{t^7} :  [ 2 ]
HS_1(k[t])_{t^8} :  [ 2 ]
HS_1(k[t])_{t^9} :  [ 2 ]
HS_1(k[t])_{t^10} :  [ 2 ]
gap> ## Poincare polynomial of Sym_*^{(p)} for small p.
gap> ##  There is a check for torsion, using a call to Fermat
gap> ##  to find Smith Normal Form of the differential matrices.
gap> PoincarePolynomialSymComplex(2);
C_0  Dimension: 1
C_1  Dimension: 6
C_2  Dimension: 6
D_1
SNF(D_1)
D_2
SNF(D_2)
2*t^2+t
gap> PoincarePolynomialSymComplex(5);
C_0  Dimension: 1
C_1  Dimension: 30
C_2  Dimension: 300
C_3  Dimension: 1200
C_4  Dimension: 1800
C_5  Dimension: 720
D_1
SNF(D_1)
D_2
SNF(D_2)
D_3
SNF(D_3)
D_4
SNF(D_4)
D_5
SNF(D_5)
120*t^5+272*t^4+t^3
\end{verbatim}

\bibliographystyle{plain}
\bibliography{refs}

\end{document}